\documentclass[noshowframe]{imsart}

\RequirePackage{amsthm,amsmath,amsfonts,amssymb,dsfont}
\RequirePackage[numbers]{natbib}
\RequirePackage[colorlinks,citecolor=blue,urlcolor=blue]{hyperref}
\RequirePackage{graphicx,float}
\RequirePackage{booktabs}
\RequirePackage{comment}

\startlocaldefs
\numberwithin{equation}{section}
\theoremstyle{plain}

\newtheorem{theorem}{Theorem}[section]
\newtheorem{lemma}[theorem]{Lemma}
\newtheorem{proposition}[theorem]{Proposition}
\newtheorem{corollary}[theorem]{Corollary}

\theoremstyle{definition}
\newtheorem{definition}[theorem]{Definition}
\newtheorem{property}[theorem]{Property}
\newtheorem{assumption}[theorem]{Assumption}
\newtheorem{remark}[theorem]{Remark}

\theoremstyle{remark}

\def \A {\mathbf{A}}
\def \a {\mathbf{a}}
\def \B {\mathbf{B}}
\def \b {\mathbf{b}}
\def \C {\mathbf{C}}
\def \D {\mathbf{D}}
\def \d {\mathbf{d}}
\def \E {\mathbf{E}}
\def \e {\mathbf{e}}
\def \F {\mathbf{F}}
\def \f {\mathbf{f}}
\def \g {\mathbf{g}}
\def \G {\mathbf{G}}
\def \g {\mathbf{g}}
\def \h {\mathbf{h}}
\def \H {\mathbf{H}}
\def \I {\mathbf{I}}
\def \i {\mathbf{i}}
\def \J {\mathbf{J}}
\def \j {\mathbf{j}}
\def \K {\mathbf{K}}
\def \k {\mathbf{k}}
\def \L {\mathbf{L}}
\def \l {\mathbf{l}}
\def \M {\mathbf{M}}
\def \m {\mathbf{m}}
\def \N {\mathbf{N}}
\def \n {\mathbf{n}}
\def \O {\mathbf{O}}
\def \o {\mathbf{o}}
\def \P {\mathbf{P}}
\def \p {\mathbf{p}}
\def \Q {\mathbf{Q}}
\def \q {\mathbf{q}}
\def \r {\mathbf{r}}
\def \R {\mathbf{R}}
\def \s {\mathbf{s}}
\def \S {\mathbf{S}}
\def \T {\mathbf{T}}
\def \t {\mathbf{t}}
\def \U {\mathbf{U}}
\def \u {\mathbf{u}}
\def \v {\mathbf{v}}
\def \V {\mathbf{V}}
\def \W {\mathbf{W}}
\def \w {\mathbf{w}}
\def \x {\mathbf{x}}
\def \X {\mathbf{X}}
\def \Y {\mathbf{Y}}
\def \y {\mathbf{y}}
\def \Z {\mathbf{Z}}
\def \z {\mathbf{z}}

\def \Acal {\mathcal{A}}
\def \Bcal {\mathcal{B}}
\def \Ccal {\mathcal{C}}
\def \Dcal {\mathcal{D}}
\def \Ecal {\mathcal{E}}
\def \Fcal {\mathcal{F}}
\def \Gcal {\mathcal{G}}
\def \Hcal {\mathcal{H}}
\def \Ical {\mathcal{I}}
\def \Jcal {\mathcal{J}}
\def \Kcal {\mathcal{K}}
\def \Lcal {\mathcal{L}}
\def \Mcal {\mathcal{M}}
\def \Ncal {\mathcal{N}}
\def \Ocal {\mathcal{O}}
\def \Pcal {\mathcal{P}}
\def \Qcal {\mathcal{Q}}
\def \Rcal {\mathcal{R}}
\def \Scal {\mathcal{S}}
\def \Tcal {\mathcal{T}}
\def \Ucal {\mathcal{U}}
\def \Vcal {\mathcal{V}}
\def \Wcal {\mathcal{W}}
\def \Xcal {\mathcal{X}}
\def \Ycal {\mathcal{Y}}
\def \Zcal {\mathcal{Z}}

\def \Cbb {\mathbb{C}}
\def \Ebb {\mathbb{E}}
\def \Fbb {\mathbb{F}}
\def \Hbb {\mathbb{H}}
\def \Ibb {\mathbb{I}}
\def \Nbb {\mathbb{N}}
\def \Pbb {\mathbb{P}}
\def \Rbb {\mathbb{R}}
\def \Sbb {\mathbb{S}}
\def \Tbb {\mathbb{T}}
\def \Vbb {\mathbb{V}}
\def \Zbb {\mathbb{Z}}

\def \drm {\mathrm{d}}
\def \erm {\mathrm{e}}
\def \irm {\mathrm{i}}
\def \Prm {\mathrm{P}}
\def \Qrm {\mathrm{Q}}

\def \bfrak {\mathfrak{b}}

\def \alphabs {\boldsymbol{\alpha}}
\def \deltabs {\boldsymbol{\delta}}
\def \gammabs {\boldsymbol{\gamma}}
\def \omegabs {\boldsymbol{\omega}}
\def \thetabs {\boldsymbol{\theta}}
\def \mubs {\boldsymbol{\mu}}
\def \xibs {\boldsymbol{\xi}}
\def \chibs {\boldsymbol{\chi}}
\def \epsilonbs {\boldsymbol{\epsilon}}

\def \Deltabs {\boldsymbol{\Delta}}
\def \Gammabs {\boldsymbol{\Gamma}}
\def \Lambdabs {\boldsymbol{\Lambda}}
\def \Omegabs {\boldsymbol{\Omega}}
\def \Pibs {\boldsymbol{\Pi}}
\def \Phibs {\boldsymbol{\Phi}}
\def \Psibs {\boldsymbol{\Psi}}
\def \Sigmabs {\boldsymbol{\Sigma}}
\def \Thetabs {\boldsymbol{\Theta}}
\def \Xibs {\boldsymbol{\Xi}}
\def \Upsilonbs {\boldsymbol{\Upsilon}}

\def \det {\mathrm{det} \ }
\def \Tr {\mathrm{Tr}\,}
\def \tr {\mathrm{tr}\,}
\def \Prob {\mathbb{P}}
\def \dist {\mathrm{dist}}
\def \card {\mathrm{card}}
\def \Res {\mathrm{Res}}
\def \Diag{\mathrm{Diag}}
\def \Vec{\mathrm{Vec}}

\DeclareMathOperator*{\dg}{dg}
\DeclareMathOperator{\Supp}{Supp}
\DeclareMathOperator{\Var}{Var}
\DeclareMathOperator*{\argmin}{argmin}
\DeclareMathOperator*{\argmax}{argmax}
\renewcommand{\Im}{\mathrm{Im}}
\renewcommand{\Re}{\mathrm{Re}}
\newcommand*\diff{\mathop{}\!\mathrm{d}}
\newcommand{\floor}[1]{\lfloor #1 \rfloor}
\renewcommand{\Im}{\mathrm{Im}}
\renewcommand{\Re}{\mathrm{Re}}
\newcommand{\supp}{\mathrm{supp}}
\newcommand{\Exp}{\mathbb{E}}
\newcommand{\rvcenter}[1]{\overset{\circ}{#1}}
\newcommand{\rvcentergroup}[1]{\overset{\circ}{\wideparen{#1}}}
\newcommand{\commentred}[1]{\textcolor{red}{\textbf{#1}}}
\newcommand{\commentblue}[1]{\textcolor{blue}{\textbf{#1}}}
\newcommand{\bs}{\boldsymbol}
\endlocaldefs

\begin{document}
\begin{frontmatter}
\title{Correlation tests and sample spectral coherence matrix in the high-dimensional regime}

\runtitle{Correlation tests and sample spectral coherence matrix in the high-dimensional regime}

\begin{aug}
\author[A]{\fnms{Philippe}~\snm{Loubaton}\ead[label=e1]{philippe.loubaton@univ-eiffel.fr}\orcid{https://orcid.org/0000-0002-5145-9319}},
\author[B]{\fnms{Alexis}~\snm{Rosuel}\ead[label=e2]{rosuelalexis1@gmail.com}\orcid{https://orcid.org/0000-0002-3038-926X}}
\and
\author[C]{\fnms{Pascal}~\snm{Vallet}\ead[label=e3]{pascal.vallet@bordeaux-inp.fr}\orcid{0000-0002-1636-0452}}
\address[A]{Laboratoire d’Informatique Gaspard Monge (CNRS, Univ. Gustave Eiffel)
5 Boulevard Descartes, 77454 Marne-la-Vallée (France) \printead[presep={,\ }]{e1}}

\address[B]{Laboratoire d’Informatique Gaspard Monge until 2021 (CNRS, Univ. Gustave Eiffel)
5 Boulevard Descartes, 77454 Marne-la-Vallée (France)\printead[presep={,\ }]{e2}}

\address[C]{Laboratoire de l’Int\'egration du Mat\'eriau au Syst\`eme (CNRS, Univ. Bordeaux, Bordeaux INP),
351 Cours de la Lib\'eration, 33400 Talence (France) \printead[presep={,\ }]{e3}}

\runauthor{P. Loubaton et al.}
\end{aug}

\begin{abstract}
   It is established that the linear spectral statistics (LSS) of the smoothed periodogram estimate of the spectral coherence matrix of a complex Gaussian high-dimensional times series $(\y_n)_{n \in \mathbb{Z}}$ with independent components satisfy at each frequency a central limit theorem (CLT) in the asymptotic regime where the sample size $N$, the dimension $M$ of the observation, and the smoothing span $B$ both converge towards $+\infty$ in such a way that $M = \Ocal(N^{\alpha})$ for $\alpha < 1$ and $\frac{M}{B} \rightarrow c$, $c \in (0,1)$. It is deduced that two recentered and renormalized versions of the LSS, one based on an average in the frequency domain and the other one based on a sum of squares also in the frequency domain, and both evaluated over a well-chosen frequency grid, also verify a central limit theorem. These two statistics are proposed to test with controlled asymptotic level the hypothesis that the components of $\y$ are independent. Numerical simulations assess the performance of the two tests. 
\end{abstract}

\begin{keyword}[class=MSC]
\kwd[Primary ]{60B20}
\kwd{62H15}
\kwd[; secondary ]{62M15}
\end{keyword}

\begin{keyword}
\kwd{High-dimensional time series, large random matrices , spectral coherence matrix}
\kwd{independence test}
\end{keyword}

\end{frontmatter}
\tableofcontents

\section{Introduction}
\subsection{The addressed problem and the results}

In this paper, we consider a zero-mean $M$-variate complex Gaussian stationary random sequence $(\y_n)_{n \in \mathbb{Z}}$, and denote 
by $(y_{1,n})_{n \in \mathbb{Z}}, \ldots, (y_{M,n})_{n \in \mathbb{Z}}$ the $M$ components of $\y$ defined by 
$\y_n = (y_{1,n}, \ldots, y_{M,n})^{T}$. We assume that the samples $\y_1, \ldots, \y_N$ are available and address the problem of testing in the frequency domain the 
hypothesis $\Hcal_0$ that the $M$ components of $\y$ are mutually independent time series 
in asymptotic regimes where $M$ and $N$ are both large
and the ratio $\frac{M}{N}$ is not supposed to be small enough to 
use conventional methods studied in the past in the low-dimensional asymptotic regime where $M$ is a fixed parameter and $N \rightarrow +\infty$. 
It is clear that $\Hcal_0$ holds if and only if the spectral density matrix $\S(\nu)$ of $\y$ is a diagonal matrix for each frequency $\nu$, or equivalently if the spectral coherence matrix $\C(\nu)$ defined by 
\begin{equation}
    \label{eq:def-C}
    \C(\nu) = \left( \mathrm{dg}(\S(\nu))\right)^{-1/2} \S(\nu)   \left( \mathrm{dg}(\S(\nu))\right)^{-1/2},
\end{equation}
verifies $\C(\nu) = \I_M$ for each $\nu \in [0,1]$, where for each $M \times M$ matrix 
$\A$, $\mathrm{dg}(\A)$ represents the diagonal matrix $\A \odot \I_M$ with $\odot$ denoting the Hadamard product (i.e., entry-wise product) and $\I_M$ is the $M$-dimensional identity matrix. 
It is thus reasonable to build test statistics depending on a relevant 
estimate of $\C(\nu)$. In particular, we take as a starting point the results presented in \cite{loubaton-rosuel-ejs-2021}, devoted to the behavior of the linear spectral statistics (LSS) of the estimator $\hat{\C}(\nu)$ of $\C(\nu)$ defined for each frequency $\nu \in [0,1]$ by 
\begin{equation}
    \label{eq:def-hatC}
    \hat{\C}(\nu) = \left( \mathrm{dg}(\hat{\S}(\nu))\right)^{-1/2} \hat{\S}(\nu)   \left( \mathrm{dg}(\hat{\S}(\nu))\right)^{-1/2},
\end{equation}
where $\hat{\S}(\nu)$ represents the frequency smoothed periodogram 
estimate of the spectral density $\S(\nu)$ of $\y$ defined by 
\begin{equation}
    \label{eq:def-hatS}
    \hat{\S}(\nu) = \frac{1}{B+1} \sum_{b=-B/2}^{B/2} \xibs_{\y}\left(\nu+\frac{b}{N}\right) \xibs_{\y}\left(\nu+\frac{b}{N}\right) ^{*},
\end{equation}
and 
\begin{equation}
    \label{eq:def-xiy}
\xibs_{\y}(\nu) = \frac{1}{\sqrt{N}} \sum_{n=1}^{N} \y_n e^{-2 i \pi (n-1) \nu},    
\end{equation}
where $B$ represents the smoothing span. In the asymptotic regime where $M = M(N) = \Ocal(N^{\alpha})$ for $\alpha \in (1/2,1)$ and $B = B(N)$ converge toward $+\infty$ in such a way that 
$c_N = \frac{M(N)}{B(N)} \rightarrow c$ where $c \in (0,1)$, \cite{loubaton-rosuel-ejs-2021} established that under $\Hcal_0$, for each frequency $\nu$, the empirical eigenvalue distribution of the $M \times M$ Hermitian matrix 
$\hat{\C}(\nu)$ converges almost surely toward the Marchenko-Pastur distribution $\mu_{MP}^{(c)}$ with parameter $c < 1$ defined by 
\[
    d\mu_{MP}^{(c)}(\lambda)=  \frac{\sqrt{(\lambda_+-\lambda)(\lambda-\lambda_-)}}{2\pi c\lambda}\mathds{1}_{\lambda\in[\lambda_-;\lambda_+]}(\lambda)\diff\lambda, \quad \lambda_\pm=(1\pm\sqrt{c})^2.
\]
Therefore, under $\Hcal_0$, for each well-chosen function $f$, the linear spectral statistics $\hat{f}_N(\nu)$ defined by
\begin{equation}
\label{eq:def-hatf}
\hat{f}_N(\nu) = \frac{1}{M} \mathrm{Tr}\left( f(\hat{\C}(\nu))\right),
\end{equation}
converges toward $\int_{\Rbb^{+}}f\diff\mu_{MP}^{(c)}$. Moreover, 
\cite{loubaton-rosuel-ejs-2021} evaluated the rate of convergence toward $0$ of the error $\hat{f}_N(\nu) - \int_{\Rbb^{+}}f\diff\mu_{MP}^{(c_N)}$. In the present paper, which can be seen as a continuation of \cite{loubaton-rosuel-ejs-2021}, we establish under $\Hcal_0$ a central limit theorem (CLT) on $\hat{f}_N(\nu)$ for each frequency $\nu$, and deduce from this two CLTs on statistics combining the LSS $\hat{f}_N(\nu)$ over well-chosen frequency grids. These results allow us to test the hypothesis $\Hcal_0$ and to evaluate analytically the asymptotic type I errors of the proposed statistics.

In order to introduce the results of this paper more precisely, we define $v_N$ by 
\begin{equation}
    \label{eq:def-vN}
    v_N = \frac{1}{B+1} \sum_{b=-B/2}^{B/2} \left( \frac{b}{N} \right)^{2},
\end{equation}
and observe that $v_N = \Ocal\left( (\frac{B}{N})^{2} \right)$. Moreover,  
if $(s_m)_{m=1, \ldots, M}$ represent the spectral densities of the scalar time series $((y_{m,n})_{n \in \mathbb{Z}})_{m=1, \ldots, M}$, we denote by $r_N(\nu)$ the term given by 
\begin{equation}
    \label{def-eq-rN}
    r_N(\nu) = \left( \frac{1}{M} \sum_{m=1}^{M} \frac{s_m'(\nu)}{s_m(\nu)} \right)^{2},
\end{equation}
where $'$ represents the differentiation operator with respect to $\nu$. Then, if $f$ is $\mathcal{C}^{\infty}$ in a neighborhood of $[\lambda_{-}, \lambda_{+}]$, we consider the corrected error term $\theta_N(f,\nu)$ defined by
\begin{equation}
    \label{eq:def-theta-f-nu}
\theta_N(f,\nu) =  \hat{f}_N(\nu) -  \int_{\Rbb^{+}}f\diff\mu_{MP}^{(c_N)}  - <D_N, f> \left(  r_N(\nu)   \; v_N \;  - \; \frac{1}{c_N} \frac{1}{B}  \right),
\end{equation}
where $D_N$ is a certain deterministic compactly supported distribution to be introduced later, carried by the support $[(1 - \sqrt{c_N})^{2}, (1 + \sqrt{c_N})^{2}]$ of the Marchenko-Pastur distribution $\mu_{MP}^{(c_N)}$ with parameter $c_N$. We first show that if
$\frac{1}{2} < \alpha < \frac{4}{5}$, there exists a variance term $\sigma_N^{2}(f)$, which only depends on the Stieltjes transform of the Marchenko-Pastur distribution $\mu_{MP}^{(c_N)}$, converging to a limit $\sigma^{2}(f)$ which, under mild additional assumptions, satisfies $\sigma^{2}(f) > 0$. Then, if $\sigma^{2}(f) > 0$, $\theta_N(f,\nu)$ verifies the central limit theorem 
\begin{equation}
\label{eq:CLT-thetaN}
\frac{B \theta_N(f,\nu)}{\sigma_N(f)} \rightarrow_{\Dcal} \Ncal(0,1),
\end{equation}
for each $\nu$. In order to test the hypothesis $\Hcal_0$ using the linear spectral statistics
$(\hat{f}_N(\nu))_{\nu \in [0,1]}$, we propose to combine the LSS $\hat{f}_N(\nu)$
on a large enough frequency grid. If $\Gcal_N$ represents the frequency grid defined by 
\begin{equation}
    \label{eq:def-Gcal}
    \Gcal_N = \{ k \frac{B+1}{N}, k=0, \ldots, K-1 \},
\end{equation}
where $K$ is defined by 
\begin{equation}
\label{eq:def-K}
K =  \left \lfloor \frac{N}{B+1} \right \rfloor,
\end{equation}
and if $\delta > 0$ verifies $ \delta < 1 - \alpha$, we consider the subset 
$\Gcal_N^{'}$ of $\Gcal_N$ defined by 
\begin{equation}
    \label{eq:def-Gcal-prime}
    \Gcal_N^{'} = \{ k \frac{B'+1}{N}, k=0, \ldots, K' - 1 \},
\end{equation}
where $B'$ and $K'$ are given by
\begin{equation}
    \label{eq:def-Bprime-Kprime}
B' = \lfloor N^{\delta} \rfloor B, \; K'= \left \lfloor \frac{N}{B'+1} \right \rfloor  = 
\frac{1}{N^{\delta}} \Ocal\left( K \right).
\end{equation}
We then establish that if 
$\frac{1}{2} < \alpha < \frac{7}{9}$, the statistics 
$\zeta_{N,1}(f)$ and $\zeta_{N,2}(f)$ defined by 
\begin{eqnarray}
\label{eq:def-statistics-zeta1-clt}
\zeta_{N,1}(f) & = & \frac{1}{\sqrt{K'}} \sum_{\nu \in \Gcal_N^{'}} B \theta_N(f,\nu), \\
\label{eq:def-statistics-zeta-clt}
\zeta_{N,2}(f) & = & \frac{1}{\sqrt{K'}} \sum_{\nu \in \Gcal_N^{'}} \left( (B \theta_N(f,\nu))^{2} - \sigma_N^{2}(f) \right), 
\end{eqnarray}
verify 
\begin{align}  
\label{eq:clt-zeta1}
& \frac{\zeta_{N,1}(f)}{\sigma_N(f)}  \rightarrow_{\Dcal} \Ncal(0,1),  \\
\label{eq:clt-zeta}
& \frac{\zeta_{N,2}(f)}{\sqrt{2} \sigma_N^{2}(f)}  \rightarrow_{\Dcal} \Ncal(0,1). 
\end{align}
Informally, this means that under $\Hcal_0$, the probability distribution of the random variables $\sum_{\nu \in \Gcal_N^{'}} B \theta_N(f,\nu)$ and 
$\sum_{\nu \in \Gcal_N^{'}} \left(B \theta_N(f,\nu)\right)^{2}$ are close to the probability distributions of the random variables $\sqrt{K'} \Ncal(0, \sigma_N^{2}(f))$ and $\sigma_N^{2}(f) K'  +   \sqrt{K'} \Ncal(0, 2 \sigma^{4}_N(f))$. We will also see (see Remark \ref{re:chi2-approximation}) that (\ref{eq:clt-zeta}) implies that the probability distribution of $\sum_{\nu \in \Gcal_N^{'}} \left(B \theta_N(f,\nu)\right)^{2}$ behaves as the probability distribution of $\sigma_N^{2}(f) \chi^{2}(K')$ where $\chi^{2}(K')$ represents a $\chi^{2}$ random variable with $K'$ degrees of freedom. The statistics $\sum_{\nu \in \Gcal_N^{'}} B \theta_N(f,\nu)$ and $\sum_{\nu \in \Gcal_N^{'}} \left(B \theta_N(f,\nu)\right)^{2}$ can thus be used to check the hypothesis $\Hcal_0$ and to control the asymptotic first-order error of the corresponding tests.

We finally mention that $\theta_N(f,\nu), \zeta_{N,1}(f)$ and $\zeta_{N,2}(f)$ depend on the term $r_N(\nu)$ defined by (\ref{def-eq-rN}), which, of course, is unknown. Under mild additional assumptions, we show that the statistics $\hat{\theta}_N(f,\nu), \hat{\zeta}_{N,1}(f)$ and $\hat{\zeta}_{N,2}(f)$ obtained by replacing $r_N(\nu)$ with the estimator $\hat{r}_N(\nu)$ 
proposed in \cite{loubaton-rosuel-ejs-2021} still verify the CLTs (\ref{eq:clt-zeta1}, \ref{eq:clt-zeta}). 

\subsection{On the literature}
\label{subsec:literature}
The problem of testing correlation among $M$ time series has a long history in the field of statistics, which has been motivated by a significant number of applications including microarray analysis in genomics \cite{almudevar2006utility}, signal detection \cite{rosuel-vallet-loubaton-mestre-ieeesp-2021}, and wireless communications \cite{Ramirez2011}, to name a few. A large number of works have addressed the case
where $M$ is much smaller than $N$, traditionally modeled by the low-dimensional asymptotic regime where $N \rightarrow +\infty$ while 
$M$ remains fixed (see e.g., \cite{loubaton-rosuel-ejs-2021} and the references therein). Testing the hypothesis $\Hcal_0$ in asymptotic regimes where $M$ and $N$ both converge toward $+\infty$ 
was mainly addressed when the components of $\y$ are temporally white. In this context, it is possible 
to build relevant test statistics from the sample covariance matrix or the sample correlation matrix $\hat{\C}_{\y}$ of $\y$ in the asymptotic regime where $\frac{M}{N} \rightarrow d$, where $0 < d < 1$. We mention \cite{jiang-aap-2004}, which studied the maximum modulus of the off-diagonal entries of $\hat{\C}_{\y}$, and established that, under $\Hcal_0$, it converges after normalization and recentering, toward a Gumbel distribution. \cite{dette-dornemann-jmva-2020} studied the behavior of the determinant of the correlation matrix, while \cite{gao-han-pan-yang-jrss-2017} proved a CLT on linear spectral statistics of $\hat{\C}_{\y}$ under $\Hcal_0$. \cite{mestre-vallet-ieeeit-2017} established independently a more general CLT valid when the components of the observations are possibly correlated signals. We mention that other papers (\cite{cai-ma-bernouilli-2013}, \cite{fan-jiang-ann-proba-2019}, \cite{morales-jimenez-et-al-stat-sinica-2021}) also considered various alternatives $\Hcal_1$. Still in the regime $\frac{M}{N} \rightarrow d$, \cite{pan-gao-yang-jasa-2014} addressed the case where the components of $\y$ share the same unknown spectral density. Under $\Hcal_0$, 
the rows of $\Y = (\y_1, \ldots, \y_N)$ are independent and identically distributed (i.i.d.), and \cite{pan-gao-yang-jasa-2014} proved that under $\Hcal_0$, the linear spectral statistics of the sample covariance matrix verify a CLT that can be used to test $\Hcal_0$. The particular context considered in \cite{pan-gao-yang-jasa-2014} also allows the use of previous works (see e.g., \cite{chang-yao-zhou-biometrika-2017}, \cite{zeng-jianfeng-qiwei-ann-stat-2019}) testing that the columns of $\Y^{T}$ are extracted from a temporally $N$-dimensional white noise sequence. To the best of our knowledge, the case where the components of $\y$ have possibly different unknown spectral densities was not addressed in previous works, except in \cite{loubaton-rosuel-ejs-2021}, \cite{rosuel-vallet-loubaton-mestre-ieeesp-2021}, and 
\cite{loubaton-rosuel-vallet-jmva-2022}, which consider 
the asymptotic regime $M = \Ocal(N^{\alpha})$ for $\alpha < 1$. These papers 
study the estimated coherence matrix $\hat{\C}(\nu)$ defined by (\ref{eq:def-hatC}) when $\frac{M}{B} \rightarrow c$, $0 < c < 1$. \cite{loubaton-rosuel-vallet-jmva-2022} considered the maximum over the frequency grid $\Gcal_N$ defined by (\ref{eq:def-Gcal}) of the modulus of the off-diagonal entries of $\hat{\C}(\nu)$, and generalized the result of 
\cite{jiang-aap-2004}. \cite{loubaton-rosuel-ejs-2021}, on which the present paper is based, proved that 
the eigenvalue distribution of $\hat{\C}(\nu)$ converges toward $\mu_{MP}^{c}$, and evaluated 
for each function $f$ the order of magnitude of the corrected error term $\psi_N(f,\nu)$ defined by 
\begin{equation}
\label{eq:def-psi-f}
\psi_N(f,\nu) = \hat{f}_N(\nu) -  \int_{\Rbb^{+}}f\diff\mu_{MP}^{(c_N)}  - <D_N, f>  r_N(\nu)   \; v_N \;  \mathds{1}_{\alpha \geq  2/3}.
\end{equation}
More precisely, the main result of \cite{loubaton-rosuel-ejs-2021} states that for each $\epsilon > 0$, 
there exists $\gamma > 0$ such that 
\begin{equation}
\label{eq:main-result-ejs-2021}
P \left( \sup_{\nu \in [0,1]} \left| \psi_N(f,\nu) \right| > N^{\epsilon} u_N \right) < e^{-N^{\gamma}},
\end{equation}
where $u_N$ is defined by 
\begin{equation}
\label{eq:def-uN}
u_N = \frac{1}{B} + \frac{\sqrt{B}}{N} + \frac{B^{3}}{N^{3}}.
\end{equation}
We notice that $u_N = \Ocal(B^{-1})$ if $1/2 < \alpha < 2/3$, $u_N = \Ocal\left(\frac{\sqrt{B}}{N} \right)$ if $2/3 \leq  \alpha < 4/5$, and $u_N = \Ocal\left(\frac{B^3}{N^3} \right)$ if $\alpha \geq 4/5$. It is therefore seen 
that the present paper improves significantly the results of \cite{loubaton-rosuel-ejs-2021} in that 
\begin{itemize}
\item we show that for $\alpha < \frac{4}{5}$, for each $\nu$, $B \theta_N(f,\nu)$ satisfies a CLT, which, in particular, implies that $\theta_N(f,\nu) = \Ocal_{P}\left(\frac{1}{B}\right)$, thus leading to a more accurate evaluation of the error $\hat{f}_N(\nu) -  \int_{\Rbb^{+}}f\diff\mu_{MP}^{(c_N)}$;
\item for $\alpha < \frac{7}{9}$, under $\Hcal_0$, we derive a CLT on the statistics $\zeta_{N,1}(f)$
  and $\zeta_{N,2}(f)$ defined by (\ref{eq:def-statistics-zeta1-clt}) and (\ref{eq:def-statistics-zeta-clt}), which allows us to evaluate the asymptotic first-order error of the corresponding tests. 
\end{itemize} 
\cite{rosuel-vallet-loubaton-mestre-ieeesp-2021} took benefit from the results of \cite{loubaton-rosuel-ejs-2021} to propose statistics that allow one to consistently test the hypothesis $\Hcal_0$ under an alternative
defined by $\y_n = \u_n + \v_n$, where the components of $\v$ are independent and $\u$ is a time series
defined as the output of a $K$-input / $M$-output unknown filter driven by a $K$-dimensional 
white noise sequence, where $K$ is an integer that remains fixed when $M$ and $N$ increase. While \cite{loubaton-rosuel-ejs-2021}, \cite{rosuel-vallet-loubaton-mestre-ieeesp-2021}, and \cite{loubaton-rosuel-vallet-jmva-2022} study frequency domain statistics, we notice that \cite{loubaton-mestre-rmta-2022} considered lag domain approaches
based on linear spectral statistics of a normalized version of the sample covariance matrix of vectors $(\y_n^{L})_{n=1, \ldots, N}$, where $\y_n^{L} = (\y_n^{T}, \y_{n+1}^{T}, \ldots, \y_{n+L-1}^{T})^{T}$ and $L = L(N)$ converges toward $+\infty$ in such a way that $\frac{ML}{N} \rightarrow \tilde{c}$, where $0 < \tilde{c} < 1$. In particular, 
\cite{loubaton-mestre-rmta-2022} established a result that can be interpreted as a lag domain version of (\ref{eq:main-result-ejs-2021}).  

We finally mention the paper \cite{deitmar-2024}, which, although not related to our testing problem, characterized, in the asymptotic regime considered in the present paper, the behavior of the empirical eigenvalue distribution of the frequency smoothed periodogram estimate (\ref{eq:def-hatS}) when the components of $\y$ are not necessarily independent. 

\subsection{General approach}
The approach developed in this paper is mainly based on the use of Bartlett's factorization (see e.g., \cite{walker-1965}, \cite{hannan-book}, \cite{brockwell-davis-book}). If for each $m$, $(\epsilon_{m,n})_{n \in \mathbb{Z}}$ represents the normalized (i.e., $\mathbb{E}|\epsilon_{m,n}|^{2} = 1$) innovation sequence of 
the time series $(y_{m,n})_{n \in \mathbb{Z}}$, and if $h_m(\nu) = \sum_{k=0}^{\infty} a_{m,k} e^{-2 i \pi k \nu}$ represents the square-integrable function with Fourier coefficients $(a_{m,k})_{k \geq 0}$ defined by 
\begin{equation}
    \label{eq:representation-causale-innovation}
    y_{m,n} = \sum_{k=0}^{+\infty} a_{m,k} \epsilon_{m,n-k},
\end{equation}
the normalized Fourier transform $\xi_{y_m}(\nu) = \frac{1}{\sqrt{N}} \sum_{n=1}^{N} y_{m,n} e^{-2 i \pi (n-1) \nu}$ of \\
$(y_{m,n})_{n=1, \ldots, N}$ can be written as $\xi_{y_m}(\nu) = h_m(\nu) \xi_{\epsilon_m}(\nu) + r_{m,\mathfrak{b}}(\nu)$, where $r_{m,\mathfrak{b}}(\nu)$ 
represents an error term which, in some sense, converges toward $0$. Bartlett's factorization 
consists in replacing $\xi_{y_m}(\nu)$ with $h_m(\nu) \xi_{\epsilon_m}(\nu)$ for each $m$ in the 
definitions (\ref{eq:def-hatS}, \ref{eq:def-hatC}) of $\hat{\S}(\nu)$ and $\hat{\C}(\nu)$, thus defining new "estimates" $\hat{\S}_{\mathfrak{b}}(\nu)$ 
and $\hat{\C}_{\mathfrak{b}}(\nu)$. We note, however, that $\hat{\S}_{\mathfrak{b}}(\nu)$ 
and $\hat{\C}_{\mathfrak{b}}(\nu)$ are "virtual" in the sense that they cannot be evaluated from 
the available observations $\y_1, \ldots, \y_N$, and rather represent useful theoretical tools.

We define the terms $\hat{f}_{N,\mathfrak{b}}(\nu)$ and $\theta_{N,\mathfrak{b}}(f,\nu)$ obtained by replacing $\hat{\C}(\nu)$ with $\hat{\C}_{\mathfrak{b}}(\nu)$. The 
main interest of Bartlett's factorization follows from the observation that, 
since the time series $(\epsilon_m)_{m \geq 1}$ are mutually independent white noise sequences, the random variables $\xi_{\epsilon_{m_1}}(\nu_1)$ and $\xi_{\epsilon_{m_2}}(\nu_2)$ are independent if $m_1 \neq m_2$ or if $\nu_2 - \nu_1$ is a nonzero integer multiple of $\frac{1}{N}$ when $m_1 = m_2$, a property that implies that the matrices $(\hat{\C}_{\mathfrak{b}}(\nu))_{\nu \in \Gcal_N}$ are mutually independent, where we recall that $\Gcal_N$ is defined by (\ref{eq:def-Gcal}). Therefore, the random variables $\left( \theta_{N,\mathfrak{b}}(f,\nu)\right)_{\nu \in \mathcal{G}_N}$ are mutually independent as well, a property that will be used in order to establish the 
CLTs on the statistics $\zeta_{N,1}(f)$ and $\zeta_{N,2}(f)$.

We first show that if
$1/2 < \alpha < \frac{4}{5}$, $\theta_{N,\mathfrak{b}}(f,\nu)$ can be written as 
\begin{equation}
    \label{eq:representation-B-thetaNb}
\theta_{N,\mathfrak{b}}(f,\nu) = w_{N,\mathfrak{b}}(f,\nu) + \epsilon_{N,\mathfrak{b}}(f,\nu),    
\end{equation}
where the random variables $(w_{N,\mathfrak{b}}(f,\nu))_{\nu \in \Gcal_N}$ are independent and identically distributed, and satisfy for each $\nu$
\begin{equation}
\label{eq:E-WN=0}
\mathbb{E}(w_{N,\mathfrak{b}}(f,\nu)) = 0,
\end{equation}
and where $\epsilon_{N,\mathfrak{b}}(f,\nu)$ is an error term verifying
\begin{equation}
\label{eq:moyenne-kappaN}
\mathbb{E}(\epsilon_{N,\mathfrak{b}}(f,\nu)) = \mathbb{E}(\theta_{N,\mathfrak{b}}(f,\nu)) = \Ocal\left(\frac{B^{4}}{N^{4}}\right) + o\left( \frac{1}{\sqrt{BN}} \right),
\end{equation}
as well as 
\begin{equation}
\label{eq:concentration-kappaNrond-4sur5}
\sup_{\nu} P\left( |\epsilon_{N,\mathfrak{b}}(f,\nu) - \mathbb{E}(\epsilon_{N,\mathfrak{b}}(f,\nu))| >   N^{\eta} \frac{1}{\sqrt{BN}} \right) < e^{-N^{\gamma}},
\end{equation}
for each $N$ large enough and for some small enough constant $\eta > 0$, where $\gamma$ is a constant depending on $\eta$ (see Theorem \ref{th:approximation-thetaNb}). Since $\alpha < \frac{4}{5}$ 
is equivalent to $\frac{B^{4}}{N^{4}} = o\left( \frac{1}{B} \right)$ and $\frac{1}{\sqrt{BN}} = o\left( \frac{1}{B}\right)$, $\epsilon_{N,\mathfrak{b}}(f,\nu)$ verifies 
\begin{equation}
\label{eq:concentration-kappaN-4sur5}
\sup_{\nu} P\left( |\epsilon_{N,\mathfrak{b}}(f,\nu)| >   \frac{1}{N^{\eta} B} \right) < e^{-N^{\gamma}},
\end{equation}
for each $N$ large enough and for some $\eta$ small enough. 
Moreover, there exists a variance term $\sigma_N^{2}(f)$, which only depends on the Stieltjes transform of the Marchenko-Pastur distribution $\mu_{MP}^{(c_N)}$, converging
to a limit $\sigma^{2}(f)$ which, under mild additional assumptions, satisfies $\sigma^{2}(f) > 0$. Then, 
if $\sigma^{2}(f) > 0$, we prove using the Stein method that $w_{N,\mathfrak{b}}(f,\nu)$ verifies 
\begin{equation}
\label{eq:convergence-WN}
\frac{B w_{N,\mathfrak{b}}(f,\nu)}{\sigma_N(f)} \rightarrow_{\Dcal} \Ncal(0,1).
\end{equation}
Equations (\ref{eq:representation-B-thetaNb}) and (\ref{eq:concentration-kappaN-4sur5}) thus imply that, for $\alpha < \frac{4}{5}$, if $\sigma^{2}(f) > 0$, for each frequency $\nu$, the following CLT
\begin{equation}
\label{eq:CLT-thetaNb}
\frac{B \theta_{N,\mathfrak{b}}(f,\nu)}{\sigma_N(f)} \rightarrow_{\Dcal} \Ncal(0,1),
\end{equation}
holds.

We then argue that
$\theta_{N,\mathfrak{b}}(f,\nu)$ and $\theta_N(f,\nu)$ are defined very similarly, and this leads to the conclusion that, for $\alpha < \frac{4}{5}$, $\theta_N(f,\nu)$ still has a representation
\begin{equation}
\label{eq:representation-B-thetaN}
\theta_N(f,\nu) = w_N(f,\nu) + \epsilon_N(f,\nu),
\end{equation}
where the terms $w_N(f,\nu)$ and $\epsilon_N(f,\nu)$ 
verify (\ref{eq:E-WN=0}), (\ref{eq:moyenne-kappaN}), (\ref{eq:concentration-kappaNrond-4sur5}), as well as (\ref{eq:concentration-kappaN-4sur5}) (see Theorem \ref{th:approximation-thetaN}). For each $\nu$, the sequence of random variables $(w_N(f,\nu))_{N \geq 1}$ has the same 
probability distribution as the sequence \\ $(w_{N,\mathfrak{b}}(f,\nu))_{N \geq 1}$ and also verifies the CLT (\ref{eq:convergence-WN}), thus leading to the 
conclusion that the CLT (\ref{eq:CLT-thetaN}) on $\frac{B \theta_N(f,\nu)}{\sigma_N(f)}$ holds when $\alpha < \frac{4}{5}$.

In order to prove the CLTs (\ref{eq:clt-zeta1}) and 
(\ref{eq:clt-zeta}) for $\alpha < \frac{7}{9}$, 
we first express $\theta_N(f,\nu)$ as 
$
\theta_N(f,\nu) = \theta_{N,\mathfrak{b}}(f,\nu) + \theta_N(f,\nu) - \theta_{N,\mathfrak{b}}(f,\nu),
$
and using the representations (\ref{eq:representation-B-thetaNb}) and (\ref{eq:representation-B-thetaN}) of 
$\theta_{N,\mathfrak{b}}(f,\nu)$ and $\theta_N(f,\nu)$, 
we obtain that if $\alpha < \frac{4}{5}$, then
\begin{equation}
\label{eq:representation-theta-bartlett}
\theta_N(f,\nu) = w_{N,\mathfrak{b}}(f,\nu) + 
 \kappa_{N,\mathfrak{b}}(f,\nu), 
\end{equation}
where $\kappa_{N,\mathfrak{b}}(f,\nu)$ is defined 
by 
\begin{equation}
\label{eq:def:kappaNb}
\kappa_{N,\mathfrak{b}}(f,\nu) = w_N(f,\nu) - w_{N,\mathfrak{b}}(f,\nu) + \epsilon_N(f,\nu).
\end{equation}
We establish that 
\begin{equation}
\label{eq:concentration-wN-wNb}
\sup_{\nu} P\left( | w_N(f,\nu) - w_{N,\mathfrak{b}}(f,\nu) | >   N^{\eta} \frac{1}{\sqrt{BN}} \right) < e^{-N^{\gamma}},
\end{equation}
holds for each $N$ large enough and for some small enough constant $\eta > 0$, where $\gamma$ is a constant depending on $\eta$. If $\alpha < \frac{7}{9}$, then 
$\frac{B^{4}}{N^{4}} = o\left( \frac{1}{\sqrt{BN}}\right)$, and, using the fact that $\epsilon_N(f,\nu)$ verifies (\ref{eq:moyenne-kappaN}) and (\ref{eq:concentration-kappaNrond-4sur5}), we obtain that 
$\epsilon_N(f,\nu)$ verifies 
\begin{equation}
\label{eq:concentration-epsilonN}
\sup_{\nu} P\left( | \epsilon_N(f,\nu) | >   N^{\eta} \frac{1}{\sqrt{BN}} \right) < e^{-N^{\gamma}}.
\end{equation}
Equation (\ref{eq:concentration-wN-wNb}) thus implies that $\kappa_{N,\mathfrak{b}}(f,\nu)$ also satisfies the concentration inequality 
\begin{equation}
\label{eq:concentration-kappaNb}
\sup_{\nu} P\left( |\kappa_{N,\mathfrak{b}}(f,\nu) | >   N^{\eta} \frac{1}{\sqrt{BN}} \right) < e^{-N^{\gamma}}.
\end{equation}
In order to establish the CLT (\ref{eq:clt-zeta1}) when $\alpha < \frac{7}{9}$, we first observe 
that if $\kappa_{N,\mathfrak{b}}(f,\nu)$ verifies the concentration 
inequality (\ref{eq:concentration-kappaNb}), then 
a straightforward evaluation leads immediately to 
\begin{equation}
\label{eq:contribution-kappaNb-clt-zeta1}   
\frac{1}{\sqrt{K'}} \sum_{\nu \in \Gcal_N'} B \kappa_{N,\mathfrak{b}}(f,\nu) = o_P(1).
\end{equation}
To prove the CLT (\ref{eq:clt-zeta1}), it is thus sufficient to establish that $\zeta_{N,1,w}(f)$ defined by 
\begin{equation}
\label{eq:def-zetaN1w}
\zeta_{N,1,w}(f) = \frac{1}{\sqrt{K'}} \sum_{\nu \in \Gcal_N'} B w_{N,\mathfrak{b}}(f,\nu),
\end{equation}
verifies the CLT $\frac{\zeta_{N,1,w}(f)}{\sigma_N(f)} \rightarrow_{\Dcal} \Ncal(0,1)$. Since the random variables 
$(B w_{N,\mathfrak{b}}(f,\nu))_{\nu \in \Gcal_N'}$ are independent and identically distributed, this is easily verified. Equation (\ref{eq:clt-zeta}) is proved similarly. 

\subsection{\texorpdfstring{Motivation of the asymptotic regime 
$M = \Ocal\left( N^{\alpha} \right)$ and $\frac{M}{B} \rightarrow c$, and discussion of the conditions $\alpha > \frac{1}{2}$, $\alpha < \frac{4}{5}$ and $\alpha < \frac{7}{9}$}{Motivation of the asymptotic regime}}
When the dimension $M$ of the observation is large and the sample size $N$ is not unlimited, the classical statistical methodologies testing $\Hcal_0$, developed in the traditional asymptotic regime where $M$ is fixed and $N \rightarrow +\infty$, may not allow us to predict the performance of the corresponding tests 
when the ratio $\frac{M}{N}$ is not small enough. Therefore, it appears relevant to address the above testing problem in asymptotic regimes where both $M$ and $N$ converge toward $+\infty$.  In this paper, we have chosen to consider the test statistics $\zeta_{1,N}(f)$ and 
$\zeta_{2,N}(f)$ that combine linear statistics of the matrices $\hat{\C}_N(\nu)$ evaluated on the discrete grid $\Gcal_N^{'}$ defined by (\ref{eq:def-Gcal-prime}). In order to evaluate the asymptotic distribution of $\zeta_{1,N}(f)$ and $\zeta_{2,N}(f)$ under $\Hcal_0$, it is of course necessary to derive asymptotic properties of $\hat{\C}_N(\nu)$ for each $\nu$. For this, it seems mandatory to choose the parameter $B$
in such a way that $\frac{B}{N} \rightarrow 0$ because it can be seen 
that, under mild additional assumptions, $\frac{B}{N} \rightarrow 0$ implies 
$\sup_{\nu} \| \dg(\hat{\S}_N(\nu)) - \S(\nu) \| \rightarrow 0$. This property allows us to approximate $\left(\dg(\hat{\S}_N(\nu))\right)^{-1/2}$ by $\left(\S(\nu)\right)^{-1/2}$, and is thus fundamental to understanding part of the properties of $\hat{\C}_N(\nu)$. If 
$\frac{M}{N}$ does not converge to $0$—say, it converges to a nonzero constant—then $\frac{B}{N} \rightarrow 0$ implies that $\hat{\C}_N(\nu)$ can be seen as an empirical covariance matrix in the ultra-high-dimensional regime, 
in the sense that $\frac{M}{B} \rightarrow +\infty$. Therefore, the techniques 
developed in this paper, based on the assumption that $M$ and $B$ are of the same order of magnitude, are no longer valid. The context where $\frac{M}{N}$ does not converge to $0$ thus has to be addressed in a completely different way, and is therefore out of the scope of the present paper (see e.g. \cite{chen2010tests, han2020test} for some examples of such studies).
In the following, we moreover focus 
on the regime $M = M(N) = \Ocal(N^{\alpha})$ where $\alpha < 1$. It is easy to check that if the smoothing span 
$B = B(N)$ defined in (\ref{eq:def-hatS}) is chosen in such a way that $\frac{B}{N} \rightarrow 0$ and 
$\frac{M}{B} \rightarrow 0$, then, under mild additional assumptions, the estimated spectral coherence matrix $\hat{\C}_N(\nu)$ defined by (\ref{eq:def-hatC}) verifies 
\begin{equation}
\label{eq:consistency-hatC}
\sup_{\nu \in [0,1]} \| \hat{\C}_N(\nu) - \I \| \rightarrow 0.
\end{equation}
It might therefore be possible to develop statistical tests comparing $\hat{\C}_N(\nu)$ to $\I$ for each $\nu$. 
In practice, if the ratio $\frac{M}{N}$ is not small enough, it may be difficult to find a value of the smoothing
span $B$ satisfying both $B << N$ and $M << B$, and for which $\hat{\C}_N(\nu)$ is close to $\I$. In this context, 
it appears more relevant to choose $B$ of the same order of magnitude as $M$, to study the behavior of 
$\hat{\C}_N(\nu)$ when $\frac{M}{B} \rightarrow c$, $c \in (0,1)$ (note that $\hat{\C}_N(\nu)$ of course does not converge
to $\I$ anymore), and to take advantage of the results to propose new testing approaches. This justifies the relevance of the asymptotic 
regime $\frac{M}{B} \rightarrow c$. We also consider $\alpha > \frac{1}{2}$ because, otherwise, $\frac{M}{N}$ would
converge quite fast to $0$, and it would be possible to find values of $B$ for which $\hat{\C}_N(\nu)$ is close to $\I$. The value $\frac{1}{2}$ could of course be replaced by another threshold, but it appears 
that assuming $\alpha > \frac{1}{2}$ allows us to simplify the exposition of the forthcoming results. We note, however, 
that addressing $\alpha < \frac{1}{2}$ would not introduce serious methodological problems.

We now discuss the condition $\alpha < \frac{4}{5}$. It is clear that, since 
the error term $\epsilon_N(f,\nu)$ in the representation (\ref{eq:representation-B-thetaN}) of $\theta_N(f,\nu)$ satisfies (\ref{eq:moyenne-kappaN}), $\epsilon_N(f,\nu)$ cannot satisfy (\ref{eq:concentration-kappaN-4sur5}) 
if $\alpha \geq \frac{4}{5}$. Therefore, even if $w_N(f,\nu)$ were shown to satisfy the CLT 
(\ref{eq:convergence-WN}) for $\alpha \geq \frac{4}{5}$, the order of magnitude 
of $\mathbb{E}(\epsilon_N(f,\nu))$ appears too large to deduce the CLT (\ref{eq:CLT-thetaN}) from the representation (\ref{eq:representation-B-thetaN}). In order to extend our results to $\alpha \geq \frac{4}{5}$, it would first be necessary 
to evaluate in closed form the $\Ocal\left( B^{4}/N^{4} \right)$ term of $\mathbb{E}\left( \theta_N(f,\nu) \right)$, and to subtract it 
from $\theta_N(f,\nu)$, thus modifying the deterministic correction of $\hat{f}_N(\nu)$. The evaluation of the $\Ocal\left( B^{4}/N^{4} \right)$ term of $\mathbb{E}\left( \theta_N(f,\nu) \right)$ requires, however, tremendous calculations
(see Appendix \ref{sec:guideline-calculation-high-order-term} where an outline of the necessary computations is provided). Moreover, as specified below, the condition $\alpha < \frac{4}{5}$ allows us to simplify the content of this paper (see Remark \ref{re:biais-variance-hats}). Therefore, addressing $\alpha \geq \frac{4}{5}$ in the present paper would not be reasonable, which explains why we prefer to leave this case for further investigation. We feel that $\alpha < 4/5$ is by itself general enough, and represents a good compromise between technical difficulty and the methodological and practical impact of the results. We also mention that evaluating the $\Ocal\left( B^{4}/N^{4} \right)$ term of $\mathbb{E}\left( \theta_N(f,\nu) \right)$ should not be sufficient to obtain a CLT on a relevant normalized and recentered version of $\hat{f}_N(\nu)$ for $\alpha$ arbitrarily close to $1$ because $\mathbb{E}\left( \theta_N(f,\nu) \right)$ also contains higher-order terms of $\frac{B}{N}$ that could become dominant with respect to $\frac{1}{B}$ after subtraction of the $\Ocal\left( B^{4}/N^{4} \right)$ term. Depending on the value of $\alpha$, some of these terms should therefore also be evaluated and subtracted from $\hat{f}_N(\nu)$ to obtain a CLT on $\hat{f}_N(\nu)$.

We now discuss the condition $\alpha < \frac{7}{9}$, which is equivalent 
to $\frac{B^{4}}{N^{4}} = o\left(\frac{1}{\sqrt{BN}}\right)$. If it does not hold, 
$\mathbb{E}(\epsilon_N(f,\nu))$ is no longer negligible with respect to $\frac{1}{\sqrt{BN}}$, 
and (\ref{eq:concentration-epsilonN}), (\ref{eq:concentration-kappaNb}), 
and (\ref{eq:contribution-kappaNb-clt-zeta1}) are no longer valid.  
In order to extend (\ref{eq:clt-zeta1}) and (\ref{eq:clt-zeta}) to $\alpha \in [\frac{7}{9}, \frac{4}{5})$, it would also be necessary to evaluate the term $\Ocal\left( B^{4}/N^{4} \right)$ of $\mathbb{E}(\theta_N(f,\nu))$ and subtract it from $\theta_N(f,\nu)$. 
The new recentered version of $\theta_N(f,\nu)$ would still have a representation 
(\ref{eq:representation-theta-bartlett}), but one in which the order of magnitude 
of the error term should be an $\Ocal\left(\frac{1}{\sqrt{BN}}\right)$ term for 
$\frac{7}{9} \leq \alpha < \frac{4}{5}$ (see Remark \ref{re:alpha-larger-7-9}). 
As explained above, the calculation of the term $\Ocal\left( B^{4}/N^{4} \right)$ of $\mathbb{E}(\theta_N(f,\nu))$ is tremendous, so we prefer to prove (\ref{eq:clt-zeta1}) and (\ref{eq:clt-zeta}) for $\alpha < \frac{7}{9}$. 

\subsection{Assumptions and notations}
\subsubsection{Assumptions}
Although the assumptions formulated on the time series $\left((y_{m,n})_{n \in \mathbb{Z}}\right)_{m \geq 1}$ are similar to those in \cite{loubaton-rosuel-ejs-2021}, 
we include them below for completeness.
\begin{assumption}
\label{eq:as-ym}
For each $m \geq 1$, $y_m = (y_{m,n})_{n \in \mathbb{Z}}$ is a zero-mean complex Gaussian time series\footnote{Every finite linear combination of the random variables $(y_{m,n})_{n \in \mathbb{Z}}$ 
is a complex Gaussian variable $x$, i.e., $\mathrm{Re}(x)$ and $\mathrm{Im}(x)$ are independent and identically distributed real Gaussian random variables.}. Moreover, if $m_1 \neq m_2$, $y_{m_1}$ and $y_{m_2}$ are independent.  
\end{assumption}

\begin{assumption}
\label{as:spectral-densities}
For each $m$, the spectral measure of $y_m$ is absolutely continuous. The corresponding 
spectral densities $(s_m)_{m \geq 1}$ verify 
\begin{equation}
    \label{eq:s-lower-bounded-from-below}
    \inf_{m \geq 1} \inf_{\nu \in [0,1]} s_m(\nu) > 0.
\end{equation}
Moreover, if $\left(\mathbb{E}(y_{m,n+k}y_{m,k}^{*})\right)_{n \in \mathbb{Z}}$ represents the autocovariance sequence of $y_m$, then we have 
\begin{equation}
    \label{eq:decroissance-rm}
    \sup_{m \geq 1} \sum_{n \in \mathbb{Z}} (1 + |n|)^{\gamma_0} \, \left| \mathbb{E}(y_{m,n+k}y_{m,k}^{*}) \right|  < +\infty, 
 \end{equation}
where $\gamma_0 > 4$.
\end{assumption}
Equation (\ref{eq:decroissance-rm}) implies that 
$s_m$ is $\mathcal{C}^{4}$, and verifies $\sup_{m \geq 1} \sup_{\nu} |s_m^{(i)}(\nu)| < +\infty$
for $i=1,2,3,4$, where $s_m^{(i)}$ represents the derivative of order $i$ of $s_m$. Moreover, 
it also follows from (\ref{eq:decroissance-rm}) that 
\begin{equation}
    \label{eq:reste-serie-rm}
    \sup_{m \geq 1} \sum_{|l| \geq n} \left| \mathbb{E}(y_{m,l+k}y_{m,k}^{*}) \right| \leq \frac{C}{n^{\gamma_0}}.
\end{equation}
It is also useful to mention that if $h_m(\nu) = \sum_{k=0}^{+\infty} a_{m,k} e^{-2 i \pi k \nu}$ is the square-integrable function defined by (\ref{eq:representation-causale-innovation}), then, by Lemma D.1 in \cite{loubaton-mestre-rmta-2022}, (\ref{eq:decroissance-rm}) implies that the $(a_{m,k})_{k \in \mathbb{N}}$ verify 
\begin{equation}
    \label{eq:decroissance-am}
    \sup_{m \geq 1} \sum_{k \in \mathbb{N}} (1 + |k|)^{\gamma} \, |a_{m,k}| < +\infty, 
 \end{equation}
for each $\gamma < \gamma_0$. Lemma D.1 in \cite{loubaton-mestre-rmta-2022} can be seen as a generalized uniform (with respect to $m$) 
Wiener-Lévy theorem. Since $\gamma_0 > 4$, $h_m$ is $\mathcal{C}^{4}$, and the $(a_{m,k})_{k \in \mathbb{N}}$ also satisfy 
\begin{equation}
    \label{eq:reste-serie-am}
    \sup_{m \geq 1} \sum_{k \geq  n} |a_{m,k}| \leq \frac{C}{n^{\gamma}},
\end{equation}
for each $\gamma < \gamma_0$. We note that $\gamma_0$ was assumed strictly larger than $3$ in \cite{loubaton-rosuel-ejs-2021}. In the present
paper, we need $\gamma_0 > 4$ to improve certain evaluations that were not needed in \cite{loubaton-rosuel-ejs-2021}. We finally recall formally the asymptotic regime considered in this paper. 
\begin{assumption}
\label{as:asymptotic-regime}
$M = M(N)$ verifies $C_1 N^{\alpha} \leq M \leq C_2 N^{\alpha}$, where $C_1$ and $C_2$ are two nonzero positive constants, and where $\frac{1}{2} < \alpha < 1$. Moreover, $B = B(N)$ is chosen in such a way that 
$c_N = \frac{M}{B} \rightarrow c$, where $c \in (0,1)$. 
\end{assumption}
We recall that in the present paper, we only consider parameters $\alpha$
verifying either $\alpha < \frac{4}{5}$ (condition required for (\ref{eq:CLT-thetaN})) 
or $\alpha < \frac{7}{9}$ (condition required for (\ref{eq:clt-zeta1}, \ref{eq:clt-zeta})).

\subsubsection{Discussion on Assumptions \ref{eq:as-ym} and (\ref{eq:decroissance-rm})}
\label{subsub:discussion-assumptions}

\paragraph{The underlying time series are assumed to be complex Gaussian distributed.} 
We first mention that in the context of 
multichannel signal processing, the $M$-variate time series $(\y_n)_{n \in \mathbb{Z}}$ with uncorrelated components that we consider in the present paper is a reasonable model for additive thermal noise due to imperfections in $M$ sensor electronics (see \cite{rosuel-vallet-loubaton-mestre-ieeesp-2021} for more details). The Gaussianity of each component 
$y_m$ of $\y$ is then a quite typical assumption, while the fact that $y_m$ is complex-valued is motivated by applications to digital communications and radar.

From a technical point of view, the Gaussian assumption is used in several places in the manuscript:
\begin{itemize}
    \item for concentration of measure results introduced in Sections 2.2, 2.4, and 2.5, and used extensively in the rest of the paper;
    \item for the stochastic representation of matrices $\hat{\C}_N(\nu)$ 
and $\tilde{\C}_N(\nu)$, which, in the non-Gaussian case, is no longer 
valid because, while the rows of matrix $\X_N(\nu)$ are mutually independent and 
verify $\mathbb{E}(\x_m(\nu)^{*} \x_m(\nu)) = \I_{B+1}$, the entries 
of each vector $\x_m(\nu)$ are not necessarily independent, so that, among other technical difficulties, the behavior of the empirical eigenvalue distribution of $\frac{\X_N(\nu) \X_N^{*}(\nu)}{B+1}$ is unclear;
   \item for various approximations related to Bartlett's factorization (Section 4) as well as for the derivation of the CLT behind the main result of our work (Sections 5 and 6).
\end{itemize}
It is of course important 
to check whether it is reasonable to expect a generalization of the results of the present paper to the non-Gaussian case, e.g., to the context where each time series $(y_{m,n})_{n \in \mathbb{Z}}$ is a linear process 
driven by a non-Gaussian innovation sequence. In the existing literature, we can mention the recent paper \cite{deitmar-2024} that studies the behavior of the empirical eigenvalue distribution of $\hat{\S}(\nu)$ when $\y$ is a Gaussian linear process whose components $(y_m)_{m=1, \ldots, M}$ are not necessarily independent time series. \cite{deitmar-2024} is based on an approximation of $\hat{\S}(\nu)$ that has some connections with the Bartlett factorization (Theorem 3.2 in \cite{deitmar-2024}). When the innovation sequence of $\y$ is non-Gaussian, using an approach based 
on the Lindeberg principle applied to the difference of the Stieltjes transform of the empirical eigenvalue distribution of $\hat{\S}(\nu)$ with the same term evaluated for a well-chosen Gaussian time series, Corollary 3.4 in \cite{deitmar-2024} proves the universality of the limit eigenvalue distribution of $\hat{\S}(\nu)$ found in the Gaussian case. In our opinion, such a result 
should still be valid for the eigenvalue distribution of $\hat{\C}_N(\nu)$ 
in the context of the present paper. However, establishing a CLT in the non-Gaussian case requires a much more precise analysis than what can be obtained using the tools in \cite{deitmar-2024}. We thus believe that 
the generalization of the results of the present paper to the non-Gaussian case
is a quite challenging topic requiring the development of new approaches.

\paragraph{On Assumption (\ref{eq:decroissance-rm}).} 
It is useful to specify at which points (\ref{eq:decroissance-rm}) appears necessary. We first mention that 
Eq. (\ref{eq:expression-1-covariance-omegabs}) requires that (\ref{eq:decroissance-rm}) holds for $\gamma_0 \geq 1$
(see the proof of Lemma A-1 in \cite{loubaton-rosuel-ejs-2021}), while (\ref{eq:representation-covariance-omegam})
holds provided that for each $m$, $s_m$ is $\Ccal^{4}$ with $\sup_{m, \nu} |s_m^{(4)}(\nu)| < +\infty$, a condition which is implied 
by (\ref{eq:decroissance-rm}). Moreover, formula (\ref{eq:expre-Trace-Phimb}) that plays a fundamental role 
to establish the properties of Bartlett's factorization based approximation of the LSS of $\hat{\C}_N(\nu)$ 
requires that $h_m$ is $\Ccal^{4}$ with $\sup_{m, \nu} |h_m^{(4)}(\nu)| < +\infty$. While we could have 
added this property to the assumptions, we have found it more appropriate to formulate Assumption (\ref{eq:decroissance-rm}) 
for $\gamma_0 > 4$, which implies (\ref{eq:decroissance-am}) for $\gamma = 4$. This is because (\ref{eq:decroissance-rm}) 
only depends on the covariance sequences $(r_m)_{m \geq 1}$, and it is thus more interpretable and easier to check in practice. 
In any case, (\ref{eq:decroissance-rm}) is of course not met when some of the time series $(y_m)_{m \geq 1}$ 
are long-range dependent processes in the sense that $s_m(\nu) \rightarrow +\infty$ as $\nu \rightarrow 0$. In such a context, the statistical inference methods based on the periodogram have quite different behavior (see e.g., \cite{beran-feng-ghosh-kulik-2013}, Chap. 5). Potential positive results on the problem addressed in the present paper would be quite different, and their derivation would certainly be challenging.

\subsubsection{Notations}
In order to simplify the notations, the asymptotic regime (\ref{as:asymptotic-regime}) will be denoted by $N \rightarrow +\infty$. 
A number of terms that are studied throughout this paper in the asymptotic regime (\ref{as:asymptotic-regime}) may depend on 
$N$, the frequency $\nu$, and sometimes a complex variable $z$. A typical example is the matrix denoted in the following 
as $\hat{\Q}_N(z,\nu) = \left( \hat{\C}_N(\nu) - z \I_M \right)^{-1}$. In order to simplify the exposition, we will very often omit 
the dependency with respect to $N, \nu$, or $z$ in the absence of ambiguity. Matrix $\hat{\Q}_N(z,\nu)$ will in particular be denoted 
$\hat{\Q}_N(z), \hat{\Q}(z)$, or $\hat{\Q}$.

If $x$ is a random variable, we denote by $x^{\circ}$ the zero-mean random variable
\begin{equation}
    \label{eq:def-xrond}
   x^{\circ} = x - \mathbb{E}(x). 
\end{equation}
A zero-mean random vector $\x$ with complex-valued entries is said to be $\Ncal_c(0,\Sigmabs)$ distributed if $\mathbb{E}(\x \x^{*}) = \Sigmabs$ and if every linear combination $y$ of the components of $\x$ is a complex Gaussian random variable (i.e., $\mathrm{Re}(y)$ and $\mathrm{Im}(y)$ are independent and identically distributed Gaussian real random variables).

If $\A$ is a matrix, $\|\A\|$ and $\|\A\|_{F}$ represent the spectral norm and the Frobenius norm of $\A$, respectively. If $\A$ is a $p \times p$ matrix, we recall that $\mathrm{dg}(\A)$ represents the diagonal 
matrix $\mathrm{dg}(\A) = \A \odot \I_p$, where $\odot$ is the Hadamard product. If $(a_i)_{i=1, \ldots, p}$ 
represent real or complex numbers, we also denote by $\mathrm{dg}\left((a_i)_{i=1, \ldots, p}\right)$ the $p \times p$ diagonal matrix with diagonal entries $(a_i)_{i=1, \ldots, p}$. If $\A$ is a Hermitian $p \times p$ matrix, the eigenvalues of $\A$ are denoted $(\lambda_k(\A))_{k=1, \ldots, p}$, and are arranged in decreasing order. If $\A$ and $\B$ are two Hermitian matrices, $\A \geq \B$ means that $\A - \B$ is a positive matrix. $\A^{T}$, $\bar{\A}$, and $\A^{*}$ represent the transpose, the conjugate, and the conjugate transpose of $\A$, respectively. Finally, $(\e_i)_{i=1, \ldots, M}$ and $(\f_j)_{j=1, \ldots, B+1}$ represent the canonical bases of $\mathbb{C}^{M}$ and $\mathbb{C}^{B+1}$, respectively.

$\mathcal{C}^{p}$ is the set of all (possibly complex-valued) functions defined on $\mathbb{R}$ whose first $p$ derivatives exist and are continuous. $\mathcal{C}_{c}^{p}$
is the subset of all compactly supported functions of $\mathcal{C}^{p}$.

In the following, if $z = x + i y \in \mathbb{C}$, we define the differential operators 
$\frac{\partial}{\partial z}$ and $\frac{\partial}{\partial \bar{z}}$ as 
$
\frac{\partial}{\partial z} = \frac{\partial}{\partial x} - i \frac{\partial}{\partial y}, \, \frac{\partial}{\partial  \bar{z}} = \frac{\partial}{\partial x} + i \frac{\partial}{\partial y}.
$
It is clear that a real- or complex-valued function $\tilde{h}(\x,\y)$ defined on $\mathbb{R}^{2N}$ can be considered as a function $h(\z, \z^{*})$, where $z = \x + i \y$. In the following, if $h$ is $\mathcal{C}^{1}$, i.e., if 
$\tilde{h}$ is itself $\mathcal{C}^1$, we denote by $\nabla h$ the vector 
\begin{equation}
    \label{eq:def-gradient-h-z-zbar}
 \nabla h(\z, \z^{*})  = \left( \begin{array}{c}  \frac{\partial h}{\partial z} \\ \frac{\partial h}{\partial  \bar{z}} \end{array} \right).
\end{equation}

A nice constant is a positive constant that does not depend on $N, M, B$, the index $m$ of the time series, the frequency $\nu$, or the complex variable $z$ of the various functions defined on subsets of $\mathbb{C}$ that are considered in this paper. In the following, $C$ is a generic notation for nice constants. Its value may change from one line to the other. In the following, when we write that a term $t_{m,N}(\nu)$ verifies $t_{m,N}(\nu) = \Ocal(a_N)$ (resp. $t_{m,N}(\nu) = o(a_N)$) for some sequence of positive real numbers 
$(a_N)_{N \geq 1}$, we mean that there exists a nice constant $C$ such that $|t_{m,N}(\nu)| \leq C a_N$ (resp. $|t_{m,N}(\nu)| \leq C b_N$,
where $\frac{b_N}{a_N} \rightarrow 0$). A nice polynomial is a polynomial whose degree and coefficients are nice constants. If $z \in \mathbb{C}^{+}$, $C(z)$ represents a generic notation 
for terms such as $P_1(|z|) P_2\left(\frac{1}{\mathrm{Im}z}\right)$, where $P_1$ and $P_2$ are nice polynomials, and its value may change from one line to the other. It is easily seen that for each $z \in \mathbb{C}^{+}$, we have $C_1(z) + C_2(z) \leq C_3(z)$ and $\left(C_1(z)\right)^{x} \leq C_2(z)$ if $x \in \mathbb{R}^{+}$. If $(t_{m,N}(z,\nu))_{N \geq 1}$ is a sequence of functions of the complex variable $z$ defined on $\mathbb{C}^{+}$ and depending
on $m$ and $\nu$, and if 
$(a_N)_{N \geq 1}$ is a sequence of positive real numbers, the notation $t_{m,N}(z,\nu) = \Ocal_z(a_N)$ (resp. 
$t_{m,N}(z,\nu) = o_z(a_N)$) means that
\begin{equation}
    \label{eq:def-Oz-aN}
    |t_{m,N}(z,\nu)| \leq C(z) \, a_N \; (\mathrm{resp.} \; |t_{m,N}(z,\nu)| \leq C(z) \, b_N),
\end{equation}
for each $z \in \mathbb{C}^{+}$, where $(b_N)_{N \geq 1}$ is a sequence of positive real numbers such that 
$\lim_{N \rightarrow +\infty} \frac{b_N}{a_N} = 0$.

If $\mu$ is a positive finite measure, the Stieltjes transform $s_{\mu}$ of $\mu$ is the function defined on $\mathbb{C} \setminus \mathrm{Supp}(\mu)$ by 
\begin{equation}
    \label{eq:def-stieltjes}
    s_{\mu}(z) = \int \frac{d \mu(\lambda)}{\lambda - z}.
\end{equation}
$s_{\mu}$ verifies $|s_{\mu}(z)| \leq \frac{\mu(\mathbb{R})}{\mathrm{Im}z}$ 
and $\mathrm{Im}(s_{\mu}(z)) > 0$ for each $z \in \mathbb{C}^{+}$. If $\mu$ is carried by $\mathbb{R}^{+}$, we also have $\mathrm{Im}(z s_{\mu}(z)) > 0$ on $\mathbb{C}^{+}$, and if $a > 0$, the function $- \frac{1}{z(1 + a s_{\mu}(z))}$ coincides with the Stieltjes transform of a probability measure carried by $\mathbb{R}^{+}$.

If $\A$ is a Hermitian $p \times p$ matrix, the empirical eigenvalue distribution of 
$\A$ is the probability measure $\mu = \frac{1}{p} \sum_{k=1}^{p} \delta_{\lambda_k(\A)}$. 
It is clear that the Stieltjes transform $s_{\mu}$ of $\mu$ is given by 
$
s_{\mu}(z) = \frac{1}{p} \Tr \Q_{\A}(z),
$
where $\Q_{\A}$ represents the resolvent of $\A$ defined as the $p \times p$
matrix-valued function defined by 
\begin{equation}
\label{eq:def-resolvente}
\Q_{\A}(z) = \left( \A - z \I \right)^{-1}.
\end{equation}
We mention that $\Q_{\A}(z)$ verifies $\| \Q_{\A}(z) \| \leq \frac{1}{\mathrm{Im}z}$ 
for $z \in \mathbb{C}^{+}$ as well as the resolvent identity
\begin{equation}
\label{eq:resolvent-identity}
\A \Q_{\A}(z) = \Q_{\A}(z) \A = \I + z \Q_{\A}(z).
\end{equation}

\subsection{Overview of the paper}
In Section \ref{sec:tools}, we present some useful tools that were also used in \cite{loubaton-rosuel-ejs-2021} and provide some new properties that are needed in the context of the present paper. We introduce in Subsection \ref{subsec:domination} the concept of stochastic domination, recall in Subsection \ref{subsec:properties-wishart} some properties of large Wishart matrices in the asymptotic regime defined by Assumption \ref{as:asymptotic-regime} as well as the Nash-Poincaré inequality and the integration by parts formula. Subsection 
\ref{subsubsec:helffer-sjostrand} is devoted to the Helffer-Sjöstrand formula,
while Subsection \ref{subsec:gaussian-concentration} presents the classical Gaussian concentration inequality and some useful adaptations to functions that are not Lipschitz on the whole space. Section \ref{subsec:hanson-wright}
is devoted to the Hanson-Wright inequality and its stochastic domination counterpart. In Section \ref{sec:background}, we provide a review of the main results derived in \cite{loubaton-rosuel-ejs-2021}. In Subsection \ref{subsec:background-stochastic-representation}, we recall that matrix $\hat{\C}_N(\nu)$ can be interpreted as a Wishart matrix, up to an error term whose order of magnitude is provided. 
In Subsection \ref{subsec:location-eigenvalues-f-compactly-supported}, we also recall that the eigenvalues of $\hat{\C}_N(\nu)$ are, with high probability, located in a neighborhood of the support of the 
Marchenko-Pastur distribution, and take advantage of this property to show that it is sufficient to 
establish the results of the present paper when the test function $f$ is compactly supported. Finally, we provide in Subsection \ref{subsec:background-study-hatphi} useful properties of the LSS of 
$\hat{\C}_N(\nu)$, and improve the accuracy of some of the evaluations of \cite{loubaton-rosuel-ejs-2021}.
Section \ref{sec:bartlett} is devoted to the properties of 
$\hat{\C}_{N,\mathfrak{b}}(\nu)$ and $\hat{f}_{N,\mathfrak{b}}(\nu)$,
the Bartlett's factorization based versions of 
$\hat{\C}_N(\nu)$ and $\hat{f}_N(\nu)$. In Subsection \ref{subsec:representations-omega}, 
we compare vector $(\xi_{y_m}(\nu - B/2N), \ldots, \xi_{y_m}(\nu + B/2N))$ and its Bartlett's factorization, and derive results that will be useful to prove that the concentration inequality (\ref{eq:concentration-wN-wNb}) holds. 
Subsection \ref{subsec:approximation-sup-lss} provides useful properties of $\hat{f}_{N,\mathfrak{b}}(\nu)$. Section \ref{sec:clt-theta-given-frequency} establishes the CLT on $B \theta_{N,\mathfrak{b}}(f,\nu)$.  
Subsection \ref{subsec:simplification-hatphib} proves the representation (\ref{eq:representation-B-thetaNb}) of $\theta_{N,\mathfrak{b}}(f,\nu)$, and contains the most technical results of the paper, while Subsection \ref{subsec:CLT-W} establishes that $\frac{B w_{N,\mathfrak{b}}(f,\nu)}{\sigma_N(f)} \rightarrow \Ncal(0,1)$ using the Stein method. Based on the results of Subsection \ref{subsec:approximation-sup-lss}, we deduce in 
Subsection \ref{subsec:CLT-B-theta} that $\theta_N(f,\nu)$ verifies (\ref{eq:representation-B-thetaN}) and satisfies the CLT (\ref{eq:CLT-thetaN}). Section \ref{sec:clt-zeta} establishes the CLTs (\ref{eq:clt-zeta1}) and (\ref{eq:clt-zeta}). In Subsection \ref{subsec:proof-representation-bartlett-theta}, we prove the representation 
(\ref{eq:representation-theta-bartlett}) of $\theta_N(f,\nu)$, and deduce from this in 
Subsection \ref{subsec:study-zeta} the properties of the statistics $\zeta_{N,1}(f)$ and $\zeta_{N,2}(f)$. 
The properties of the statistics $\hat{\theta}_N(f,\nu)$, $\hat{\zeta}_{N,1}(f)$, and $\hat{\zeta}_{N,2}(f)$ obtained by replacing $r_N(\nu)$ with the estimator $\hat{r}_N(\nu)$ 
proposed in \cite{loubaton-rosuel-ejs-2021} are addressed in Subsection \ref{subsec:estimation-rN}. We provide in Section \ref{sec:power} a brief power analysis, and establish that under certain alternatives, the proposed tests are consistent under mild extra assumptions. We finally present in Section \ref{sec:simulations} 
numerical experiments that assess the performance of the test statistics  
$\hat{\zeta}_{N,1}(f)$ and $\hat{\zeta}_{N,2}(f)$ in various scenarios. 
The proofs of a number of technical results are provided in the Appendix.

\section{Useful tools}
\label{sec:tools}
\subsection{Stochastic domination}
\label{subsec:domination}
We recall the concept of stochastic domination, adapted from \cite{erdos-knowles-yau-ihp-2013},  and introduced in \cite{loubaton-rosuel-ejs-2021} in order to manage in a  convenient way the various Gaussian exponential concentration inequalities that are used in this paper. We also add an extra definition and a new Lemma. 
\begin{definition}
\label{def:stochastic-domination}
We consider $X = (X^{(N)}(u), u \in U^{(N)}, N \in \mathbb{N})$  a family of non negative
random variables, where $U^{(N)}$ is a set that may depend on $N$. If $a = (a_N)_{N \in \mathbb{N}}$
is a sequence of positive real numbers, the family $X$ is said to be stochastically dominated 
by $a$, denoted $X \prec a$, sometimes $X^{(N)} \prec a_N$ or 
$X^{(N)}(u)  \prec a_N$, if for each $\epsilon > 0$, there exists  $\gamma > 0$ depending only on $\epsilon$ such that 
\begin{equation}
    \label{eq:definition-stochastic-domination}
    \sup_{u \in U^{(N)}} P\left( X^{(N)}(u) > a_N N^{\epsilon}\right) \leq e^{-N^{\gamma}},
\end{equation}
for each $N$ large enough. 
\end{definition}
 \begin{definition}
If $X = (X^{(N)}(u), u \in U^{(N)}, N \in \mathbb{N})$ is a family of possibly complex-valued random variables, we will say that $X^{(N)} = \Ocal_{\prec}(a_N)$ if $|X^{(N)}| \prec a_N$, and that 
$X^{(N)} = o_{\prec}(a_N)$ if there exists a sequence $(b_N)_{N \in \mathbb{N}}$ 
such that $X^{(N)} = \Ocal_{\prec}(b_N)$ and $\frac{b_N}{a_N}  \leq \frac{1}{N^{\delta}}$ for some $\delta > 0$ for each $N$ large enough.
\end{definition}

\begin{definition}
If $X = (X^{(N)}(z,u): N \geq 1, z \in \mathbb{C}^{+}, u \in U^{(N)})$ is a family of possibly complex-valued random variables such that for each $u \in U^{(N)}$, $z \rightarrow X^{(N)}(z,u)$ is defined on $\mathbb{C}^{+}$ and if $(a_N)_{N \geq 1}$ is a sequence of positive real numbers, we say that 
$X^{(N)}(z) = \mathcal{O}_{\prec,z}(a_N)$ if there exists 2 nice polynomials $P_1$ and $P_2$ and a family $Y = (Y^{(N)}(u), N \geq 1, u \in U^{N})$ of positive random variables verifying 
$Y^{(N)} \prec a_N$ and 
\begin{equation}
    \label{eq:def-Ozprec}
    |X^{(N)}(z,u)| \leq P_1(|z|) \, P_2\left(\frac{1}{\Im z}\right) \, Y^{(N)}(u),
\end{equation}
for each $(z,u) \in \mathbb{C}^{+} \times U^{(N)}$. Moreover, we say that $X^{(N)}(z) = o_{\prec,z}(a_N)$ if the family $Y$ verifies $Y^{(N)}(u)  = o_{\prec}(a_N)$.
\end{definition}

In the context of this paper, the families of random variables under investigation will 
frequently depend on $m = 1, \ldots, M=M(N)$, $M(N) \leq N$ and on the frequency $\nu \in [0,1]$, 
i.e. $X = (X^{(N)}_{m}(\nu), \nu \in [0,1], m=1, \ldots, M(N))$. In order to simplify the 
notations, the stochastic domination $X = \Ocal_{\prec}(a_N)$ will be sometimes denoted 
as $X_m(\nu) = \Ocal_{\prec}(a_N)$, or even $X_m  = \Ocal_{\prec}(a_N)$. \\

We also recall the following properties of the stochastic domination relationship
(see \cite{loubaton-rosuel-ejs-2021}). 
\begin{property}
  \label{pr:properties-stochastic-domination}
  \begin{itemize}
      \item (i) If $X_1^{(N)} \prec a_{1,N}$ and  $X_2^{(N)} \prec a_{2,N}$, then 
    $X_1^{(N)} + X_2^{(N)} \prec  a_{1,N} + a_{2,N}$ and 
    $X_1^{(N)} X_2^{(N)} \prec  a_{1,N} a_{2,N}$,
    \item (ii) $X^{(N)} \prec a_N$ if and only if $X^{(N)} \prec N^{\epsilon} a_N$ for each 
    $\epsilon > 0$.
    \end{itemize}
\end{property}
We finally mention the following useful property. 
\begin{lemma}
\label{le:domination-moyenne}
If the set $U^{(N)}$ is given by $U^{(N)} = \{ 1,2, \ldots, M(N) \} \times V^{(N)}$ for a certain set $V^{(N)}$ and for $M(N) \leq N$,  then if $X^{(N)} = \Ocal_{\prec}(a_N)$ (respectively 
$X^{(N)} = o_{\prec}(a_N)$), the family $Y^{(N)}(v), v \in V^{(N)}$ defined by 
$$
Y^{(N)}(v) = \frac{1}{M(N)} \sum_{m=1}^{M(N)} X^{(N)}_m(v),
$$
verifies $Y^{(N)} = \Ocal_{\prec}(a_N)$ (resp. $Y^{(N)} = o_{\prec}(a_N)$). 
\end{lemma}
\begin{proof} 
We just proof the Lemma under the condition that $X^{(N)} = \Ocal_{\prec}(a_N)$, because 
$X^{(N)} = o_{\prec}(a_N)$ means that  $X^{(N)} = \Ocal_{\prec}(b_N)$ 
where $\frac{b_N}{a_N} \leq \frac{1}{N^{\delta}}$; it is thus sufficient to replace $a_N$ by $b_N$ in the following proof. We remark $P(|Y^{(N)}(v)| > N^{\epsilon} a_N) = P(|\sum_{m=1}^{M} X^{(N)}_m(v)|> M N^{\epsilon} a_N)$
verifies 
\begin{align*}
    P(|Y^{(N)}(v)| > N^{\epsilon} a_N) 
    &\leq P(\sum_{m=1}^{M} |X^{(N)}_m(v)|> M N^{\epsilon} a_N) 
    \\
    &\leq \sum_{m=1}^{M} P(|X^{(N)}_m(v)|>  N^{\epsilon} a_N).
\end{align*}
As $\sup_{m,v} P(|X^{(N)}_m(v)|>  N^{\epsilon} a_N) \leq e^{-N^{\gamma}}$ 
for some $\gamma > 0$, we obtain that $\sup_{v} P(|Y^{(N)}(v)| > N^{\epsilon} a_N) \leq M e^{-N^{\gamma}} \leq e^{-N^{\gamma/2}}$ for each $N$ large enough.
\end{proof}
In order to simplify the presentation of the next results, we finally introduce the following definitions. 
\begin{definition}
\label{def:exponentially-high-small-probability}
If $\Lambda = \left((\Lambda_N(u))_{u \in U^{(N)}}\right)_{N \geq 1}$ is a family of events, $\Lambda$ is said to hold with 
exponentially high (resp. small) probability if 
there exists $\gamma > 0$ such that $\inf_{u \in U^{(N)}} P(\Lambda_N(u)) \geq 1 - e^{-N^{\gamma}}$ 
(resp.  $\sup_{u \in U^{(N)}} P(\Lambda_N(u)) \leq  e^{-N^{\gamma}}$) for each $N$ large enough. 
\end{definition}
\subsection{Properties of the eigenvalues and of the resolvent of large Wishart matrices.}
\label{subsec:properties-wishart}
We first recall (see e.g. \cite{loubaton-rosuel-ejs-2021}) the following result. 
\begin{proposition}
\label{prop:Lambda-holds-with-high-proba}
If $\X = \X_N(u), u \in U^{(N)}$ is a family of 
$M \times (B+1)$ matrix with i.i.d. $\Ncal_c(0,1)$ distributed entries, and if $M=M(N)$ and $B = B(N)$ satisfy Assumption 
\ref{as:asymptotic-regime}, then, for each $\epsilon > 0$, the family of events $\Lambda_{N,\epsilon}(u)$ defined by 
\begin{equation}
\label{eq:def-Lambda-N-epsilon}
\Lambda_{N,\epsilon}(u) = \left\{ \lambda_k\left( \frac{\X_{N}(u) \X^{*}_{N}(u)}{B+1} \right) \in 
\mathrm{Supp} \mu_{MP}^{c} + \varepsilon, 
k=1, \ldots, M \right\},
\end{equation}
holds with exponential high probability. 
\end{proposition}

We now introduce the Nash-Poincaré inequality and the integration by parts formula, two Gaussian tools 
which were first used in conjunction in \cite{pastur-2005} to analyse the asymptotic behaviour of large Gaussian random matrices. We also refer to \cite{pastur-shcherbina-2011} for an exhaustive reference. 
\begin{proposition}
    \label{prop:nash-poincare-ipp-iid}
    If $\Z$ is a $N$--dimensional vector $\Ncal_c(0,\sigma^{2} \I_N)$ distributed, and if $h(\Z, \Z^{*})$ is a 
    $\mathcal{C}^1$ function with polynomialy bounded first derivatives, then, we have
    \begin{equation}
        \label{eq:nash-poincare-iid}
        \mathrm{Var}\left( h(\Z, \Z^{*}) \right) \leq \sigma^{2} \, \mathbb{E} \left( \| \nabla h \|^{2} \right),
    \end{equation}
    and 
    \begin{equation}
        \label{eq:ipp-iid}
        \mathbb{E}\left( \Z_{ij} h(\Z,\Z^{*}) \right) = \sigma^{2} \mathbb{E} \left( \frac{\partial h}{\partial \bar{\Z}_{ij}} \right).
    \end{equation}
\end{proposition}
Proposition \ref{prop:nash-poincare-ipp-iid} can be used to analyse the asymptotic behaviour (the asymptotic regime is defined by 
Assumption \ref{as:asymptotic-regime}) of the expectation $\mathbb{E}(\Q_N(z))$ of the resolvent 
of the Wishart matrix $\frac{\X_N \X_N^{*}}{B+1}$ where $\X_N$ is a $M \times (B+1)$ random matrix with 
i.i.d. $\Ncal_c(0,1)$ entries. We first notice that the properties of the probability distribution of 
$\X_N$ implies that $\mathbb{E}(\Q_N(z)) = \beta_N(z) \I_M$ where $\beta_N(z) = \mathbb{E}(\Q_{m,m}(z))$ for each $m$.
Using Proposition \ref{prop:nash-poincare-ipp-iid} in the case $\Z = \mathrm{Vec}(\X_N)$, it can be shown that 
$\beta_N(z)$ can be written as 
\begin{equation}
    \label{eq:expre-beta-t}
    \beta_N(z) = t_N(z) + \epsilon_N(z),
\end{equation}
where $t_N$ represents the Stieltjes transform of the Marcenko-Pastur distribution $\mu_{MP}^{(c_N)}$ 
and where the error term $\epsilon_N$ verifies $\epsilon_N(z) = \Ocal_z\left( \frac{1}{B^{2}}\right)$. We recall that 
$t_N(z)$ satisfies the equation 
\begin{equation}
    \label{eq:equation-MP-tN}
    t_N(z) = \frac{1}{-z + \frac{1}{1 + c_N t_N(z)}},
\end{equation}
for each $z \in \mathbb{C} \setminus \mathrm{Supp}(\mu_{MP}^{(c_N)})$. If $\tilde{t}_N(z)$ is defined by 
\begin{equation}
    \label{eq:equation-MP-tilde-tN}
    \tilde{t}_N(z) = - \frac{1}{z(1+c_N t_N(z))},
\end{equation}
it is also useful to rewrite (\ref{eq:equation-MP-tN}) as 
\begin{equation}
    \label{eq:equation-MP-tN-tilde-tN}
    t_N(z) = - \frac{1}{z(1+ \tilde{t}_N(z))}.
\end{equation}
It is well known that $\tilde{t}_N$ is the Stieltjes transform of the probability measure 
$c_N \mu_{MP}^{(c_N)} + (1 - c_N) \delta_0$. Among other properties of $t_N$ and $\tilde{t}_N$, we mention 
that $c \left|z t_N(z) \tilde{t}_N(z)\right|^{2} < 1$ if $z \in \mathbb{C}^{+}$, and that 
\begin{equation}
    \label{eq:upper-bound-1-czttilde-carre}
    \frac{1}{1 - c \left|z t_N(z) \tilde{t}_N(z)\right|^{2}} \leq \frac{C(C^{2} + |z|^{2})}{(\mathrm{Im}z)^{4}},
\end{equation}
for each $z \in \mathbb{C}^{+}$ (see e.g. Lemma 1.1 in \cite{loubaton-jotp-2016}). Moreover, the function $w_N(z)$ defined by $w_N(z) = \frac{1}{z t_N(z) \tilde{t}_N(z)}$ verifies $\psi_N(w_N(z)) = z$ for each $z$, where 
$\psi_N(w) = \frac{(w+1)(w+c_N)}{w}$. Moreover, 
if $\mathcal{C}$ is a negatively oriented simple contour enclosing $\mathrm{Supp}(\mu_{MP}^{(c_N)})$,  
then, $w_N(\mathcal{C})$ is a negatively oriented simple contour enclosing $[-\sqrt{c}_N, \sqrt{c}_N]$. \\

We finally notice that $\Q_N^{'}(z) = \Q_N^{2}(z)$, that $\mathbb{E}(\Q_N^{2}(z)) = \beta_N^{'}(z) \I_M$
and that $\epsilon_N^{'}(z) = \beta_N^{'}(z) - t_N^{'}(z)$ verifies $\epsilon_N^{'}(z) = \Ocal_z\left( \frac{1}{B^{2}}\right)$, where $'$ represents the
differentiation operator w.r.t. $z$ in this context.

\subsection{The Helffer-Sjöstrand formula}
\label{subsubsec:helffer-sjostrand}
We briefly recall the Helffer-Sjöstrand formula, which, as mentioned in \cite{loubaton-rosuel-ejs-2021}, can be seen as an alternative to the Stieltjes inversion formula. If $f$ is a $\Ccal^{k+1}$ compactly supported function defined on $\mathbb{R}$, it allows to recover $\int f d\mu$ in terms of the Stieltjes transform $s_{\mu}(z)$ where $\mu$ is a finite positive  measure. For this, we consider the following extension $\Phi_k(f)$ of $f$ defined by 
$$
\Phi_k(f)(x+iy) = \sum_{l=0}^{k} \frac{(iy)^{l}}{l!} f^{(l)}(x) \, \rho(y),
$$
where $\rho$ is $\Ccal^{\infty}$, compactly supported, and takes the value 1 in a neighbourhood 
of $0$. If $\bar{\partial}$ is the differential operator $\bar{\partial} = \partial_x + i \partial_y$, then $\Phi_k(f)$ verifies 
\begin{equation}
    \label{eq:derivee-extension-f}
 \bar{\partial} \Phi_k(f)(x+iy) = \frac{(iy)^{k}}{k!} f^{(k+1)}(x),   
\end{equation}
when $y$ is located in the neighbourhood of $0$ on which $\rho(y) = 1$. This implies that if 
$q(z)$ is a function defined on $\mathbb{C}^{+}$ verifying $|q(z)| \leq 
P_1(|z|) P_2\left(\frac{1}{\mathrm{Im}z}\right)$ with $\mathrm{deg}(P_2) \leq k$, then 
$ \bar{\partial} \Phi_k(f)(z) q(z)$ is well defined and bounded on $\mathbb{C}^{+}$. 
Therefore, the integral 
\begin{equation}
    \label{eq:definition-integrale-helffer-sjostrand}
\int_{\mathbb{C}^{+}} \bar{\partial}\Phi_k(f)(z) \, q(z) \, \diff x \, \diff y
\end{equation}
is well defined. We notice that if the support of $f$ is included into the interval $[a_1, a_2]$, and if we assume without restriction that $\rho$ is supported by $[-1, 1]$, then, 
\begin{equation}
  \label{eq:integrale-Cplus-D}
\int_{\mathbb{C}^{+}} \bar{\partial}\Phi_k(f)(z) \, q(z) \, \diff x \, \diff y = \int_{\Dcal} \bar{\partial}\Phi_k(f)(z) \, q(z) \, \diff x \, \diff y,
\end{equation}
where $\Dcal = [a_1, a_2] \times [0,1]$. It is also useful to point out that if $k$ is large enough, and that 
$(q_N(z,u), z \in \mathbb{C}^{+}, u \in U^{(N)}, N \geq 1)$ is a deterministic family of functions defined on $\mathbb{C}^{+}$ verifying
$q_N(z,u) = \Ocal_{z}(a_N)$ (resp. $q_N(z,u) =  o_{z}(a_N)$), then
\begin{equation}
    \label{eq:transfert-O-z-helffer-sjostrand}
\int_{\mathbb{C}^{+}} \bar{\partial}\Phi_k(f)(z) \, q_N(z,u) \, \diff x \, \diff y = \Ocal(a_N) \; (\mathrm{resp.} \; o(a_N)).  
\end{equation}
If the family $(q_N(z,u), z \in \mathbb{C}^{+}, u \in U^{(N)}, N \geq 1)$ is random and that $q_N(z,u) = \Ocal_{\prec,z}(a_N)$ (resp. $q_N(z,u) \prec o_{\prec,z}(a_N)$), then
\begin{equation}
    \label{eq:transfert-O-prec-z-helffer-sjostrand}
\int_{\mathbb{C}^{+}} \bar{\partial}\Phi_k(f)(z) \, q_N(z,u) \, \diff x \, \diff y = \Ocal_{\prec}(a_N) \; (\mathrm{resp.} \; o_{\prec}(a_N))
\end{equation}
The Helffer-Sjöstrand formula states that if $f$ is a compactly supported $\Ccal^{(k+1)}$ 
function defined on $\mathbb{R}$ and if $\mu$ is a finite positive measure, then, for each $1 \leq l \leq k$, we have 
\begin{equation}
    \label{eq:definition-helffer-sjostrand}
    \int f d\mu  = \frac{1}{\pi} \Re \int_{\mathbb{C}^{+}} \bar{\partial}\Phi_l(f)(z) \, s_{\mu}(z) \, \diff x \, \diff y.
\end{equation}
The right-hand side of (\ref{eq:definition-helffer-sjostrand}) is well defined because 
$|\s_{\mu}(z)| \leq \frac{1}{\mathrm{Im}(z)}$ for $z \in \mathbb{C}^{+}$. We also notice that the Helffer-Sjöstrand formula is still valid
if the measure $\mu$ is replaced by a compactly supported distribution. We refer the reader to e.g. Subsection 9.1 in \cite{loubaton-jotp-2016}. 
\subsection{Gaussian concentration inequalities}
\label{subsec:gaussian-concentration}
We consider a family $(\X_N)_{N \geq 1}$ of $N$-dimensional $\mathcal{N}(0,\I)$-distributed random vectors. It is well known that if $h_N$ is a Lipschitz function from 
$\mathbb{R}^{N}$ to $\mathbb{R}$ with Lipschitz constant $\sigma_N$, then the 
following standard Gaussian concentration inequality holds: 
\begin{equation}
    \label{eq:standard-gaussian-concentration}
 P\left( |h_N(\X_N) - \mathbb{E}(h_N(\X_N))| \geq  t \right) \leq C_1 \exp-\left[C_2\left(\frac{t}{\sigma_N}\right)^{2} \right],
\end{equation}
for some universal positive constants $C_1$ and $C_2$. In the following, 
we will need to adapt the inequality (\ref{eq:standard-gaussian-concentration})
to the case where $h_N$ is $\sigma_N$-Lipschitz on the set $\X_N(A_N)$ 
where $A_N$ is an event verifying 
\begin{equation}
    \label{eq:probability-A-complementaire}
    P\left( A_N^{c}\right) \leq e^{-N^{\gamma}},
\end{equation}
for each $N$ large enough for some constant $\gamma > 0$. Then, we have the following refinement of
(\ref{eq:standard-gaussian-concentration}) established in the Appendix \ref{sec:proof-lemma-concentration}. 
\begin{lemma}
    \label{le:conditional-concentration}
    We consider a sequence $(h_N)_{N \geq 1}$ of real valued functions 
    defined on $\mathbb{R}^{N}$ and satisfying the following properties:
    \begin{itemize}
        \item 
    \begin{equation}
        \label{eq:second-moments-f-X}
        \sup_{N \geq 1} \mathbb{E}\left( (h_N(\X_N))^{2} \right) < +\infty,
    \end{equation}
    \item
   $h_N$ is  for each $N$
    $\sigma_N$-Lipschitz on the set $\X_N(A_N)$ where 
    $(A_N)_{N \geq 1}$ are events satisfying (\ref{eq:probability-A-complementaire})
    and where $\sigma_N \geq C \frac{1}{N^{a}}$ for some $a > 0$.
    \end{itemize}
    Then, the following concentration inequality holds:
  \begin{equation}
    \label{eq:conditional-gaussian-concentration}
 P\left( |h_N(\X_N) - \mathbb{E}(h_N(\X_N))| \geq t \right) \leq C_1' \exp-\left[C_2' (t/\sigma_N)^{2}\right]  + e^{-N^{\gamma}},
\end{equation}  
for some constants $C_1'$ and $C_2'$ for each $N$ large enough. 
\end{lemma}
 We notice that, in terms of stochastic domination, (\ref{eq:conditional-gaussian-concentration}), under the hypotheses formulated in Lemma \ref{le:conditional-concentration}, $h_N(\X_N)$ verifies
\begin{equation}
\label{eq:stochastic-domination-conditional-concentration}
|h_N(\X_N) - \mathbb{E}(h_N(\X_N)| \prec \sigma_N.
\end{equation}
We now consider $\X = \X_N(u), u \in U^{(N)}, N \geq 1$ a family of 
$\Ncal(0,\I_N)$ random vectors, and $(A_N(u), u \in U^{(N)}, N \geq 1)$
a family of events verifying 
\begin{equation}
    \label{eq:event-exponentially-probability}
    \sup_{u \in U^{(N)}} P\left((A_N(u))^{c}\right) \leq e^{-N^{\gamma}},
\end{equation}
for each $N$ large enough. Then, provided $h_N$ is $\sigma_N$--Lipschitz on 
$\X_N(u)(A_N(u))$ for each $u$, the family of random variables 
$h_N(\X_N(u)) - \mathbb{E}(h_N(\X_N(u))), u \in U^{(N)}, N \geq 1$
verifies 
\begin{equation}
\label{eq:family-stochastic-domination-conditional-concentration}
h_N(\X_N(u)) - \mathbb{E}(h_N(\X_N(u))) = \Ocal_{\prec}(\sigma_N).
\end{equation}
\begin{remark}
\label{re:evaluation-sigmaN}
If $h_N$ is Lipschitz on $\X_N(A_N)$, it is of course useful to be able to evaluate one of its Lipschitz constant. When $h_N$ is a $\Ccal^{1}$ function, and when $\X_N(A_N)$ is convex, we claim that 
if $\sigma_N$ is defined by 
\begin{equation}
\label{eq:evaluation-sigmaN}
\sigma_N = \sup_{x \in \X_N(A_N)} \| (\nabla h_N)(x)\|,
\end{equation}
then, we have 
\begin{equation}
\label{eq:sigmaN-lipschitz-constant}
|h_N(x) - h_N(y)| \leq \sigma_N \| x - y \|,
\end{equation}
for each pair $(x,y)$ of elements of $\X_N(A_N)$. To justify (\ref{eq:sigmaN-lipschitz-constant}), we remark that there exists $z \in [x,y]$
for which $h_N(x) - h_N(y) = (x-y)^{T} (\nabla h_N)(z)$. As  $\X_N(A_N)$ is convex, 
$z \in \X_N(A_N)$, which, in turn, leads to (\ref{eq:sigmaN-lipschitz-constant}).
\end{remark}
We remark that (\ref{eq:standard-gaussian-concentration}), Lemma \ref{le:conditional-concentration} and (\ref{eq:stochastic-domination-conditional-concentration}) and  (\ref{eq:family-stochastic-domination-conditional-concentration}) are still 
valid when vectors $\X_N$ (or $\X_N(u), u \in U^{(N)}$) are $N$--dimensional $\Ncal_c(0, \I)$. In this context, $h_N$ is a real-valued function $h_N(\X_N, \X_N^*)$
depending on the entries of $\X_N$ and $\X_N^*$. $h_N(\X_N, \X_N^*)$ can of course be written 
as $\tilde{h}_N(\sqrt{2} \mathrm{Re}(\X_N), \sqrt{2} \mathrm{Im}(\X_N))$ for some real valued-function 
$\tilde{h}_N$ defined on $\mathbb{R}^{2N}$. As $(\sqrt{2} \mathrm{Re}(\X_N), \sqrt{2} \mathrm{Im}(\X_N))$ is a $\Ncal(0, \I_{2N})$ random 
vector, $h_N(\X_N, \X_N^*) = \tilde{h}_N(\sqrt{2} \mathrm{Re}(\X_N), \sqrt{2} \mathrm{Im}(\X_N))$ verifies (\ref{eq:standard-gaussian-concentration}) and Lemma \ref{le:conditional-concentration}. Finally, if $h_N$ is complex-valued, writing $h_N(\X_N,\X_N^{*}) = \mathrm{Re}(h_N(\X_N,\X_N^{*})) +
i \mathrm{Im}(h_N(\X_N,\X_N^{*}))$ also leads to the conclusion that (\ref{eq:standard-gaussian-concentration}) and Lemma \ref{le:conditional-concentration} are valid. \\

In the following, if $\X_N(u), u \in U^{(N)}$ is a family of $\Ncal_{c}(0, \I_N)$ random vectors, we will frequently consider family of functions $h_N(\X_N(u),\X_N^{*}(u)), u \in U^{(N)}$ given by 
\begin{equation}
    \label{eq:exemple-fN-helffer-sjostrand}
h_N(\X_N(u),\X_N^{*}(u)) = \frac{1}{\pi} \mathrm{Re} \int_{\Dcal} \bar{\partial}  \Phi_k(f)(z) \, q_N(z,\X_N(u),\X_N^{*}(u)) \, \diff x \diff y,
\end{equation}
where $k$ is a large enough integer, where $f$ is a $\Ccal^{\infty}$ compactly supported function with support included in 
an interval $[a_1, a_2]$, $\Dcal = [a_1, a_2] \times [0,1]$ and where $q_N$ satisfies certain properties. Then, the following Lemma allows to evaluate the stochastic domination order of 
the family $h_N(\X_N(u),\X_N^{*}(u)), u \in U^{(N)}$. 
\begin{lemma}
  \label{le:concentration-integrale-helffer-sjostrand}
We assume that $q_N$ verifies 
\begin{itemize}
    \item For each $\x_N \in \mathbb{C}^{N}$, $z \rightarrow q_N(z,\x_N,\x_N^{*})$ is defined on $\mathbb{C}^{+}$.
    \item $(\x_N, \x_N^{*}) \rightarrow  q_N(z,\x_N,\x_N^{*})$ is $\Ccal^{1}$. 
    \item $q_N(z,\X_N(u),\X_N^{*}(u))  \leq P_1(|z|) P_2\left(\frac{1}{\mathrm{Im}z}\right)$ for some nice 
    polynomials $P_1$ and $P_2$ for each $u \in U^{(N)}$ and on an event $A_N(u)$
    verifying $$\sup_{u \in U^{(N)}} P(A_N(u)^{c}) \leq e^{-N^{\gamma}}$$ for some nice constant $\gamma > 0$.
    \item There exists  nice polynomials $P_1$ and $P_2$ such that $$\mathbb{E}|q_N(z,\X_N(u),\X_N^{*}(u))|^{2} \leq P_1(|z|) P_2\left(\frac{1}{\mathrm{Im}z}\right),$$ for each $z \in \mathbb{C}^{+}$ on the event $A_N(u)$ for each $u$.
    \item 
    \begin{equation}
      \label{eq:condition-gradient-q}
      \| \nabla q_N(z,\X_N(u),\X_N^{*}(u)) \|^{2} \leq  C(z) \,  \sigma^{2}_N,
  \end{equation}
  for each $z \in \mathbb{C}^{+}$ where $C(z) = P_1(|z|) P_2\left(\frac{1}{\mathrm{Im}z}\right)$ for some nice polynomials $P_1$ and $P_2$, and on the event $A_N(u)$ for each $u$.
 \end{itemize}
  We also assume that 
  $\X_N(u)(A_N(u))$ is a convex subset of $\mathbb{C}^{N}$ \footnote{in the sense that $\mathrm{Re}(\X_N(u)(\A_N)), \mathrm{Im}(\X_N(u)(\A_N))$ is a convex subset of $\mathbb{R}^{2N}$} for each $u$. Then, we have
  \begin{equation}
      \label{eq:stochastic-domination-integrale-helffer-sjostrand}
      \left| h_N(\X_N(u),\X_N^{*}(u)) - \mathbb{E}\left( h_N(\X_N(u),\X_N^{*}(u)) \right) \right|  \prec \sigma_N.
  \end{equation}
\end{lemma}
\begin{proof}
We first evaluate $\|\nabla h_N(\X_N(u),\X_N^{*}(u)) \|^{2}$, and remark
 that, as $\Dcal$
is compact, the Schwartz inequality implies that 
\begin{equation}
\label{eq:majoration-norme-gradient-theta2}
\|\nabla h_N(\X_N(u),\X_N^{*}(u)) \|^{2}  \leq C \, T^{2}, 
\end{equation}
where $T$ is the term defined by 
\begin{equation}
\label{eq:def-T-concentration}
T = \left( \int_{\mathcal D} \, \left| \bar{\partial}  \Phi_k(f)(z) \right|^{2}\,
\| \nabla q_N(z,\X_N(u),\X_N^{*}(u)) \|^{2} \, \diff x \diff y \right)^{1/2}
\end{equation}
(\ref{eq:condition-gradient-q}) implies that $T^{2} \leq C \sigma_N^{2}$ on the event 
$A_N(u)$. Using again the Schwartz inequality, we also verify immediately that 
$\mathbb{E}|h_N(\X_N(u),\X_N^{*}(u))|^{2} \leq C$ for some nice constant $C$. Therefore, Remark \ref{re:evaluation-sigmaN}, 
Lemma \ref{le:conditional-concentration}, and its stochastic domination 
counterpart (\ref{eq:family-stochastic-domination-conditional-concentration}) lead to (\ref{eq:stochastic-domination-integrale-helffer-sjostrand}). 
\end{proof}

\subsection{Hanson-Wright inequality}
\label{subsec:hanson-wright}
This well known inequality allows to control deviations of a quadratic form from its expectation. We recall the inequality for a
$\Ncal_c(0,\I_N)$ vector $\x_N = (x_1, \ldots, x_N)^{T}$. If $\A_N$ is a $N \times N$ deterministic matrix, then
\begin{equation}
  \label{eq:hanson-wright-1}
  \mathbb{P}\left( |\x_N^{*} \A_N \x_N - \Tr \A_N| > t \right) \leq 2 \exp -C \min\left( \frac{t}{\|\A_N\|}, \frac{t^{2}}{\|\A_N\|_F^{2}} \right).
\end{equation}
In order to formulate (\ref{eq:hanson-wright-1}) in the stochastic domination framework, we notice that, as $\|\A_N\| \leq \| \A_N \|_F$,
(\ref{eq:hanson-wright-1}) implies that
\begin{equation}
  \label{eq:hanson-wright-2}
  \mathbb{P}\left( |\x_N^{*} \A_N \x_N - \Tr \A_N| > t \right) \leq 2 \exp -C \min\left( \frac{t}{\|\A_N\|_F}, \frac{t^{2}}{\|\A_N\|_F^{2}} \right).
\end{equation}
In the stochastic domination framework, this leads to
\begin{equation}
  \label{eq:hanson-wright-stochastic-domination}
  |\x_N^{*} \A_N \x_N - \Tr \A_N| \prec \| \A_N \|_F.
\end{equation}
We finally add the following useful properties. If $M = M(N)$ is a sequence of integers satisfying $\frac{M(N)}{N} = \Ocal(1)$, and $(\x_{m,N})_{m=1, \ldots, M}$ and $(\A_{m,N})_{m=1, \ldots, M}$ 
are respectively families of $\Ncal_c(0,\I_N)$ vectors and of deterministic
matrices verifying $\sup_{m=1, \ldots, M} \|\A_{m,N}\|_F \leq \kappa_N$,
then, the union bound leads immediately to 
$$
\max_{m=1,\ldots, M} \left|  \x_{m,N}^{*} \A_{m,N} \x_{m,N} - \Tr \A_{m,N} \right| \prec  \kappa_N.
$$
In other words, if we denote by  $\D_N$ the diagonal $M \times M$ matrix given by
$$
\D_N = \dg \left( \x_{m,N}^{*} \A_{m,N} \x_{m,N} - \Tr \A_{m,N}, m=1, \ldots, M \right),
$$
then, we have
\begin{equation}
  \label{eq:sup-hanson-wright-domination-stochastique}
    \| \D_N \| \prec \kappa_N.
    \end{equation}
The moments of $\| \D_N \|$ can also be evaluated. In particular, 
for each $\epsilon > 0$, we have
\begin{equation}
  \label{eq:moyenne-sup-hanson-wright}
  \mathbb{E}\left( \sup_{m=1, \ldots, M}  |\x_{m,N}^{*} \A_{m,N} \x_{m,N} - \Tr \A_{m,N}|^{k} \right) = \Ocal(N^{\epsilon}  (\kappa_N)^{k}),
\end{equation}
for each integer $k$, or equivalently, 
\begin{equation}
  \label{eq:moyenne-sup-hanson-wright-D}
  \mathbb{E}\left( \| \D_N \|^{k} \right) = \Ocal(N^{\epsilon} (\kappa_N)^{k}).
\end{equation}
The proof of (\ref{eq:moyenne-sup-hanson-wright-D}) is provided in the Appendix.

\section{Background on the asymptotic behaviour of the LSS of the estimated spectral coherence matrix}
\label{sec:background}
In this section, we first recall some useful results derived in \cite{loubaton-rosuel-ejs-2021}. For this, we introduce
the modified estimated spectral coherence matrix $\tilde{\C}_N(\nu)$ defined by
\begin{equation}
  \label{eq:def-tildeC}
\tilde{\C}_N(\nu) = \left(\D_N(\nu)\right)^{-1/2} \, \hat{\S}_N(\nu)  \left(\D_N(\nu)\right)^{-1/2},
\end{equation}
where $\D_N(\nu)$ represents the diagonal $M \times M$ matrix
\begin{equation}
  \label{eq:def-DN}
  \D_N(\nu) =  \dg \left( s_m(\nu), m=1, \ldots, M \right).
\end{equation}
If we denote by
$\hat{s}_m(\nu)$ the diagonal entry $\left(\hat{\S}_N(\nu)\right)_{m,m}$ of $\hat{\S}_N(\nu)$, and by
$\hat{\D}_N(\nu)$ the diagonal matrix
\begin{equation}
  \label{eq:def-hat-DN}
  \hat{\D}_N(\nu) =  \dg \left( \hat{s}_m(\nu), m=1, \ldots, M \right),
\end{equation}
it is seen that $\tilde{\C}_N(\nu)$
is obtained from $\hat{\C}_N(\nu)$ by replacing matrix $\hat{\D}_N(\nu)$ by the deterministic matrix $\D_N(\nu)$. 
Matrix $\tilde{\C}_N(\nu)$ is of course simpler to analyze than $\hat{\C}_N(\nu)$, and appears useful to study
 $\hat{\C}_N(\nu)$.
\subsection{\texorpdfstring{Stochastic representation of $\tilde{\C}(\nu)$ and $\hat{\C}(\nu)$}{Stochastic representation}}
\label{subsec:background-stochastic-representation}
The approach developed in \cite{loubaton-rosuel-ejs-2021} is based on the observation that it exists 
a $M \times (B+1)$ matrix $\X_N(\nu)$ with i.i.d. $\mathcal{N}_c(0,1)$ entries and error matrices $\tilde{\Deltabs{}}_N(\nu)$ 
and $\Deltabs_N(\nu)$ such that 
\begin{eqnarray}
\label{eq:representation-tildeC}
\tilde{\C}_N(\nu) & = & \frac{\X_N(\nu) \X_N(\nu)^{*}}{B+1} + \tilde{\Deltabs}_N(\nu), \\
\label{eq:representation-hatC}
\hat{\C}_N(\nu)  & = & \frac{\X_N(\nu) \X_N(\nu)^{*}}{B+1} + \Deltabs_N(\nu),
\end{eqnarray}
where the families $\| \tilde{\Deltabs}_N(\nu) \|, N \geq 1, \nu \in [0,1]$ and  $\| \Deltabs_N(\nu) \|, N \geq 1, \nu \in [0,1]$ verify 
\begin{eqnarray}
\label{eq:domination-tildeDelta}
\| \tilde{\Deltabs} \| \prec \frac{B}{N}, \\
\label{eq:domination-Delta-1}
\| \Deltabs \| \prec \frac{B}{N} + \frac{1}{\sqrt{B}}, \\
\label{eq:domination-hatC-tildeC}
\| \Deltabs - \tilde{\Deltabs} \| = \| \hat{\C} - \tilde{\C}\| \prec \left( \frac{B}{N} \right)^{2} +  \frac{1}{\sqrt{B}}. 
\end{eqnarray}
We refer the reader to Theorem 1.1 in \cite{loubaton-rosuel-ejs-2021}.
As we need similar results in the context of the Bartlett's factorization, we provide a sketch of proof of (\ref{eq:domination-tildeDelta}),  (\ref{eq:domination-Delta-1})
and (\ref{eq:domination-hatC-tildeC}) in order to convince the reader that no further proof will be necessary to establish the new related results. (\ref{eq:domination-tildeDelta}) is based on the observation that for each $m=1, \ldots, M$, 
the covariance matrix of the $(B+1)$--dimensional random row vector $\omegabs_m(\nu)$ defined by 
\begin{equation}
    \label{eq:def-omegam}
  \omegabs_m(\nu) = \left( \xi_{y_m}(\nu-\frac{B}{2N}), \ldots, \xi_{y_m}(\nu+\frac{B}{2N})\right),
\end{equation}
is given by 
\begin{equation}
\label{eq:expression-1-covariance-omegabs}
\mathbb{E}(\omegabs_m(\nu)^{*} \omegabs_m(\nu)) = \mathrm{\Diag}\left( s_m(\nu + b/N), b= -B/2, \ldots, B/2\right) + \Upsilonbs_{m}(\nu),
\end{equation}
where the entries of $\Upsilonbs_{m}(\nu)$ are $\Ocal\left(\frac{1}{N}\right)$ terms. 
Expanding $s_m(\nu + b/N)$ around $\nu$ up to the fourth order (we recall that $s_m$ is $\mathcal{C}^{4}$ because $\gamma_0$ defined by (\ref{eq:decroissance-rm}) 
verifies $\gamma_0 > 4$) and using the symmetry w.r.t. $\nu$ of the set  $\{ \nu + \frac{b}{N}, b=-B/2, \ldots, \frac{B}{2} \}$ lead immediately to 
\begin{equation}
    \label{eq:representation-covariance-omegam}
    \mathbb{E}(\omegabs_m(\nu)^{*} \omegabs_m(\nu)) = s_m(\nu) \left( I + \Phibs_m(\nu) \right),
\end{equation}
where $\Phibs_m(\nu)$ is a $(B+1) \times (B+1)$ Hermitian matrix verifying 
\begin{eqnarray}
\label{eq:bound-norm-Phim}
\sup_{N \geq 1, \nu \in [0,1]} \| \Phibs_m(\nu) \| & = & \mathcal{O}\left( \frac{B}{N} \right), \\ 
\label{eq:expre-Trace-Phim}
\frac{1}{B+1} \Tr \Phibs_m(\nu) & = & \frac{1}{2} \frac{s_m^{''}(\nu)}{s_m(\nu)} v_N + \Ocal\left( \left(\frac{B}{N}\right)^{4} + \frac{1}{N} \right), \\
\label{eq:bound-trace-Phim}
\sup_{N \geq 1, \nu \in [0,1]} \left| \frac{1}{B+1} \Tr \Phibs_m(\nu) \right| & =  & \mathcal{O}\left( \left(\frac{B}{N}\right)^{2} \right).
\end{eqnarray}
(\ref{eq:bound-trace-Phim}) holds because $\frac{1}{N} = o\left( \frac{B^{2}}{N^{2}} \right)$ for $\alpha > \frac{1}{2}$.
We also notice that in \cite{loubaton-rosuel-ejs-2021}, the error term in Eq. (\ref{eq:expre-Trace-Phim}) was $\Ocal\left( \left(\frac{B}{N}\right)^{3} + \frac{1}{N} \right)$ because $s_m$ was assumed $\mathcal{C}^{3}$ so that $s_m(\nu+b/N)$ was expanded up to the third order around $\nu$. Therefore, $\omegabs_m(\nu)$ can be represented as 
\begin{equation}
    \label{eq:representation-omegam}
    \omegabs_m(\nu) = \sqrt{s_m(\nu)} \x_m(\nu) (I + \Phibs_m(\nu))^{\frac{1}{2}} = \sqrt{s_m(\nu)} \x_m(\nu) (I + \Psibs_m(\nu)),
\end{equation}
 where $\Psibs_m(\nu)$ verifies 
\begin{eqnarray}
\label{eq:bound-norm-Psim}
\sup_{N \geq 1, \nu \in [0,1]} \| \Psibs_m(\nu) \| & = & \mathcal{O}\left( \frac{B}{N} \right), \\ 
\label{eq:bound-trace-Psim}
\sup_{N \geq 1, \nu \in [0,1]} \left| \frac{1}{B+1} \Tr \Psibs_m(\nu) \right| & = & \mathcal{O}\left( \left(\frac{B}{N}\right)^{2} \right),
\end{eqnarray}
and where $\x_m(\nu)$ is $\mathcal{N}_c(0, \I_{B+1})$ distributed and where $\x_{m_1}(\nu)$ and $\x_{m_2}(\nu)$ are independent if $m_1 \neq m_2$. The representation (\ref{eq:representation-omegam}) 
where $\Phibs_m(\nu),\Psibs_m(\nu)$ verify (\ref{eq:bound-norm-Phim}) to (\ref{eq:bound-trace-Phim}) and (\ref{eq:bound-norm-Psim}), (\ref{eq:bound-trace-Psim}) is the key tool 
of the approach developed in \cite{loubaton-rosuel-ejs-2021}, and allows in particular to 
prove (\ref{eq:representation-tildeC}) and (\ref{eq:representation-hatC}).
If $\X_N(\nu)$ and $\Gammabs_N(\nu)$ represent the $M \times (B+1)$ matrices with rows $(\x_m(\nu))_{m=1, \ldots, M}$ and $(\x_m(\nu) \Psibs_m(\nu))_{m=1, \ldots, M}$, then the $M \times (B+1)$ matrix $\Sigmabs_N(\nu)$ defined by 
\begin{equation}
    \label{eq:def-Sigma}
    \Sigmabs_N(\nu) = \left( \zeta_{\y}(\nu - \frac{B}{2N}), \ldots, \zeta_{\y}(\nu + \frac{B}{2N}) \right),
\end{equation}
can be written as
\begin{equation}
    \label{eq:expre-Sigma}
  \Sigmabs_N(\nu) = \mathrm{dg}(\sqrt{s_m(\nu)}, m=1, \ldots, M) \, \left( \X_N(\nu) + \Gammabs_N(\nu) \right),   
\end{equation}
while the estimate $\hat{\S}_N(\nu)$ is given by
\begin{equation}
\label{eq:def-hatS-bis}
 \hat{\S}_N(\nu) = \frac{\Sigmabs_N(\nu) (\Sigmabs_N(\nu))^{*}}{B+1}.  
\end{equation}
Therefore, $\tilde{\C}_N(\nu)$ can be represented as 
\begin{equation}
    \label{eq:expre-tildeC}
 \tilde{\C}_N(\nu) = \frac{(\X_N(\nu) + \Gammabs_N(\nu)) (\X_N(\nu) + \Gammabs_N(\nu))^{*}}{B+1} = \frac{\X_N(\nu) (\X_N(\nu))^{*}}{B+1} 
+ \tilde{\Deltabs}_N(\nu),
\end{equation}
where $\tilde{\Deltabs}_N(\nu) = \frac{\X_N(\nu) (\Gammabs_N(\nu))^{*} + \Gammabs_N(\nu) (\X_N(\nu))^{*} + \Gammabs_N(\nu) (\Gammabs_N(\nu))^{*}}{B+1}$. 
(\ref{eq:domination-tildeDelta}) then follows from the following Lemma (see Proposition 1.1 in \cite{loubaton-rosuel-ejs-2021}). 
\begin{lemma}
\label{le:domination-stochastique-Gamma}
The family $(\| \frac{\Gammabs_N(\nu)}{\sqrt{B+1}} \|, N \geq 1, \nu \in [0,1])$ verifies 
\begin{equation}
    \label{eq:domination-stochastique-Gamma}
\| \frac{\Gammabs_N(\nu)}{\sqrt{B+1}} \| \prec \frac{B}{N}.    
\end{equation}
\end{lemma}
The proof in \cite{loubaton-rosuel-ejs-2021} is based on the following concentration inequality obtained using the epsilon net argument: if 
$\Z_N$ represents $\Z_N = \frac{\Gammabs_N(\nu) (\Gammabs_N(\nu))^{*}}{B+1}$, for each $\epsilon$ small enough, we have
\begin{equation}
    \label{eq:concentration-normZ-EZ}
    \Prob \left[ \|\Z_N - \mathbb{E}(\Z_N) \| >  t_N \right] \le  C_0 \exp\left\{-CB \frac{t_N}{(B/N)^2} + 2M\log\frac{1}{\epsilon}\right\},
\end{equation}
for each $t_N \geq \left(\frac{B}{N}\right)^{2}$, where $C_0$ and $C$ are universal constants. (\ref{eq:concentration-normZ-EZ}) clearly implies that $\|\Z_N - \mathbb{E}(\Z_N) \| \prec \left(\frac{B}{N}\right)^{2}$. Moreover, it is easily checked that $E(\Z_N)$ is the diagonal matrix with diagonal entries $\frac{1}{B+1} \mathrm{Tr}( \Psibs_m \Psibs_m^{*}) = \mathcal{O}\left( \frac{B}{N}\right)^{2}$ for $m=1, \ldots, M$. Therefore, $\| \mathbb{E}(\Z_N) \| = \mathcal{O}\left( \frac{B}{N}\right)^{2}$, which, in turn, leads to the conclusion that $\| \Z_N \| \prec \left(\frac{B}{N}\right)^{2}$, and that the family $\left( \left \| \frac{\Gammabs_N(\nu) }{\sqrt{B+1}} \| \right| \right)_{N \geq 1, \nu \in [0,1]}$ verifies (\ref{eq:domination-stochastique-Gamma}).
As the family  $\left( \left \| \frac{\X_N(\nu) }{\sqrt{B+1}} \| \right| \right)_{N \geq 1, \nu \in [0,1]}$ is stochastically dominated by $1$ (see Proposition \ref{prop:Lambda-holds-with-high-proba}), we eventually conclude that (\ref{eq:domination-tildeDelta}) holds. \\

We also mention that \cite{loubaton-rosuel-ejs-2021} (see Lemma 10) took benefit of the concentration
inequality (\ref{eq:concentration-normZ-EZ}) to establish  that
\begin{equation}
  \label{eq:E-norm-Gamma-over-B}
  \mathbb{E} \left\| \frac{\Gammabs_N(\nu)}{\sqrt{B+1}} \right \|^{k} = \Ocal\left( \left(\frac{B}{N}\right)^{k} \right).
\end{equation}
As the moments of $\| \frac{\X_N(\nu)}{\sqrt{B+1}} \|$ are $\Ocal(1)$ terms, (\ref{eq:E-norm-Gamma-over-B}) 
and the Schwartz inequality imply that 
\begin{equation}
\label{eq:E-norm-X-over-B}
\mathbb{E} \left\| \tilde{\Deltabs}_N(\nu) \right \|^{k} = \Ocal\left( \left(\frac{B}{N}\right)^{k} \right),
\end{equation}
 (\ref{eq:domination-Delta-1}) and (\ref{eq:domination-hatC-tildeC}) follow from the observation that 
 $$
 \Deltabs(\nu) = \Thetabs(\nu) + \tilde{\Deltabs}(\nu),
 $$
 where $\Thetabs(\nu) = \hat{\C}(\nu) - \tilde{\C}(\nu)$ can be written as
 \begin{align}
\label{equation:decomposition_Theta}
    \Thetabs(\nu) = (\hat{\D}(\nu)^{-\frac{1}{2}}-\D(\nu)^{-\frac{1}{2}})\hat{\S}(\nu)\hat{\D}(\nu)^{-\frac{1}{2}} + \D(\nu)^{-\frac{1}{2}}\hat{\S}(\nu)(\hat{\D}(\nu)^{-\frac{1}{2}}-\D(\nu)^{-\frac{1}{2}}).
\end{align}
It is proved in
\cite{loubaton-rosuel-ejs-2021} that 
$\|\hat{\S}\| \prec 1$, $|\hat{s}_m| + \frac{1}{|\hat{s}_m|} \prec 1$  and
\begin{equation}
  \label{eq:stoxhastic-domination-order-hats-s}
  |\hat{s}_m - s_m| \prec \frac{1}{\sqrt{B}} + \left( \frac{B}{N} \right)^{2}.
\end{equation}
 The properties of the family $\left(\hat{s}_m(\nu)\right)_{m=1, \ldots, M, \nu \in [0,1]}$ imply immediately that $\| \hat{\D}^{-1}\| \prec 1$, $\| \hat{\D} - \D \| \prec \frac{1}{\sqrt{B}} + \left( \frac{B}{N} \right)^{2}$, $ \| \hat{\D}^{-1/2}-\D^{-1/2}\| \prec  \frac{1}{\sqrt{B}} + \left(\frac{B}{N} \right)^{2}$
and that (\ref{eq:domination-hatC-tildeC}) holds. This evaluation and (\ref{eq:domination-tildeDelta}) eventually lead to (\ref{eq:domination-Delta-1}). \\

\begin{remark}
  \label{re:biais-variance-hats}
  We notice that the term $\frac{1}{\sqrt{B}}$ in (\ref{eq:stoxhastic-domination-order-hats-s}) is due to $\hat{s}_m - \mathbb{E}(\hat{s}_m)$, while $\left(\frac{B}{N} \right)^{2} $ is the order of magnitude of the bias $\mathbb{E}(\hat{s}_m) - s_m$, which, 
due to the formula
\begin{equation}
\label{eq:expre-hats}
\hat{s}_m =  \frac{\| \omegabs_m \|^{2}}{B+1} = s_m \, \frac{\x_m^{*}(\I + \Phibs_m)\x_m}{B+1},
\end{equation}
is given by 
\begin{equation}
\label{eq:expre-bias-hats}
\mathbb{E}(\hat{s}_m) - s_m =  s_m \, \frac{1}{M}\Tr \frac{\Phibs_m}{B+1} = \Ocal\left( \frac{B}{N} \right)^{2},
\end{equation}
(see Eq. (\ref{eq:bound-trace-Phim})). The condition $\alpha < \frac{4}{5}$ is equivalent to 
$\left(\frac{B}{N}\right)^{2} = o\left(\frac{1}{\sqrt{B}}\right)$. Therefore, if $\alpha < \frac{4}{5}$, 
the bias $\mathbb{E}(\hat{s}_m) - s_m$ is negligible w.r.t. $\hat{s}_m - \mathbb{E}(\hat{s}_m)$. We thus have
\begin{equation}
\label{eq:domination-hatD-D-alpha-inferieur-4-5}
\| \hat{\D} - \D \| \prec \frac{1}{\sqrt{B}},
\end{equation}
as well as $ \| \hat{\D}^{-1/2}-\D^{-1/2}\| \prec \frac{1}{\sqrt{B}}$ and 
\begin{equation}
    \label{eq:domination-hatC-tildeC-2}
 \| \hat{\C} - \tilde{\C}\|  \prec \frac{1}{\sqrt{B}},   
\end{equation}
for $\alpha < \frac{4}{5}$. When $\alpha > \frac{4}{5}$, the bias term $\mathbb{E}(\hat{s}_m) - s_m$ is dominant, and some evaluations of the present paper need to adapted. 
We also mention that the expression (\ref{eq:expre-hats}) of $\hat{s}_m$ allows to 
use (\ref{eq:moyenne-sup-hanson-wright-D}) to deduce that 
\begin{equation}
\label{eq:E-norm-hatD-rond}
\mathbb{E}\left( \| \hat{\D} - \mathbb{E}(\hat{\D}) \|^{k} \right) = \Ocal\left( B^{-k/2 + \epsilon}\right),
\end{equation}
for each $\epsilon > 0$. When 
$\alpha < \frac{4}{5}$, we also have 
\begin{equation}
\label{eq:E-norm-hatD-D}
\mathbb{E}\left( \| \hat{\D} - \D) \|^{k} \right) = \Ocal\left( B^{-k/2 + \epsilon}\right),
\end{equation}
for each $\epsilon > 0$. It will be shown in the following that $\mathbb{E}(\| \hat{\C} - \tilde{\C} \|^{k}) = \Ocal\left(B^{-k/2 + \epsilon}\right)$ for $\alpha < \frac{4}{5}$.
\end{remark}
\subsection{\texorpdfstring{Location of the eigenvalues of $\tilde{\C}$ and  $\hat{\C}$}{Location of the eigenvalues}}
\label{subsec:location-eigenvalues-f-compactly-supported}
As matrices  $\tilde{\C}$ and  $\hat{\C}$ are close from a Wishart matrix, it is natural to 
expect that their eigenvalues behave as those of $\frac{\X \X^*}{B+1}$ with exponentially high probability. This point is 
addressed in \cite{loubaton-rosuel-ejs-2021}.  We introduce
the events $\Lambda^{\tilde{\C}}_{N,\epsilon}(\nu)$ and $\Lambda^{\hat{\C}}_{N,\epsilon}(\nu)$ defined by
\begin{eqnarray}
\label{eq:def-Lambda_C_tilde}
\Lambda^{\tilde{\C}}_{N,\epsilon}(\nu)  & = & \{\lambda_k(\tilde{\C}_N(\nu))\subset\Supp\mu_{MP}^{(c)}+\epsilon, k=1, \ldots, M \}, \\
\Lambda^{\hat{\C}}_{N,\epsilon}(\nu)  & = & \{ \lambda_k(\hat{\C}_N(\nu))\subset\Supp\mu_{MP}^{(c)}+\epsilon,  k=1, \ldots, M \}
\label{eqdef-Lambda_C_hat}.
\end{eqnarray}
Then, the following result is proved in \cite{loubaton-rosuel-ejs-2021}.
\begin{proposition}
\label{prop:localisation-eigenvalues-hatC-tildeC}
For each $\epsilon > 0$, the two collections of of events 
\begin{align*}
    \left\{\Lambda^{\tilde{\C}}_{N,\epsilon}(\nu): N \geq 1, \nu \in [0,1]\right\},    
\end{align*}
and 
\begin{align*}
    \left\{\Lambda^{\hat{\C}}_{N,\epsilon}(\nu): N \geq 1, \nu \in [0,1]\right\},     
\end{align*}
hold with exponential high probability.
\end{proposition}
We now claim that Proposition \ref{prop:localisation-eigenvalues-hatC-tildeC} implies 
that in order to establish a CLT on 
$B  \theta_N(f,\nu)$ and on the statistics $\zeta_{N,1}(f)$ and $\zeta_{N,2}(f)$ defined by (\ref{eq:def-statistics-zeta1-clt}) and (\ref{eq:def-statistics-zeta-clt}) for each function $f$ defined on $\mathbb{R}^{+}$, $\mathcal{C}^{\infty}$ in a neighbourhood 
of $[\lambda_{-}, \lambda_{+}]$, it is sufficient to prove the CLT when $f$ is  compactly supported.  More precisely, we consider $\epsilon>0$ and define $\chi: \Rbb\to\Rbb$ as a $\mathcal{C}^{\infty}$ function such that:
\begin{equation}   
\label{eq:definition_chi}
\chi(\lambda) = 
     \begin{cases}
       1 &\quad\text{if }  \lambda \in \Supp\mu_{MP}^{(c)}+\epsilon \\
       0 &\quad\text{if }\lambda \notin \Supp\mu_{MP}^{(c)}+2\epsilon
     \end{cases}.
\end{equation}  
If $f$ is $\mathcal{C}^{\infty}$ in a neighborhood of $[\lambda_{-}, \lambda_{+}]$, we define the compactly supported function $\bar{f}$ given by $\bar{f}= f \times\chi$, which, of course, verifies $f = \bar{f}$ on $\Supp\mu_{MP}^{(c)}+\epsilon$. Then, the following result holds. 
\begin{proposition}
\label{prop:phi-compactly-supported}
The families $(B \theta_N(f,\nu))_{\nu \in [0,1]}$ and $(B \theta_N(\bar{f},\nu))_{\nu \in [0,1]}$ verify
\begin{equation}
\label{eq:phi-compactly-supported}
B \left(  \theta_N(f,\nu) - \theta_N(\bar{f},\nu) \right) = o_{\prec}\left( \frac{1}{N^{a}}\right),
\end{equation}
and 
\begin{equation}
\label{eq:phi-compactly-supported-2}
 \left( B  \theta_N(f,\nu) \right)^{2} - \left( B \theta_N(\bar{f},\nu) \right)^{2} = o_{\prec}\left( \frac{1}{N^{a}}\right),
\end{equation}
for each $a > 0$. 
\end{proposition}
\begin{proof}
We just establish (\ref{eq:phi-compactly-supported}). We first evaluate the left-hand side of (\ref{eq:phi-compactly-supported}). For this, we express $ \theta_N(f,\nu)$ as 
\begin{eqnarray*}
     \theta_N(f,\nu) & = & 
     \hat{f}_N(\nu) \mathds{1}_{\Lambda^{\hat{\C}}_{N,\epsilon}(\nu)} - \int f(\lambda) d \mu^{(c_N)}_{MP}   \\
 &    - &  <D_N,f> \left( r_N(\nu) v_N - (c_N B)^{-1} \right) + \hat{f}_N(\nu) \mathds{1}_{\left(\Lambda^{\hat{\C}}_{N,\epsilon}(\nu)\right)^{c}}.
\end{eqnarray*}
As $c_N \rightarrow c$, for $N$ large enough, $\Supp\mu_{MP}^{(c_N)}$ is contained in $\Supp\mu_{MP}^{(c)}+\epsilon$.  As $f = \bar{f}$ on $\Supp\mu_{MP}^{(c)}+\epsilon$, 
$f$ also coincides with $\bar{f}$ on $\Supp\mu_{MP}^{(c_N)}$ for $N$ large enough. 
Therefore, the equalities $\int f d\mu_{MP}^{(c_N)} = \int \bar{f} d\mu_{MP}^{(c_N)}$ and $<D_N, f > = < D_N, \bar{f}>$ hold. 
Moreover, it is clear that $ \hat{f}_N(\nu) \mathds{1}_{\Lambda^{\hat{\C}}_{N,\epsilon}(\nu)} =  \hat{\bar{f}}_N(\nu) \mathds{1}_{\Lambda^{\hat{\C}}_{N,\epsilon}(\nu)}$. Therefore, $\theta_N(f,\nu) - \theta_N(\bar{f},\nu)$ is given by 
\begin{equation}
    \label{eq:expre-thetaf-thetafbar}
    \theta_N(f,\nu) - \theta_N(\bar{f},\nu) = \hat{f}_N(\nu) \mathds{1}_{\left(\Lambda^{\hat{\C}}_{N,\epsilon}(\nu)\right)^{c}} -
    \hat{\bar{f}}_N(\nu) \mathds{1}_{\left(\Lambda^{\hat{\C}}_{N,\epsilon}(\nu)\right)^{c}}. 
\end{equation}
It thus remains to check that $B \hat{f}_N(\nu) \mathds{1}_{\left(\Lambda^{\hat{\C}}_{N,\epsilon}(\nu)\right)^{c}} $ and
 $B \hat{\bar{f}}_N(\nu) \mathds{1}_{\left(\Lambda^{\hat{\C}}_{N,\epsilon}(\nu)\right)^{c}} $ are  $o_{\prec}\left( \frac{1}{N^{a}}\right)$ 
terms for each $a > 0$. We just check this for the former term. We remark that for each $\delta > 0$, it holds that 
$$
\sup_{\nu} P\left(  B | \hat{f}_N(\nu)| \mathds{1}_{\left(\Lambda^{\hat{\C}}_{N,\epsilon}(\nu)\right)^{c}} >  \frac{N^{\delta}}{N^{a}} \right)
\leq \sup_{\nu} P\left( \left(\Lambda^{\hat{\C}}_{N,\epsilon}(\nu)\right)^{c} \right).
$$
Proposition \ref{prop:localisation-eigenvalues-hatC-tildeC} implies that 
$$
\sup_{\nu}  \mathbb{P}\left(\left( \Lambda^{\hat{\C}}_{N,\epsilon}(\nu)\right)^{c} \right) < e^{-N^{\gamma}},
$$
for each $N$ large enough, thus leading to the conclusion that $$B \hat{f}_N(\nu) \mathds{1}_{\left(\Lambda^{\hat{\C}}_{N,\epsilon}(\nu)\right)^{c}} = o_{\prec}\left( \frac{1}{N^{a}}\right)$$ for each $a$. This completes the proof of Proposition \ref{prop:phi-compactly-supported}. 
\end{proof}
In the following, we therefore only consider functions $f$ defined on  
$\mathbb{R}^{+}$, $\mathcal{C}^\infty$ in a neighbourhood of $[\lambda_{-}, \lambda_{+}]$, and that vanish outside a neighborhood of $[\lambda_{-}, \lambda_{+}]$.
From now on, the set $\mathcal{D}$ that appears in the forthcoming Helffer-Sjöstrand formulas (see Eq. (\ref{eq:integrale-Cplus-D})) represents $[a_1, a_2] \times [0,1]$ where
$[a_1, a_2] = [\lambda_{-} - \epsilon, \lambda_{+} + \epsilon]$ where $\epsilon > 0$ is small enough. 

\subsection{\texorpdfstring{Behaviour of the LSS $\hat{f}_N(\nu)$}{Behaviour of the LSS}}
\label{subsec:background-study-hatphi}
In order to evaluate the behaviour of $\psi_N(\nu) = \hat{f}_N(\nu) - \int f d\mu^{(c_N)}_{MP} - \\ <D_N,f> r_N(\nu) v_N \mathds{1}_{\alpha \geq 2/3}$, we recall that \cite{loubaton-rosuel-ejs-2021}
studies each term of the decomposition 
\begin{align}
    \nonumber 
    \frac{1}{M} \mathrm{Tr}\left( f(\hat{\C}(\nu))\right) -  \int_{\Rbb^{+}}f\diff\mu_{MP}^{(c_N)}  = \frac{1}{M} \mathrm{Tr}\left( f(\hat{\C}(\nu))\right) - \frac{1}{M} \mathrm{Tr}\left( f(\tilde{\C}(\nu))\right) + \\ 
    \nonumber \frac{1}{M} \mathrm{Tr}\left( f(\tilde{\C}(\nu))\right) - \mathbb{E} \left[ \frac{1}{M} \mathrm{Tr}\left( f(\tilde{\C}(\nu))\right) \right] + \\
    \nonumber \mathbb{E} \left[ \frac{1}{M} \mathrm{Tr}\left( f(\tilde{\C}(\nu))\right) 
    - \frac{1}{M} \mathrm{Tr}\left( f(\frac{\X(\nu)\X^*(\nu)}{B+1}) \right) \right] + \\ 
    \label{eq:fundamental-decomposition}
    \mathbb{E} \left[  \frac{1}{M} \mathrm{Tr}\left( f(\frac{\X(\nu)\X^*(\nu)}{B+1}) \right) \right] - \int_{\Rbb^{+}}f\diff\mu_{MP}^{(c_N)}.
\end{align}
We denote by $(T_i)_{i=1, \ldots, 4}$ the four terms of the right-hand side of (\ref{eq:fundamental-decomposition}).
\cite{loubaton-rosuel-ejs-2021} first established that $T_2 = \Ocal_{\prec}(B^{-1})$ using the standard Gaussian concentration inequality (\ref{eq:standard-gaussian-concentration}). The study of $T_1, T_3, T_4$ is based on the Helffer-Sjöstrand formula and on an evaluation of $\frac{1}{M} \Tr \left( \hat{\Q}_N(z) - \tilde{\Q}_N(z)  \right)$, 
$\mathbb{E}\left( \frac{1}{M} \Tr \left( \tilde{\Q}_N(z) - \Q_N(z) \right) \right)$
and $\mathbb{E}\left( \frac{1}{M} \Tr \left(  \Q_N(z) - t_N(z) \right)  \right)$, where 
$\hat{\Q}_N(z)$, $\tilde{\Q}_N(z)$, $\Q_N(z)$ represent the resolvents of matrices $\hat{\C}_N$, 
$\tilde{\C}_N$, $\frac{\X_N \X_N^{*}}{B+1}$ respectively, while we recall that $t_N(z)$ is the Stieltjes transform of the Marcenko-Pastur distribution $\mu_{MP}^{(c_N)}$.
\cite{loubaton-rosuel-ejs-2021} proved that 
$$
\frac{1}{M} \Tr \left( \hat{\Q}_N(z) - \tilde{\Q}_N(z)  \right) - \tilde{p}_N(z) \,  \tilde{r}_N(\nu) \, v_N \, \mathds{1}_{\alpha \geq 2/3} = \Ocal_{\prec,z}(u_N),
$$
where $\tilde{p}_N(z)$ and $\tilde{r}_N(\nu)$ are defined by
\begin{align}
\label{eq:def-tildep}
\tilde{p}_N(z) = & (z t_N(z))' =   \frac{(zt_N(z)\tilde{t}_N(z))^{2}}{1 - c (zt_N(z)\tilde{t}_N(z))^{2}}, \\
\label{eq:def-tilder}
\tilde{r}_N(\nu) = & \frac{1}{2M} \sum_{m=1}^{M} \frac{s^{''}_m(\nu)}{s_m(\nu)},
\end{align}
and where we recall that $u_N$ is defined by (\ref{eq:def-uN}). Moreover, Proposition 1.3 in \cite{loubaton-rosuel-ejs-2021} as well as (\cite{these-alexis}, Chapter 2) imply that $\mathbb{E}\left( \frac{1}{M} \Tr \left( \tilde{\Q}_N(z) - \Q_N(z) \right)\right)$ verifies 
\begin{align}
    \label{eq:Etrace-tildeQ-Q-precise-1}
 &\mathbb{E}\left( \frac{1}{M} \Tr \left( \tilde{\Q}_N(z) - \Q_N(z) \right) \right) 
 = 
 \notag\\
 &\qquad 
 p_N(z)  \left( \frac{1}{B+1} \Tr \left( \frac{1}{M} \sum_{m=1}^{M}  \Phibs_m \right)^{2} \right)  
 \notag\\
 &\qquad 
 - \tilde{p}_N(z) \left( \frac{1}{M} \sum_{m=1}^{M} \frac{1}{B+1} \Tr \Phibs_m \right) +  \mathcal{O}_z\left( \left( \frac{B}{N}\right)^{3} + \frac{1}{N} \right),  
\end{align}
which further gives
\begin{align}
 &\mathbb{E}\left( \frac{1}{M} \Tr \left( \tilde{\Q}_N(z) - \Q_N(z) \right) \right) 
 = 
 \\
 &\qquad 
p_N(z) r_N(\nu)  v_N - \tilde{p}_N(z) \tilde{r}_N(\nu) v_N  + 
  \mathcal{O}_z\left( \left( \frac{B}{N}\right)^{3} +  \frac{1}{N} \right), 
    \label{eq:Etrace-tildeQ-Q-precise-2}
\end{align}
where $p_N(z)$ represents the Stieltjes transform of the distribution $D_N$ introduced in (\ref{eq:def-theta-f-nu}) and is given by
\begin{equation}
  \label{eq:def-p}
  p_N(z)  =  -c_N \frac{(zt_N(z)\tilde{t}_N(z))^{3}}{1 - c (zt_N(z)\tilde{t}_N(z))^{2}}.
\end{equation}  
Finally, it holds that
\begin{equation}
    \label{eq:E-Q-t}
\mathbb{E}\left( \frac{1}{M} \Tr \left(  \Q_N(z) - t_N(z) \right)  \right) = \mathcal{O}_z\left(\frac{1}{B^{2}}\right).  
\end{equation}
Putting all the pieces together, \cite{loubaton-rosuel-ejs-2021} deduced that
$$
\hat{f}_N(\nu) - \int f  d\mu^{(c_N)}_{MP} - <D_N,f> r_N(\nu) v_N \mathds{1}_{\alpha \geq 2/3} = \Ocal_{\prec}(u_N).
$$
We however mention that (\ref{eq:Etrace-tildeQ-Q-precise-1}) and (\ref{eq:Etrace-tildeQ-Q-precise-2}) can be improved. More precisely, using the approach developed in (\cite{these-alexis}, Chapter 2), it appears 
possible to establish that
\begin{align}
    \label{eq:Etrace-tildeQ-Q-precise-1-improved}
 \mathbb{E}\left( \frac{1}{M} \Tr \left( \tilde{\Q}_N(z) - \Q_N(z) \right) \right) = & 
 p_N(z)  \left( \frac{1}{B+1} \Tr \left( \frac{1}{M} \sum_{m=1}^{M}  \Phibs_m \right)^{2} \right)   \\ 
 \notag 
 & 
- \tilde{p}_N(z) \left( \frac{1}{M} \sum_{m=1}^{M} \frac{1}{B+1} \Tr \Phibs_m \right) \\
\notag
& 
+ \mathcal{O}_z\left( \left( \frac{B}{N}\right)^{4} + \frac{1}{N} \right).
 \end{align}
Moreover, as a consequence of (\ref{eq:expre-Trace-Phim}), 
the equality
\begin{equation}
    \label{eq:moyenne-trace-Phim}
  \frac{1}{M} \sum_{m=1}^{M} \frac{1}{B+1} \Tr \Phibs_m = \tilde{r}_N(\nu) v_N  + 
  \mathcal{O}\left( \left( \frac{B}{N}\right)^{4} +  \frac{1}{N} \right),   
\end{equation}
holds, while it can be shown that 
\begin{equation}
    \label{eq:trace-sum-phim-carre}
  \frac{1}{B+1} \Tr \left( \frac{1}{M} \sum_{m=1}^{M}  \Phibs_m \right)^{2} =  r_N(\nu)  v_N  + 
  \mathcal{O}\left( \left( \frac{B}{N}\right)^{4} +  \frac{1}{N} \right).   
\end{equation}
Therefore, we eventually obtain that
\begin{align}
 &\mathbb{E}\left( \frac{1}{M} \Tr \left( \tilde{\Q}_N(z) - \Q_N(z) \right) \right)   =
 \notag\\
 &\qquad p_N(z) r_N(\nu)  v_N - \tilde{p}_N(z) \tilde{r}_N(\nu) v_N  + 
  \mathcal{O}_z\left( \left( \frac{B}{N}\right)^{4} +  \frac{1}{N} \right).
  \label{eq:Etrace-tildeQ-Q-precise-2-improved}
\end{align}
Replacing $\mathcal{O}_z\left( \left( \frac{B}{N}\right)^{3}\right)$ by $\mathcal{O}_z\left( \left( \frac{B}{N}\right)^{4}\right)$ is useful in the context of the present paper because it allows to establish 
that 
\begin{equation}
\label{eq:expansion-expectration-Btheta}
\mathbb{E}\left(  \theta_N(f,\nu) \right) = \Ocal\left( (B^{4}/N^{4}) \right) + o\left( \frac{1}{\sqrt{NB}}\right),
\end{equation}
for $\alpha < \frac{4}{5}$, while (\ref{eq:Etrace-tildeQ-Q-precise-1}) and (\ref{eq:Etrace-tildeQ-Q-precise-2}) would only lead to 
\begin{equation}
\label{eq:expansion-expectration-Btheta-weak}
\mathbb{E}\left(  \theta_N(f,\nu) \right) = \Ocal\left( (B^{3}/N^{3}) \right) +  o\left( \frac{1}{\sqrt{NB}}\right).
\end{equation}
Based on (\ref{eq:expansion-expectration-Btheta-weak}), the convergence in distribution of $B \theta_N(f,\nu)$ towards a zero mean Gaussian random variable would then depend on 
the condition $B^{4}/N^{3} \rightarrow 0$, i.e. $\alpha < 3/4$, a more restrictive assumption. 
A brief justification of (\ref{eq:Etrace-tildeQ-Q-precise-1-improved}, \ref{eq:Etrace-tildeQ-Q-precise-2-improved}) is provided in the Appendix \ref{sec:improvement-bias}.  
\section{Bartlett's factorization}
\label{sec:bartlett}
The Bartlett's factorization (see e.g. \cite{walker-1965}, \cite{hannan-book} Chap. 5, Theorem 1, p. 248, \cite{brockwell-davis-book}, Theorem 10.3.1 p. 346) consists in writing $\xi_{y_m}(\nu)$ as 
\begin{equation}
\label{eq:walker}
\xi_{y_m}(\nu) = h_m(\nu) \xi_{\epsilon_m}(\nu) + r_{m,\mathfrak{b}}(\nu)
\end{equation}
where the reminder $r_{m,\mathfrak{b}}(\nu)$ is supposed to represent an error term converging towards $0$ in an appropriate sense. In the following, we denote  $\omegabs_{m,\mathfrak{b}}(\nu)$, $\omegabs_{m,\epsilon}(\nu)$, and $\omegabs_{m,r}(\nu)$ the $(B+1)$--dimensional vectors defined by
\begin{align*}
\omegabs_{m,\mathfrak{b}}(\nu)  =  \left(h_m(\nu - \frac{B}{2N}) \xi_{\epsilon_m}(\nu - \frac{B}{2N}),\ldots, h_m(\nu + \frac{B}{2N}) \xi_{\epsilon_m}(\nu + \frac{B}{2N})\right), \\
\omegabs_{m,\epsilon}(\nu)  =  \left( \xi_{\epsilon_m}(\nu-\frac{B}{2N}), \ldots, \xi_{\epsilon_m}(\nu+\frac{B}{2N})\right), \\  
\omegabs_{m,r}(\nu)  =  \left( r_{m,\mathfrak{b}}(\nu-\frac{B}{2N}), \ldots, r_{m,\mathfrak{b}}(\nu+\frac{B}{2N})\right).
\end{align*}
It is clear that vector $\omegabs_m(\nu)$ defined by (\ref{eq:def-omegam}) can be written as
\begin{equation}
\label{eq:expre-omegabs-barlett}
\omegabs_m(\nu) = \omegabs_{m,\mathfrak{b}}(\nu) + \omegabs_{m,r}(\nu).
\end{equation}
As recalled in Subsection \ref{subsec:background-stochastic-representation}, the representation 
(\ref{eq:representation-omegam}) of $\omegabs_m(\nu)$ is the key tool 
to derive the results in \cite{loubaton-rosuel-ejs-2021}. In Section \ref{subsec:representations-omega}, 
we derive similar representations of vectors $ \omegabs_{m,\mathfrak{b}}(\nu)$
and $\omegabs_{m,r}(\nu)$, which will allow to obtain alternative 
representations of $\omegabs_m(\nu)$, $\Sigmabs_N(\nu)$, $\tilde{\C}_N(\nu)$
and $\hat{\C}_N(\nu)$. 

\subsection{\texorpdfstring{Properties of vectors $\omegabs_{m,\mathfrak{b}}(\nu)$ and 
$\omegabs_{m,r}(\nu)$, and alternative representation of vector $\omegabs_{m}(\nu)$}{Some properties and alternative representations}}

\label{subsec:representations-omega}
In this paragraph, we derive properties of $\omegabs_{m,\mathfrak{b}}(\nu)$ and 
$\omegabs_{m,r}(\nu)$ that will be used to:
\begin{itemize}
\item evaluate the properties of $\theta_{N,\mathfrak{b}}(f,\nu)$ and establish the CLT on  
$B \theta_{N,\mathfrak{b}}(f,\nu)$
\item obtain a new representation of vector $\omegabs_{m}(\nu)$ that will allow to make a
connection between $\theta_{N,\mathfrak{b}}(f,\nu)$ and $\theta_{N}(f,\nu)$ and to establish 
the representation (\ref{eq:representation-theta-bartlett}) where the error term 
$\kappa_{N,\mathfrak{b}}(f,\nu)$ verifies the concentration inequality
(\ref{eq:concentration-kappaNb}) if $\alpha < \frac{7}{9}$.
\end{itemize}
We first state the following obvious, but important property of vector $\omegabs_{m,\mathfrak{b}}(\nu)$. 
\begin{proposition}
    \label{prop:independence-bartlett-calGN}
For each $m=1, \ldots, M$, the vectors $(\omegabs_{m,\mathfrak{b}}(\nu))_{\nu \in \Gcal_N}$
are mutually independent.
\end{proposition}
\begin{proof}
If we consider 2 different frequencies $\nu_1$ and $\nu_2$ that belong to $\Gcal_N$, it is clear that the two sets of frequencies $\Acal_i = \{ \nu_i - \frac{B}{2N}, \ldots, \nu_i + \frac{B}{2N} \}$, $i=1,2$ are disjoints. As $(\epsilon_{m,n})_{n \in \mathbb{Z}}$ is a Gaussian white noise sequence, the Gaussian random vectors $\xi_{\epsilon_m}(\nu_1)$ and $\xi_{\epsilon_m}(\nu_2)$ are mutually independent. This, of course, implies that 
the vectors $\omegabs_{m,\mathfrak{b}}(\nu_1)$ and $\omegabs_{m,\mathfrak{b}}(\nu_2)$ share this property. 
\end{proof}
In order to establish the other properties of $\omegabs_{m,\mathfrak{b}}(\nu)$, we remark that 
\begin{align}
\omegabs_{m,\mathfrak{b}}(\nu) & = h_m(\nu) \omegabs_{m,\epsilon}(\nu) \\
\notag 
& + \omegabs_{m, \epsilon}(\nu) \, \dg \left( h_m(\nu - B/2N) - h_m(\nu), \ldots,  h_m(\nu + B/2N) - h_m(\nu) \right)  \notag \\
     & = h_m(\nu) \left( \omegabs_{m,\epsilon}(\nu) + \omegabs_{m, \epsilon}(\nu) \, \Psibs_{m,\mathfrak{b}}(\nu) \right) \notag \\
\label{eq:representation-omegaw}
     & = \sqrt{s_m(\nu)} \x_{m,\mathfrak{b}}(\nu)  \left( I + \Psibs_{m,\mathfrak{b}}(\nu) \right),
\end{align}
where $\x_{m,\mathfrak{b}}(\nu)$ is defined by 
\begin{equation}
    \label{eq:def-xmb}
\x_{m,\mathfrak{b}}(\nu) =  \frac{h_m(\nu)}{\sqrt{s_m(\nu)}} \,  \omegabs_{m,\epsilon}(\nu),    
\end{equation}
and where $\Psibs_{m,\mathfrak{b}}(\nu)$ represents the diagonal $(B+1) \times (B+1)$ matrix
$$
\Psibs_{m,\mathfrak{b}}(\nu) =  \dg \left( \frac{h_m(\nu - B/2N) - h_m(\nu)}{h_m(\nu)}, \ldots,  \frac{h_m(\nu + B/2N) - h_m(\nu)}{h_m(\nu)} \right).
$$
As $(\epsilon_{m,n})_{n \in \mathbb{Z}}$ is an uncorrelated sequence, the covariance matrix of the row vector $ \omegabs_{m,\epsilon}(\nu)$
is equal to $\I_{B+1}$, and vector $\x_{m,\mathfrak{b}}(\nu)$ shares this property.  $\x_{m,\mathfrak{b}}(\nu)$ is thus
$\Ncal_{c}(0, \I_{B+1})$ distributed. 
Expending $h_m(\nu + b/N)$ around $\nu$ up to the second order leads immediately to 
\begin{equation}
\label{eq:proprietes-Psimw}
\| \Psibs_{m,\mathfrak{b}} \| = \mathcal{O}\left(\frac{B}{N}\right), \, \frac{1}{B+1} \mathrm{Tr} \Psibs_{m,\mathfrak{b}} = \mathcal{O}\left(\left(\frac{B}{N}\right)^{2}\right),
\end{equation}
(\ref{eq:representation-omegaw}) implies that $\omegabs_m$ can be written as 
$$
\omegabs_m = \sqrt{s_m} \left( \x_{m,\mathfrak{b}} + \x_{m,\mathfrak{b}} \Psibs_{m,\mathfrak{b}} + \frac{\omegabs_{m,r}}{\sqrt{s_m}} \right).
$$
This suggests to express $\frac{\omegabs_{m,r}}{\sqrt{s_m}}$ as
\begin{equation}
\label{eq:decompostion-omegamr}
\frac{\omegabs_{m,r}}{\sqrt{s_m}} = \frac{\omegabs_{m,r}}{\sqrt{s_m}} | \mathrm{sp}(\x_{m,\mathfrak{b}}) + 
 \frac{\omegabs_{m,r}}{\sqrt{s_m}} | (\mathrm{sp}(\x_{m,\mathfrak{b}}))^{\perp},
\end{equation}
where, if  $A$ represents a Hilbert subspace of the set of all square integrable random variables, the symbol $|\, A$ represents the orthogonal projection operator on $A$ while $\mathrm{sp}(\x_{m,\mathfrak{b}})$ is the 
space generated by the components of $\x_{m,\mathfrak{b}}$. We write the two terms at the right-hand side of 
(\ref{eq:decompostion-omegamr}) as 
\begin{eqnarray}
 \label{eq:def-Psi1mr}
 \frac{\omegabs_{m,r}}{\sqrt{s_m}} | \mathrm{sp}(\x_{m,\mathfrak{b}}) & = & \x_{m,\mathfrak{b}} \, \Psibs_{m,r}^{1}, \\
 \label{eq:def-Psi2mr}
 \frac{\omegabs_{m,r}}{\sqrt{s_m}} | (\mathrm{sp}(\x_{m,\mathfrak{b}}))^{\perp} & = & \x_{m,r} \, \Psibs_{m,r}^{2},
\end{eqnarray}
where $\x_{m,r}$ is $\mathcal{N}_c(0, \I)$ distributed, independent from $\x_{m,\mathfrak{b}}$, 
and $\Psibs_{m,r}^{1}$ and  $\Psibs_{m,r}^{2}$ are $(B+1) \times (B+1)$ matrices. In Appendix \ref{sec:proof-properties-Psi1r-Psi2r}, we establish the following properties of $\Psibs_{m,r}^{1}$ and $\Psibs_{m,r}^{2}$.
\begin{proposition}
\label{prop:properties-Psi1r-Psi2r}
$\Psibs_{m,r}^{1}$ and  $\Psibs_{m,r}^{2}$ verify:
\begin{equation}
 \label{eq:norme-Psi2r}
 \| \Psibs_{m,r}^{2} \| = \Ocal\left( \left(\frac{B}{N}\right)^{1/2}\right),
\end{equation}
and 
\begin{equation}
     \label{eq:norme-Psi1mr}
   \| \Psibs_{m,r}^{1} \| \leq C \, \frac{B}{N}, \; \left|\frac{1}{B+1} \Tr  \Psibs_{m,r}^{1} \right| \leq C \, \frac{1}{N}, 
 \end{equation}
 for some nice constant $C$
\end{proposition}
We thus obtain that $\omegabs_m$ can be represented as
\begin{eqnarray}
\label{eq:representation-omegam-bartlett}
\omegabs_m & = & \sqrt{s_m} \left( \x_{m,\mathfrak{b}}(\I + \Psibs_{m,\mathfrak{b}} + \Psibs^{1}_{m,r}) + \x_{m,r}   \Psibs_{m,r}^{2} \right) \\
           & = & \sqrt{s_m} (\x_{m,\mathfrak{b}}, \x_{m,r}) \left( \begin{array}{c} \I + \Psibs_{m,\mathfrak{b}} + \Psibs^{1}_{m,r} \\ \Psibs_{m,r}^{2} \end{array} \right),
\end{eqnarray}
where $\| \Psibs_{m,\mathfrak{b}} \| = \mathcal{O}\left( \frac{B}{N} \right)$, 
$\| \Psibs^{1}_{m,r} \| = \mathcal{O}\left( \frac{B}{N} \right)$, 
$\| \Psibs^{2}_{m,r} \| = \mathcal{O}\left( \sqrt{\frac{B}{N}} \right)$, and where $(\x_{m,\mathfrak{b}}, \x_{m,r})$ is $\mathcal{N}_c(0, \I_{2(B+1)})$ distributed. Moreover,  if $m_1 \neq m_2$,  
 $(\x_{m_1,\mathfrak{b}}, \x_{m_1,r})$ and  $(\x_{m_2,\mathfrak{b}}, \x_{m_2,r})$ are mutually independent. We mention that, altough the Bartlett's factorization is a well-known tool, the representation (\ref{eq:representation-omegam-bartlett}) with $\Psibs_{m,\mathfrak{b}}, \Psibs^{1}_{m,r}, \Psibs^{2}_{m,r}$ verifying (\ref{eq:proprietes-Psimw}, \ref{eq:norme-Psi1mr}, \ref{eq:norme-Psi2r}) of vector $\omegabs_m$ seems to be new in the Gaussian case. \\
 
 We note that (\ref{eq:representation-omegam-bartlett}) allows one to obtain an alternative representation of $\tilde{\C}$ (see (\ref{eq:expre-tildeC})). More precisely, we denote by $\X_{\mathfrak{b}}$ the Gaussian $M \times (B+1)$ i.i.d. matrix  with rows $(\x_{m,\mathfrak{b}})_{m=1, \ldots, M}$, and by $\Gammabs_{\mathfrak{b}}$, 
 $\Gammabs^{1}_{r}, \Gammabs^{2}_{r}$ the $M \times (B+1)$ matrices with rows $(\x_{m,\mathfrak{b}} \Psibs_{m,\mathfrak{b}})_{m=1, \ldots, M}$, $(\x_{m,\mathfrak{b}} \Psibs^{1}_{m,r})_{m=1, \ldots, M}$, 
 $( \x_{m,r}   \Psibs_{m,r}^{2})_{m=1, \ldots, M}$ respectively. Then, the comparison between (\ref{eq:representation-omegam}) and (\ref{eq:representation-omegam-bartlett}) implies that $\X + \Gammabs = \X_{\mathfrak{b}} + \Gammabs_{\mathfrak{b}} + \Gammabs^{1}_r + \Gammabs^{2}_r$, and that matrix $\tilde{\C}$ can be written as 
 \begin{equation}
     \label{eq:expre-tildeC-bartlett}
 \tilde{\C} = \frac{ (\X_{\mathfrak{b}} + \Gammabs_{\mathfrak{b}} +  \Gammabs^{1}_{r} + \Gammabs^{2}_{r})  (\X_{\mathfrak{b}} + \Gammabs_{\mathfrak{b}} +  \Gammabs^{1}_{r} + \Gammabs^{2}_{r})^{*}}{B+1}.
 \end{equation}
The spectral norms of matrices $(\Psibs_{m,\mathfrak{b}})_{m=1, \ldots, M}$ 
and $(\Psibs_{m,r}^{1})_{m=1, \ldots, M}$ are $\mathcal{O}\left(\frac{B}{N}\right)$ terms. Adapting the proof of Proposition 1.1 in \cite{loubaton-rosuel-ejs-2021}, it is easily seen
 that matrices $\frac{ \Gammabs_{\mathfrak{b}}  \Gammabs_{\mathfrak{b}}^{*}}{B+1}$ and $\frac{ \Gammabs^{1}_{r}  (\Gammabs^{1}_{r})^{*}}{B+1}$
 satisfy the concentration inequality (\ref{eq:concentration-normZ-EZ}). Therefore, 
 the result presented in Lemma \ref{le:domination-stochastique-Gamma} can 
 be extended as follows:
 \begin{equation}
     \label{eq:domination-stochastique-Gammab-Gamma1r}
     \| \frac{\Gammabs_{\mathfrak{b}}}{\sqrt{B+1}} \| \prec \frac{B}{N}, \; \| \frac{\Gammabs^{1}_{r}}{\sqrt{B+1}} \| \prec \frac{B}{N}.
 \end{equation}
 Moreover, $ \frac{\Gammabs_{\mathfrak{b}}}{\sqrt{B+1}}$ and $ \frac{\Gammabs^{1}_{r}}{\sqrt{B+1}}$ satisfy (\ref{eq:E-norm-Gamma-over-B}), i.e. 
 \begin{equation}
   \label{eq:E-norm-Gammab-Gamma1r}
   \mathbb{E} \left\| \frac{\Gammabs_{\mathfrak{b}}}{\sqrt{B+1}} \right\|^{k} = \Ocal\left(\left(\frac{B}{N}\right)^{k}\right), \;  \mathbb{E}\left\| \frac{\Gammabs^{1}_{r}}{\sqrt{B+1}} \right\|^{k} = \Ocal\left(\left(\frac{B}{N}\right)^{k}\right).
   \end{equation}
 The matrix $\frac{ \Gammabs^{2}_{r}  (\Gammabs^{2}_{r})^{*}}{B+1}$ also satisfies a concentration inequality similar to
 (\ref{eq:concentration-normZ-EZ}), but in which the term $\left(\frac{B}{N}\right)^{2}$ has to be replaced
 by $\frac{B}{N}$ because $\| \Psibs_{m,r}^{2} \| = \mathcal{O}\left(\left(\frac{B}{N}\right)^{1/2}\right)$. Therefore, $ \frac{\Gammabs^{2}_{r}}{\sqrt{B+1}}$ verifies 
\begin{equation}
     \label{eq:domination-stochastique-Gamma2r}
     \| \frac{\Gammabs^{2}_{r}}{\sqrt{B+1}} \| \prec \left(\frac{B}{N}\right)^{1/2},
\end{equation}
as well as
\begin{equation}
   \label{eq:E-norm-Gamma2r}
    \mathbb{E}\left\| \frac{\Gammabs^{2}_{r}}{\sqrt{B+1}} \right\|^{k} = \Ocal\left(\left(\frac{B}{N}\right)^{k/2}\right).
\end{equation}

\subsection{\texorpdfstring{Bartlett's factorization based approximation of the LSS of $\hat{\C}_N(\nu)$}{Bartlett's factorization}}
\label{subsec:approximation-sup-lss}

If $f$ is a compactly supported function, $\mathcal{C}^{\infty}$ in a neighbourhood 
of $[\lambda_{-}, \lambda_{+}]$, the Bartlett's factorization based approximation of the LSS $\hat{f}(\nu)$, 
denoted $\hat{f}_{\mathfrak{b}}(\nu)$, is defined in the same way than  $\hat{f}(\nu)$
by replacing $\xi_{y_m}(\nu)$ by its Bartlett's factorization $h_m(\nu) \xi_{\epsilon_m}(\nu)$ for each $m$ and each $\nu$. More precisely, 
we denote by $\Sigmabs_{\mathfrak{b}}(\nu)$ the $M \times (B+1)$ matrix defined 
by 
\begin{equation}
    \label{eq:def-Sigmab}
 \Sigmabs_{\mathfrak{b}}(\nu) = \left( \begin{array}{c} \omega_{1,\mathfrak{b}}(\nu) \\ \vdots \\ \omega_{M,\mathfrak{b}}(\nu) \end{array} \right), 
\end{equation}
which, by (\ref{eq:representation-omegaw}), can also be written as 
\begin{equation}
    \label{eq:expre-Sigmab}
  \Sigmabs_{\mathfrak{b}}(\nu) = (\D(\nu))^{1/2} \, \left( \X_{\mathfrak{b}}(\nu) + 
  \Gammabs_{\mathfrak{b}}(\nu)\right).
\end{equation}
Proposition \ref{prop:independence-bartlett-calGN} leads to the obvious corollary. 
\begin{corollary}
\label{coro:independance-XNb-GammaNb}
If $\nu_1$ and $\nu_2$ are 2 different frequencies of $\Gcal_N$, then, the entries of 
$\left(\X_{\mathfrak{b}}(\nu_1), \Gammabs_{\mathfrak{b}}(\nu_1)\right)$ are independent 
from those of $\left(\X_{\mathfrak{b}}(\nu_2), \Gammabs_{\mathfrak{b}}(\nu_2)\right)$.  
\end{corollary}
The Bartlett's factorization based estimate $\hat{\S}_{\mathfrak{b}}(\nu)$ of $\S(\nu)$ is 
given by 
\begin{equation}
    \label{eq:def-HatSb}
 \hat{\S}_{\mathfrak{b}}(\nu) = \frac{  \Sigmabs_{\mathfrak{b}}(\nu)  \Sigmabs_{\mathfrak{b}}(\nu)^{*}}{B+1}.  
\end{equation}
We denote by $\hat{\D}_{\mathfrak{b}}(\nu)$ the diagonal matrix with diagonal entries $\hat{s}_{m,\mathfrak{b}}(\nu) = \left(\hat{\S}_{\mathfrak{b}}(\nu)\right)_{m,m}$ given by 
\begin{equation}
\label{eq:expre-hatsb}
\hat{s}_{m,\mathfrak{b}} = \frac{\| \omegabs_{m,\mathfrak{b}} \|^{2}}{B+1} = s_m \, \frac{\x_{m,\mathfrak{b}}^{*}(\I + \Phibs_{m,\mathfrak{b}}) \x_{m,\mathfrak{b}}}{B+1},
\end{equation}
 where matrix $\Phibs_{m,\mathfrak{b}}$ is the diagonal matrix defined by 
 \begin{equation}
     \label{eq:def-Phimb}
 \Phibs_{m,\mathfrak{b}} = (\I + \Psibs_{m,\mathfrak{b}})  (\I + \Psibs_{m,\mathfrak{b}})^{*} - \I = \Psibs_{m,\mathfrak{b}} + \Psibs_{m,\mathfrak{b}}^{*} +  \Psibs_{m,\mathfrak{b}} \Psibs_{m,\mathfrak{b}}^{*}.  
 \end{equation}
 $\Phibs_{m,\mathfrak{b}}$ can be interpreted as the analogue of matrix $\Phibs_m$
 defined by (\ref{eq:representation-omegam}) because the covariance matrix of $\omegabs_{m,\mathfrak{b}}$ is given by 
 $$
 \mathbb{E}\left(\omegabs_{m,\mathfrak{b}}^{*} \omegabs_{m,\mathfrak{b}} \right) = 
 s_m(\nu) \left( \I + \Phibs_{m,\mathfrak{b}} \right).
 $$
 In particular, using $s_m(\nu) = |h_m(\nu)|^{2}$, it is easily seen that 
 \begin{eqnarray}
\label{eq:bound-norm-Phimb}
\sup_{N \geq 1, \nu \in [0,1]} \| \Phibs_{m,\mathfrak{b}}(\nu) \| & = & \mathcal{O}\left( \frac{B}{N} \right), \\ 
\label{eq:expre-Trace-Phimb}
\frac{1}{B+1} \Tr \Phibs_{m,\mathfrak{b}}(\nu) & = & \frac{1}{2} \frac{s_m^{''}(\nu)}{s_m(\nu)} v_N + \Ocal\left( \left(\frac{B}{N}\right)^{4} \right), \\
\label{eq:bound-trace-Phimb}
\sup_{N \geq 1, \nu \in [0,1]} \left| \frac{1}{B+1} \Tr \Phibs_{m,\mathfrak{b}}(\nu) \right| & = &  \mathcal{O}\left( \left(\frac{B}{N}\right)^{2} \right). 
\end{eqnarray}
We notice in particular that
\begin{equation}
\label{eq:difference-Tr-Phimb-Phim}
\frac{1}{B+1} \Tr \Phibs_{m,\mathfrak{b}}(\nu) - \frac{1}{B+1} \Tr \Phibs_{m}(\nu) = \Ocal\left( \left(\frac{B}{N}\right)^{4} + \frac{1}{N} \right).  
\end{equation}
The Bartlett's factorization based estimate $\hat{\C}_{\mathfrak{b}}(\nu)$ of $\C(\nu)$ is defined 
by 
\begin{equation}
    \label{eq:def-hatCb}
\hat{\C}_{\mathfrak{b}}(\nu) = \left( \hat{\D}_{\mathfrak{b}}(\nu) \right)^{-1/2} \, \hat{\S}_{\mathfrak{b}}(\nu) \, \left( \hat{\D}_{\mathfrak{b}}(\nu) \right)^{-1/2},   
\end{equation}
and the LSS  Bartlett's factorization based estimate $\hat{f}_{N,\mathfrak{b}}(\nu)$ 
is given by
\begin{equation}
    \label{eq:def-hatphib}
 \hat{f}_{N,\mathfrak{b}}(\nu) = \frac{1}{M} \mathrm{Tr} f\left(\hat{\C}_{N,\mathfrak{b}}(\nu) \right) =  \frac{1}{M} \sum_{m=1}^{M} f\left(\lambda_m(\hat{\C}_{N,\mathfrak{b}}(\nu)\right),   
\end{equation}
while we define $\theta_{N,\mathfrak{b}}(f,\nu)$ by 
\begin{equation}
\label{eq:def-thetab}
\theta_{N,\mathfrak{b}}(f,\nu) = \hat{f}_{N,\mathfrak{b}}(\nu) - \int f \, d\mu_{MP}^{c_N} - 
<D_N, f> \left( r_N(\nu) \, v_N - \frac{1}{c_N B} \right).
\end{equation}
The properties of $\hat{f}(\nu)$ derived in \cite{loubaton-rosuel-ejs-2021} extend immediately to 
$\hat{f}_{\mathfrak{b}}(\nu)$ because the results in \cite{loubaton-rosuel-ejs-2021} 
are based on the representation (\ref{eq:representation-omegam}) of $\omegabs_{m}(\nu)$ where
$\Phibs_m(\nu)$ and $\Psibs_m(\nu)$ verify (\ref{eq:bound-norm-Phim}) to (\ref{eq:bound-trace-Phim}) and (\ref{eq:bound-norm-Psim}), (\ref{eq:bound-trace-Psim}) respectively.  (\ref{eq:representation-omegam}) is now replaced by the representation (\ref{eq:representation-omegaw}) and the matrices $\Phibs_{m,\mathfrak{b}}(\nu)$ and 
$\Psibs_{m,\mathfrak{b}}(\nu)$ are now diagonal and verify (\ref{eq:bound-norm-Phimb}) to (\ref{eq:bound-trace-Phimb}) and (\ref{eq:proprietes-Psimw}) respectively, while matrix $\Gammabs_{\mathfrak{b}}(\nu)$ introduced above verifies (\ref{eq:domination-stochastique-Gammab-Gamma1r}). In particular, the approach of \cite{loubaton-rosuel-ejs-2021} can be adapted to establish the analog of the stochastic representations (\ref{eq:representation-tildeC}) and (\ref{eq:representation-hatC}), i.e., if 
$\tilde{\C}_{\mathfrak{b}}(\nu)$ is defined by 
\begin{equation}
    \label{eq:def-tildeCb}
 \tilde{\C}_{\mathfrak{b}}(\nu) = \left( \D(\nu) \right)^{-1/2} \, \hat{\S}_{\mathfrak{b}}(\nu) \, 
 \left( \D(\nu) \right)^{-1/2} = \frac{(\X_{\mathfrak{b}}(\nu) + 
  \Gammabs_{\mathfrak{b}}(\nu))\X_{\mathfrak{b}}(\nu) + 
  \Gammabs_{\mathfrak{b}}(\nu))^{*}}{B+1},
\end{equation}
then, we have 
\begin{eqnarray}
\label{eq:representation-tildeCb}
\tilde{\C}_{N,\mathfrak{b}}(\nu) & = & \frac{\X_{N,\mathfrak{b}}(\nu) \X_{N,\mathfrak{b}}(\nu)^{*}}{B+1} + \tilde{\Deltabs}_{N,\mathfrak{b}}(\nu), \\
\label{eq:representation-hatCb}
\hat{\C}_{N,\mathfrak{b}}(\nu)  & = & \frac{\X_{N,\mathfrak{b}}(\nu) \X_{N,\mathfrak{b}}(\nu)^{*}}{B+1} + \Deltabs_{N,\mathfrak{b}}(\nu).
\end{eqnarray}
The families $\| \tilde{\Deltabs}_{N,\mathfrak{b}}(\nu) \|, N \geq 1, \nu \in [0,1]$ and  $\| \Deltabs_{N,\mathfrak{b}}(\nu) \|, N \geq 1, \nu \in [0,1]$ verify 
\begin{eqnarray}
\label{eq:domination-tildeDeltab}
\| \tilde{\Deltabs}_{\mathfrak{b}} \| \prec \frac{B}{N}, \\
\label{eq:domination-Deltab-1}
\| \Deltabs_{\mathfrak{b}} \| \prec \frac{B}{N} + \frac{1}{\sqrt{B}}.
\end{eqnarray}
Moreover, matrix $\Deltabs_{\mathfrak{b}}$ is given by 
\begin{equation}
    \label{eq:def-Deltabsb}
 \Deltabs_{\mathfrak{b}} =   \tilde{\Deltabs}_{\mathfrak{b}} + \Thetabs_{\mathfrak{b}}  ,
\end{equation}
where $\Thetabs_{\mathfrak{b}} = \hat{\C}_{\mathfrak{b}} - \tilde{\C}_{\mathfrak{b}}$ can be written as 
 \begin{align}
\label{equation:decomposition_Thetab}
    \Thetabs_{\mathfrak{b}} = (\hat{\D}_{\mathfrak{b}}^{-1/2}-\D_{\mathfrak{b}}^{-1/2})\hat{\S}_{\mathfrak{b}} \, \hat{\D}_{\mathfrak{b}}^{-1/2} + \D_{\mathfrak{b}}^{-1/2}\, \hat{\S}_{\mathfrak{b}}(\hat{\D}_{\mathfrak{b}}^{-1/2}-\D_{\mathfrak{b}}^{-1/2}),
\end{align}
and verifies 
\begin{equation}
\label{eq:domination-hatCb-tildeCb}
\| \Thetabs_{\mathfrak{b}} \| =  \| \hat{\C}_{\mathfrak{b}} - \tilde{\C}_{\mathfrak{b}} \| \prec \frac{1}{\sqrt{B}} + 
\left( \frac{B}{N} \right)^{2},
\end{equation}
because $|\hat{s}_{m,\mathfrak{b}} - s_m| \prec \frac{1}{\sqrt{B}} + \left( \frac{B}{N} \right)^{2}$. 
We recall that $\alpha < 4/5$ implies that $\left( \frac{B}{N} \right)^{2} = o(B^{-1/2})$. 
In particular, matrix $\hat{\D}_{\mathfrak{b}}$ verifies (\ref{eq:domination-hatD-D-alpha-inferieur-4-5}) and (\ref{eq:E-norm-hatD-D}), i.e. 
\begin{eqnarray}
\label{eq:domination-hatDb-D-alpha-inferieur-4-5}
\| \hat{\D}_{\mathfrak{b}} - \D \| & \prec & \frac{1}{\sqrt{B}}, \\
\label{eq:E-norm-hatDb-D}
\mathbb{E}\left( \| \hat{\D}_{\mathfrak{b}} - \D) \|^{k} \right) & = & \Ocal\left(B^{-k/2 + \epsilon}\right),
\end{eqnarray}
for each $\epsilon > 0$. Moreover,  $\| \hat{\C}_{\mathfrak{b}} - \tilde{\C}_{\mathfrak{b}} \| \prec \frac{1}{\sqrt{B}}$, and it will be shown in the following that $\mathbb{E}(\| \hat{\C}_{\mathfrak{b}} - \tilde{\C}_{\mathfrak{b}} \|^{k}) = \Ocal\left(B^{-k/2 + \epsilon}\right)$.
We finally notice that the results presented in section \ref{subsec:background-study-hatphi} still hold in the context of $\hat{f}_{\mathfrak{b}}(\nu)$. In particular, if 
we denote by $\tilde{\Q}_{N,\mathfrak{b}}(z)$ and $\Q_{N,\mathfrak{b}}(z)$ the resolvents of matrices 
$\tilde{\C}_{N,\mathfrak{b}}$ and $\frac{\X_{N,\mathfrak{b}} \X_{N,\mathfrak{b}}^{*}}{B+1}$ respectively, 
then, following (\cite{these-alexis}, Chap. 2) and the Appendix \ref{sec:improvement-bias}, it can be shown that 
\begin{align}
    \label{eq:Etrace-tildeQb-Qb}
 \mathbb{E}\left( \frac{1}{M} \Tr \left( \tilde{\Q}_{N,\mathfrak{b}}(z) - \Q_{N,\mathfrak{b}}(z) \right) \right) = & 
 p_N(z)  \left( \frac{1}{B+1} \Tr \left( \frac{1}{M} \sum_{m=1}^{M}  \Phibs_{m,\mathfrak{b}} \right)^{2} \right) \\ 
 \notag 
 & 
 - \tilde{p}_N(z)  \left( \frac{1}{M} \sum_{m=1}^{M} \frac{1}{B+1} \Tr \Phibs_{m,\mathfrak{b}} \right) \\ \notag 
 & 
 +  \mathcal{O}_z\left( \left( \frac{B}{N}\right)^{4} + \frac{1}{B^{2}} \right) ,  
 \end{align}
where we recall that $p_N$ and $\tilde{p}_N$ are defined by (\ref{eq:def-p}) and (\ref{eq:def-tildep}).
 A simple calculation shows that 
  \begin{equation}
\label{eq:behaviour-trace-sum-Phimb-carre}
  \left( \frac{1}{B+1} \Tr \left( \frac{1}{M} \sum_{m=1}^{M}  \Phibs_{m,\mathfrak{b}} \right)^{2} \right) = \left( \frac{1}{M} \sum_{m=1}^{M} \frac{s_m'}{s_m} \right)^{2} v_N + \mathcal{O}\left( \left( \frac{B}{N}\right)^{4} \right) ,
 \end{equation}
 and 
 \begin{equation}
\label{eq:behaviour-moyenne-trace-Phimb}
 \frac{1}{M} \sum_{m=1}^{M} \frac{1}{B+1} \Tr \Phibs_{m,\mathfrak{b}} =  \left( \frac{1}{2M} \sum_{m=1}^{M} \frac{s^{''}_m}{s_m} \right) v_N  + 
  \mathcal{O}_z\left( \left( \frac{B}{N}\right)^{4}  \right) .
\end{equation}
This leads immediately to 
\begin{align}
    \label{eq:Etrace-tildeQb-Qb-precise}
 \mathbb{E}\left( \frac{1}{M} \Tr \left( \tilde{\Q}_{N,\mathfrak{b}}(z) - \Q_{N,\mathfrak{b}}(z) \right) \right)  =  &  
 p_N(z) r_N(\nu) v_N  - \tilde{p}_N(z) \tilde{r}_N(\nu) v_N \\ \notag 
 & 
 +  \mathcal{O}_z\left( \left( \frac{B}{N}\right)^{4} + \frac{1}{B^{2}} \right) .
 \end{align}
We finally remark that, as the equality $ \mathbb{E}\left( \frac{1}{M} \Tr \Q_{N,\mathfrak{b}}(z)\right) = \mathbb{E}\left( \frac{1}{M} \Tr \Q_{N}(z)\right)$ of course holds, (\ref{eq:Etrace-tildeQ-Q-precise-2-improved}) and (\ref{eq:Etrace-tildeQb-Qb}) imply that 
\begin{align}
\label{eq:egalite-biais-barlett}  
  \mathbb{E}\left( \frac{1}{M} \Tr \left( \tilde{\Q}_{N}(z) - \tilde{\Q}_{N,\mathfrak{b}}(z) \right) \right) = & 
  \mathcal{O}_z\left( \left( \frac{B}{N}\right)^{4} + \frac{1}{N} + \frac{1}{B^{2}} \right) \\
   = & \mathcal{O}_z\left( \left( \frac{B}{N}\right)^{4} + \frac{1}{N}  \right) ,
\end{align}
because $\frac{1}{B^{2}} = o\left( \frac{1}{N}\right)$ (recall that $\alpha > \frac{1}{2}$). 

\section{\texorpdfstring{CLT for $B \theta_{N,\mathfrak{b}}(f,\nu)$ at a given frequency}{CLT}}
\label{sec:clt-theta-given-frequency}

We first establish in Subsection \ref{subsec:simplification-hatphib} that if $\alpha < \frac{4}{5}$, the representation (\ref{eq:representation-B-thetaNb}) holds. In Subsection 
\ref{subsec:CLT-W}, we prove that $\frac{B w_{N,\mathfrak{b}}(f,\nu)}{\sigma_N(f)} \rightarrow_{\Dcal} \Ncal(0,1)$ for some variance term $\sigma_N^{2}(f)$ that is expressed as a double contour integral, 
and deduce from the representation (\ref{eq:representation-B-thetaNb}) that 
$\frac{B \theta_{N,\mathfrak{b}}(f,\nu)}{\sigma_N(f)} \rightarrow_{\Dcal} \Ncal(0,1)$. 
We then argue that  the similarity of the models defining vectors $(\omegabs_m)_{m=1, \ldots, M}$ and $(\omegabs_{m,\mathfrak{b}})_{m=1, \ldots, M}$ leads to the conclusion that $B \theta_N(f,\nu)$ has the same properties than $B \theta_{N,\mathfrak{b}}(f,\nu)$, i.e. (\ref{eq:representation-B-thetaN}) holds and $\frac{B \theta_{N}(f,\nu)}{\sigma_N(f)} \rightarrow_{\Dcal} \Ncal(0,1)$. 
\subsection{Proof of representation (\ref{eq:representation-B-thetaNb})}
\label{subsec:simplification-hatphib}
In order to derive (\ref{eq:representation-B-thetaNb}), we approximate 
$\theta_{N,\mathfrak{b}}(f,\nu)$ by a simpler expression which will represent 
the term $w_{N,\mathfrak{b}}(f,\nu)$. In order to introduce the corresponding result, 
we first remark that (\ref{eq:expre-hatsb}) implies that matrix $ (\hat{\D}_{\mathfrak{b}}(\nu) - \D(\nu)) (\D(\nu))^{-1}$ can be written as
\begin{equation}
\label{eq:decomposition-Dx-I}
(\hat{\D}_{\mathfrak{b}}(\nu) - \D(\nu)) (\D(\nu))^{-1} = \D_{\x_{\mathfrak{b}}(\nu)} - \I + \D_{2,\mathfrak{b}}(\nu) +  \D_{3,\mathfrak{b}}(\nu),
\end{equation}
where $\D_{\x_{\mathfrak{b}}(\nu)}$, $\D_{2,\mathfrak{b}}(\nu)$ and  $\D_{3,\mathfrak{b}}(\nu)$ are the three diagonal matrices defined by 
\begin{eqnarray}
   \label{eq:def-Dx}
\D_{\x_{\mathfrak{b}}(\nu)} & = & \mathrm{dg}\left( \frac{1}{B+1} \| \x_{m,\mathfrak{b}}(\nu)\|^{2}, m=1, \ldots, M \right), \\
\label{eq:def-D2b}
\D_{2,\mathfrak{b}}(\nu) & = & \dg \left( \frac{\x_{m,\mathfrak{b}}(\nu) \Phibs_{m,\mathfrak{b}}(\nu) \x_{m,\mathfrak{b}}(\nu)^{*}}{B+1} - \frac{1}{B+1} \Tr \Phibs_{m,\mathfrak{b}}(\nu)\right)_{m=1, \ldots, M},  \\
\label{eq:def-D3b}
\D_{3,\mathfrak{b}}(\nu) & = & \dg \left( \frac{1}{B+1} \Tr \Phibs_{m,\mathfrak{b}}(\nu), m=1, \ldots, M \right).
\end{eqnarray}
These matrices are easily seen to verify 
\begin{align}
\label{eq:Dx-I-carre-mm}
\mathbb{E} \left((\D_{\x_{\mathfrak{b}}} - \I)_{m,m}^{2} \right)  = &
\frac{1}{B+1}, \\
\label{eq:trace-Dx-I-carre}
\mathbb{E}\left( \frac{1}{M} \Tr (\D_{\x_{\mathfrak{b}}} - \I)^{2} \right)  = &
\frac{1}{B+1}, \\
\label{eq:D2-carre-mm}
\mathbb{E}\left( (\D_{2,\mathfrak{b}})_{m,m}^{2} \right)  = & 
\Ocal\left( \frac{B}{N^{2}}\right), \\
\label{eq:trace-D2-carre}
\mathbb{E}\left( \frac{1}{M} \Tr \D_{2,\mathfrak{b}}^{2} \right)  = & 
\Ocal\left( \frac{B}{N^{2}}\right),
\end{align}
as well as
\begin{eqnarray}
\label{eq:domination-Dxb-I}
\| \D_{\x_{\mathfrak{b}}} - \I \| & \prec & \frac{1}{\sqrt{B}}, \\
\label{eq:domination-D2b}
\| \D_{2,\mathfrak{b}} \| & \prec & \frac{\sqrt{B}}{N}, \\
\label{eq:E-norm-Dxb-I}
\mathbb{E}\left( \|  \D_{\x_{\mathfrak{b}}} - \I \|^{k} \right) & = & \Ocal\left(B^{-k/2 + \epsilon}\right), \\
\label{eq:E-norm-D2b}
\mathbb{E}\left( \|  \D_{2,\mathfrak{b}} \|^{k} \right) & = & \Ocal\left(\left(\frac{\sqrt{B}}{N}\right)^{k/2 + \epsilon}\right), \\
\label{eq:norm-D3b}
\| \D_{3,\mathfrak{b}} \| = \Ocal\left(\frac{B^{2}}{N^{2}}\right),
\end{eqnarray}
for each $\epsilon > 0$ (see (\ref{eq:sup-hanson-wright-domination-stochastique}), (\ref{eq:moyenne-sup-hanson-wright-D}), (\ref{eq:bound-trace-Phimb})). 
In this Subsection, we establish the following theorem. 
\begin{theorem}
\label{th:approximation-thetaNb}
If $\alpha < \frac{4}{5}$, $\theta_{N,\mathfrak{b}}(f,\nu)$ can be written as 
\begin{equation}
    \label{eq:approximation-thetaNb}
\theta_{N,\mathfrak{b}}(f,\nu) = w_{N,\mathfrak{b}}(f,\nu) + \epsilon_{N,\mathfrak{b}}(f,\nu) ,   
\end{equation}
where $w_{N,\mathfrak{b}}(f,\nu)$ is defined by 
\begin{align}
&w_{N,\mathfrak{b}}(f,\nu) =  
\notag\\
&\qquad\qquad\frac{1}{\pi} \mathrm{Re} \int_{\Dcal} \bar{\partial}  \Phi_k(f)(z) \,  
\Bigl( 
\frac{1}{M} \Tr \Q_{N,\mathfrak{b}}(z,\nu) \D_{\x_{\mathfrak{b}}(\nu)} 
\notag\\
&\qquad\qquad\qquad\qquad +z \frac{1}{M} \Tr \Q^{2}_{N,\mathfrak{b}}(z,\nu) (\D_{\x_{\mathfrak{b}}(\nu)} - \I) 
\Bigr)^{\circ} \, \diff x \diff y,
\label{eq:expre-wNb}
\end{align}
and where the family $(\epsilon_{N,\mathfrak{b}}(f,\nu), \nu \in [0,1])$ verifies 
\begin{eqnarray}
    \label{eq:domination-stochastique-erreur-approximation-clt-4/5}
 \left(\epsilon_{N,\mathfrak{b}}(f,\nu)\right)^{\circ} & = & o_{\prec}\left( \frac{1}{\sqrt{NB}}\right), \\
  \label{eq:E-erreur-approximation-clt-4/5}
  \mathbb{E} \left(\epsilon_{N,\mathfrak{b}}(f,\nu)\right) & = & \mathbb{E}\left( \theta_{N,\mathfrak{b}}(f,\nu) \right) =   \Ocal\left( \left( \frac{B}{N} \right)^{4} \right) +  o\left( \frac{1}{\sqrt{NB}}\right).
\end{eqnarray}
Moreover, the random variables $(w_{N,\mathfrak{b}}(f,\nu))_{\nu \in \Gcal_N}$ 
are independent and identically distributed.
\end{theorem}
\begin{proof}
That $(w_{N,\mathfrak{b}}(f,\nu))_{\nu \in \Gcal_N}$ is a i.i.d. sequence is an obvious 
consequence of Corollary \ref{coro:independance-XNb-GammaNb}. In order to establish the 
rest of Theorem \ref{th:approximation-thetaNb}, we study the behaviour of the function $\gamma_N(z,f)$ defined by
\begin{equation}
  \label{eq:def-gamma-f}
  \gamma_{N,\mathfrak{b}}(z) = \frac{1}{M} \Tr \hat{\Q}_{N,\mathfrak{b}}(z) - t_N(z)- p_N(z) \,  \left( r_N(\nu) v_N - (c_N B)^{-1} \right),
\end{equation}
because the Helffer-Sjöstrand formula leads to
\begin{equation}
  \label{eq:HS-formula-thetaNb-f}
\theta_{N,\mathfrak{b}}(f,\nu) =  \frac{1}{\pi} \mathrm{Re} \int_{\Dcal} \bar{\partial}  \Phi_k(f)(z) \, \gamma_{N,\mathfrak{b}}(z) \, \diff x \diff y  . 
\end{equation}
In the following, we express $\gamma_{N,\mathfrak{b}}(z,f)$ as
\begin{equation}
  \label{eq:decomposition-gammab}
 \gamma_{N,\mathfrak{b}}(z) = \gamma_{1,N,\mathfrak{b}}(z) + \gamma_{2,N,\mathfrak{b}}(z) +  \frac{1}{M} \Tr \Q_{N,\mathfrak{b}}(z) - t_N(z),
\end{equation}
where $\gamma_{1,N,\mathfrak{b}}(z)$ and $\gamma_{{2,N,\mathfrak{b}}}(z)$ are defined by
\begin{align}
  \label{eq:def-gamma1b}
  \gamma_{1,N,\mathfrak{b}}(z) = & \frac{1}{M} \Tr \hat{\Q}_{N,\mathfrak{b}}(z) - \frac{1}{M} \Tr \tilde{\Q}_{N,\mathfrak{b}}(z) \\ \notag 
  & 
 - \left( \tilde{p}_N(z) \tilde{r}_N(\nu) v_N  - \frac{p_N(z)}{c_N} \frac{1}{B} \right) \\
  \label{eq:def-gamma2b}
  \gamma_{2,N,\mathfrak{b}}(z,f) = & \frac{1}{M} \Tr \tilde{\Q}_{N,\mathfrak{b}}(z) - \frac{1}{M} \Tr \Q_{N,\mathfrak{b}}(z) \\ \notag 
  & 
 - \left(p_N(z) r_N(\nu) v_N -  \tilde{p}_N(z) \tilde{r}_N(\nu) v_N \right), 
\end{align}
Therefore, $\theta_{N,\mathfrak{b}}(f,\nu)$ can be written as
\begin{equation}
  \label{eq:decomposition-thetaN-f}
  \theta_{N,\mathfrak{b}}(f,\nu) = \theta_{1,N,\mathfrak{b}}(f,\nu) + \theta_{2,N,\mathfrak{b}}(f,\nu) + \frac{1}{M} \Tr f \left( \frac{\X_{\mathfrak{b}} \X_{\mathfrak{b}}^{*}}{B+1} \right) - \int f \, d\mu_{MP}^{(c_N)},
\end{equation}
where $(\theta_{i,N,\mathfrak{b}}(f,\nu))_{i=1,2}$ are defined by
\begin{equation}
  \label{eq:HS-formula-thetaiNb-f}
\theta_{i,N,\mathfrak{b}}(f,\nu) =  \frac{1}{\pi} \mathrm{Re} \int_{\Dcal} \bar{\partial}  \Phi_k(f)(z) \, \gamma_{i,N,\mathfrak{b}}(z) \, \diff x \diff y.   
\end{equation}
In the following, we study the 3 terms of the decomposition (\ref{eq:decomposition-thetaN-f}). 
\\

The behaviour of the third term
of (\ref{eq:decomposition-thetaN-f})  is well known because it is well established
that $\mathbb{E} \left(  \frac{1}{M} \Tr \Q_{N,\mathfrak{b}}(z) - t_N(z) \right) = \Ocal_z(B^{-2}) = o_z\left( \frac{1}{\sqrt{BN}} \right)$, a result which implies 
that $\mathbb{E} \left( \frac{1}{M} \Tr f \left( \frac{\X_{\mathfrak{b}} \X_{\mathfrak{b}}^{*}}{B+1} \right) \right)- \int f \, d\mu_{MP}^{(c_N)} =  \Ocal(B^{-2}) = o\left( \frac{1}{\sqrt{BN}} \right)$. Therefore, we also have 
\begin{align}
  \label{eq:behaviour-Qb-t}
  \frac{1}{M} \Tr \Q_{N,\mathfrak{b}}(z) - t_N(z) =  & \frac{1}{M} \Tr \Q^{\circ}_{N,\mathfrak{b}}(z) + o_z\left( \frac{1}{\sqrt{BN}} \right), \\
  \label{eq:behaviour-trfXX-intf}
  \frac{1}{M} \Tr f \left( \frac{\X_{\mathfrak{b}} \X_{\mathfrak{b}}^{*}}{B+1} \right) - \int f \, d\mu_{MP}^{(c_N)}  = & \frac{1}{\pi} \mathrm{Re} \int_{\Dcal} \bar{\partial}  \Phi_k(f)(z) \, 
  \frac{1}{M} \Tr \Q^{\circ}_{N,\mathfrak{b}}(z) \, \diff x \diff y \\ \notag 
  & 
  +  o\left( \frac{1}{\sqrt{BN}} \right).
\end{align}
The behaviour of $\theta_{2,N,\mathfrak{b}}(f)$ is given by the following Proposition 
established in the Paragraph \ref{subsubsec:study-theta2}. 
\begin{proposition}
  \label{prop:behaviour-theta2b-f}
   $\theta_{2,N,\mathfrak{b}}(f)$ verifies
  \begin{equation}
    \label{eq:behaviour-E-theta2b-f}
    \mathbb{E} (\theta_{2,N,\mathfrak{b}}(f)) = \Ocal\left( \frac{B^{4}}{N^{4}} + \frac{1}{B^{2}} \right)  = \Ocal\left( \frac{B^{4}}{N^{4}} \right) +  o \left( \frac{1}{\sqrt{BN}} \right),
    \end{equation}
  and
  \begin{equation}
    \label{eq:behaviour-theta2b-f-recentered}
    \left(\theta_{2,N,\mathfrak{b}}(f) \right)^{\circ} = \Ocal_{\prec}\left( \frac{1}{N} \right) =  o_{\prec} \left( \frac{1}{\sqrt{BN}} \right).
  \end{equation}
\end{proposition}
In order to characterize $\theta_{1,N,\mathfrak{b}}(f)$, we introduce the term 
$\bar{\theta}_{1,N,\mathfrak{b}}(f)$ given by 
\begin{equation}
    \label{eq:def-bar-theta1b}
\bar{\theta}_{1,N,\mathfrak{b}}(f) =  \frac{1}{\pi} \mathrm{Re} \int_{\Dcal} \bar{\partial}  \Phi_k(f,z) \, \left( \frac{1}{M} \Tr (\Q_{N,\mathfrak{b}}(z) + z \Q^{2}_{N,\mathfrak{b}}(z))(\D_{\x_{\mathfrak{b}}} - \I) \right) \, \diff x \diff y.   
\end{equation}
Then, we have following result, proved in Paragraph \ref{subsubsec:simplification-hatphib-recentre} and in Appendix \ref{sec:details-proof-theorem2}.
\begin{theorem}
\label{th:simplification-theta1b}
For $\alpha < \frac{4}{5}$, $\theta_{1,N,\mathfrak{b}}(f)$ verifies
\begin{equation}
    \label{eq:simplification-hatphib-recentre}
\theta_{1,N,\mathfrak{b}}(f) - \mathbb{E}\left( \theta_{1,N,\mathfrak{b}}(f) \right) = 
\bar{\theta}_{1,N,\mathfrak{b}}(f) - \mathbb{E} \left( \bar{\theta}_{1,N,\mathfrak{b}}(f) \right) + o_{\prec}\left( \frac{1}{\sqrt{BN}}\right),
\end{equation}
as well as 
\begin{equation}
\label{eq:behaviour-E-theta1b}
    \mathbb{E}\left( \theta_{1,N,\mathfrak{b}}(f) \right) =  \Ocal\left( \left(\frac{B}{N}\right)^{4} \right) + o\left( \frac{1}{\sqrt{BN}}\right).
\end{equation}
\end{theorem}
We finally gather the results of Proposition \ref{prop:behaviour-theta2b-f} and Theorem  \ref{th:simplification-theta1b}, and complete the proof of Theorem  \ref{th:approximation-thetaNb}. For this, we take the decomposition (\ref{eq:decomposition-thetaN-f}) as a starting
point. 
  We add (\ref{eq:simplification-hatphib-recentre}) and (\ref{eq:behaviour-E-theta1b}), and remark that
  the $o_{\prec}\left( \frac{1}{\sqrt{BN}}\right)$ term at the right-hand side of (\ref{eq:simplification-hatphib-recentre}) is a zero mean term.
  We thus obtain
  that $ \theta_{1,N,\mathfrak{b}}(f)$ can be written as 
  \begin{align*}
  \theta_{1,N,\mathfrak{b}}(f) = & \frac{1}{\pi} \mathrm{Re} \int_{\Dcal} \bar{\partial}  \Phi_k(f)(z) \,   \left(   \frac{1}{M} \Tr (\Q_{N,\mathfrak{b}}(z) + z  \Q^{2}_{N,\mathfrak{b}}(z)) \, (\D_{\x_{\mathfrak{b}}} - \I) \right)^{\circ}  \diff x \diff y  \\ 
  & 
  + \epsilon_{1,N,\mathfrak{b}}(f),
  \end{align*}
where $\epsilon_{1,N,\mathfrak{b}}(f,\nu)$ verifies $\left(\epsilon_{1,N,\mathfrak{b}}(f,\nu)\right)^{\circ} = o_{\prec}\left( \frac{1}{\sqrt{BN}}\right)$
  and $$\mathbb{E}(\epsilon_{1,N,\mathfrak{b}}(,\nu)) = \mathbb{E}(\theta_{1,N,\mathfrak{b}}(f)) =   \Ocal\left( \left(\frac{B}{N}\right)^{4} \right) + o\left( \frac{1}{\sqrt{BN}}\right).$$  
Proposition \ref{prop:behaviour-theta2b-f} implies that
 $\theta_{2,N,\mathfrak{b}}(f) = \epsilon_{2,N,\mathfrak{b}}(f,\nu)$ where $\epsilon_{2,N,\mathfrak{b}}(f)$ verifies $\left(\epsilon_{2,N,\mathfrak{b}}(f)\right)^{\circ} = o_{\prec}\left( \frac{1}{\sqrt{BN}}\right)$
  and $$\mathbb{E}(\epsilon_{2,N,\mathfrak{b}}(f)) = \mathbb{E}(\theta_{2,N,\mathfrak{b}}(f)) =   \Ocal\left( \left(\frac{B}{N}\right)^{4} \right) + o\left( \frac{1}{\sqrt{BN}}\right).$$ We finally observe that (\ref{eq:behaviour-Qb-t}) implies that 
\begin{align*}
   &\frac{1}{M} \mathrm{Tr} f \left( \frac{\X_{\mathfrak{b}} \X_{\mathfrak{b}}^{*}}{B+1} \right) - \int f \, d\mu_{MP}^{(c_N)} = 
   \\
   &\qquad\frac{1}{\pi} \mathrm{Re} \int_{\Dcal} \bar{\partial}  \Phi_k(f)(z) \,  \frac{1}{M} \mathrm{Tr}  \Q_{N,\mathfrak{b}}^{\circ}(z) \, \diff x \diff y  + \epsilon_{3,N,\mathfrak{b}}(f) ,
\end{align*}
   where $$\epsilon_{3,N,\mathfrak{b}}(f) = \mathbb{E} \left(  \frac{1}{M} \mathrm{Tr} f \left( \frac{\X_{\mathfrak{b}}\X_{\mathfrak{b}}^{*}}{B+1} \right)\right) - \int f \, d \mu^{c_N}_{MP} =  o\left( \frac{1}{\sqrt{BN}}\right).$$ This, in turn, implies that $\theta_{N,\mathfrak{b}}(f)$
   verifies (\ref{eq:approximation-thetaNb}) where $ \epsilon_{N,\mathfrak{b}}(f)$ is defined by
   $ \epsilon_{N,\mathfrak{b}}(f) = \sum_{i=1}^{3}  \epsilon_{i,N,\mathfrak{b}}(f)$. 
\end{proof}

\subsubsection{Proof of Proposition \ref{prop:behaviour-theta2b-f}}
\label{subsubsec:study-theta2}
\begin{proof}
  (\ref{eq:behaviour-E-theta2b-f}) is an immediate consequence of (\ref{eq:Etrace-tildeQb-Qb-precise}) and of the
  Helffer-Sjöstrand formula. In order to
  establish (\ref{eq:behaviour-theta2b-f-recentered}), we apply Lemma \ref{le:concentration-integrale-helffer-sjostrand} for $U^{(N)} = [0,1]$ which thus coincides with the set of all frequencies $\nu$. 
  For $u=\nu$, $\X_N(u)$ corresponds to the $M(B+1)$--dimensional vector $\mathrm{vec}(\X_{N,\mathfrak{b}}(\nu))$, $h_N(\X_N(u), \X_N(u)^{*})$ and $q_N(z,\X_N(u), \X_N(u)^{*})$ are $\theta_{N,\mathfrak{b}}(f,\nu)$ and $\gamma_{2,N,\mathfrak{b}}(z,f,\nu)$  respectively. We omit to mention from now on that the previous terms depend on the frequency $\nu$, except when we introduce the relevant family of events $A_{N}(\nu)$. We now evaluate $\| \nabla  \gamma_{2,N,\mathfrak{b}}(z,f) \|^{2}$. $\nabla \gamma_{2,N,\mathfrak{b}}(z,f)$ is given by
$$
\nabla \gamma_{2,N,\mathfrak{b}}(z,f) = 
\nabla \frac{1}{M} \Tr \left(\tilde{\Q}_{N,\mathfrak{b}}(z) - \Q_{N,\mathfrak{b}}(z)\right),
$$
We express $ \frac{1}{M} \Tr \left(\tilde{\Q}_{N,\mathfrak{b}}(z) - \Q_{N,\mathfrak{b}}(z)\right)$ as 
\begin{align*}
& -\frac{1}{M} \Tr \tilde{\Q}_{N,\mathfrak{b}}(z) \left( \tilde{\C}_{N,\mathfrak{b}} - \frac{\X_{N,\mathfrak{b}} \X_{N,\mathfrak{b}}^{*}}{B+1} \right) \Q_{N,\mathfrak{b}}(z) = \\
& -\frac{1}{M} \Tr \tilde{\Q}_{N,\mathfrak{b}}(z) \left( \frac{(\X_{N,\mathfrak{b}} + \Gammabs_{N,\mathfrak{b}}) (\X_{N,\mathfrak{b}} + \Gammabs_{N,\mathfrak{b}})^{*}}{B+1} - \frac{\X_{N,\mathfrak{b}} \X_{N,\mathfrak{b}}^{*}}{B+1} \right) \Q_{N,\mathfrak{b}}(z) = \\
& -\frac{1}{M} \Tr \tilde{\Q}_{N,\mathfrak{b}}(z) \left( \frac{\X_{N,\mathfrak{b}}  \Gammabs_{N,\mathfrak{b}}^{*} + \Gammabs_{N,\mathfrak{b}} \X_{N,\mathfrak{b}}^{*} +  \Gammabs_{N,\mathfrak{b}}  \Gammabs_{N,\mathfrak{b}}^{*}}{B+1} \right) \Q_{N,\mathfrak{b}}(z).
\end{align*}
Using the resolvent identity (\ref{eq:resolvent-identity}) for matrices $\frac{\X_{N,\mathfrak{b}} \X_{N,\mathfrak{b}}^{*}}{B+1}$ and $\tilde{\C}_{N,\mathfrak{b}}$ as well as
$\| \Psibs_{m,\mathfrak{b}} \| = \Ocal\left(\frac{B}{N}\right)$, a straightforward calculation implies that 
\begin{equation}
\label{eq:majoration-norme-gradient-gamma-2-z}
 \| \nabla \gamma_{2,\mathfrak{b}}(z,f) \|^{2} \leq 
 \frac{C(z)}{B^{2}} \left( T_1 + T_2 \right),
\end{equation}
where $T_1 = \frac{1}{M} \Tr \tilde{\Deltabs}_{N,\mathfrak{b}} \tilde{\Deltabs}_{N,\mathfrak{b}}^{*}$ and 
$$
T_2 = \frac{1}{M} \Tr \frac{\Gammabs_{N,\mathfrak{b}} \Gammabs_{N,\mathfrak{b}}^{*}}{B+1} + C \frac{B^{2}}{N^{2}}  \frac{1}{M} \Tr \tilde{\C}_{N,\mathfrak{b}}.
$$
For each $\delta > 0$, we denote by $A_{N,\delta}(\nu)$ the event defined by 
\begin{equation}
    \label{eq:def-event-A}
A_{N,\delta}(\nu) = \{ \| \frac{\X_{\mathfrak{b}}(\nu)}{\sqrt{B+1}} \| \leq 3,  \| \frac{\Gammabs_{\mathfrak{b}}(\nu)}{\sqrt{B+1}} \| \leq N^{\delta} \frac{B}{N} \}.
\end{equation}
Proposition \ref{prop:Lambda-holds-with-high-proba} and $ \| \frac{\Gammabs_{\mathfrak{b}}}{\sqrt{B+1}} \| \prec \frac{B}{N}$ (see (\ref{eq:domination-stochastique-Gammab-Gamma1r})) imply that there exits $\gamma > 0$
for which $\sup_{\nu} P(A_{N,\delta}(\nu)) \leq e^{-N^{\gamma}}$ for each $N$ large enough. Moreover, on 
the set  $A_{N,\delta}$, matrix $\tilde{\Deltabs}_{N, \mathfrak{b}}$ 
verifies $\| \tilde{\Deltabs}_{N, \mathfrak{b}} \| \leq C N^{\delta} \frac{B}{N}$
while $\tilde{\C}_{N, \mathfrak{b}}$ satisfies $\| \tilde{\C}_{N, \mathfrak{b}} \| \leq C$. We deduce immediately from (\ref{eq:majoration-norme-gradient-gamma-2-z}) that on the event $A_{N,\delta}$, 
we have 
$$
 \| \nabla \gamma_{2,\mathfrak{b}}(z,f) \|^{2} \leq N^{2 \delta} \frac{C(z)}{N^{2}}.
$$
Moreover, $\gamma_{2,\mathfrak{b}}(z,f)$ is clearly a $\Ocal_{z}(1)$ term, while 
the set $\X_{N,\mathfrak{b}}(A_{N,\delta})$ is convex. Therefore, Lemma \ref{le:concentration-integrale-helffer-sjostrand} leads to the conclusion that 
for each $\delta > 0$, the family $((\theta_{2,N,\mathfrak{b}}(f,\nu))^{\circ}, \nu \in [0,1]$ 
verifies $(\theta_{2,N,\mathfrak{b}}(f,\nu))^{\circ} = \Ocal_{\prec}\left( \frac{N^{\delta}}{N}\right)$. (\ref{eq:behaviour-theta2b-f-recentered}) thus follows from Property \ref{pr:properties-stochastic-domination}, item (ii). 
\end{proof}

\subsubsection{Proof of Theorem \ref{th:simplification-theta1b}}
\label{subsubsec:simplification-hatphib-recentre}
\begin{proof}
In order to establish Theorem \ref{th:simplification-theta1b}, we remark that it is sufficient to prove that 
\begin{align}
&\frac{1}{\pi} \mathrm{Re} \int_{\Dcal} \bar{\partial}  \Phi_k(f)(z) \,   
\Bigl(  
\frac{1}{M} \Tr (\hat{\Q}_{N,\mathfrak{b}}(z) - \tilde{\Q}_{N,\mathfrak{b}}(z) ) - 
  \frac{1}{M} \Tr (\Q_{N,\mathfrak{b}}(z) 
  \notag\\
 &\qquad\qquad\qquad\qquad + z  \Q^{2}_{N,\mathfrak{b}}(z)) \, (\D_{\x_{\mathfrak{b}}} - \I) \Bigr)^{\circ}  \diff x \diff y 
  = o_{\prec}\left( \frac{1}{\sqrt{BN}}\right),
  \label{eq:approximation-tr-hatQ-tildeQ}
\end{align}
and 
\begin{align}
    &\mathbb{E} \left(  \frac{1}{M} \Tr \left( \hat{\Q}_{N,\mathfrak{b}}(z) - \tilde{\Q}_{N,\mathfrak{b}}(z)\right) \right) = 
    \notag\\
    &\qquad\tilde{p}_N(z) \tilde{r}_N(\nu) v_N  - p_N(z) \, (c_N B)^{-1}  + \Ocal_z\left( \left(\frac{B}{N}\right)^{4} \right)
    + o_{z}\left( \frac{1}{\sqrt{BN}}\right).
    \label{eq:approximation-E-tr-hatQ-tildeQ}
\end{align}
We just explain the general approach of the proof, and provide the details of the proof of (\ref{eq:approximation-tr-hatQ-tildeQ}) and (\ref{eq:approximation-E-tr-hatQ-tildeQ}) in Appendix \ref{sec:details-proof-theorem2}. 
We express $ \hat{\Q}_{N,\mathfrak{b}}(z) - \tilde{\Q}_{N,\mathfrak{b}}(z)$ as
\begin{equation}
  \label{eq:expre-hatQ-tildeQ}
\hat{\Q}_{N,\mathfrak{b}}(z) - \tilde{\Q}_{N,\mathfrak{b}}(z) =  - \hat{\Q}_{\mathfrak{b}} (\hat{\C}_{\mathfrak{b}} - \tilde{\C}_{\mathfrak{b}}) \tilde{\Q}_{\mathfrak{b}}, 
\end{equation}
and deduce that 
$$
\frac{1}{M} \Tr \left( \hat{\Q}_{\mathfrak{b}} - \tilde{\Q}_{\mathfrak{b}}\right) = - \frac{1}{M} \Tr \left( \hat{\Q}_{\mathfrak{b}} (\hat{\C}_{\mathfrak{b}} - \tilde{\C}_{\mathfrak{b}}) \tilde{\Q}_{\mathfrak{b}}\right),
$$
We express $\hat{\C}_{\mathfrak{b}} - \tilde{\C}_{\mathfrak{b}}$ using that $\hat{\C}_{\mathfrak{b}} =  \hat{\D}_{\mathfrak{b}}^{-1/2} \D^{1/2} \tilde{\C}_{\mathfrak{b}} \D^{1/2} \hat{\D}_{\mathfrak{b}}^{-1/2} $. We  expand for each $m$  $\frac{s_m^{1/2}}{\hat{s}_{m,\mathfrak{b}}^{1/2}}$ around $s_m$ up to the third order, and obtain that
$ \hat{\D}_{\mathfrak{b}}^{-1/2} \D^{1/2}$ can be written as 
\begin{equation}
\label{eq:third-order-expansion-hatDb-D}
\hat{\D}_{\mathfrak{b}}^{-1/2} \D^{1/2}  =  \I - \frac{1}{2} \left( \hat{\D}_{\mathfrak{b}} - \D \right) \D^{-1} + \frac{3}{8} \, \left((\hat{\D}_{\mathfrak{b}} - \D ) \D^{-1}\right)^{2} + \hat{\F}_{\mathfrak{b}},
\end{equation} 
where $\hat{\F}_{\mathfrak{b}}$ is the diagonal matrix with diagonal entries 
$$
\left(\hat{\F}_{\mathfrak{b}}\right)_{m,m} = - \frac{5}{16} \, s_m^{1/2} \,  \frac{(\hat{s}_{m,\mathfrak{b}} - s_m)^{3}}{\hat{\theta}_m^{7/2}},
$$
where $\hat{\theta}_m$ is located between $s_m$ and $\hat{s}_{m,\mathfrak{b}}$. 
This allows to express $\hat{\C}_{\mathfrak{b}} - \tilde{\C}_{\mathfrak{b}}$ as 
 \begin{align}
 \label{eq:decomposition-hatCb-tildeCb}
  \hat{\C}_{\mathfrak{b}} - \tilde{\C}_{\mathfrak{b}}  = & -\frac{1}{2} (\hat{\D}_{\mathfrak{b}} - \D) \D^{-1} \tilde{\C}_{\mathfrak{b}} -\frac{1}{2} \tilde{\C}_{\mathfrak{b}} (\hat{\D}_{\mathfrak{b}} - \D) \D^{-1} + \\
  \notag
    & \frac{3}{8} \left((\hat{\D}_{\mathfrak{b}} - \D ) \D^{-1}\right)^{2} \tilde{\C}_{\mathfrak{b}} +  \frac{3}{8} \tilde{\C}_{\mathfrak{b}} \left((\hat{\D}_{\mathfrak{b}} - \D ) \D^{-1}\right)^{2} + \\
    \notag
    & \frac{1}{4}  (\hat{\D}_{\mathfrak{b}} - \D) \D^{-1} \tilde{\C}_{\mathfrak{b}}   (\hat{\D}_{\mathfrak{b}} - \D) \D^{-1} + \Upsilonbs_1,
 \end{align}
where $\Upsilonbs_1$ represents the corresponding error term. In order to establish 
(\ref{eq:approximation-tr-hatQ-tildeQ}) and (\ref{eq:approximation-E-tr-hatQ-tildeQ}), 
we study the contribution of the various terms at the right-hand side of (\ref{eq:decomposition-hatCb-tildeCb}) to the left-hand sides of (\ref{eq:approximation-tr-hatQ-tildeQ}) and (\ref{eq:approximation-E-tr-hatQ-tildeQ}). The proof can be divided in 3 steps:
\begin{itemize}
\item Step 1: study of the contribution of $\Upsilonbs_1$
\item Step 2: study of the contribution of  the 3 quadratic terms in (\ref{eq:decomposition-hatCb-tildeCb}) (i.e. the third, fourth, and fifth term
at the right-hand side of (\ref{eq:decomposition-hatCb-tildeCb}))
\item Step 3: study of the contribution of the two linear terms of 
the right-hand side of (\ref{eq:decomposition-hatCb-tildeCb})
\end{itemize}
In Appendix \ref{subsec:needed-results}, we present an overview of the proofs  of  Step 1, Step 2 and Step 3, and provide the details in Appendices \ref{subsec:proof-step1}, \ref{subsec:proof-step2}, and \ref{subsec:proof-step3}. 
\end{proof} 

\begin{remark}
In order to  illustrate formula (\ref{eq:approximation-E-tr-hatQ-tildeQ}), we consider the case when $\alpha < 2/3$ (i.e.  $\left( \frac{B}{N} \right)^{2} = o(B^{-1})$)
and  $f(\lambda)  = \log \lambda$.
In this context, $\hat{f}_{\mathfrak{b}}$ is given by
$$
\hat{f}_{\mathfrak{b}} = \frac{1}{M} \log \det \left( \hat{\D}_{\mathfrak{b}}^{-1/2} \D_{\mathfrak{b}}^{1/2} \tilde{\C}_{\mathfrak{b}} \D_{\mathfrak{b}}^{1/2}  
\hat{\D}_{\mathfrak{b}}^{-1/2} \right),
$$
and $\tilde{f}_{\mathfrak{b}} = \frac{1}{M} \log \det \tilde{\C}_{\mathfrak{b}}$. Therefore, 
$\hat{f}_{\mathfrak{b}} - \tilde{f}_{\mathfrak{b}}$ is given by
$$
\hat{f}_{\mathfrak{b}} - \tilde{f}_{\mathfrak{b}} = \frac{1}{M} \sum_{m=1}^{M} \log s_m - \log \hat{s}_{m,\mathfrak{b}},
$$
Expending $\log \hat{s}_{m,\mathfrak{b}}$ around $s_m$, it is easy to check that, for $\alpha < 2/3$,  $\mathbb{E}(\hat{f}_{\mathfrak{b}} - \tilde{f}_{\mathfrak{b}}) = \frac{1}{2 (B+1)} + o(B^{-1})$. Using the Helffer-Sjöstrand formula, we get from (\ref{eq:approximation-E-tr-hatQ-tildeQ}) that
$\mathbb{E}(\hat{f}_{\mathfrak{b}} - \tilde{f}_{\mathfrak{b}}) = -\frac{1}{c B} <D,f> + o(B^{-1})$, 
and verify that  $<D,f> = -\frac{c}{2}$. For this, 
if $[a_1, a_2]$ is an interval containing  $[(1-\sqrt{c})^{2}, (1+\sqrt{c})^{2}]$ with $a_1 > 0$, we first use the Stieltjes 
inversion formula
$$
<D,f> = \lim_{\epsilon \rightarrow 0} \frac{1}{\pi} \int_{a_1}^{a_2} \log\lambda \; \Im p(\lambda + i \epsilon) d\lambda ,
$$
or equivalently
\begin{equation}
\label{eq:integrale-contour}
<D,f> = \lim_{\epsilon \rightarrow 0} \frac{1}{2 i \pi} \int_{(\partial \mathcal{R}_{\epsilon})_{-}} \log z \; p(z) dz,
\end{equation}
where $(\partial \mathcal{R}_{\epsilon})_{-}$ is the negatively oriented contour defined by
$$
(\partial \mathcal{R}_{\epsilon})_ = \{ \lambda \pm i \epsilon, a_1 \leq \lambda \leq a_2 \} \cup \{ a_1 + i y, -\epsilon \leq y \leq \epsilon \} \cup  \{ a_2 + i y, -\epsilon \leq y \leq \epsilon \}.
$$
The right-hand side of  (\ref{eq:integrale-contour}) does not depend on $\epsilon$. Therefore, for each $\epsilon > 0$, we have 
$$
<D,f> = \frac{1}{2 i \pi} \int_{(\partial \mathcal{R}_{\epsilon})_{-}} \log z \; p(z) dz.
$$
In order to evaluate directly $<D,f>$, we observe that the properties of 
$w$ defined in Section \ref{subsec:properties-wishart} imply that 
$$
<D,f> = -c \, \frac{1}{2 i \pi} \int_{(\mathcal{C})_{-}} \frac{\log \psi(w)}{w^{3}} \, dw,
$$
where $(\mathcal{C})_{-}$ is a negatively oriented simple contour enclosing $[-\sqrt{c}, \sqrt{c}]$. Integrating by parts, 
we obtain that 
$$
<D,f> = -c \, \frac{1}{2} \, \frac{1}{2 i \pi} \int_{(\mathcal{C})_{-}} \frac{\psi'(w)}{\psi(w)} \frac{1}{w^{2}} dw = - \frac{c}{2} ,
$$
as expected
\end{remark}
\
\subsection{\texorpdfstring{CLT on $B w_{N,\mathfrak{b}}(f,\nu)$}{CLT}}
\label{subsec:CLT-W}
Theorem \ref{th:approximation-thetaNb} shows that in order to prove a 
CLT on $B \theta_{N,\mathfrak{b}}(f,\nu)$, it is sufficient to do the same job for the 
zero-mean random variable $W_{N,\mathfrak{b}}(f,\nu)$ defined by 
\begin{align}
  W_N(f,\nu) 
  & = B w_{N,\mathfrak{b}}(f, \nu) 
  \notag\\ 
  & = \frac{1}{\pi \, c_N} \mathrm{Re} \int_{\Dcal} \bar{\partial}  \Phi_k(f)(z) \, 
  \Bigl( 
  \Tr \Q_{N,\mathfrak{b}}(z) \D_{\x_{\mathfrak{b}}} 
  \notag\\
  &\qquad\qquad\qquad+ z \,  \Tr  \Q^{2}_{N,\mathfrak{b}}(z) (\D_{\x_{\mathfrak{b}}} - \I) \Bigr)^{\circ} \, \diff x \diff y  ,
  \label{eq:def-WNfnu}
\end{align}
In this section, we thus establish that $W_N(f,\nu)$ verifies a CLT. For this, we first introduce some useful notations. All along this section, the function $\bar{\partial}  \Phi_k(f)(z)$ is denoted by $h(z)$. 
We also introduce the function $s_N(z)$ defined by 
\begin{equation}
    \label{eq:def-sN}
    s_N(z) = \frac{\sqrt{c_N} (z t_N(z) \tilde{t}_N(z))^{2}}{1 - c_N (z t_N(z) \tilde{t}_N(z))^{2}}.
\end{equation}
and consider the function $\omega_N(z_1,z_2)$ given by 
\begin{equation}
    \label{eq:def-thetaz1z2}
    \omega_N(z_1,z_2) = s_N(z_1) s_N(z_2)  \left( \frac{1}{(1 - c_N z_1 t(z_1)\tilde{t}(z_1)z_2 t(z_2)\tilde{t}(z_2))^{2}} - 1 \right).
\end{equation}
As $ c_N |z t_N(z) \tilde{t}_N(z)|^{2} < 1$ when $z \in \mathbb{C}^{+}$ (see Subsection \ref{subsec:properties-wishart}), it is clear that $s_N$ and 
$\omega_N$ are holomorphic on $\mathbb{C}^{+}$ and  $\left(\mathbb{C}^{+}\right)^{2}$ respectively. When $N \rightarrow +\infty$, $c_N$, $t_N(z)$, $\tilde{t}_N(z)$ converge towards $c, t(z), \tilde{t}(z)$ respectively. Therefore, if $s(z)$ and 
$\omega(z_1,z_2)$ are defined as $s_N(z)$ and $\omega_N(z_1,z_2)$ by replacing  $c_N$, $t_N(z)$, $\tilde{t}_N(z)$ by their limits, we have of course $s_N(z) \rightarrow s(z)$ for $z \in \mathbb{C}^{+}$ and 
$\omega_N(z_1,z_2) \rightarrow \omega(z_1,z_2)$ for $z_1,z_2 \in \mathbb{C}^{+}$ when $N \rightarrow +\infty$.
Lemma 9.2 in \cite{loubaton-jotp-2016}  implies that for each integer $l \geq 1$, the function $z \rightarrow s(z) (\sqrt{c} z t(z) \tilde{t}(z))^{l}$ coincides with the Stieltjes transform of a distribution carried by the interval $[(1-\sqrt{c})^{2}, (1+\sqrt{c})^{2}]$ that is denoted by $D_l$. We first state the following Lemma proved in the Appendix \ref{sec:proof-le-positivite-sigmaN}. 
\begin{lemma}
  \label{le:positivite-sigmaN}
  We define $\sigma^{2}_N(f)$ by the double integral
\begin{small}  
\begin{equation}
\label{eq:sigma2}
  \sigma_N^2(f) =
  \frac{1}{4 \pi^2 \, c_N^{2}}
  \iint_{\Dcal \times \Dcal}
 g(z_1, z_2) \drm z_1 \drm z_2,
\end{equation} 
\end{small}
where $\drm z$ stands for $\diff x \diff y$ and where $g(z_1,z_2)$ is given by 
\begin{align*}
g(z_1,z_2) = & h(z_1) h(z_2) \omega_N(z_1,z_2)
  + h(z_1) \overline{h(z_2)} \omega_N(z_1,\bar{z_2}) \\ 
  & 
  + \overline{h(z_1)} h(z_2) \omega_N(\bar{z_1},z_2)
  + \overline{h(z_1)} \overline{h(z_2)} \omega_N(\bar{z_1},\bar{z_2}).
\end{align*}
Then, $\sigma^{2}_N(f)$ converges towards the
term $\sigma^{2}(f)$ defined by (\ref{eq:sigma2}) when $\omega_N$ is replaced by $\omega$.  
Moreover, $\sigma^{2}(f) \geq 0$ and if the test function $f$ verifies 
\begin{equation}
    \label{eq:condition-sigma-positif}
    <D_l, f> \neq 0,
\end{equation}
for some integer $l \geq 1$, then, $\sigma^{2}(f)$
verifies $\sigma^{2}(f)> 0$, and for each $N$ large enough, $\sigma_N^{2}(f) > \frac{\sigma^{2}(f)}{2} > 0$. 
\end{lemma}
We now prove the following result.  
\begin{theorem}
\label{th:CLT-W}
We assume that condition (\ref{eq:condition-sigma-positif}) holds. Then, the sequence of random variables 
$(\frac{W_N(f,\nu)}{\sigma_N(f)})_{N \geq 1}$ verifies 
\begin{equation}
    \label{eq:clt-WN-normalise}
\frac{W_N(f,\nu)}{\sigma_N(f)} \rightarrow_{\Dcal} \Ncal(0,1).    
\end{equation}
Moreover, $\mathbb{E}(W_N(f,\nu))^{2}$ and $\mathbb{E}(W_N(f,\nu))^{4}$ verify 
\begin{eqnarray}
\label{eq:convergence-EWN2}
\mathbb{E}(W_N(f,\nu))^{2} - \sigma_N^{2}(f) & = & \Ocal\left(\frac{1}{B}\right), \\
\label{eq:convergence-EWN4}
\mathbb{E}(W_N(f,\nu))^{4} - 3 \sigma_N^{4}(f) & = & \Ocal\left(\frac{1}{B}\right).
\end{eqnarray}
\end{theorem}
The proof of this result is a consequence of the following Proposition established in the Appendix \ref{sec:proof-prop:master_eq} using the Stein method. For ease of reading, we denote $W_N  = W_N(f,\nu)$  and $\sigma_N = \sigma_N(f)$.
\begin{proposition}
  \label{prop:master_eq}
  Let $\phi$ be a $\Ccal^{1}$ function defined on $\mathbb{R}$ such that
  \begin{align}
  \label{eq:master-equation-condition-phiW}
    \Ebb[|\phi'(W_N)|^{2}] = \Ocal(1), \;  \Ebb[|\phi(W_N)|^{2}] = \Ocal(1).
  \end{align}
  Then
  \begin{align}
  \label{eq:master-equation}
    \Ebb\left[W_N \phi(W_N)\right] = \sigma^{2}_N \Ebb\left[\phi'(W_N)\right] + \Delta_N,
  \end{align}
  where
  \begin{align}
  \label{eq:master-equation-propriete-Delta}
    \left|\Delta_N \right|
    \leq
   \Ebb\Biggl[Y_N \Bigl(|\phi(W_N)| + |\phi'(W_N)|\Bigr)\Biggr], 
  \end{align}
  for some positive random variable $Y_N$ which does not depend on $\phi$, and such that
  $\Ebb[Y_N^{2}] \leq \frac{C}{B^{2}}$ for some nice constant $C$.
\end{proposition}
Before giving the proof of Theorem \ref{th:CLT-W}, we first claim that 
for each integer $k$, 
\begin{equation}
    \label{eq:WN-finite-moment}
\mathbb{E}|W_N|^{k} \leq C_k,    
\end{equation}
where $C_k$ is a constant that only depends on $k$. The proof is provided in the Appendix 
\ref{sec:proof-eq:WN-finite-moment}. In order to obtain (\ref{eq:clt-WN-normalise}), we apply Proposition \ref{prop:master_eq} to the function $\phi_u(w) = e^{iuw}$ for $u \in \mathbb{R}$, which, of course, verifies (\ref{eq:master-equation-condition-phiW}). We denote 
by $\Delta_N(u)$ the error term at the right-hand side of (\ref{eq:master-equation}). It is clear that $\mathbb{E}(\phi_u(W_N)) = \psi_{W_N}(u)$
where $\psi_{W_N}$ represents the characteristic function of the random variable $W_N$. Moreover, 
(\ref{eq:WN-finite-moment}) for $k=1$ implies that 
$\mathbb{E}(iW_N e^{iuW_N}) = \psi^{'}_{W_N}(u)$ where $'$ stands for the differentiation operator w.r.t. $u$ in this context. Therefore, (\ref{eq:master-equation}) leads to
$$
\psi^{'}_{W_N}(u) = -u \sigma_N^{2}  \psi_{W_N}(u) + i \Delta_N(u).
$$
Solving this equation, we obtain that $\psi_{W_N}(u)$ is given by 
$$
\psi_{W_N}(u) = e^{-\sigma_N^{2} u^{2}/2} + e^{-\sigma_N^{2} u^{2}/2} \int_{0}^{u} i \Delta_N(s) e^{\sigma_N^{2} s^{2}/2} ds.
$$
(\ref{eq:master-equation-propriete-Delta}) implies that $|\Delta_N(u)| \leq C \frac{1 + |u|}{B}$. Thus, the second term 
of the righthanside of the above equation is a $\Ocal(B^{-1})$ term, and 
$\psi_{W_N}(u) - e^{-\sigma_N^{2} u^{2}/2} \rightarrow 0$ for each $u$. As  $\sigma_N^{2}(f)$ is bounded away from zero (see Lemma \ref{le:positivite-sigmaN}),
this leads to (\ref{eq:clt-WN-normalise}). In order to justify (\ref{eq:convergence-EWN2}), we apply Proposition \ref{prop:master_eq} to the function $\phi(w) = w e^{iuw}$. (\ref{eq:WN-finite-moment}) for $k=2$ leads to the conclusion that this function verifies (\ref{eq:master-equation-condition-phiW}), so that the master equation (\ref{eq:master-equation}) holds for each $u$. Taking $u=0$ in (\ref{eq:master-equation}) and using the evaluation $\Delta_N(0) = \Ocal\left(\frac{1}{B}\right)$ leads to (\ref{eq:convergence-EWN2}). To obtain (\ref{eq:convergence-EWN4}), it is sufficient to apply  Proposition \ref{prop:master_eq} to the function $\phi(w) = w^{3} e^{iuw}$, 
and to set $u=0$ in the corresponding master equation (\ref{eq:master-equation}). \\

We finally deduce from Theorem \ref{th:approximation-thetaNb} and Theorem \ref{th:CLT-W} that 
if function $f$ verifies condition (\ref{eq:condition-sigma-positif}), $B \theta_{N,\mathfrak{b}}(f,\nu)$ verifies 
\begin{equation}
    \label{eq:CLT-B-thetaNb}
 \frac{B \theta_{N,\mathfrak{b}}(f,\nu)}{\sigma_N(f)} \rightarrow \Ncal(0,1)   ,
\end{equation}
for $\alpha < \frac{4}{5}$. 

\subsection{\texorpdfstring{CLT on $B \theta_N(f,\nu)$}{CLT}}
\label{subsec:CLT-B-theta}

The models defining vectors $(\omega_{m,\mathfrak{b}})_{m=1, \ldots, M}$ and 
$(\omega_{m})_{m=1, \ldots, M}$ are very similar  (see 
Eqs. (\ref{eq:representation-omegam}) and (\ref{eq:representation-omegaw})). In particular,  matrices 
$(\Phibs_{m})_{m=1, \ldots, M}$ verify (\ref{eq:bound-norm-Phim}, \ref{eq:expre-Trace-Phim}, \ref{eq:moyenne-trace-Phim},\ref{eq:trace-sum-phim-carre}), quite similar 
to  (\ref{eq:bound-norm-Phimb}, \ref{eq:expre-Trace-Phimb}, \ref{eq:behaviour-moyenne-trace-Phimb}, \ref{eq:behaviour-trace-sum-Phimb-carre}) satisfied 
by the matrices $(\Phibs_{m,\mathfrak{b}})_{m=1, \ldots, M}$. 
This immediately implies that $\theta_N(f,\nu)$ verifies Theorem \ref{th:approximation-thetaNb}. More precisely, if $\D_{\x(\nu)}$ represents the diagonal $M \times M$ matrix with diagonal entries $\left( \frac{\| \x_m(\nu) \|^{2}}{B+1}\right)_{m=1, \ldots, M}$, 
we have the following result which establishes the representation (\ref{eq:representation-B-thetaN}) of $\theta_N(f,\nu)$. 
\begin{theorem}
\label{th:approximation-thetaN}
$\theta_{N}(f,\nu)$ can be written as 
\begin{equation}
    \label{eq:approximation-thetaN}
\theta_{N}(f,\nu) = w_N(f,\nu) + \epsilon_N(f,\nu),
\end{equation}
where $w_N(f,\nu)$ is defined by 
\begin{align}
w_N(f,\nu) &=  
\frac{1}{\pi} \mathrm{Re} \int_{\Dcal} \bar{\partial}  \Phi_k(f)(z) \,  
\Bigr( \frac{1}{M} \Tr \Q_{N}(z) \D_{\x(\nu)} 
\notag\\
&\qquad\qquad+ z \frac{1}{M} \Tr \Q^{2}_{N}(z) (\D_{\x(\nu)} - \I) \Bigr)^{\circ} \, \diff x \diff y, 
\label{eq:expre-wN}
\end{align}
and where the family $(\epsilon_{N}(f,\nu), \nu \in [0,1])$ verifies 
(\ref{eq:domination-stochastique-erreur-approximation-clt-4/5}) and 
(\ref{eq:E-erreur-approximation-clt-4/5}).
\end{theorem}
It is clear that for each $\nu$, the sequences of random variables 
$(w_{N,\mathfrak{b}}(f,\nu))_{N \geq 1}$ and $(w_N(f,\nu))_{N \geq 1}$ share the same probability distribution. Therefore, Theorem \ref{th:CLT-W} and representation (\ref{eq:approximation-thetaN}) imply that 
\begin{equation}
    \label{eq:CLT-B-thetaN}
 \frac{B \theta_N(f,\nu)}{\sigma_N(f)} \rightarrow \Ncal(0,1)   ,
\end{equation}
for $\alpha < \frac{4}{5}$ provided function $f$ verifies condition (\ref{eq:condition-sigma-positif}). 
\begin{remark}
\label{re:sigma2-independent-from-nu}
It is important to notice that the asymptotic variance $\sigma^{2}_N(f)$ does not depend on the spectral densities $(s_m(\nu))_{m=1, \ldots, M}$. Therefore, $\sigma_N^{2}(f)$ also
coincides with the asymptotic variance of $\theta_{N}(f,\nu)$ when for each $m$, $(y_{m,n})_{n \in \mathbb{Z}}$ is an uncorrelated sequence, i.e. when the spectral density of $(y_{m,n})_{n \in \mathbb{Z}}$ reduces to a constant denoted $s_m$. To elaborate from this, we assume that this assumption holds, and denote by $\D_{N,iid}$ the constant diagonal matrix $\D_{N,iid} = \mathrm{diag}\left( s_m, m=1, \ldots, M \right)$. In this context, 
it is clear that for each $m$, the components of the $(B+1)$--dimensional vector $\omegabs_{m}(\nu)$ 
given by (\ref{eq:def-omegam}) are i.i.d. $\Ncal_c(0,s_m)$ distributed random variables. 
Matrix $\Phibs_m(\nu)$ defined by (\ref{eq:representation-covariance-omegam}) is thus reduced to $0$ for each $m$, and we deduce from this that matrix $\Sigmabs_N(\nu)$ given by (\ref{eq:def-Sigma}) is equal to 
$$
\Sigmabs_N(\nu) = \D_{N,iid}^{1/2} \X_N(\nu).
$$
The estimated spectral coherence matrix $\hat{\C}_N(\nu)$ is thus equal to 
$$
\hat{\C}_N(\nu) = \hat{\D}^{-1/2}_{N,iid}(\nu)  \D_{N,iid}^{1/2} \, \frac{\X_N(\nu)  \X_N^{*}(\nu)}{B+1}  \, \hat{\D}^{-1/2}_{N,iid}(\nu)  \D_{N,iid}^{1/2},
$$
where $\hat{\D}_{N,iid}(\nu) = \mathrm{dg}\left ( \frac{\Sigmabs_N(\nu)\Sigmabs_N^{*}(\nu)}{B+1}\right)$. As the diagonal terms of $\hat{\C}_N(\nu)$ are equal to $1$, the diagonal matrix $ \hat{\D}^{-\frac{1}{2}}_{N,iid}(\nu)  \D_{N,iid}^{\frac{1}{2}}$ coincides with $\left(\mathrm{dg}\left(  \frac{\X_N(\nu)  \X_N^{*}(\nu)}{B+1} \right)\right)^{-\frac{1}{2}} $, so that $\hat{\C}_N(\nu)$
can be written as 
$$
\hat{\C}_N(\nu) = \left(\mathrm{dg}\left(  \frac{\X_N(\nu)  \X_N^{*}(\nu)}{B+1} \right)\right)^{-\frac{1}{2}} \, 
 \frac{\X_N(\nu)  \X_N^{*}(\nu)}{B+1} \, \left(\mathrm{dg}\left(  \frac{\X_N(\nu)  \X_N^{*}(\nu)}{B+1} \right)\right)^{-\frac{1}{2}}.
$$
We also remark that, as $r_N(\nu)$ is identically $0$, $\theta_{N}(f,\nu)$ is reduced to 
$$
\theta_{N,iid}(f,\nu) =  \hat{f}_N(\nu) -  \int_{\Rbb^{+}}f\diff\mu_{MP}^{(c_N)}  + <D_N, f> \,  \frac{1}{c_N} \frac{1}{B}.  
$$
$\hat{\C}_N(\nu)$ thus coincides with the sample autocorrelation matrix $\hat{\C}$ build from a $M \times (B+1)$ Gaussian
random matrix $\X$ with i.i.d. $\Ncal_c(0,1)$ entries, while $\theta_N(f,\nu)$ coincides with 
the recentered LSS  $\theta_{N,iid}(f,\nu)$ of the eigenvalues of $\hat{\C}$ given by
$$
\theta_{N,iid}(f) = \frac{1}{M} \Tr f(\hat{\C}) - \int_{\Rbb^{+}}f\diff\mu_{MP}^{(c_N)}  + <D_N, f> \,  \frac{1}{c_N} \frac{1}{B}  .
$$
This discussion implies that, whatever the spectral densities $(s_m(\nu))_{m=1, \ldots,M}$, $\sigma^{2}_N(f)$  coincides with the asymptotic 
variance of the recentered LSS $ \theta_{N,iid}(f,\nu)$ of the sample autocorrelation 
matrix $\hat{\C}$ build from a $M \times (B+1)$ random matrix with i.i.d. standard complex Gaussian entries. We finally notice that the CLT on $\theta_{N,iid}(f)$ precisely 
coincides with the results presented in \cite{gao-han-pan-yang-jrss-2017} and \cite{mestre-vallet-ieeeit-2017}. These central limit theorems are not formulated as in the present paper 
because the test functions $f$ considered in \cite{gao-han-pan-yang-jrss-2017} and \cite{mestre-vallet-ieeeit-2017} are supposed analytic in a neighbourhood of $[(1 - \sqrt{c})^{2}, (1 + \sqrt{c})^{2}]$. 
\end{remark}
\begin{remark}
\label{re:clt-alpha-inferieur-2tiers}
When the components of $\y$ are not reduced to i.i.d. sequences, it is also interesting to notice that when $\alpha < \frac{2}{3}$, the term $r_N(\nu) \, v_N$ is asymptotically negligible w.r.t. $ \frac{1}{c_N} \frac{1}{B}$. The CLT (\ref{eq:CLT-B-thetaN}) on the linear spectral statistics of $\hat{\C}_N(\nu)$ is thus exactly the same as if all the components of $\y$ were i.i.d. sequences. In other words, if $\alpha < \frac{2}{3}$, the presence of 
error matrix $\Gammabs_N(\nu)$ in the expression (\ref{eq:def-Sigma}) of $\Sigmabs_N(\nu)$ has no impact on the 
CLT on the LSS of $\hat{\C}_N(\nu)$. We however mention that, in practice, for finite values of 
$M$ and $N$, even if $\alpha < \frac{2}{3}$, a better fit between the distribution of 
$\frac{B \theta_N(f,\nu)}{\sigma_N(f)}$ and the Gaussian standard distribution is observed when the term 
$r_N(\nu) v_N$ is taken into account in the recentering term of $\hat{f}_N(\nu)$. We refer the reader to the 
Section \ref{sec:simulations} for more details.
\end{remark}

\section{\texorpdfstring{CLT for the statistics $\zeta_{N,1}(f)$ and  $\zeta_{N,2}(f)$ when $\alpha < \frac{7}{9}$}{CLT}}
\label{sec:clt-zeta}
In this section, we assume that $\alpha < \frac{7}{9}$, and establish the CLTs
 (\ref{eq:clt-zeta1}) and (\ref{eq:clt-zeta}) verified by the statistics  $\zeta_{N,1}(f)$ and  $\zeta_{N,2}(f)$ defined by (\ref{eq:def-statistics-zeta1-clt}) and  (\ref{eq:def-statistics-zeta-clt}). For this, we first establish the representation 
(\ref{eq:representation-theta-bartlett}) of $\theta_{N}(f,\nu)$

\subsection{\texorpdfstring{Proof of representation (\ref{eq:representation-theta-bartlett}) of $\theta_N(f,\nu)$}{Proof of representation}}
\label{subsec:proof-representation-bartlett-theta}
In order to justify (\ref{eq:representation-theta-bartlett}), we express 
$\theta_N(f,\nu)$ as 
\begin{equation}
    \label{eq:equation-triviale}
    \theta_N(f,\nu) = \theta_{N,\mathfrak{b}}(f,\nu) + \theta_N(f,\nu) - \theta_{N,\mathfrak{b}}(f,\nu). 
\end{equation}
Theorems \ref{th:approximation-thetaNb} and \ref{th:approximation-thetaN} imply that
$$
\theta_N(f,\nu) - \theta_{N,\mathfrak{b}}(f,\nu) = w_N(f,\nu) - w_{N,\mathfrak{b}}(f,\nu) + \epsilon_N(f,\nu) - \epsilon_{N,\mathfrak{b}}(f,\nu),
$$
so that, using again Theorem \ref{th:approximation-thetaNb}, Eq. (\ref{eq:equation-triviale}) can be rewritten as 
\begin{equation}
\label{eq:expre-theta-thetabar}
\theta_N(f,\nu) = w_{N,\mathfrak{b}}(f,\nu)  + w_N(f,\nu) - w_{N,\mathfrak{b}}(f,\nu) + \epsilon_N(f,\nu).
\end{equation}
In order to establish (\ref{eq:representation-theta-bartlett}), it thus remains to justify that the error term $\kappa_{N,\mathfrak{b}}(f,\nu)$ given by 
\begin{equation}
\label{eq:expre-kappab-clt}
\kappa_{N,\mathfrak{b}}(f,\nu) = w_N(f,\nu) - w_{N,\mathfrak{b}}(f,\nu) 
 + \epsilon_N(f,\nu),
\end{equation}
verifies $\kappa_{N,\mathfrak{b}}(f,\nu) = \Ocal_{\prec}\left( \frac{1}{\sqrt{BN}}\right)$
if $\frac{1}{2} < \alpha < \frac{7}{9}$. 
 $\alpha < \frac{7}{9}$ is equivalent to $\frac{B^{4}}{N^{4}} = o\left( \frac{1}{\sqrt{BN}}\right)$. Therefore, 
as $\epsilon_N(f,\nu)$ verifies (\ref{eq:domination-stochastique-erreur-approximation-clt-4/5})
and  (\ref{eq:E-erreur-approximation-clt-4/5}), we have just to prove the following Proposition. 
 \begin{proposition}
 \label{th:hatphi-hatphib}
 The family of random variables $( w_N(f,\nu) - w_{N,\mathfrak{b}}(f,\nu)$ , $N \geq 1, \nu \in [0,1])$ verifies
 \begin{equation}
     \label{eq:domination-stochastique-hatphi-hatphib}
    \left( w_N(f,\nu) - w_{N,\mathfrak{b}}(f,\nu) \right) =  O_{\prec}\left( \frac{1}{\sqrt{BN}} \right). \\
 \end{equation}
 \end{proposition}
\begin{proof}
In order to simplify the exposition, we just establish that
\begin{align}
    \label{eq:equivalent-theorem-hatphi-hatphib-bis}
\frac{1}{\pi} \mathrm{Re} \int_{\Dcal} \bar{\partial}  \Phi_k(f)(z) \, \left( \frac{1}{M} \Tr \left(\Q_{N}(z,\nu) \D_{\x(\nu)} - \Q_{N,\mathfrak{b}}(z,\nu) \D_{\x_{\mathfrak{b}}(\nu)} \right) \right)^{\circ} \diff x \diff y  = & \\
\notag 
\Ocal_{\prec}\left( \frac{1}{\sqrt{BN}}\right).  
\end{align}
because the contribution of the term \\ $\frac{1}{M} \Tr \left(\Q^{2}_{N}(z,\nu) (\D_{\x(\nu)} - \I) - \Q^{2}_{N,\mathfrak{b}}(z,\nu) (\D_{\x_{\mathfrak{b}}(\nu)} -\I) \right)^{\circ}$ can be addressed similarly. For this, we express $\frac{1}{M} \Tr \left(  \Q_{N}(z,\nu) \D_{\x(\nu)} - \Q_{N,\mathfrak{b}}(z,\nu) \D_{\x_{\mathfrak{b}}(\nu)} \right)$ as 
\begin{align}
\label{eq:expre-QD-QbDb}
 \frac{1}{M} \Tr \left( \Q_{N}(z) \D_{\x}  - \Q_{N,\mathfrak{b}}(z)  \D_{\x_{\mathfrak{b}}} \right) = & 
 \frac{1}{M} \Tr \left( \Q_{N}(z) (\D_{\x} -  \D_{\x_{\mathfrak{b}}}) \right) \\ \notag 
 & 
 + \frac{1}{M} \Tr \left(  (\Q_{N}(z)  - \Q_{N,\mathfrak{b}}(z))  \D_{\x_{\mathfrak{b}}} \right),
\end{align}
and study separately the contribution to the left-hand side of (\ref{eq:equivalent-theorem-hatphi-hatphib-bis}) of the two terms at the right-hand side of (\ref{eq:expre-QD-QbDb}). The first term is given  by
\begin{align}
 &\frac{1}{M} \Tr \left( \Q(z) (\D_{\x} -  \D_{\x_{\mathfrak{b}}}) \right) =  
 \notag\\
 &\beta(z) \frac{1}{M} \Tr \left( \D_{\x} -  \D_{\x_{\mathfrak{b}}} \right) + 
 \frac{1}{M} \sum_{m=1}^{M} \Q^{\circ}_{m,m} ( \D_{\x} -  \D_{\x_{\mathfrak{b}}})_{m,m},
 \label{eq:first-term-proof-prop-6-1}
\end{align}
where we recall that $\beta(z) = \mathbb{E}\left( \frac{1}{M} \Tr \Q(z) \right)$. In order to study the contribution of this term to the left-hand side of (\ref{eq:equivalent-theorem-hatphi-hatphib-bis}), we need to evaluate the diagonal elements of $ \D_{\x} -  \D_{\x_{\mathfrak{b}}}$. For this, we prove in the Appendix \ref{sec:proof-lemma-hatD-hatDb} the following useful Lemma. 
\begin{lemma}
\label{le:hats-hatsb}
Matrix $(\hat{\D} - \hat{\D}_{\mathfrak{b}})\D^{-1}$
verifies 
\begin{eqnarray}
\label{eq:hats-hatsb-rond}   
\left(\left((\hat{\D} - \hat{\D}_{\mathfrak{b}})\D^{-1} \right)_{m,m}\right)^{\circ} &  = & \Ocal_{\prec}\left(\frac{1}{\sqrt{NB}}\right), \\
\label{eq:E-hats-hatsb}   
\mathbb{E}\left(\left((\hat{\D} - \hat{\D}_{\mathfrak{b}})\D^{-1} \right)_{m,m}\right) &  = & \Ocal\left(\frac{1}{N}\right), \\
\label{eq:E-hats-hatsb-square}
\mathbb{E}\left( \left( (\hat{\D} - \hat{\D}_{\mathfrak{b}})\D^{-1} \right)_{m,m}^{2} \right)  &  = & \Ocal\left(\frac{1}{NB}\right),
\end{eqnarray}
and 
\begin{eqnarray}
\label{eq:tr-hatD-hatDb-rond} 
\frac{1}{M} \Tr \left((\hat{\D} - \hat{\D}_{\mathfrak{b}})\D^{-1}\right)^{\circ}  & =  & \Ocal_{\prec}\left(\frac{1}{B\sqrt{N}}\right), \\
\label{eq:tr-hatD-hatDb}
\mathbb{E} \left( \frac{1}{M} \left((\hat{\D} - \hat{\D}_{\mathfrak{b}})\D^{-1} \right) \right) &  = & \Ocal\left(\frac{1}{N}\right).
\end{eqnarray}
\end{lemma}
Decomposition (\ref{eq:decomposition-Dx-I}) and its analog 
\begin{equation}
\label{eq:decomposition-hatD-D-over-D-analog}
\left((\hat{\D} - \D)\D^{-1}\right) = \D_{\x} - \I + \D_2 + \D_3,
\end{equation}
where $\D_2$ and $\D_3$ are defined in the same way that $\D_{2,\mathfrak{b}}$ and $\D_{3,\mathfrak{b}}$ respectively, 
imply that 
$$
 \frac{1}{M} \Tr \left( \D_{\x} -  \D_{\x_{\mathfrak{b}}} \right)^{\circ} =  \frac{1}{M} \Tr \left((\hat{\D} - \hat{\D}_{\mathfrak{b}})\D^{-1}\right)^{\circ} - \frac{1}{M} \Tr (\D_{2} - \D_{2,\mathfrak{b}}) ,
$$
because matrices $\D_{2}$ and $\D_{2,\mathfrak{b}}$ are zero mean and matrices $\D_{3}$ and $\D_{3,\mathfrak{b}}$ are deterministic. The Hanson-Wright inequality leads immediately to $\frac{1}{M} \Tr \D_{2} = \Ocal_{\prec}\left(\frac{1}{N}\right)$ and $\frac{1}{M} \Tr \D_{2,\mathfrak{b}} = \Ocal_{\prec}\left(\frac{1}{N}\right)$. Therefore, (\ref{eq:tr-hatD-hatDb-rond}) implies that 
$$ \frac{1}{M} \Tr \left( \D_{\x} -  \D_{\x_{\mathfrak{b}}} \right)^{\circ} =
\Ocal_{\prec}\left( \frac{1}{N}\right) = o_{\prec}\left(\frac{1}{\sqrt{NB}}\right).$$
Here, we have used that $\alpha > \frac{1}{2}$ implies that $\frac{1}{B \sqrt{N}} = o\left( \frac{1}{N} \right)$.
The contribution of $\beta(z) \frac{1}{M} \Tr  ( \D_{\x} -  \D_{\x_{\mathfrak{b}}})$
to the left-hand side of (\ref{eq:equivalent-theorem-hatphi-hatphib-bis}) is thus
a $o_{\prec}\left(\frac{1}{\sqrt{NB}}\right)$ term. We now prove that  
\begin{equation}
\label{eq:E-tr-Qrond-Dx-Dxb}
\mathbb{E}\left(  \frac{1}{M} \sum_{m=1}^{M} \Q^{\circ}_{m,m} ( \D_{\x} -  \D_{\x_{\mathfrak{b}}})_{m,m} \right) = \Ocal_z\left( \frac{1}{B \sqrt{N}} +  \frac{1}{N} \right) = \Ocal_z\left( \frac{1}{N}\right),
\end{equation}
and 
\begin{align}
&\frac{1}{\pi} \mathrm{Re} \int_{\Dcal} \bar{\partial}  \Phi_k(f)(z) \, \left( \frac{1}{M} \sum_{m=1}^{M} \Q^{\circ}_{m,m} ( \D_{\x} -  \D_{\x_{\mathfrak{b}}})_{m,m} \right) \, \diff x \diff y 
\notag\\
&\qquad =  \Ocal_{\prec} \left( \frac{1}{B \sqrt{N}} +  \frac{1}{N} + \frac{B^{7/2}}{N^{4}} \right) 
\notag\\ 
&\qquad 
= \Ocal_{\prec}\left( \frac{1}{N} +  \frac{B^{7/2}}{N^{4}} \right),
\label{eq:tr-Qrond-Dx-Dxb}
\end{align}
which, in turn, will imply that 
\begin{align}
&\frac{1}{\pi} \mathrm{Re} \int_{\Dcal} \bar{\partial}  \Phi_k(f)(z) \, \left( \frac{1}{M} \sum_{m=1}^{M} \Q^{\circ}_{m,m} ( \D_{\x} -  \D_{\x_{\mathfrak{b}}})_{m,m} \right)^{\circ} \, \diff x \diff y 
\notag\\
& \qquad=  \Ocal_{\prec}\left( \frac{1}{N} +  \frac{B^{7/2}}{N^{4}} \right) 
\notag\\ 
& \qquad = o_{\prec}\left( \frac{1}{\sqrt{BN}}\right).
\label{eq:tr-Qrond-Dx-Db-rond}
\end{align}
Decomposition (\ref{eq:decomposition-Dx-I}) and its analog (\ref{eq:decomposition-hatD-D-over-D-analog})
also imply that 
\begin{equation}
    \label{eq:expre-Dx-Dxb}
\left(\D_{\x} -  \D_{\x_{\mathfrak{b}}}\right)_{m,m} = \left((\hat{\D} - \hat{\D}_{\mathfrak{b}})\D^{-1} \right)_{m,m} - 
\left( \D_{2} - \D_{2,\mathfrak{b}} \right)_{m,m} - \left( \D_{3} - \D_{3,\mathfrak{b}} \right)_{m,m}.
\end{equation}
As $\D_{3}$ and $\D_{3,\mathfrak{b}}$ are deterministic, that (\ref{eq:E-hats-hatsb-square}) holds and that
$\mathbb{E}|(\D_{2})_{m,m}|^{2}$ and $\mathbb{E}|(\D_{2,\mathfrak{b}})_{m,m}|^{2}$
are both $\Ocal\left( \frac{B}{N^{2}}\right) = o\left( \frac{1}{N}\right)$ terms (see 
Eq. (\ref{eq:D2-carre-mm})), 
the evaluation $\mathbb{E}|\Q^{\circ}_{m,m}|^{2} =  \Ocal_z(B^{-1})$ (easily obtained using the Nash-Poincaré inequality (\ref{eq:nash-poincare-iid})) and the Schwartz inequality lead to (\ref{eq:E-tr-Qrond-Dx-Dxb}). In order to justify (\ref{eq:tr-Qrond-Dx-Dxb}),
we still use the decomposition (\ref{eq:expre-Dx-Dxb}). 
  (\ref{eq:expre-Trace-Phim}) and (\ref{eq:expre-Trace-Phimb}) imply that $\left( \D_{3} - \D_{3,\mathfrak{b}} \right)_{m,m} = \Ocal\left( \frac{B^{4}}{N^{4}} + \frac{1}{N} \right)$
  while $(\D_{2})_{m,m}$ and $(\D_{2,\mathfrak{b}})_{m,m}$ are $\Ocal_{\prec}\left( \frac{\sqrt{B}}{N}\right)$ terms (see (\ref{eq:domination-D2b}) which is of course also verified my matrix $\D_2$). 
Therefore, (\ref{eq:hats-hatsb-rond}) leads to
  $$
\left(\D_{\x} -  \D_{\x_{\mathfrak{b}}}\right)_{m,m} = 
\Ocal_{\prec}\left( \frac{\sqrt{B}}{N} + \frac{1}{\sqrt{BN}} +  \frac{B^{4}}{N^{4}} + \frac{1}{N} \right).
  $$
Moreover, (\ref{eq:proof-behaviour-omegam}) in Appendix (\ref{subsec:proof-step2}) implies that 
$$
\frac{1}{\pi} \mathrm{Re} \int_{\Dcal} \bar{\partial}  \Phi_k(f)(z) \, \Q^{\circ}_{m,m}(z) 
\, \diff x \diff y = \Ocal_{\prec}\left( \frac{1}{\sqrt{B}}\right).
$$
Therefore, the use of Property \ref{pr:properties-stochastic-domination}, item (i), and of Lemma \ref{le:domination-moyenne} allows to obtain (\ref{eq:tr-Qrond-Dx-Dxb}), and therefore (\ref{eq:tr-Qrond-Dx-Db-rond}).

We now study the contribution of the second term of the right-hand side of (\ref{eq:expre-QD-QbDb}) to (\ref{eq:equivalent-theorem-hatphi-hatphib-bis}), i.e. 
\begin{equation}
\label{eq:contribution-second-term-QD-QbDb}
\frac{1}{\pi} \mathrm{Re} \int_{\Dcal} \bar{\partial}  \Phi_k(f)(z) \, \left( \frac{1}{M} \Tr \left(\Q_{N}(z,\nu)  - \Q_{N,\mathfrak{b}}(z,\nu)\right)  \D_{\x_{\mathfrak{b}}(\nu)} \right)^{\circ} \diff x \diff y .
\end{equation}
For this, we remark that $\frac{1}{M} \Tr \left(\Q_{N}(z)  - \Q_{N,\mathfrak{b}}(z)\right)  \D_{\x_{\mathfrak{b}}}$
can be seen as a function of $\X_N, \X_N^{*}$ and $\X_{N,\mathfrak{b}}, \X_{N,\mathfrak{b}}^{*}$. $\X_N$ is moreover a function of 
$\X_{N,\mathfrak{b}}$ and $\X_{N,r}$ because, by (\ref{eq:representation-omegam}) and (\ref{eq:representation-omegam-bartlett}), the equality 
$$
\x_m = (\x_{m,\mathfrak{b}}, \x_{m,r}) \left( \begin{array}{c} I + \Psibs_{m,\mathfrak{b}} + \Psibs^{1}_{m,r} \\ \Psibs_{m,r}^{2} \end{array} \right) \left( \I + \Psibs_m \right)^{-1},
$$
holds for each $m$. Therefore, $\frac{1}{M} \Tr \left(\Q_{N}(z)  - \Q_{N,\mathfrak{b}}(z)\right)  \D_{\x_{\mathfrak{b}}}$
is a function of $\tilde{\X}_N = (\X_{N,\mathfrak{b}}, \X_{N,r})$ and $\tilde{\X}_N^{*}$. As $\tilde{\X}_N$ is i.i.d. 
with $\Ncal_c(0,1)$ entries, it seems reasonable to apply Lemma \ref{le:concentration-integrale-helffer-sjostrand}. In order to simplify the following calculations, it is more appropriate to rewrite  $\frac{1}{M} \Tr \left(\Q_{N}(z)  - \Q_{N,\mathfrak{b}}(z)\right)  \D_{\x_{\mathfrak{b}}}$ as 
\begin{align}
&\frac{1}{M} \Tr \Bigl(\Q_{N}(z) - \tilde{\Q}_N(z) + \tilde{\Q}_N(z) - \tilde{\Q}_{N,\mathfrak{b}}(z) 
\notag\\
&\qquad\qquad+ \tilde{\Q}_{N,\mathfrak{b}}(z) - \Q_{N,\mathfrak{b}}(z) \Bigr)  \left( \D_{\x_{\mathfrak{b}}} - \I + \I \right).
\end{align}
We first claim that the contribution of $\frac{1}{M} \Tr \left((\Q_{N}(z) - \tilde{\Q}_N(z))(\D_{\x_{\mathfrak{b}}} - \I)\right)$ and $\frac{1}{M} \Tr \left((\Q_{N,\mathfrak{b}}(z) - \tilde{\Q}_{N,\mathfrak{b}}(z))(\D_{\x_{\mathfrak{b}}} - \I)\right)$ to (\ref{eq:contribution-second-term-QD-QbDb}) are $\Ocal_{\prec}\left( \frac{1}{N} \right)$ terms. To check this, it is sufficient to adapt the approach developed to establish
(\ref{eq:evaluation-tilde-delta11}) in Appendix \ref{subsec:proof-step3}. The same result holds for the terms $\frac{1}{M} \Tr \left(\Q_{N}(z) - \tilde{\Q}_N(z)\right) $ and 
$\frac{1}{M} \Tr \left(\tilde{\Q}_{N,\mathfrak{b}}(z) - \Q_{N,\mathfrak{b}}(z) \right)$ that were addressed 
in Proposition \ref{prop:behaviour-theta2b-f}. It thus remains to consider the terms
$\frac{1}{M} \Tr \left( \tilde{\Q}_N(z) - \tilde{\Q}_{N,\mathfrak{b}}(z) \right)$ and \\
$\frac{1}{M} \Tr \left( (\tilde{\Q}_N(z) - \tilde{\Q}_{N,\mathfrak{b}}(z)) (\D_{\x_{\mathfrak{b}}} - \I)  \right)$. 
The former term is briefly evaluated in the Appendix  \ref{sec:proof-domination-stochastique-zeta-tildeQ-tildeQb} where it is proved using Lemma \ref{le:concentration-integrale-helffer-sjostrand} that $\omega_N(f,\nu)$ defined by 
\begin{equation}
    \label{eq:def-zeta-tildeQ-tildeQb}
\omega_N(f,\nu) = \frac{1}{\pi} \mathrm{Re} \int_{\Dcal} \bar{\partial}  \Phi_k(f)(z) \,  \frac{1}{M} \Tr \left( \tilde{\Q}_N(z,\nu) - \tilde{\Q}_{N,\mathfrak{b}}(z,\nu) \right)^{\circ} \diff x \diff y,   
\end{equation}
verifies 
\begin{equation}
    \label{eq:domination-stochastique-zeta-tildeQ-tildeQb}
    \omega_N(f,\nu) = \Ocal_{\prec}\left( \frac{1}{\sqrt{BN}}\right).
\end{equation}
The term $\frac{1}{M} \Tr \left( (\tilde{\Q}_N(z) - \tilde{\Q}_{N,\mathfrak{b}}(z)) (\D_{\x_{\mathfrak{b}}} - \I)  \right)$ is evaluated similarly, except that, as in the context of the proof of  (\ref{eq:stochastic-domination-tr-D-I-square}) in Appendix \ref{subsec:proof-step2},  we use the trick introduced in the proof of
Lemma 7 in \cite{loubaton-rosuel-ejs-2021}, and replace matrix $\D_{\x_{\mathfrak{b}}} - \I$ 
by matrix $\D_{\epsilon,g}$ defined by (\ref{eq:def-Depsilong}). The application of Lemma \ref{le:concentration-integrale-helffer-sjostrand} to matrix $\tilde{\X}_N$ leads to 
$$
\frac{1}{\pi} \mathrm{Re} \int_{\Dcal} \bar{\partial}  \Phi_k(f)(z) \,  \frac{1}{M} \Tr \left( (\tilde{\Q}_N(z,\nu) - \tilde{\Q}_{N,\mathfrak{b}}(z,\nu)) \D_{\epsilon,g} \right)^{\circ} \diff x \diff y = 
\Ocal_{\prec}\left( \frac{B^{\epsilon}}{\sqrt{BN}}\right),
$$
for each $\epsilon$, which, in turn, implies that 
\begin{align*}
&\frac{1}{\pi} \mathrm{Re} \int_{\Dcal} \bar{\partial}  \Phi_k(f)(z) \,  \frac{1}{M} \Tr \left( (\tilde{\Q}_N(z,\nu) - \tilde{\Q}_{N,\mathfrak{b}}(z,\nu))  (\D_{\x_{\mathfrak{b}}} - \I) \right)^{\circ} \diff x \diff y 
\\
&\qquad= 
\Ocal_{\prec}\left( \frac{1}{\sqrt{BN}}\right).
\end{align*}
This completes the proof of Proposition \ref{th:hatphi-hatphib}. 
\end{proof}

\subsection{\texorpdfstring{Study of $\zeta_{N,1}(f)$ and  $\zeta_{N,2}(f)$}{Study of zeta}}
\label{subsec:study-zeta}
We finally study the statistics  $\zeta_{N,1}(f)$ and  $\zeta_{N,2}(f)$  defined by
(\ref{eq:def-statistics-zeta1-clt}) and (\ref{eq:def-statistics-zeta-clt}) when $\alpha < \frac{7}{9}$.
\subsubsection{Proof of (\ref{eq:clt-zeta1})}
We start from (\ref{eq:representation-theta-bartlett}) and
recall that  $\kappa_{N,\mathfrak{b}}(f,\nu) = \Ocal_{\prec}\left( \frac{1}{\sqrt{NB}}\right)$.  Lemma \ref{le:domination-moyenne} implies that 
\begin{equation}
    \label{eq:contribution-deltaN-to-zetaN}
    \frac{1}{\sqrt{K'}} \sum_{\nu \in \Gcal^{'}_N} B \kappa_{N,\mathfrak{b}}(f,\nu) = \Ocal_{\prec}\left( \left( \frac{K' B}{N}\right)^{1/2}\right) = \Ocal_{\prec}(N^{-\delta/2}) = o_P(1),
\end{equation}
(we recall that $K'$ verifies (\ref{eq:def-Bprime-Kprime})). Therefore, in order to prove the CLT (\ref{eq:clt-zeta1}), it is sufficient to check that 
$$
\zeta_{N,1,w}(f) = \frac{1}{\sqrt{K'}} \sum_{\nu \in \Gcal^{'}_N}  B w_{N,\mathfrak{b}}(f,\nu),
$$
verifies 
\begin{equation}
\label{eq:clt-zeta1-w}
\frac{\zeta_{N,1,w}(f)}{\sigma_N(f)} \rightarrow_{\Dcal} \Ncal(0,1).
\end{equation}
The random variables $(B w_{N,\mathfrak{b}} (f,\nu))_{\nu \in \Gcal^{'}_N}$ are i.i.d. The standard CLT thus leads to 
$$
\frac{\zeta_{N,1,w}(f)}{\left(\mathbb{E}((B w_{N,\mathfrak{b}})^{2})\right)^{1/2}} \rightarrow_{\Dcal} \Ncal(0,1),
$$
while (\ref{eq:convergence-EWN2}) implies that $\frac{\left(\mathbb{E}((B w_{N,\mathfrak{b}})^{2})\right)^{1/2}}{\sigma_N(f)} = 1 + \Ocal(\frac{1}{B})$. This, in turn,
justifies (\ref{eq:clt-zeta1}).

\subsubsection{Proof of (\ref{eq:clt-zeta})}
 (\ref{eq:representation-theta-bartlett}) implies that
$\left(B \theta_N(f,\nu)\right)^{2} = \left(B w_{N,\mathfrak{b}} (f,\nu)\right)^{2} + \delta_{N}(f,\nu)$ where $\delta_N(f,\nu)$ is given by 
$$
\delta_N(f,\nu) = \left(B \kappa_{N,\mathfrak{b}}(f,\nu)\right)^{2} + 2 \, B \kappa_{N,\mathfrak{b}}(f,\nu) \, B w_{N,\mathfrak{b}} (f,\nu).
$$
Replacing $\D_{\x_{\mathfrak{b}}} - \I$ by the diagonal matrix $\D_{\epsilon,g}$ 
defined by (\ref{eq:def-Depsilong}), and using Lemma 
\ref{le:concentration-integrale-helffer-sjostrand}, it is easy to check that $w_{N,\mathfrak{b}} (f,\nu) =   \Ocal_{\prec}(B^{-1})$. 
As $\kappa_{N,\mathfrak{b}}(f,\nu) = \Ocal_{\prec}\left( \frac{1}{\sqrt{NB}}\right)$, we obtain that 
$\delta_N(f,\nu) = \Ocal_{\prec}\left(\left( \frac{B}{N} \right)^{1/2}\right)$. Lemma \ref{le:domination-moyenne} implies that 
\begin{equation}
    \label{eq:contribution-deltaN-to-zetaN-bis}
    \frac{1}{\sqrt{K'}} \sum_{\nu \in \Gcal^{'}_N} \delta_N(f,\nu) = \Ocal_{\prec}\left( \left( \frac{K' B}{N}\right)^{1/2}\right) = \Ocal_{\prec}(N^{-\delta/2}) = o_P(1).
\end{equation}
Therefore, in order to prove the CLT (\ref{eq:clt-zeta}), it is sufficient to check that 
$$
\zeta_{N,2,w}(f) = \frac{1}{\sqrt{K'}} \sum_{\nu \in \Gcal^{'}_N}  \left( \left(B w_{N,\mathfrak{b}}(f,\nu)\right)^{2} - \sigma^{2}_N(f)\right),
$$
verifies 
\begin{equation}
\label{eq:clt-zeta-w}
\frac{\zeta_{N,2,w}(f)}{\sqrt{2} \sigma^{2}_N(f)} \rightarrow_{\Dcal} \Ncal(0,1).
\end{equation}
(\ref{eq:convergence-EWN2}) implies that
\begin{align*}
 \zeta_{N,2,w}(f)  & =  \frac{1}{\sqrt{K'}} \sum_{\nu \in \Gcal^{'}_N}  \left( \left(B w_{N,\mathfrak{b}}(f,\nu)\right)^{2} - \sigma^{2}_N(f)\right)^{\circ} + \Ocal\left( \frac{\sqrt{K'}}{B} \right) \\
  & = \frac{1}{\sqrt{K'}} \sum_{\nu \in \Gcal^{'}_N}  \left( \left(B w_{N,\mathfrak{b}}(f,\nu)\right)^{2} - \sigma^{2}_N(f)\right)^{\circ} + o_(1),
\end{align*}
because $\frac{\sqrt{K'}}{B} = \Ocal\left( \left( \frac{N^{1-\delta}}{B^{3}} \right)^{1/2} \right) = o(1)$ for $\alpha \geq \frac{1}{3}$. 
(\ref{eq:convergence-EWN2}) and (\ref{eq:convergence-EWN4}) lead to 
$$
\mathbb{E}\left[\left( \left(B w_{N,\mathfrak{b}}(f,\nu)\right)^{2} - \sigma^{2}_N(f)\right)^{\circ}\right]^{2} = 2 \sigma_N^{4}(f) + 
O(B^{-1}).
$$
Therefore, (\ref{eq:clt-zeta-w}) is an immediate consequence of the standard CLT on the empirical mean 
of the zero mean i.i.d. random variables \\ $\left(\left( \left(B w_{N,\mathfrak{b}}(f,\nu)\right)^{2} - \sigma^{2}_N(f)\right)^{\circ}\right)_{\nu \in \Gcal^{'}_N}$. 
\begin{remark}
\label{re:chi2-approximation}    
(\ref{eq:clt-zeta-w}) implies that if $\chi^{2}(K')$ represents a $\chi^{2}$ random 
variable with $K'$ degrees of freedom, then, for each $a$
\begin{equation}
    \label{eq:approximation-chi2-zeta2}
    P\left( \frac{\zeta_{N,2}(f)}{\sqrt{2} \sigma^{2}_N(f)} > a \right) - 
    P\left( \frac{\chi^{2}(K') - K'}{\sqrt{2K'}} > a \right) \rightarrow 0,
\end{equation}
when $N \rightarrow +\infty$. To justify (\ref{eq:approximation-chi2-zeta2}), it is sufficient to remark that the standard CLT implies that 
$\frac{\chi^{2}(K') - K'}{\sqrt{2K'}}$ converges in distribution towards 
a standard Gaussian random variable. Therefore, (\ref{eq:clt-zeta}) leads to 
the conclusion that the 2 terms at the left-hand side of (\ref{eq:approximation-chi2-zeta2}) converge towards the same limit, which, of course, implies 
that (\ref{eq:approximation-chi2-zeta2}) holds. In other words, in distribution, 
$\sum_{k \in \Gcal_N^{'}} \left( \theta_N(f)\right)^{2}$ can be approximated by  
$\sigma^{2}_N(f) \chi^{2}(K')$. We will see in Section \ref{sec:simulations} devoted to the numerical simulations that this $\chi^{2}$ approximation may allow to predict more accurately than the Gaussian approximation based on (\ref{eq:clt-zeta-w}) the type I error of the test consisting in comparing $\frac{\zeta_{N,2}(f)}{\sqrt{2} \sigma^{2}_N(f)}$ to a threshold. 
\end{remark}

\begin{remark}
\label{re:alpha-larger-7-9}
If $\alpha \in (\frac{7}{9}, \frac{4}{5})$,  (\ref{eq:clt-zeta1}) and (\ref{eq:clt-zeta}) could be generalized
if the $\Ocal\left( \frac{B^{4}}{N^{4}}\right)$ term of $\mathbb{E}(\theta_N(f,\nu))$ were evaluated in closed 
form and substracted from $\theta_N(f,\nu)$. In this case, the recentered statistic 
would still have a representation (\ref{eq:representation-theta-bartlett}) in which the error term would be obtained by substracting the above mentioned $\Ocal\left( \frac{B^{4}}{N^{4}}\right)$ term from 
$\kappa_{N,\mathfrak{b}}$ given by (\ref{eq:expre-kappab-clt}). It is reasonable to conjecture that the new error term would appear to be a 
$\Ocal\left( \frac{B^{6}}{N^{6}}\right) + \Ocal_{\prec}\left( \frac{1}{\sqrt{BN}}\right)$ term, or equivalently, $\mathbb{E}(\theta_N(f,\nu))$ does not contain $\Ocal\left( \frac{B^{5}}{N^{5}}\right)$ term. This behaviour is sustained by the observation that the $\Ocal\left( \frac{B^{3}}{N^{3}}\right)$ term of 
$\mathbb{E}(\theta_N(f,\nu))$ is reduced to $0$. 
$\alpha < \frac{4}{5}$ implies that $\frac{B^{6}}{N^{6}} = o\left( \frac{1}{\sqrt{BN}}\right)$,  and that the new error term would be a $\Ocal_{\prec}\left( \frac{1}{\sqrt{BN}}\right)$ term. Therefore, the recentered 
statistics obtained from $\theta_N(f,\nu)$ would still have a representation (\ref{eq:representation-theta-bartlett}), and  (\ref{eq:clt-zeta1}) and (\ref{eq:clt-zeta}) could be generalized. As mentioned above, the closed form evaluation of 
the $\Ocal\left( \frac{B^{4}}{N^{4}}\right)$ term of $\mathbb{E}(\theta_N(f,\nu))$ is a tremendous calculation. 
\end{remark}

\begin{remark}
\label{eq:explanation-Gprime}
We remark that if $\zeta_{N,1}(f)$ and  $\zeta_{N,2}(f)$ were built on a combination of the $B \theta_N(f,\nu)$ for 
$\nu \in \mathcal{G}_N$, or equivalently if the parameter $\delta$ in Eq. (\ref{eq:contribution-deltaN-to-zetaN}) 
was equal to $0$, the rough evaluation of the contribution of the error terms
$(\kappa_{N,\mathfrak{b}}(f,\nu))_{\nu \in \Gcal_N}$ and 
$(\delta_N(f,\nu))_{\nu \in \Gcal_N}$ to 
$\zeta_{N,1}(f)$ and $\zeta_{N,2}(f)$ (see Eqs. (\ref{eq:contribution-deltaN-to-zetaN}) and 
(\ref{eq:contribution-deltaN-to-zetaN-bis}))  would imply that they would be $ \Ocal_P(1)$ terms. This explains why we choose to 
combine the $B \theta_N(f,\nu)$ on the smaller frequency grid $\Gcal_N^{'}$. However, intuitively, 
the random variables $(B \kappa_{N,\mathfrak{b}}(f,\nu))_{\nu \in \Gcal_N}$ and
$(\delta_N(f,\nu))_{\nu \in \Gcal_N}$, while not mutually independent, should nearly behave as independent random variables. This is because for each $m$, the renormalized Fourier transform $\xi_{y_m}(\nu)$ of $(y_{m,n})_{n=1, \ldots, N}$,  verifies 
$\mathbb{E}(\xi_{y_m}(\nu_2) \xi_{y_m}(\nu_1)^{*}) = \Ocal(N^{-1})$ if $\nu_2 - \nu_1$ is a non zero integer multiple of $\frac{1}{N}$. Therefore, it is reasonable to expect that the entries of $\hat{\C}(\nu_1)$ and 
$\hat{\C}(\nu_2)$ are nearly independent, in a sense to be defined, for $\nu_1, \nu_2 \in \Gcal_N$, $\nu_1 \neq \nu_2$, and that the same property should hold for
$B \kappa_{N,\mathfrak{b}}(f,\nu_1)$ and $B \kappa_{N,\mathfrak{b}}(f,\nu_2)$ as well as for 
$\delta_N(f,\nu_1)$ and $\delta_N(f,\nu_2)$. It might therefore 
be possible that
$\frac{1}{\sqrt{K}} \sum_{\nu \in \Gcal_N} B \kappa_{N,\mathfrak{b}}(f,\nu) = o_P(1)$
and 
$\frac{1}{\sqrt{K}} \sum_{\nu \in \Gcal_N} \delta_N(f,\nu) = o_P(1)$, even if the order of magnitude of each individual terms $B \kappa_{N,\mathfrak{b}}(f,\nu)$ and $\delta_N(f,\nu)$ are $\Ocal_P(\frac{B^{1/2}}{N^{1/2}})$. To establish such a result, a reasonable approach would consist in showing that
$$\mathbb{E}\left( \frac{1}{\sqrt{K}} \sum_{\nu \in \Gcal_N} B \kappa_{N,\mathfrak{b}}(f,\nu) \right)^{2} = o(1),$$ and 
$$\mathbb{E}\left( \frac{1}{\sqrt{K}} \sum_{\nu \in \Gcal_N} \delta_N(f,\nu) \right)^{2} = o(1).$$ For this, it would be necessary to evaluate the order of magnitude of the two terms
$$\mathbb{E}( B \kappa_{N,\mathfrak{b}}(f,\nu_1)  B \kappa_{N,\mathfrak{b}}(f,\nu_2)), \ \mathbb{E}( \delta_N(f,\nu_1) \delta_N(f,\nu_2)),$$ for $\nu_1, \nu_2 \in \Gcal_N$ using the integration by parts formula. However, this calculation appears tremendous. As the proof of such a result would only avoid the use of $\Gcal_N^{'}$ instead of $\Gcal_N$, we prefer to let this point for a future work. 
\end{remark}

\subsection{\texorpdfstring{Estimation of $r_N(\nu)$}{Estimation of r}}
\label{subsec:estimation-rN}
In practice, the term $r_N(\nu)$, defined by (\ref{def-eq-rN}), is unknown, and has thus to be estimated 
in order to be able to define a test statistic from the available observations. For this, we 
use the estimator $\hat{r}_N(\nu)$ proposed in \cite{loubaton-rosuel-ejs-2021}, and defined by 
\begin{equation}
\label{eq:def-hat-rN}
\hat{r}_N(\nu) = \left( \frac{1}{M} \sum_{m=1}^{M} \frac{\hat{s}^{'}_{m,L}(\nu)}{\hat{s}_{m,L}(\nu)}\right)^{2},
\end{equation}
where $\hat{s}_{m,L}(\nu)$ is the lag-window estimator of $s_m(\nu)$ given by 
\begin{equation}
\label{eq:def-hat-sm}
\hat{s}_{m,L}(\nu) = \sum_{l=-L}^{L} \hat{r}_{m,l} e^{- 2 i \pi l \nu},
\end{equation}
and $ \hat{r}_{m,l} = \frac{1}{N} \sum_{n=1}^{N-l} y_{m,n+l} y_{m,n}^{*}$ for $l \geq 0$ and 
$ \hat{r}_{m,l} =  \hat{r}_{m,-l}^{*}$ for $ l < 0$. $L$ is an integer that has to be chosen in a relevant way, while in formula (\ref{eq:def-hat-rN}), $\hat{s}^{'}_{m,L}(\nu)$ represents the derivative of $\hat{s}_{m,L}(\nu)$ 
w.r.t. $\nu$. We thus replace in practice $\theta_N(f,\nu), \zeta_{N,1}(f)$ and $\zeta_{N,2}(f)$ 
by the statistics $\hat{\theta}_N(f,\nu), \hat{\zeta}_{N,1}(f)$ and $\hat{\zeta}_{N,2}(f)$  obtained  
by replacing $r_N(\nu)$ by $\hat{r}_N(\nu)$. We now verify that, provided the term $\gamma_0$ (supposed to 
verify $\gamma_0 > 4$)  
defined by (\ref{eq:decroissance-rm}) is large enough, the statistics $\hat{\theta}_N(f,\nu), \hat{\zeta}_{N,1}(f)$ and $\hat{\zeta}_{N,2}(f)$ verify the CLT (\ref{eq:CLT-thetaN}) for $\alpha < \frac{4}{5}$, 
and (\ref{eq:clt-zeta1}) and (\ref{eq:clt-zeta}) for $\alpha < \frac{7}{9}$. More precisely, the following result holds. 
\begin{proposition}
\label{prop:clt-hat-rN}
If $L = L(N) = \mathcal{O}\left( N^{\frac{1}{2 \gamma_0 + 1}} \right)$, then, if $\alpha < \frac{4}{5}$, 
$\hat{\theta}_N(f,\nu)$ verifies  (\ref{eq:CLT-thetaN}) if $\gamma_0 > 4$ and 
\begin{equation}
\label{eq:seuil-gamma0-clt-hat-theta}
\gamma_0 > \frac{3 \alpha - 1}{5 - 6 \alpha} \, \mathds{1}_{\alpha \geq 2/3}.
\end{equation}
Moreover, if $\alpha < \frac{7}{9}$, then, $\hat{\zeta}_{N,1}(f)$ and $\hat{\zeta}_{N,2}(f)$ verify
(\ref{eq:clt-zeta1}) and (\ref{eq:clt-zeta}) if $\gamma_0 > 4$ and 
\begin{equation}
\label{eq:seuil-gamma0-clt-hat-zeta}
\gamma_0 > \frac{5 \alpha - 1}{2(4 - 5 \alpha)}  \, \mathds{1}_{\alpha \geq 3/5}.
\end{equation}
\end{proposition}
\begin{proof}
We recall that by Proposition 4 in \cite{loubaton-rosuel-ejs-2021}, for $L =  \mathcal{O}\left( N^{\frac{1}{2 \gamma_0 + 1}} \right)$, the estimator $\hat{r}_N(\nu)$ verifies 
\begin{equation}
\label{eq:accuracy-hat-rN}
|\hat{r}_N(\nu) - r_N(\nu)| \prec \frac{1}{N^{\frac{\gamma_0-1}{2 \gamma_0+1}}}.
\end{equation}
In order to check that $\hat{\theta}_N(f,\nu)$ satisfies (\ref{eq:CLT-thetaN}) if $\gamma_0$ 
verifies (\ref{eq:seuil-gamma0-clt-hat-theta}), we express $B \hat{\theta}_N(f,\nu)$ as 
$$
B \hat{\theta}_N(f,\nu) = B \theta_N(f,\nu) + <D_N, f> \, B v_N \, \left( \hat{r}_N(\nu) - r_N(\nu) \right).
$$
Using (\ref{eq:accuracy-hat-rN}), we obtain immmediately that 
$B  v_N \, \left( \hat{r}_N(\nu) - r_N(\nu) \right) = o_{\prec}(1)$ if $\gamma_0$ verifies 
(\ref{eq:seuil-gamma0-clt-hat-theta}). This implies that  $\hat{\theta}_N(f,\nu)$ satisfies (\ref{eq:CLT-thetaN}). 
To justify the CLTs on (\ref{eq:clt-zeta1}) and (\ref{eq:clt-zeta}), it is sufficient to 
remark that $\sqrt{NB} v_N \, \left( \hat{r}_N(\nu) - r_N(\nu) \right) = o_{\prec}(1)$ if $\gamma_0$ verifies 
(\ref{eq:seuil-gamma0-clt-hat-zeta}). 
\end{proof}
We finally remark that (\ref{eq:seuil-gamma0-clt-hat-theta}) (resp. (\ref{eq:seuil-gamma0-clt-hat-zeta})) 
holds whatever $\alpha < \frac{4}{5}$ (resp. $\alpha < \frac{7}{9}$) if $\gamma_0 > 7$ (resp. $\gamma_0 > 13$).

\section{Power analysis}
\label{sec:power}
A thorough study of the power of the proposed tests would need to establish a central limit theorem for the statistics $\zeta_{N,1}(f)$ and $\zeta_{N,2}(f)$ under a large class of alternatives. As this task is far from being obvious, we believe that a deep power analysis of $\zeta_{N,1}(f)$ and $\zeta_{N,2}(f)$ deserves an extra  paper, and is therefore not in the scope of the present contribution. However, we provide in this paragraph a rough analysis suggesting that $\zeta_{N,1}(f)$ and $\zeta_{N,2}(f)$ provide consistent tests under a simple alternative when the fonction $f(\lambda)$ coincides with $f(\lambda) = (\lambda - 1)^{2}$ considered in the section devoted to the numerical experiments. 

We assume that under the alternative $\Hcal_1$, the observation $(\y^{1}_n)_{n \in \mathbb{Z}}$ is given by 
\begin{equation}
    \label{eq:modele-alternative}
\y^{1}_n = \A_N^{*} \y_n    ,
\end{equation}
where $(\y_n)_{n \in \mathbb{Z}}$ is generated as under hypothesis $\Hcal_0$, and where $\A_N = (\a_{1,N}, \ldots, \a_{M,N})$ is a deterministic $M \times M$ matrix verifying 
\begin{equation}
    \label{eq:condition-norme-AN}
     \sup_{N} \| \A_N \| \leq a_{+} < +\infty,
\end{equation}
and 
\begin{equation}
\label{eq:condition-norme-colonnes-AN}
\inf_N \min_{m=1, \ldots, M} \|\a_{m,N}\| \geq a_{-} > 0,
\end{equation}
for each $N$ large enough. \\

In the following, we provide examples of matrices $\A_N$ such that, for each level $\tilde{\alpha}$, 
\begin{equation}
    \label{eq:consistency-zeta1}
\mathbb{P}(|\zeta_{N,1}^1(f)| > q_{1 - \tilde{\alpha}}) \rightarrow 1,
\end{equation}
and 
\begin{equation}
    \label{eq:consistency-zeta2}
\mathbb{P}(|\zeta_{N,2}^1(f)| > q_{1 - \tilde{\alpha}}) \rightarrow 1,
\end{equation}
when $N \rightarrow +\infty$ and when $f(\lambda) = (\lambda - 1)^{2}$. Here, $\zeta_{N,i}^1(f)$ for $i=1,2$ represents the statistics obtained from $\zeta_{N,i}(f)$ by replacing $\y$ by the observation $\y^1$ under the hypothesis $\Hcal_1$. In the following, we just consider property (\ref{eq:consistency-zeta1}) 
because (\ref{eq:consistency-zeta2}) is an immediate consequence of the forthcoming calculations. \\
 
We first define some notations. We denote by $\hat{\S}_{\y^{1}}(\nu)$ and 
$\hat{\C}_{\y^1}(\nu)$ the estimators of the spectral density $\S_{\y^{1}}(\nu)$
and of the spectral coherence matrix $\C_{\y^{1}}(\nu)$ of $\y^{1}$ given by 
\begin{eqnarray*}
\hat{\S}_{\y^{1}}(\nu) & = & \A_N^{*} \hat{\S}(\nu) \A_N, \\
\hat{\C}_{\y^{1}}(\nu) & = & \hat{\D}^{-1/2}_{\y^{1}}(\nu) \A_N^{*} \hat{\S}(\nu) \A_N \hat{\D}^{-1/2}_{\y^1}(\nu),
\end{eqnarray*}
where $\hat{\D}_{\y^1}(\nu) = \mathrm{dg}\left( \hat{\S}_{\y^1}(\nu)\right) = 
\mathrm{Diag}\left( \a_m^{*} \hat{\S}(\nu) \a_m,  m=1, \ldots, M \right)$
is the estimate of the diagonal matrix $\D_{\y^1}(\nu) = \mathrm{Diag}\left( \a_m^{*} \S(\nu) \a_m,  m=1, \ldots, M \right)$. We recall that $\S(\nu)$ is diagonal, so that 
matrix $\S(\nu)$ coincides with \\ $\D(\nu) = \mathrm{Diag}(s_m(\nu), m=1, \ldots, M)$. As matrix 
$\hat{\S}(\nu)$ can be written as $\hat{\S}(\nu) = \D^{1/2} (\nu) \tilde{\C}(\nu) \D^{1/2} (\nu)$, $\hat{\C}_{\y^1}(\nu)$ can be expressed as 
\begin{equation}
    \label{eq:expre-hat-C-y}
\hat{\C}_{\y^1}(\nu) = \hat{\D}^{-1/2}_{\y^1}(\nu) \A_N^{*} \D(\nu)^{1/2} \tilde{\C}(\nu) \D(\nu)^{1/2}  \A_N \hat{\D}^{-1/2}_{\y^1}(\nu) .  
\end{equation}
It is easy to check that (\ref{eq:condition-norme-AN}) and (\ref{eq:condition-norme-colonnes-AN}) imply that 
$\|  \hat{\D}_{\y^1}(\nu) -  \D_{\y^1}(\nu) \| \prec \frac{1}{\sqrt{B}}$
as well as
$\|  \hat{\D}^{-1/2}_{\y^1}(\nu) -  \D^{-1/2}_{\y^1}(\nu) \| \prec \frac{1}{\sqrt{B}}$. As we also have 
$$
\| \tilde{\C}(\nu) - \frac{\X(\nu) \X^{*}(\nu)}{B+1} \| \prec \frac{B}{N},
$$
we obtain that 
\begin{equation}
    \label{eq:approx-hat-C-y}
\left\|  \hat{\C}_{\y^1}(\nu) - \Z_N(\nu) \right\| \prec \frac{1}{\sqrt{B}} + \frac{B}{N},
\end{equation}
where $\Z_N(\nu)$ is the matrix defined by 
\begin{equation}
    \label{eq:def-Z}
\Z_N(\nu) = \D^{-1/2}_{\y^1}(\nu) \A_N^{*} \D^{1/2} (\nu)
 \frac{\X(\nu) \X^{*}(\nu)}{B+1} \D^{1/2} (\nu) \A_N \D^{-1/2}_{\y^1}(\nu) . 
\end{equation}
(\ref{eq:approx-hat-C-y}) suggests that the linear spectral statistics $\hat{f}_1(\nu)$  defined  by
\begin{equation}
    \label{eq:def-hat-f-1}
\hat{f}_1(\nu) = \frac{1}{M} \Tr f\left(\hat{\C}_{\y^1}(\nu)\right)  ,   
\end{equation}
and $\hat{f}_{Z}(\nu)$ given by 
\begin{equation}
    \label{eq:def-hat-f-Z}
\hat{f}_{Z}(\nu) = \frac{1}{M} \Tr f\left(\Z_N(\nu)\right),
\end{equation}
have the same behaviour. This is a crucial observation. More precisely, we denote by $\theta_N^{1}(f,\nu)$ the term $\theta_N(f,\nu)$ defined by (\ref{eq:def-theta-f-nu}) under 
$\Hcal_0$, but now evaluated under 
hypothesis $\Hcal_1$, i.e., 
\begin{equation}
    \label{eq:def-thetaN-H1}
    \theta_N^{1}(f,\nu) = \hat{f}_1(\nu) - \int f(\lambda) d\mu_{MP}^{(c_N)}(\lambda) - <D_N,f>\left(r_N^{1}(\nu) v_N - \frac{1}{c_N B}\right),
\end{equation}
where $r_N^{1}(\nu)$ is defined by
\begin{equation}
    \label{eq:def-rN-H1}
   r_N^{1}(\nu) = \left( \frac{1}{M} \sum_{m=1}^{M} \frac{(s_{m}^{1})^{'}(\nu)}{s_m^{1}(\nu)}\right),
\end{equation}
and $s_m^{1}(\nu)$ represents $\left(\S_{\y^{1}}(\nu)\right)_{m,m}$. 
We define $\theta_{N,Z}(f)$ as the term given by 
\begin{equation}
    \label{eq:def-theta-Z}
 \theta_{N,Z}(f,\nu) =  \hat{f}_{\Z}(\nu) -  \int f(\lambda) d\mu_{MP}^{(c_N)}(\lambda) - <D_N,f>(r_N^{1}(\nu) v_N - \frac{1}{c_N B})  . 
\end{equation}
Then, showing that $\hat{f}_1(\nu) \simeq \hat{f}_{Z}(\nu)$ will imply that
the behaviours of $\theta_N^{1}(f,\nu)$ and $\theta_{N,Z}(f,\nu)$ under $\Hcal_1$
will coincide. This is a useful observation because the asymptotic behaviour of 
$\hat{f}_{Z}(\nu)$ is of course well known. \\

We now establish that 
 $\hat{f}_1(\nu) - \hat{f}_{Z}(\nu) \rightarrow 0$, and evaluate the order of 
 magnitude of the error. The function $f(\lambda) = (\lambda -1)^{2}$ is of course not compactly supported  so that it is not possible to use the Helffer-Sjöstrand formula to evaluate $\hat{f}_1(\nu) - \hat{f}_{Z}(\nu)$ 
in terms of the difference between the normalized trace of the resolvent 
of $\hat{\C}_{\y^{1}}$ and the normalized trace of the resolvent 
of $\Z$. As shown in Paragraph \ref{subsec:location-eigenvalues-f-compactly-supported}, under $\Hcal_0$, the study of the linear statistic $\hat{f}(\nu)$ and of $\zeta_{N,1}(f)$ and $\zeta_{N,2}(f)$ is equivalent to that of 
$\hat{\bar{f}}(\lambda)$ and $\zeta_{N,1}(\bar{f})$ and $\zeta_{N,2}(\bar{f})$
where $\bar{f}$ is a $\Ccal^{\infty}$ compactly supported function that coincides with $f$ in a neigbourhood of the support $\mathrm{Supp}(\mu_{MP}^{(c)})$ of the Marcenko-Pastur distribution $\mu_{MP}^{(c)}$. This useful result is a direct consequence of Proposition \ref{prop:localisation-eigenvalues-hatC-tildeC}. We now generalize this property to the context of the alternative $\Hcal_1$. For this, we first need to study the location of the eigenvalues of $\hat{\C}_{\y^1}(\nu)$ and of $\Z(\nu)$. In order to introduce 
the corresponding result, we define by $\H_N(\nu)$ the positive matrix 
given by 
\begin{equation}
    \label{eq:def-HN}
    \H_N(\nu) = \D^{-1/2}_{\y^1}(\nu) \A_N^{*} \D(\nu) \A_N \D^{-1/2}_{\y^1}(\nu),
\end{equation}
and notice that conditions (\ref{eq:condition-norme-AN}, \ref{eq:condition-norme-colonnes-AN}) imply that $\H_N(\nu)$ verifies 
$$
\| \H_N(\nu) \| \leq \tilde{a} ,
$$
where $\tilde{a}$ is defined by 
\begin{equation}
    \label{eq:def-tilde-a}
  \tilde{a} = \frac{s_{max}}{s_{min}} \, \frac{a_{+}^{2}}{a_{-}^{2}}  ,
\end{equation}
where $s_{min} > 0$ and $s_{max} < +\infty$ verify $s_{min} \leq s_m(\nu) \leq s_{max}$ for each $m$ and each $\nu$. 
It is useful to notice that $\tilde{a} \geq 1$. 
For each $\epsilon > 0$, we define the events $\Lambda_{N,\epsilon}^{\hat{\C}_{\y^1}}(\nu)$ and 
$\Lambda_{N,\epsilon}^{\Z}(\nu)$ by 
\begin{eqnarray}
    \label{eq:def-Lambda-hatC-y1}
&  \Lambda^{\hat{\C}_{\y^1}}_{N,\epsilon}(\nu)  = \{ \| \hat{\C}_{\y^1}(\nu) \| \leq 
 \tilde{a} (1 + \sqrt{c} + \epsilon)^{2} \}, \\
    \label{eq:def-Lambda-Z}
&  \Lambda^{\Z}_{N,\epsilon}(\nu)  = \{ \| \Z_N(\nu) \| \leq 
 \tilde{a} (1 + \sqrt{c} + \epsilon)^{2} \} .
\end{eqnarray}
Then, we have the following proposition.
\begin{proposition}
\label{prop:localisation-vp-hatCy1-Z}
The families of events $\left(\Lambda^{\Z}_{N,\epsilon}(\nu)\right)_{\nu \in [0,1]}$ 
and $\left(\Lambda^{\hat{\C}_{\y^1}}_{N,\epsilon}(\nu)\right)_{\nu \in [0,1]}$
hold with exponentially high probability. 
\end{proposition}
\begin{proof}
We first remark that $\| \Z_N \| \leq \| \H_N \| \| \frac{\X_N \X_N^{*}}{B+1} \|$. Therefore, 
\begin{align*}
\left(\Lambda^{\Z}_{N,\epsilon}(\nu)\right)^{c} \subset & \left\{ \| \frac{\X_N(\nu) \X_N^{*}(\nu)}{B+1} \| > \frac{1}{\|\H_N(\nu)\|} \tilde{a} (1 + \sqrt{c} + \epsilon)^{2} \right \} \\
\subset & \left\{ \| \frac{\X_N(\nu) \X_N^{*}(\nu)}{B+1} \| > (1 + \sqrt{c} + \epsilon)^{2} \right \},
\end{align*}
and Proposition \ref{prop:Lambda-holds-with-high-proba} implies that 
$\left(\Lambda^{\Z}_{N,\epsilon}(\nu)\right)_{\nu \in [0,1]}$ holds with exponentially high probability. To address $\Lambda^{\hat{\C}_{\y^1}}_{N,\epsilon}(\nu)$, we use the inequality 
$$
\| \hat{\C}_{\y^1} \| \leq \| \Z \| + \| \hat{\C}_{\y^1} - \Z \|.
$$
Therefore, we have 
\begin{align*}
& P\left( \left(\Lambda^{\hat{\C}_{\y^1}}_{N,\epsilon}(\nu)\right)^{c} \right) \leq  
P \left(\left(\Lambda^{\Z}_{N,\epsilon/2}(\nu)\right)^{c} \right) + \\
& 
P\left( \| \hat{\C}_{\y^1} - \Z \| > \tilde{a} \left( (1+\sqrt{c}+\epsilon)^{2}
- (1+\sqrt{c}+\epsilon/2)^{2}\right)\right).
\end{align*}
(\ref{eq:approx-hat-C-y}) thus implies that $\left(\Lambda^{\hat{\C}_{\y^1}}_{N,\epsilon}(\nu)\right)_{\nu \in [0,1]}$
holds with exponentially high probability. This completes the proof of 
Proposition \ref{prop:localisation-vp-hatCy1-Z}. 
\end{proof}
We now define $\bar{f}$ as any function defined 
on $\mathbb{R}$, $\Ccal^{\infty}$ in a neigbourhood of $[0, \tilde{a}(1+\sqrt{c}+2 \epsilon)^{2}]$, and verifying 
\begin{eqnarray}
\label{eq:def-bar-f-H1}
&    \bar{f}(\lambda) = f(\lambda) & , \, \lambda \in [0, \tilde{a}(1+\sqrt{c}+ \epsilon)^{2}], \\
\notag
&   \bar{f}(\lambda) =  0 & , \, \lambda >  \tilde{a}(1+\sqrt{c}+2 \epsilon)^{2}.
\end{eqnarray}
Reasoning as in the proof of Proposition \ref{prop:phi-compactly-supported}, 
we deduce from Proposition \ref{prop:localisation-vp-hatCy1-Z} that the families 
$(\hat{f}_{Z}(\nu) - \hat{\bar{f}}_{Z}(\nu))_{\nu \in [0,1]}$ and 
$(\hat{f}_1(\nu) - \hat{\bar{f}}_1(\nu))_{\nu \in [0,1]}$ verify
\begin{eqnarray*}
(\hat{f}_{Z}(\nu) - \hat{\bar{f}}_{Z}(\nu))_{\nu \in [0,1]} = & \Ocal_{\prec}\left(\frac{1}{N^{p}}\right), \\
(\hat{f}_1(\nu) - \hat{\bar{f}}_1(\nu))_{\nu \in [0,1]} = & \Ocal_{\prec}\left(\frac{1}{N^{p}}\right),
\end{eqnarray*}
for each $p > 0$. We also remark that, as $\tilde{a} \geq 1$ and $c_N \rightarrow c$, for each $N$ large enough, then 
$f(\lambda) = \bar{f}(\lambda)$ on the support $[(1- \sqrt{c}_N)^{2}, (1+ \sqrt{c}_N)^{2}]$ of the Marcenko-Pastur 
distribution $\mu_{MP}^{(c_N)}$. Therefore, we also have 
\begin{align}
\label{eq:egalite-f-barf-on-calSN}
& \int f(\lambda) d\mu_{MP}^{(c_N)}  =  \int \bar{f}(\lambda) d\mu_{MP}^{(c_N)}, \\
\label{eq:egalite-D-f-D-barf}
& <D_N, f>  =  <D_N, \bar{f}>,
\end{align}
for each $N$ large enough. 
Then, (\ref{eq:egalite-f-barf-on-calSN}) and (\ref{eq:egalite-D-f-D-barf}) imply that 
$ \theta_N^{1}(f,\nu) =  \theta_N^{1}(\bar{f},\nu) + \Ocal_{\prec}\left( \frac{1}{N^{p}}\right)$ and that $\zeta_{N,1}^1(f) = \zeta_{N,1}^1(\bar{f}) + \Ocal_{\prec}\left( \frac{1}{N^{p}}\right)$ for $i=1,2$ and for each $p$ . If we denote by 
 $(\zeta_{N,i}^{Z}(f))_{i=1,2}$ the statistics defined by 
\begin{align*}
& \zeta_{N,1}^{Z}(f)  =  \frac{1}{\sqrt{K'}} \sum_{\nu \in \Gcal_N^{'}} B \, \theta_{N,Z}(f,\nu), \\
& \zeta_{N,2}^{Z}(f)  =  \frac{1}{\sqrt{K'}} \sum_{\nu \in \Gcal_N^{'}} \left[ \left( B  \, \theta_{N,Z}(f,\nu)\right)^{2} - 
\sigma^{2} \right]  ,
\end{align*}
we also have
\begin{align}
& \theta_{N,Z}(f,\nu)  =  \theta_{N,Z}(\bar{f},\nu)+ \Ocal_{\prec}\left( \frac{1}{N^{p}}\right), \\
& \zeta_{N,i}^{Z}(f)  =   \zeta_{N,i}^{Z}(\bar{f}) +  \Ocal_{\prec}\left( \frac{1}{N^{p}}\right),
\end{align}
for $i=1,2$ and for each $p$. 

In the following, we can thus replace $\hat{f}_{Z}(\nu)$ and 
$\hat{f}_1(\nu)$ by $\hat{\bar{f}}_{Z}(\nu)$ and $\hat{\bar{f}}_1(\nu)$
without modifying significantly the behaviours of $\zeta_{N,i}^{1}(f)$ and 
$\zeta_{N,i}^{Z}(f)$. In order to simplify the notations, 
we prefer to denote $\hat{\bar{f}}_{Z}(\nu)$ and $\hat{\bar{f}}_1(\nu)$ by 
$\hat{f}_{Z}(\nu)$ and $\hat{f}_1(\nu)$. 
We deduce from this discussion the following result. 
\begin{proposition}
\label{prop:approx-hat-f1-hat-fZ}
$\hat{f}_1(\nu)$ can be written as 
\begin{equation}
    \label{eq:expre-hat-f-1}
\hat{f}_1(\nu) = \hat{f}_{Z}(\nu) + \delta(\nu)  , 
\end{equation}
as well as 
\begin{align*}
& \theta_N^{1}(f,\nu) =   \theta_{N,Z}(f,\nu) + B \delta(\nu) ,\\
& \zeta_{N,i}^1(f) =  \zeta_{N,i}^{Z}(f) + \Ocal_{\prec}\left( \sqrt{K'} B\left( \frac{1}{\sqrt{B}} + \frac{B}{N}\right)\right),
\end{align*}
where the family $(\delta(\nu))_{\nu \in [0,1]}$ verifies $\delta(\nu) = \Ocal_{\prec}\left( \frac{1}{\sqrt{B}} + \frac{B}{N} \right)$. 
\end{proposition}
\begin{proof}
The proof follows immediately from (\ref{eq:approx-hat-C-y}) and the Helffer-Sjöstrand formula applied to function $\bar{f}$.
\end{proof} 
This result shows that up to a $\Ocal_{\prec}\left( \frac{1}{\sqrt{B}} + \frac{B}{N} \right)$ term, it is sufficient to evaluate $\hat{f}_{Z}(\nu)$ in order to 
study the behaviour of $\hat{f}_1(\nu)$.  We remark that the order of magnitude of the error term $\delta(\nu)$ is very rough, and that a much tighter result could probably be obtained by establishing a CLT on $\hat{f}_1(\nu)$. \\

We now consider the behaviour of $\hat{f}_{Z}(\nu)$. It is well established (the reader may for example use the results of 
\cite{hachem-loubaton-najim-aap-2007}) that for each $\nu$
the eigenvalue distribution of $\Z_N(\nu)$ has the same asymptotic 
behaviour than the  deterministic probability distribution 
$\mu^{1}_N(\nu)$ whose Stieltjes transform $t^1_N(z,\nu)$ is solution 
of the equation 
$$
t^{1}_N(z,\nu) = \frac{1}{M} \Tr \left[ -z \left( \I + \H_N(\nu) \tilde{t}^1_N(z,\nu) \right) \right]^{-1} ,
$$
where $\tilde{t}^1_N(z,\nu) = c_N t^1_N(z,\nu) - \frac{1 -c_N}{z}$ is the Stieltjes transform
of $\tilde{\mu}^{1}_N(\nu) = c_N \mu^{1}_N(\nu) + (1 - c_N) \delta_0$. It is useful to evaluate the 
support of $\mu_N^{1}(\nu)$ in order to relate $\int f(\lambda) d\mu^{1}_N(\nu,\lambda)$ and 
$\int \bar{f}(\lambda) d\mu^{1}_N(\nu,\lambda)$
More precisely, the following Lemma is established in the Appendix \ref{sec:proof-lemma-support-mu1}.  
\begin{lemma}
    \label{le:support-mu1}
    For each $\epsilon > 0$, and for each $N$, the support of $\mu^{1}_N$ is included into $[0, \tilde{a} (1 + \sqrt{c} + \epsilon)^{2}]$. 
\end{lemma}
If $\bar{f}$ is defined by (\ref{eq:def-bar-f-H1}), this lemma of course implies that 
$$
\int f(\lambda) d\mu^{1}_N(\nu) =  \int \bar{f}(\lambda) d\mu^{1}_N(\nu),
$$
because $f(\lambda) = \bar{f}(\lambda)$ for $\lambda \in [0, \tilde{a} (1 + \sqrt{c} + \epsilon)^{2}]$. 
We deduce from this that the LSS $\hat{f}_{Z}$ verifies 
\begin{equation}
\label{eq:approx-hat-f-Z}
\hat{f}_{Z}(\nu) = \int f(\lambda) d\mu^{1}_N(\nu,\lambda) + \delta_{Z}(\nu),
\end{equation}
where $\delta_{Z}(\nu)$ verifies $\delta_{Z}(\nu) = \Ocal_{\prec}\left(\frac{1}{B}\right)$. To justify this
property, if $\Q_{Z}(z,\nu)$ represents 
the resolvent of $\Z(\nu)$, it is sufficient to establish using the Gaussian concentration 
inequality that $\frac{1}{M} \Tr \left(\Q_{Z}(z,\nu) - \mathbb{E}\left( \Tr \Q_{Z}(z,\nu) \right) \right) = \Ocal_{\prec,z}\left(\frac{1}{B}\right)$ and to prove using the integration by part formula that $\mathbb{E} \left( \frac{1}{M} \Tr \Q_{Z}(z,\nu) \right)-  t^{1}_N(z,\nu) = 
\Ocal_z\left( \frac{1}{B^{2}} \right)$, an evaluation valid because $\X_N(\nu)$ is complex Gaussian. The use of the Helffer-Sjöstrand formula to function $\bar{f}$ leads immediately to 
(\ref{eq:approx-hat-f-Z}). \\ 

We now evaluate the behaviour of $\theta_N^{1}(f,\nu)$ defined by (\ref{eq:def-thetaN-H1}). Using (\ref{eq:expre-hat-f-1}) and (\ref{eq:approx-hat-f-Z}), we obtain 
that $\theta_N^{1}(f,\nu) $ can be written as 
\begin{equation}
    \label{eq:expre-theta-1}
 \theta_N^{1}(f,\nu) = \int f(\lambda) d\mu^{1}_N(\nu,\lambda)  - 
 \int f(\lambda) d\mu^{c_N}_{MP}(\lambda) +  \delta_N^1(\nu),
\end{equation}
where $\delta^1_N(\nu)$ is defined by 
$$
\delta^1_N(\nu) = \delta_N(\nu) + \delta_{Z,N}(\nu) - < D_N,f> \left( r^1_N(\nu) v_N -\frac{1}{c_N B} \right) ,
$$
and verifies
\begin{equation}
\label{eq:order-magnitude-delta-1}
\delta^1_N(\nu) = \Ocal_{\prec}\left( \frac{1}{\sqrt{B}} + \frac{B}{N} \right).
\end{equation}
The statistic $\zeta_{1,N}^1(f)$ is thus given by 
\begin{equation}
\label{eq:decomposition-zeta-1}
\zeta_{N,1}^1(f) = d_N^1(f)  + \epsilon^1_N(f),
\end{equation}
where 
$$
\epsilon^1_N(f) = \frac{1}{\sqrt{K'}} \sum_{\nu \in \Gcal_N^{'}}  B \, \delta_N^1(\nu) ,
$$
and where the deterministic term $d^1_N(f)$ is defined by 
\begin{equation}
\label{eq:def-d1N}
d_N^1(f) =  \frac{1}{\sqrt{K'}} \left( \sum_{\nu \in \Gcal_N^{'}} B \left(\int f(\lambda) d\mu^{1}_N(\nu,\lambda)  - 
 \int f(\lambda) d\mu^{c_N}_{MP}(\lambda)\right) \right) .
\end{equation}
We deduce from (\ref{eq:decomposition-zeta-1}) a general sufficient condition 
for the consistency of the test based on $\zeta_{1,N}$ under the 
present alternative. 
\begin{proposition}
\label{prop:general-condition-for-consistency}
If the matrices $(\A_N)_{N \geq 1}$ verify 
 \begin{equation}
     \label{eq:condition-consistency-zeta-1}
 \min_{\nu \in \Gcal_N^{'}} \int f(\lambda) d\mu^{1}_N(\nu,\lambda)  - 
 \int f(\lambda) d\mu^{c_N}_{MP}(\lambda) > \kappa > 0,
  \end{equation}
for each $N$ large enough, where $\kappa$ is some constant, then, 
(\ref{eq:consistency-zeta1}) holds.
\end{proposition}
\begin{proof}
If (\ref{eq:condition-consistency-zeta-1}) holds, 
$d_N^1(f)$ verifies $d_N^1(f) > \sqrt{K'} B \, \kappa$, 
so that $\zeta_{N,1}^1(f) > \kappa \sqrt{K'} B - |\epsilon^1_N(f)|$. As 
$|\epsilon^1_N(f)| \prec \sqrt{K'} B \left( \frac{1}{\sqrt{B}} + \frac{B}{N} \right)$, 
it is clear that for each level $\tilde{\alpha}$, 
$$
\lim_{N \rightarrow +\infty} \mathbb{P}\left( \zeta_{N,1}^1(f) > q_{1-\tilde{\alpha}} \right) = 1,
$$
which of course implies that (\ref{eq:consistency-zeta1}). 
\end{proof}
\begin{remark}
    \label{eq:weaker-condition-for-consistency}
    We remark that (\ref{eq:consistency-zeta1}) still holds if (\ref{eq:condition-consistency-zeta-1}) is replaced by 
\begin{equation}
     \label{eq:weaker-condition-consistency-zeta-1}
 \min_{\nu \in \Gcal_N^{'}} \int f(\lambda) d\mu^{1}_N(\nu,\lambda)  - 
 \int f(\lambda) d\mu^{c_N}_{MP}(\lambda) > \kappa \frac{1}{N^{\beta}},
  \end{equation}
where $\beta$ is chosen in such a way that 
$$
\frac{1}{\sqrt{B}} + \frac{B}{N} = o\left(  \frac{1}{N^{\beta}} \right).
$$   
\end{remark}
\begin{remark}
The Proposition \ref{prop:general-condition-for-consistency} and the Remark \ref{eq:weaker-condition-for-consistency} are, of course, still valid if 
$f$ is any function $\Ccal^{\infty}$ in a neighborhood of $[0, \tilde{a}(1 + \sqrt{c} + \epsilon)^{2}]$.
\end{remark}
We now check (\ref{eq:condition-consistency-zeta-1}) for $f(\lambda) = (\lambda - 1)^{2}$, and first notice that 
\begin{align*}
& \int \lambda d\mu_{MP}^{c_N}(\lambda, \nu) = 1, \\
& \int \lambda^{2} d\mu_{MP}^{c_N}(\lambda, \nu) = 1 + c_N,
\end{align*}
from which we deduce that 
\begin{align*}
& \int f(\lambda) d\mu_{MP}^{c_N}(\lambda, \nu) = c_N.
\end{align*}
In order to evaluate $\int f(\lambda) d\mu^1_N(\lambda,\nu)$, we use the following 
formulas 
\begin{align}
\label{eq:premier-moment-mu1}
& \int \lambda d\mu^1_N(\lambda, \nu) = 
\lim_{y \rightarrow +\infty} \mathrm{Re} \left[ -iy \left( 1 + iy t_N(iy, \nu) \right) \right], \\
\label{eq:deuxieme-moment-mu1}
& \int \lambda^{2} d\mu^1_N(\lambda, \nu) = \lim_{y \rightarrow +\infty} \mathrm{Re} \left[ -iy \left(iy ( 1 + iy t_N(iy, \nu)) + \int \lambda  d\mu^1_N(\lambda, \nu)\right) \right].
\end{align}
(\ref{eq:premier-moment-mu1}) is established e.g. in \cite{hachem-loubaton-najim-aap-2007}, and (\ref{eq:deuxieme-moment-mu1}) can be proved similarly. 
A simple calculation eventually provides
\begin{align}
\label{eq:expression-premier-moment-mu1}
& \int \lambda d\mu^1_N(\lambda, \nu) = \frac{1}{M} \Tr \H_N(\nu), \\
\label{eq:expression-deuxieme-moment-mu1}
& \int \lambda^{2} d\mu^1_N(\lambda, \nu) = \frac{1}{M} \Tr \H_N^{2}(\nu) + c_N \left(\frac{1}{M} \Tr \H_N(\nu)\right)^2.
\end{align}
Using that $\frac{1}{M} \Tr \H_N(\nu) = 1$,  we obtain that for each $\nu$, 
$$
 \int f(\lambda) d\mu^{1}_N(\nu,\lambda)  - 
 \int f(\lambda) d\mu^{c_N}_{MP}(\lambda) = \frac{1}{M} \Tr \H_N^{2}(\nu) - 1 .
$$
It thus remains to study  $\frac{1}{M} \Tr \H_N^{2}(\nu) - 1$. For this, 
we remark that 
$$
\frac{1}{M} \Tr \H_N^{2}(\nu) - 1 = \frac{1}{M} \sum_{i \neq j} \frac{|\a_i^{*} \D(\nu) \a_j|^{2}}{\a_i^{*} \D(\nu) \a_i \; \a_j^{*} \D(\nu) \a_j },
$$
and provide examples of matrices $\A_N$ for which there exists a constant $\kappa > 0$ such that for each $N$ large enough,
\begin{align}
    \min_{\nu \in \Gcal_N^{'}} \frac{1}{M} \Tr \H_N^{2}(\nu) - 1 > \kappa.
    \label{eq:condition_consistency}
\end{align}

\paragraph{Example 1.} 
We first consider the case where $\A_N$ is given by 
\begin{equation}
    \label{eq:matrice-I+G}
\A_N = \I_M + \sigma \frac{\G_N}{\sqrt{M}}  ,  
\end{equation}
where $\sigma > 0$ is a deterministic constant and where $\G_N$ is a $M \times M$
random matrix with $\Ncal_c(0,1)$ i.i.d. entries independent from $\y^{1}$. 
We prove in the following that there exists an event of probability one on which any fixed realization of the sequence $(\G_N)_{N\geq 1}$ generates a deterministic sequence $(\A_N)_{N\geq 1}$ according to \eqref{eq:matrice-I+G} and such that \eqref{eq:condition_consistency} holds.
We denote by $\g_{1,N}, \ldots, \g_{M,N}$ the columns of $\G_N$. 

We first check that conditions (\ref{eq:condition-norme-AN}) and (\ref{eq:condition-norme-colonnes-AN}) hold almost surely w.r.t. the probability distribution of matrices $\G_N$ for $N$ large enough. On the one hand, as $\| \G_N \| \rightarrow 2$ if $N \rightarrow +\infty$, (\ref{eq:condition-norme-AN}) is verified. On the other hand, 
$\| \a_{m,N} \|^{2} = 1 + 2 \frac{\sigma}{\sqrt{M}} \mathrm{Re}(\G_{m,m})  + 
\frac{\sigma^{2}}{M} \| \g_{m} \|^{2}$. It is clear that $\G_{m,m} = \Ocal_{\prec}(1)$ and that $\frac{\sigma^{2}}{M} \| \g_{m} \|^{2} = \sigma^{2} + \Ocal_{\prec}\left( \frac{1}{\sqrt{M}}\right)$. Therefore, 
\begin{equation}
 \label{eq:evaluation-norme-amN}   
\| \a_{m,N} \|^{2} = 1 + \sigma^{2} + \Ocal_{\prec}\left( \frac{1}{\sqrt{M}}\right),
\end{equation}
a property which implies that 
\begin{equation}
    \label{eq:convergence-norme-colonnes-AN}
\lim_{N \rightarrow +\infty} \left| \sup_{m=1, \ldots,M} \| \a_{m,N} \|^{2} - (1 + \sigma^{2}) \right| = 0     .
\end{equation}
Therefore, (\ref{eq:condition-norme-colonnes-AN}) holds almost surely for $N$ large enough. We now evaluate $\frac{1}{M} \Tr \H^{2}(\nu) - 1$ and first notice that, 
as $\sup_{m} s_m(\nu) \leq s_{max} < +\infty$, the inequality 
\begin{equation}
    \label{eq:minoration-Tr-H2-1}
 \frac{1}{M} \Tr \H^{2}(\nu) - 1 \geq \frac{1}{s_{max}} \frac{1}{M} \sum_{i \neq j} \frac{|\a_i^{*} \D(\nu) \a_j|^{2}}{ \|\a_i| \|^{2}  \|\a_j \|^{2}}   ,
\end{equation}
holds. We are thus back to study the right hand side of (\ref{eq:minoration-Tr-H2-1}). 
For this, we establish in the Appendix \ref{sec:proof-prop-power-A-random} the following result. 
\begin{proposition}
    \label{prop:minorant}
    It holds that 
\begin{equation}
    \label{eq:behaviour-Tr-H2}
   \max_{\nu \in \Gcal_N^{'}} \left| \frac{1}{M} \sum_{i \neq j} \frac{|\a_i^{*} \D(\nu) \a_j|^{2}}{ \|\a_i| \|^{2}  \|\a_j \|^{2}} - \mathbb{E} \left(  \frac{1}{M} \sum_{i \neq j} \frac{|\a_i^{*} \D(\nu) \a_j|^{2}}{ (1+\sigma^{2})^{2}} \right) \right| \rightarrow 0 \; a.s. 
\end{equation}    
\end{proposition}
An easy calculation leads to 
$$
\mathbb{E} \left(  \frac{1}{M} \sum_{i \neq j} \frac{|\a_i^{*} \D(\nu) \a_j|^{2}}{ (1+\sigma^{2})^{2}} \right) = \frac{\sigma^{2}(\sigma^{2} + 2)}{(1+\sigma^{2})^{2}} 
\left( 1 - \frac{1}{M} \right) \frac{1}{M} \Tr \D^{2}(\nu) > \frac{1}{2} \frac{\sigma^{2}(\sigma^{2} + 2)}{(1+\sigma^{2})^{2}} s_{min}^{2},
$$
for each $M \geq 2$. Proposition \ref{prop:minorant} thus implies that, 
on a set of probability 1,  for each $N$ large enough, $ \min_{\nu \in \Gcal_N^{'}} \frac{1}{M} \Tr \H^{2}(\nu) - 1 \geq \kappa$
for some constant $\kappa$ . We thus conclude that almost surely, for each $N$ large enough, $d^{1}_N(f) > \sqrt{K'} B \kappa$ for some constant $\kappa$, and eventually that, as expected, (\ref{eq:consistency-zeta1}) holds. \\

\paragraph{Example 2.}
We finally consider the case where 
$\A_N^{*} \A_N = \Sigmabs_N$ where 
\begin{equation}
    \label{eq:matrix-A-Autoregressive}
    \a_i^{*} \a_j = (\Sigmabs_N)_{i,j} = \frac{\sigma^{|i-j|}}{1 - \sigma^{2}},
\end{equation}
where $0 < \sigma < 1$. $\A_N$ verifies (\ref{eq:condition-norme-AN}) and (\ref{eq:condition-norme-colonnes-AN}). We justify that under the assumption that the spectral densities $(s_m)_{m \in \mathbb{Z}}$ all coincide, then (\ref{eq:consistency-zeta1}) holds. In this context, $\frac{1}{M} \Tr \H_N^{2}(\nu) - 1$ 
coincides with $\frac{1}{M} \sum_{i \neq j} \frac{|\a_i^{*} \a_j|^{2}}{ \| \a_i \|^{2} \;  \| \a_j \|^{2}}$ which does not depend on the frequency $\nu$. 
Using (\ref{eq:matrix-A-Autoregressive}), we obtain immediately that 
$$
\frac{1}{M} \sum_{i \neq j} \frac{|\a_i^{*} \a_j|^{2}}{ \| \a_i \|^{2} \;  \| \a_j \|^{2}} = 
\frac{2}{M} \sum_{i < j} \sigma^{2(j-i)} = 2 \sum_{u=1}^{M-1} \sigma^{2u} - \frac{2}{M} \sum_{u=1}^{M-1} u \sigma^{2u},
$$
converges towards $\frac{2 \sigma^{2}}{1 - \sigma^{2}}$. Therefore, 
$$
 \int f(\lambda) d\mu^{1}_N(\nu,\lambda)  - 
 \int f(\lambda) d\mu^{c_N}_{MP}(\lambda) >  \frac{\sigma^{2}}{1 - \sigma^{2}},
$$
for each $N$ large enough, which, in turn, leads to  (\ref{eq:consistency-zeta1}). 

\section{Numerical simulations}
\label{sec:simulations}
This section presents simulated examples to demonstrate the finite sample performance of the test statistics defined in Equations \ref{eq:def-statistics-zeta1-clt} and \ref{eq:def-statistics-zeta-clt}, which we restate here for clarity:
\begin{align*}
 &\xi_{N,0}(f, \nu) = \frac{B \theta_N(f,\nu) }{\sigma_N(f)}, \\
 &\xi_{N,1}(f) = \frac{\zeta_{N,1}(f)}{\sigma_N(f)} = \frac{1}{|\Gcal_N^{'}|^{1/2}} \sum_{\nu \in \Gcal_N^{'}} \xi_{N,0}(f, \nu) \to_\mathcal{D} \Ncal(0, 1),
 \\ 
  &\xi_{N,2}(f) =   \frac{\zeta_{N,2}(f)}{\sqrt{2}\sigma^2_N(f)} = \frac{1}{\sqrt{2}} \frac{1}{|\Gcal_N^{'}|^{1/2}} \sum_{\nu \in \Gcal_N^{'}} \left( \xi_{N,0}(f, \nu)^{2} - 1 \right) \to_\mathcal{D} \Ncal(0, 1).
\end{align*}

As mentioned in Remark (\ref{re:chi2-approximation}), $ \xi_{N,2}(f)$ can be approximated in distribution by $\frac{\chi^{2}(K') - K'}{\sqrt{2 K'}}$ (where we recall that $K'=|\Gcal_N^{'}|$), and we will refer to this approximation as the $\chi^{2}(K')$ approximation in this section . \\

Furthermore, we notice that, as $(\xi_{N,0}(f, \nu))_{\nu \in \Gcal_N'}$ are nearly independent and each one converges in distribution towards a standard Gaussian random variable, it is reasonable to expect that the maximum of $\xi_{N,0}(f, \nu)^{2}$ over the grid $\Gcal_N^{'}$ converges, after proper recentering, towards a Gumbel distribution. More precisely, the statistic $\xi_{N,3}(f)$ defined by 
$$
\xi_{N,3}(f) = \frac{1}{2} \left(\max_{\nu\in\Gcal_N'} \xi_{N,0}(f, \nu)^{2} -  \left( 2 \log K' -  \log(\log K') - \log \pi  \right) \right), 
$$
should verify 
$$
\xi_{N,3}(f) \to_\mathcal{D} \Gcal,
$$
where $\Gcal$ is the Gumbel distribution with cumulative distribution function $x \mapsto e^{-e^{-x}}$. The proof of the behavior of $\xi_{N,3}(f)$ does not seem an easy task, and is outside the scope of this paper. However, we also evaluate by numerical simulations the accuracy of this approximation of $\xi_{N,3}(f)$. \\

Our test statistics will be compared against the one developed in \cite{pan-gao-yang-jasa-2014} which was developed under the assumption 
that the spectral densities $(s_m)_{m=1, \ldots, M}$ all coincide. In this context, the rows of the $M \times N$ matrix 
$\Y_N$ with elements $(\Y_N)_{m,n} = y_{m,n}$ are i.i.d. It will be referred to "PGY" (which stands for the author names Pan, Gao and Yang) in the tables below. For convenience, we provide a summary of the corresponding test statistic implementation:

\begin{enumerate}
\item The time series is split in half along the spatial dimension: $\y_n = \left( \begin{array}{c} \y_n^{(1)} \\  \y_n^{(2)} \end{array} \right)$, where $\y_n^{(1)}$ consists of the first $\lfloor M/2 \rfloor$ dimensions, and $\y_n^{(2)}$ comprises the remaining dimensions.
\item Compute the covariance matrices $\S_N^{(i)} = \frac{1}{\lfloor M/2 \rfloor} \Y_N^{(i)*} \Y_N^{(i)}$ for $i=1,2$, where $\Y_N^{(i)}=(\y_1^{(i)}, \ldots, \y_N^{(i)})$.
\item Define $\tilde{G}_N(x) = G_N^{(1)}(x) - G_N^{(2)}(x) = N (F_N^{(1)}(x) - F_N^{(2)}(x))$, where $F_N^{(i)}$ is the cumulative distribution function of the eigenvalue distribution of $\S_N^{(i)}$.
\item It is shown that for two functions $f,g$ satisfying certain conditions (e.g. analytic on an open region containing the support of the limiting distribution of $\frac{1}{M} \Y^{*} \Y$), and a regime where $\frac{M}{N} \to d>0$, the random variable $(\int f d \tilde{G}_N, \int g d \tilde{G}_N)^{T}$ converges in distribution to a bivariate normal distribution with mean 0 and known covariance $\mathbf{\tilde{\Omega}}$. Note that this covariance matrix needs to be estimated, which in practice can take a non-negligible amount of time. 
\item The test statistic $\xi_{N,gpy}$ defined below is shown to converge to a $\chi^2(2)$ random variable under the null hypothesis.
$$\xi_{N,gpy} = \left( \int f d \tilde{G}_N, \int g d \tilde{G}_N \right) \mathbf{\tilde{\Omega}}^{-1} \begin{pmatrix}
\int f d \tilde{G}_N \\
\int g d \tilde{G}_N
\end{pmatrix}, $$
for $f(\lambda) = \lambda$ and $g(\lambda) = \lambda^{2}$. 

\end{enumerate}

It is important to note that the regime used in \cite{pan-gao-yang-jasa-2014} differs from the one considered in our paper: \cite{pan-gao-yang-jasa-2014} assumes $M=O(N)$, while here it is required that $M=o(N)$. We also recall that the results of  \cite{pan-gao-yang-jasa-2014} are valid if all $M$ time series share the same spectral density. 

We begin by evaluating the behavior of the various test statistics under the null hypothesis. Subsequently, we introduce spatial dependence and measure the power of the proposed tests. To calculate the sizes and power values, we first use the asymptotic normality of the test statistics. Let $z_{1 - \tilde{\alpha}}$ denote the $(1 - \tilde{\alpha})$-quantile of the asymptotic null distribution $\Ncal(0,1)$. With $R$ replications of the dataset simulated under the null hypothesis, we compute the empirical size as:
$$
    \hat{\alpha} = \frac{\text{Number of } {\xi_{i,N}^{H_0}(f) \ge z_{1 - \tilde{\alpha}} }}{R},
$$
where $\xi_{i,N}^{H_0}(f)$ represents the value of the test statistic $\xi_{i,N}(f)$ based on data simulated under the null hypothesis for $i=1,2$.  We proceed similarly when the limiting distribution is $\chi^{2}(K')$ (for $\xi_{N,2}$) or Gumbel (for $\xi_{N,3}$). In our simulations, we set $R = 10^4$ as the number of repetitions and use a significance level of $\tilde{\alpha} = 10\%$. Similarly, we calculate the empirical power as:
$$
    \hat{\beta} = \frac{\text{Number of } {\xi_{i,N}^A(f) \ge z_{1 - \tilde{\alpha}} }}{R},
$$
where $\xi_{i,N}^A(f)$ represents the value of the test statistic $\xi_{i,N}$ based on data simulated under the alternative hypothesis, and where $z_{1-\tilde{\alpha}}$ represents the 
$(1-\tilde{\alpha})$-quantile of the $\Ncal(0,1)$, $\chi^{2}(K')$ or Gumbel distributions. 

Throughout our analysis, we will use $f: x \mapsto (x-1)^2$ as our test function. For a fair comparison between our proposal and 
\cite{pan-gao-yang-jasa-2014} which uses 2 fonctions, we have implemented a version of the approach of \cite{pan-gao-yang-jasa-2014} 
using  the single function $f: x \mapsto (x-1)^2$. Also, unless otherwise stated, we will use $M = \lfloor N^\alpha \rfloor$ with $\alpha=2/3$ and $B = \lfloor M/c \rfloor$ with $c=1/2$. Last, we will use $\Gcal_N'=\Gcal_N$: although our theorem requires to sample the frequencies on the subset $\Gcal_N'$, we have observed that the performance of the tests tends to be better when $\Gcal_N'=\Gcal_N$.

\subsection{\texorpdfstring{Numerical simulations under $\Hcal_0$}{Numerical simulations under the null hypothesis}}
To evaluate the performance of our proposed test statistics, we generate sample data using several Data Generation Processes (DGPs). We begin with a $M$-dimensional ARMA(1,1) process $(\y_n)_{n\ge0}$ defined by:
\begin{equation}
\label{eq:dgp1}
DGP_1: y_{m, n} - \phi_m y_{m, n-1} = \epsilon_{m,n} + \psi_m \epsilon_{m,n-1}, 
\end{equation}
where $\{\epsilon_{m,n}, n=1,\ldots,N, m=1,\ldots,M\}$ are i.i.d. standard complex normal random variables $\Ncal_{c}(0,1)$, while $\phi_m$ and $\psi_m$ are in $(-1,1)$ for all $m$. The index $n$ will always refer to the time dimension (of size $N$), while $m$ represents the spatial dimension (of size $M$). Our proposed test statistics aim to detect spatial dependence in these time series.

\subsubsection{Graphical representation of the empirical distribution of test statistics}
We compute the proposed statistics $\xi_{N,1}$ to $\xi_{N,3}$ under $\Hcal_0$ using the simulated data for $\phi_m = 0.1$ and 
$\psi_m = 0.5$ for all $m$. Figure \ref{figure:zeta-all-distribution} compares the empirical distributions of these statistics against their theoretical limits. Two versions of each test statistic is calculated:
\begin{itemize}
\item "Estimated": for each $\nu \in \Gcal_N$, $r_N(\nu)$ is estimated from the observation by $\hat{r}_N(\nu)$ defined in (\ref{eq:def-hat-rN}). 
\item "Oracle": $r_N(\nu)$ is assumed to be known for each $\nu \in \Gcal_N$. 
\end{itemize}
The proximity of these two versions demonstrates the robustness of our estimation procedure for $\phi_m =0.1$. Also, loosely speaking, $\xi_{N,2}$ is the sum of the square of $|\Gcal_N|$ almost independent $\Ncal(0,1)$ variables. It is therefore natural to observe for finite $N$ a better fit of this test statistics against a $\chi^2(|\Gcal_N|)$ distribution instead of the Gaussian limit. 
\begin{figure}[H]
\centering
\includegraphics[scale=0.23]{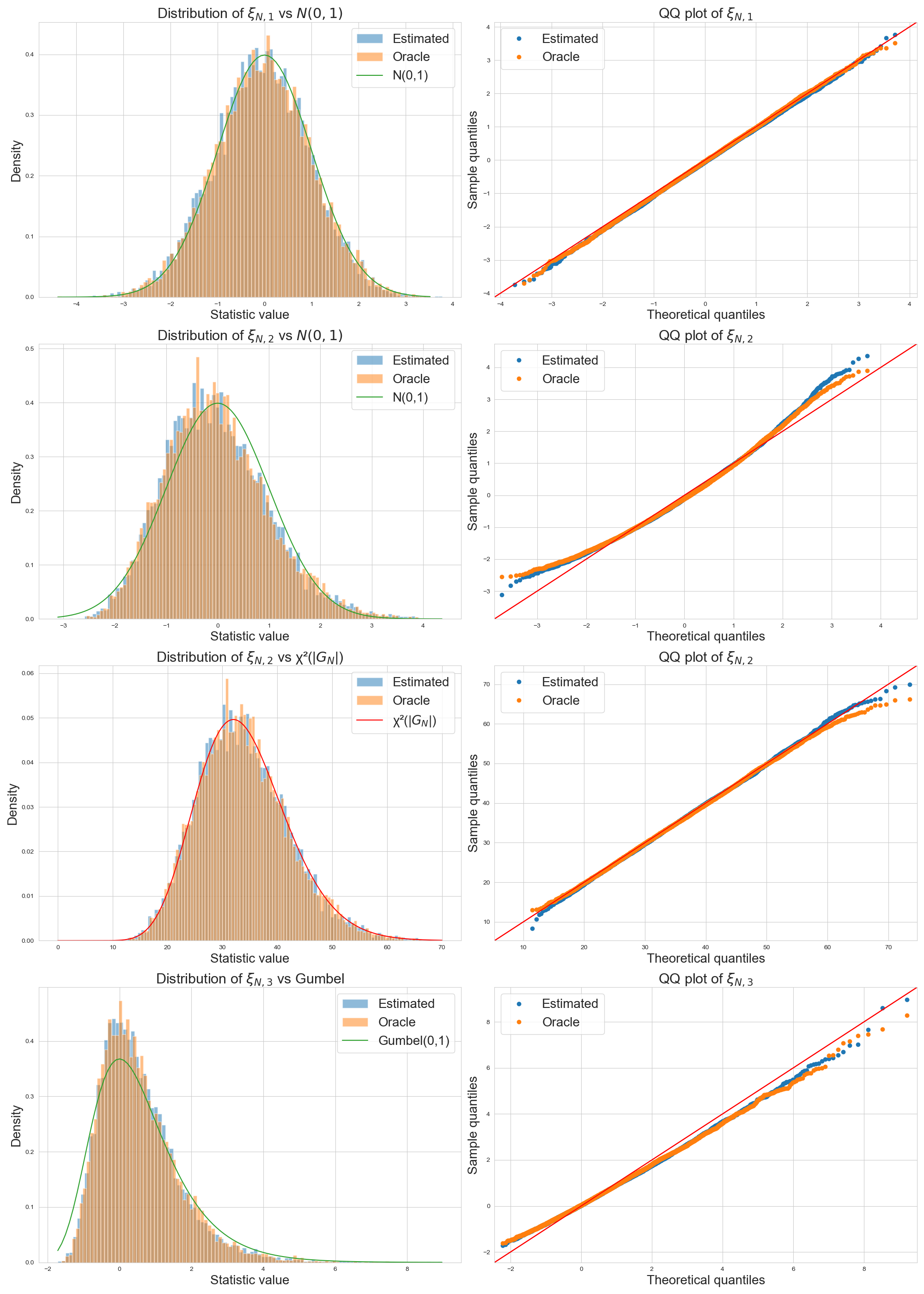}
\caption{$\xi_{N,1}, \xi_{N,2}$ (against $\Ncal(0,1)$ or a $\chi^2(|\Gcal_N|)$ random variable) and $\xi_{N,3}$ against their respective limiting distributions: histograms (left) and qq-plots (right). Data generated by DGP 1, $N=10^4, B = 301, M=120, L=5$. $\phi_m=0.1$ and $\psi_m=0.5$ for all $m$. $f: x \mapsto (x-1)^2$. $10^4$ repetitions.}
\label{figure:zeta-all-distribution}
\end{figure}

\subsubsection{Performance when $\alpha$ is close from  $\frac{7}{9}$.}
We have proved that the central limit theorem (\ref{eq:CLT-B-thetaN}) holds under the assumption that $\alpha < \frac{4}{5}$, and (\ref{eq:clt-zeta1}, \ref{eq:clt-zeta}) are valid if $\alpha < \frac{7}{9}$. These limitations are due to the presence of an unknown  deterministic $\Ocal\left( \frac{B^{5}}{N^{4}}\right)$ term in $ B \theta_N(f,\nu)$. In the following, we thus evaluate the numerical impact of this term w.r.t. $\alpha$ on the behavior of the means and standard deviations of $\xi_{N,0}(f,\nu)$ and $\xi_{N,1}(f)$. For a range of $N$ and $\alpha$, we simulate data under $\Hcal_0$ using DGP 1 with Gaussian errors, and plot in Figure \ref{figure:clt_arma_large_alpha_nu} the empirical mean (in logarithmic scale) and the standard deviation of the sampled $\xi_{N,0}(f,\nu)$ for a fixed frequency $\nu$ over $R=10^4$ repetitions. The term $\frac{B^{5}}{N^{4}}$ for $N=10000$ is also represented. 

\begin{figure}[H]
\centering
\includegraphics[width=\columnwidth, height=\textheight, keepaspectratio]{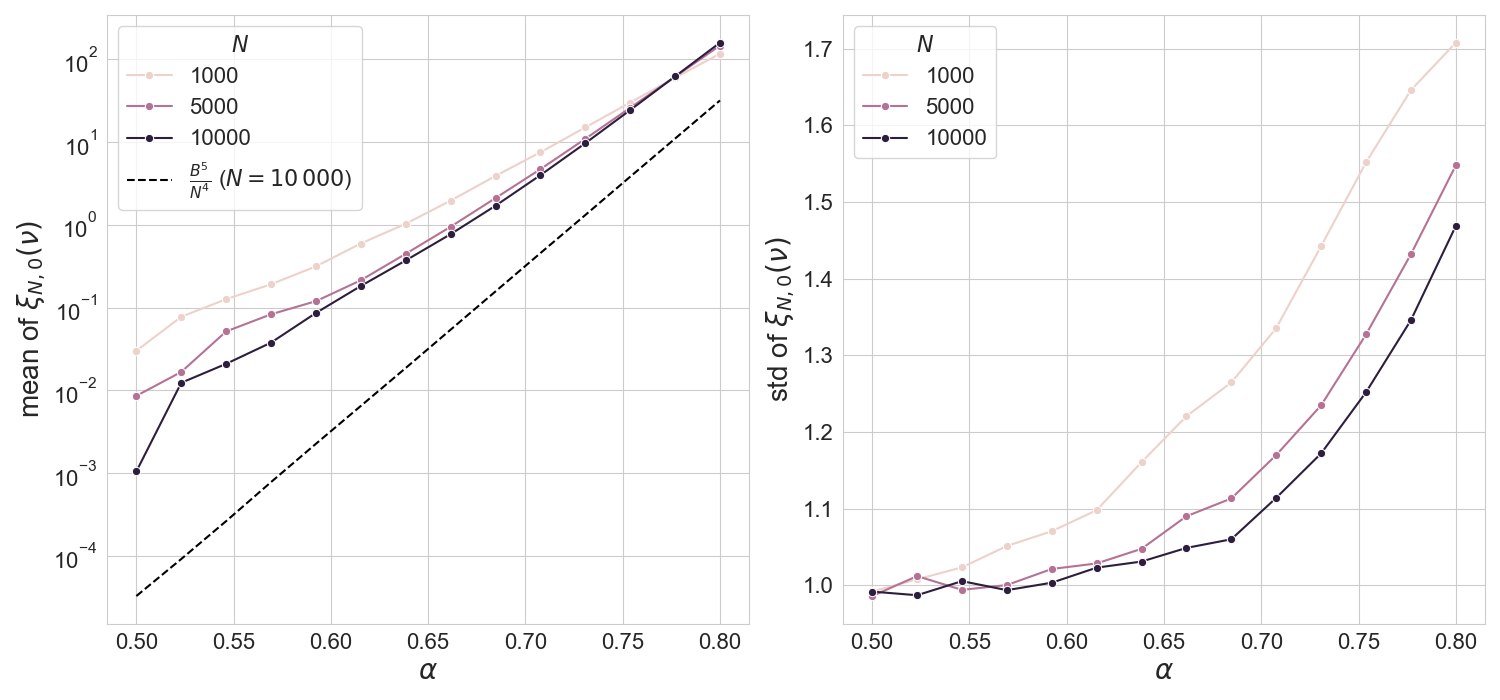}
\caption{Sample mean (left) and standard deviation (right) of $\xi_{N,0}(f, \nu)$ over $10^4$ repetitions as a function of $\alpha$ and $N$. Data generated by DGP 1, $c=1/2$, $\phi_m=0.1$ and $\psi_m=0.5$ for all $m$. $f: x \mapsto (x-1)^2$. $\nu=0.3$.  $r_N(\nu)$ is assumed to be known.}
\label{figure:clt_arma_large_alpha_nu}
\end{figure}

As seen in Figure \ref{figure:clt_arma_large_alpha_nu} left, the empirical mean of 
$\xi_{N,0}(f,\nu)$ tends to diverge when $\alpha$ increases. When $\alpha > 0.6$, 
for $N=5000$ and $N = 10.000$, this empirical mean behaves as the term proportional to $\frac{B^5}{N^4}$ evaluated for $N = 10.000$, and represented on Figure \ref{figure:clt_arma_large_alpha_nu} left. This tends to confirm that the divergence of the empirical mean of $\xi_{N,0}(f,\nu)$ is due to the $\Ocal\left(\frac{B^{5}}{N^{4}}\right)$ term of $\xi_{N,0}(f,\nu)$, up to a multiplicative constant. Moreover, while the sample variance remains close to one for $\alpha<0.7$, we also see a small degradation for larger values of $\alpha$. This emphasizes that when $\alpha$ becomes closer from the threshold $\frac{4}{5}$, for $N \leq 10.000$, $\xi_{N,0}(f,\nu)$ does not behave at all as a zero mean random variable, while its variance remains reasonably close from 1. In Figure \ref{figure:clt_arma_large_alpha} are represented the same quantities for $\xi_{N,1}(f)$. It is also observed that the test statistics is becoming unreliable when $\alpha > 0.75$. This tends to show that for $N \leq 10.000$, the use of $\xi_{N,1}(f)$ and $\xi_{N,2}(f)$ is questionable for values of 
$\alpha$ close from $\frac{4}{5}$. 

\begin{figure}[H] 
\centering
\includegraphics[scale=0.33]{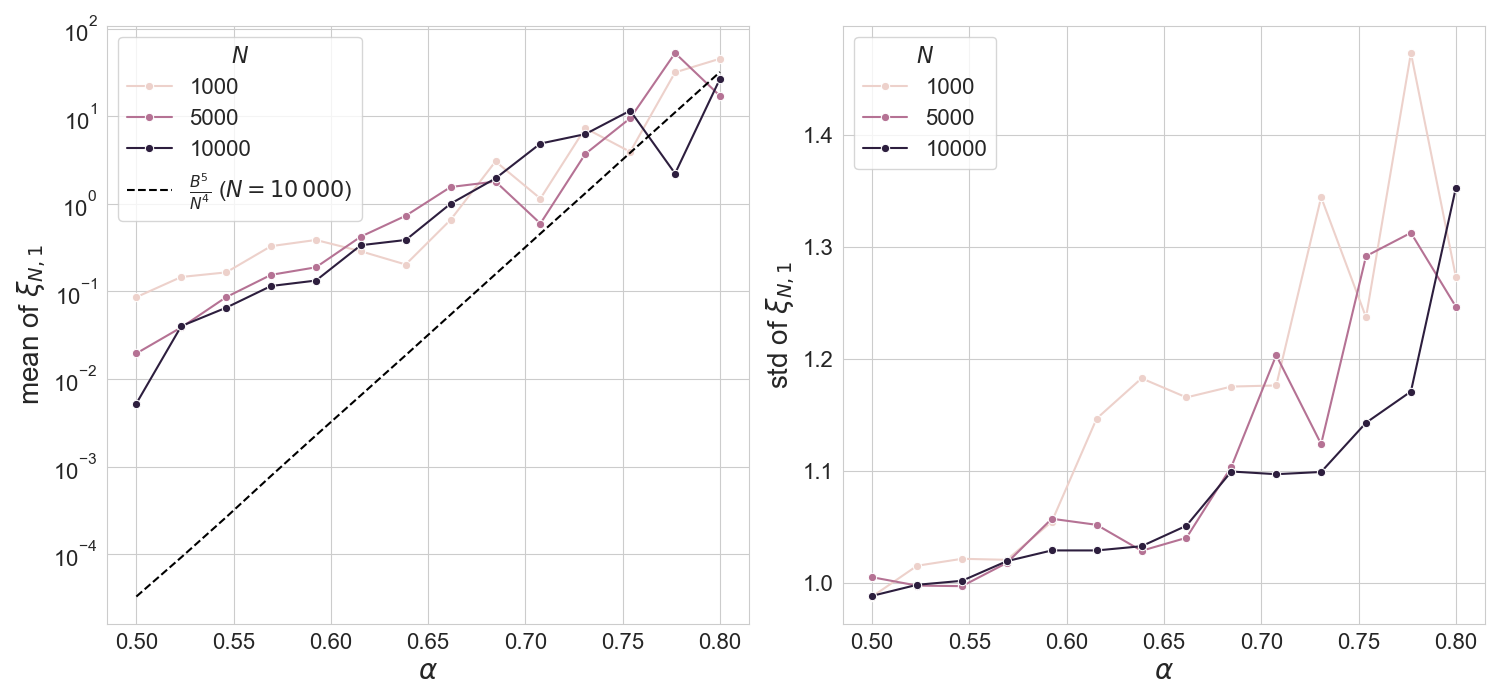}
\caption{Sample mean (left) and standard deviation (right) of $\xi_{N,1}(f)$ over $10^4$ repetitions as a function of $\alpha$ and $N$. Data generated by DGP 1, $c=1/2$, $\phi_m=0.1$ and $\psi_m=0.5$ for all $m$. $f: x \mapsto (x-1)^2$. $r_N(\nu)$ is assumed to be known.}
\label{figure:clt_arma_large_alpha}
\end{figure}

\subsubsection{Type I Error Analysis}
Table \ref{table:type-I-error-arma} presents the empirical sizes of both our test statistics and the PGY test statistic across various sample sizes $N$. As $N$ increases, the type I error rates converge to the target level of $10\%$. The PGY test statistics demonstrate adequate control of type I error when identical ARMA coefficients are applied across all time series. However, when each time series possesses its own distinct spectral density, the performance of the PGY test of course deteriorates while our test maintains consistent performance.

\begin{table}[h]
\centering
\setlength{\tabcolsep}{4pt}
\resizebox{\textwidth}{!}{
\begin{tabular}{lrrrrrrrrrr}
\hline
\toprule
ARMA coef. & \multicolumn{5}{r}{constant} & \multicolumn{5}{r}{random} \\
$\xi$ & $\xi_1$ & $\xi_2 (\Ncal)$ & $\xi_2 (\chi^2)$ & $\xi_3$ & $\xi_{pgy}$ & $\xi_1$ & $\xi_2 (\Ncal)$ & $\xi_2 (\chi^2)$ & $\xi_3$ & $\xi_{pgy}$ \\
N &  &  &  &  &  &  &  & & & \\
\midrule
1000 & 0.137 & 0.119 & 0.153 & 0.109 & 0.104 & 0.113 & 0.082 & 0.098 & 0.078 & 0.614 \\
2000 & 0.131 & 0.109 & 0.134 & 0.090 & 0.100 & 0.114 & 0.087 & 0.096 & 0.072 & 0.694 \\
3000 & 0.106 & 0.095 & 0.122 & 0.092 & 0.097 & 0.095 & 0.077 & 0.088 & 0.070 & 0.734 \\
4000 & 0.116 & 0.098 & 0.119 & 0.091 & 0.102 & 0.105 & 0.082 & 0.093 & 0.079 & 0.756 \\
5000 & 0.112 & 0.096 & 0.115 & 0.084 & 0.106 & 0.102 & 0.081 & 0.092 & 0.074 & 0.773 \\
6000 & 0.111 & 0.097 & 0.114 & 0.083 & 0.103 & 0.101 & 0.081 & 0.092 & 0.075 & 0.780 \\
7000 & 0.112 & 0.097 & 0.113 & 0.085 & 0.095 & 0.107 & 0.085 & 0.093 & 0.069 & 0.789 \\
8000 & 0.111 & 0.093 & 0.113 & 0.083 & 0.096 & 0.101 & 0.085 & 0.091 & 0.070 & 0.802 \\
\bottomrule
\hline
\end{tabular}
}

\caption{Sample type I error at $10\%$. On the left side ("constant") are shown type I errors for $\phi_m = 0.1$ and $\psi_m = 0.5$ for all $m$. On the right side ("random") $\phi_m$ and $\psi_m$ are sampled independently as uniform $U(-0.5, 0.5)$ variable. $\alpha=2/3$, $c=1/2$, $L=\lfloor N^{\frac{1}{4}} \rfloor$, $f: x \mapsto (x-1)^2$. $10^4$ repeats.}
\label{table:type-I-error-arma}
\end{table}

Although robust to heterogeneous ARMA specification, it is observed that the primary discrepancy between the expected and empirical levels of our test statistics arises from the estimation of the spectral densities $s_m(\nu)$ used to compute the corrective term $r_N(\nu)$. Table \ref{table:type-various-ar} displays the $\xi_{N,2}(\chi^2)$ type I error rates for both the "estimated" and "oracle" cases. The results show that the "estimated" version of the test statistics begins to break down when $\phi_m > 0.5$. This deterioration occurs because at high values of the AR parameter, the spectral density becomes highly concentrated around $\nu=0$ or $\nu=1/2$ (depending on the sign of $\phi_m$), which the lag-window estimator fails to approximate accurately. In contrast, the "oracle" version maintains good control of its type I error up to $\phi_m = 0.7$ where the speed of convergence of the type I error towards $10\%$ is more slowly. This might be due to  the observation that, as the correlation in the data sample strengthens, the number of effectively independent samples decreases. 

\begin{table}[h]
\centering
\setlength{\tabcolsep}{4pt}
\resizebox{\textwidth}{!}{
\begin{tabular}{lrrrrrrrrrr}
\hline
\toprule
spectral densities & \multicolumn{5}{r}{estimated} & \multicolumn{5}{r}{oracle} \\
$\phi_m$ & 0.1 & 0.3 & 0.5 & 0.6 & 0.7 & 0.1 & 0.3 & 0.5 & 0.6 & 0.7 \\
$N$ &  &  &  &  &  &  &  &  &  &  \\
\midrule
1000 & 0.083 & 0.087 & 0.166 & 0.971 & 1.000 & 0.082 & 0.096 & 0.112 & 0.137 & 0.277 \\
2000 & 0.083 & 0.093 & 0.125 & 0.565 & 1.000 & 0.083 & 0.092 & 0.107 & 0.132 & 0.370 \\
3000 & 0.092 & 0.092 & 0.112 & 0.402 & 1.000 & 0.087 & 0.095 & 0.108 & 0.125 & 0.173 \\
4000 & 0.093 & 0.097 & 0.117 & 0.326 & 1.000 & 0.090 & 0.095 & 0.104 & 0.120 & 0.183 \\
5000 & 0.092 & 0.103 & 0.114 & 0.223 & 1.000 & 0.096 & 0.099 & 0.107 & 0.118 & 0.172 \\
6000 & 0.096 & 0.099 & 0.111 & 0.229 & 1.000 & 0.099 & 0.099 & 0.101 & 0.112 & 0.168 \\
7000 & 0.092 & 0.098 & 0.105 & 0.135 & 1.000 & 0.098 & 0.099 & 0.106 & 0.113 & 0.157 \\
8000 & 0.095 & 0.097 & 0.110 & 0.179 & 1.000 & 0.099 & 0.095 & 0.102 & 0.111 & 0.156 \\
\bottomrule
\hline
\end{tabular}
}
\caption{Sample type I error at $10\%$ for $\xi_{N,2}(\chi^2)$. Data generated as DGP 1, $\psi_m = 0$ for all $m$. On the left side ("estimated") are shown type I errors in the case where the corrective term $r_N(\nu)$ is computed using a lag-window estimator of the spectral densities $s_m(\nu)$, while on the right side ("oracle") the true spectral densities are provided. $\alpha=2/3$, $c=1/2$, $L=\lfloor N^{\frac{1}{4}} \rfloor$, $f: x \mapsto (x-1)^2$. $10^4$ repeats.}
\label{table:type-various-ar}
\end{table}

In the expression (\ref{eq:def-theta-f-nu}) of $\theta_N(f,\nu)$, we have seen that for $\alpha < 2/3$, the term proportional to $r_N(\nu)v_N$ is negligible compared to the one scaled as $1/B$ (see Remark \ref{re:clt-alpha-inferieur-2tiers}). That means that it may be possible to bypass the computation of $r_N(\nu)v_N$ at a small type-I error cost. In Table \ref{table:type-dgp-1-correction} is shown a comparison where we choose to either compute all corrective terms or ignore the one proportional to $r_N(\nu)v_N$. As we can see, the type I error is significantly degraded in the second case.

\begin{table}[h]
\centering
\begin{tabular}{lrrrrrrrr}
\hline
\toprule
correction & \multicolumn{4}{r}{with $r_N(\nu)$ correction} & \multicolumn{4}{r}{no $r_N(\nu)$ correction} \\
$\xi$ & $\xi_1$ & $\xi_2 (\Ncal)$ & $\xi_2 (\chi^2)$ & $\xi_3$ & $\xi_1$ & $\xi_2 (\Ncal)$ & $\xi_2 (\chi^2)$ & $\xi_3$  \\
$N$ &  &  &  &  &  &  &  &  \\
\midrule
1000 & 0.098 & 0.079 & 0.089 & 0.072 & 0.244 & 0.120 & 0.146 & 0.119 \\
2000 & 0.115 & 0.096 & 0.097 & 0.071 & 0.278 & 0.127 & 0.150 & 0.113 \\
3000 & 0.107 & 0.087 & 0.093 & 0.077 & 0.254 & 0.110 & 0.133 & 0.108 \\
4000 & 0.099 & 0.082 & 0.086 & 0.073 & 0.247 & 0.106 & 0.129 & 0.103 \\
5000 & 0.100 & 0.091 & 0.091 & 0.070 & 0.269 & 0.113 & 0.138 & 0.105 \\
6000 & 0.102 & 0.092 & 0.090 & 0.081 & 0.253 & 0.103 & 0.128 & 0.097 \\
7000 & 0.100 & 0.091 & 0.090 & 0.074 & 0.254 & 0.105 & 0.127 & 0.099 \\
8000 & 0.104 & 0.090 & 0.095 & 0.081 & 0.270 & 0.113 & 0.134 & 0.103 \\
\bottomrule
\hline
\end{tabular}

\caption{Sample type I error at $10\%$. Data generated as DGP1, $\phi_m = 0.1$ and $\psi_m = 0.5$ for all $m$. On the left side ("with $r_N(\nu)$ correction") are shown type I errors in the case where the corrective term $r_N(\nu)$ is computed using a lag-window estimator of the spectral densities $s_m(\nu)$, while on the right side ("no $r_N(\nu)$ correction") the term proportional to $r_N(\nu)$ is ignored. $\alpha=0.6$, $c=1/2$, $L=\lfloor N^{\frac{1}{4}} \rfloor$, $f: x \mapsto (x-1)^2$, $10^4$ repeats.}
\label{table:type-dgp-1-correction}
\end{table}

\subsubsection{Robustness to non-Gaussian innovations}

We repeat the previous experiment using DGP 1, but with innovations drawn from a non-Gaussian distribution. Specifically, the i.i.d. errors $\epsilon_{m,n}$ in \eqref{eq:dgp1} are generated as $\epsilon_{m,n} = \sqrt{\tau_n} w_{m,n}$ where $\{w_{m,n}:m=1,\ldots,M, n=1,\ldots,N\}$ are i.i.d. $\Ncal_{c}(0,1)$ and $\tau_{1},\ldots,\tau_N$ are i.i.d. positive random variables independent of the sequence $(w_{m,n})$. We notice that the time series $(\epsilon_{m})_{m=1, \ldots, M}$ are not independent (for each $m_1 \neq m_2$ and each $n$, 
$\epsilon_{m_1,n}$ and $\epsilon_{m_2,n}$ are dependent random variables), 
but they remain uncorrelated. 
In particular, we consider the following cases frequently used in the field of signal processing \cite{ollila2012complex}:
\footnote
{
    We denote by $\Gcal(\alpha,\lambda)$ the Gamma distribution with density
    $x \mapsto \Gamma(\alpha)^{-1}\lambda^{\alpha} x^{\alpha-1}\erm^{-\lambda x}$ where $\Gamma$ denotes the usual Gamma function.
}
\begin{enumerate}
    \item \textit{complex Student's distribution.} When choosing $\tau_1^{-1} \sim \Gcal\left(\frac{k}{2},\frac{k}{2}\right)$, $\epsilon_{m,n}$ follows the so-called complex Student's distribution with $k$ degrees of freedom;
    \item \textit{K-distribution.} When choosing $\tau_1 \sim \Gcal\left(k,k\right)$, $\epsilon_{m,n}$ follows the K distribution with $k$ degrees of freedom.
\end{enumerate}

Figure \ref{figure:zeta-all-distribution-non-gaussian} displays the empirical distributions of the statistics $\xi_{N,1}$ to $\xi_{N,3}$ (using the estimated $r_N(\nu)$) against their respective theoretical limits over $R=10^4$ repetitions.

\begin{figure}[H] 
\centering
\includegraphics[scale=0.23]{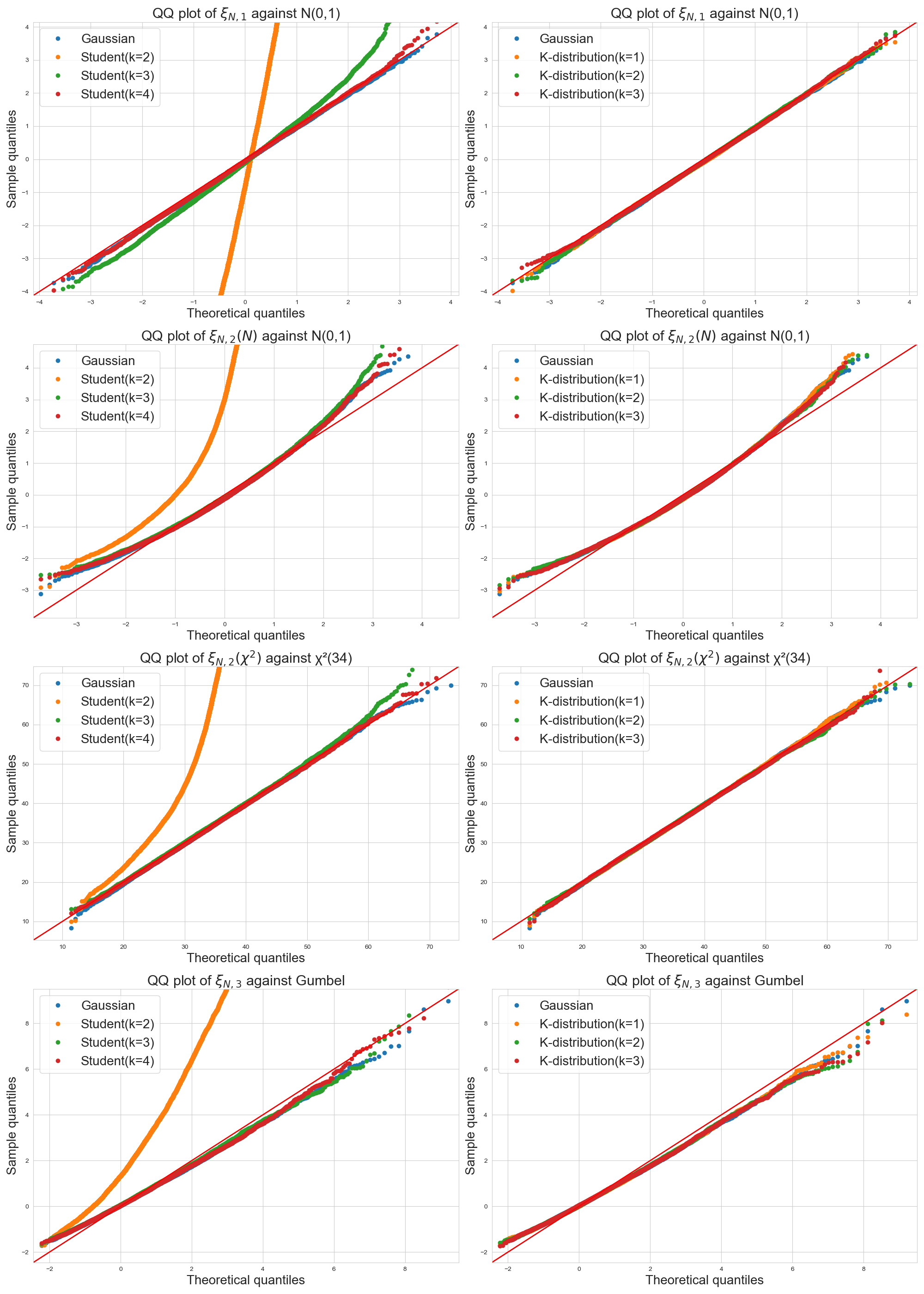}
\caption{Empirical distributions of $\xi_{N,1}$ (against a $\Ncal(0,1)$), $\xi_{N,2}$ (against a $\Ncal(0,1)$ and a $\chi^2(|\Gcal_N|)$) and $\xi_{N,3}$ (against a Gumbel random variable). Data generated by DGP 1 with non-Gaussian innovations (complex Student's distribution on the left, K-distribution on the right), $N=10^4, M=120, B = 301, L=5$. $\phi_m=0.1$ and $\psi_m=0.5$ for all $m$. $f: x \mapsto (x-1)^2$. $10^4$ repetitions.}
\label{figure:zeta-all-distribution-non-gaussian}
\end{figure}

The QQ-plots in Figure \ref{figure:zeta-all-distribution-non-gaussian} show good agreement between the empirical distributions and their theoretical counterparts, similar to the Gaussian case results depicted in Figure \ref{figure:zeta-all-distribution}. An exception occurs when the noise follows a complex Student's distribution with 2 degrees of freedom (which lacks a finite second moment). The type I errors presented in Table \ref{table:type-I-error-arma-non-gaussian} show that $\xi_{N, 2}$ is more robust than $\xi_{N,1}$, and that the performance of the PGY test is quite poor. The lack of robustness of the PGY test follows clearly from the fact that the presence of the $(\tau_n)_{n=1, \ldots, N}$ modifies the limit eigenvalue distribution of matrices $\S^{i}_N$, $i=1,2$ w.r.t. the Gaussian case. 

\begin{table}[h]
\centering
\setlength{\tabcolsep}{4pt}
\resizebox{\textwidth}{!}{
\begin{tabular}{lrrrrrrrrrr}
\hline
\toprule
Innovations & \multicolumn{5}{r}{Gaussian} & \multicolumn{5}{r}{Student (df=3)} \\
$\xi$ & $\xi_1$ & $\xi_2 (\Ncal)$ & $\xi_2 (\chi^2)$ & $\xi_3$ & $\xi_{pgy}$ & $\xi_1$ & $\xi_2 (\Ncal)$ & $\xi_2 (\chi^2)$ & $\xi_3$ & $\xi_{pgy}$ \\
N &  &  &  &  &  &  &  & & & \\
\midrule
1000 & 0.121 & 0.123 & 0.154 & 0.114 & 0.096 & 0.202 & 0.154 & 0.174 & 0.122 & 0.678 \\
2000 & 0.135 & 0.100 & 0.131 & 0.084 & 0.104 & 0.176 & 0.131 & 0.154 & 0.102 & 0.754 \\
3000 & 0.098 & 0.086 & 0.108 & 0.088 & 0.103 & 0.155 & 0.107 & 0.143 & 0.101 & 0.769 \\
4000 & 0.117 & 0.107 & 0.123 & 0.082 & 0.101 & 0.182 & 0.109 & 0.138 & 0.099 & 0.790 \\
5000 & 0.111 & 0.100 & 0.135 & 0.106 & 0.109 & 0.177 & 0.092 & 0.118 & 0.084 & 0.781 \\
6000 & 0.109 & 0.084 & 0.112 & 0.093 & 0.111 & 0.185 & 0.113 & 0.129 & 0.094 & 0.811 \\
7000 & 0.101 & 0.089 & 0.114 & 0.090 & 0.092 & 0.162 & 0.106 & 0.127 & 0.098 & 0.796 \\
8000 & 0.092 & 0.100 & 0.110 & 0.102 & 0.103 & 0.182 & 0.105 & 0.120 & 0.088 & 0.836 \\
\bottomrule
\hline
\end{tabular}
}
\caption{Sample type I error at $10\%$ when innovations are generated either by a Gaussian distribution or a Student distribution with 3 degrees of freedom. $\phi_m = 0.1$ and $\psi_m = 0.5$ for all $m$. $\alpha=2/3$, $c=1/2$, $L=\lfloor N^{\frac{1}{4}} \rfloor$, $f: x \mapsto (x-1)^2$. $10^4$ repeats.}
\label{table:type-I-error-arma-non-gaussian}
\end{table}

\subsection{Testing for Spatial Dependence}
To evaluate the power of our proposed test statistics under various spatial dependence structures, we consider three additional data generation processes. Each DGP introduces a different form of spatial dependence, allowing us to assess the robustness and effectiveness of our method across diverse scenarios.

\subsubsection{DGP 2: AR(1)-type Covariance Structure}
We first generate $(\x_n)_{n=1, \ldots, N}$ under DGP 1, then introduce spatial dependence using an AR(1)-type covariance matrix $\Sigmabs=(\sigma_{k,h})_{k,h=1}^M$, where:
\begin{equation}
\label{eq:dgp3-covariance}
\sigma_{k,h} = \frac{\sigma^{|k-h|}}{1-\sigma^2}.
\end{equation}

We compute $\mathbf{A}$ such that $\mathbf{A}\mathbf{A}^* = \Sigmabs$, and calculate $\mathbf{y}_n = \mathbf{A}\mathbf{x}_n$. Table \ref{table:power-dgp3} displays the empirical power of $\xi_{i,N}$ under this scenario. The results tend to demonstrate that $\xi_{N,1}$ and $\xi_{N,2}$ provide the best performance in terms of detection of the alternative.

\begin{table}[h]
\centering
\setlength{\tabcolsep}{4pt}
\resizebox{\textwidth}{!}{
\begin{tabular}{lrrrrrrrrrr}
\hline
\toprule
$\sigma$ & \multicolumn{5}{r}{0.05} & \multicolumn{5}{r}{0.5} \\
$\xi$ & $\xi_1$ & $\xi_2 (\Ncal)$ & $\xi_2 (\chi^2)$ & $\xi_3$ & $\xi_{pgy}$ & $\xi_1$ & $\xi_2 (\Ncal)$ & $\xi_2 (\chi^2)$  & $\xi_3$ & $\xi_{pgy}$ \\
$N$ &  &  &  &  & &  &  &  &  & \\
\midrule
1000 & 0.189 & 0.140 & 0.183 & 0.127 & 0.096 & 1.000 & 1.000 & 1.000 & 1.000 & 0.154 \\
2000 & 0.478 & 0.220 & 0.277 & 0.177 & 0.087 & 1.000 & 1.000 & 1.000 & 1.000 & 0.172 \\
3000 & 0.747 & 0.335 & 0.408 & 0.248 & 0.082 & 1.000 & 1.000 & 1.000 & 1.000 & 0.178 \\
4000 & 0.918 & 0.507 & 0.589 & 0.329 & 0.084 & 1.000 & 1.000 & 1.000 & 1.000 & 0.197 \\
5000 & 0.982 & 0.686 & 0.754 & 0.422 & 0.086 & 1.000 & 1.000 & 1.000 & 1.000 & 0.196 \\
6000 & 0.997 & 0.832 & 0.877 & 0.519 & 0.080 & 1.000 & 1.000 & 1.000 & 1.000 & 0.204 \\
7000 & 1.000 & 0.922 & 0.948 & 0.617 & 0.081 & 1.000 & 1.000 & 1.000 & 1.000 & 0.209 \\
8000 & 1.000 & 0.973 & 0.984 & 0.710 & 0.083 & 1.000 & 1.000 & 1.000 & 1.000 & 0.217 \\
\bottomrule
\hline
\end{tabular}
}
\caption{Measured power of $\xi_{i,N}$ for DGP 2. $\sigma=0.05$ (left) and $\sigma=0.5$ (right), $\phi_m=0.1$ and $\psi_m=0.5$ for all $m$. $\alpha=2/3$, $c=1/2$, $L=\lfloor N^{\frac{1}{4}} \rfloor$, $f: x \mapsto (x-1)^2$, $10^4$ repeats.}
\label{table:power-dgp3}
\end{table}

\subsubsection{DGP 3: Random Spatial Mixing}
DGP 3 is similar to DGP 2, but employs a randomly generated spatial mixing matrix:
\begin{equation}
\mathbf{y}_n = (\I_M + \frac{\sigma}{\sqrt{M}}  \A)\mathbf{x}_n, \quad n = 1, 2, \ldots, N,
\label{eq:common_random}
\end{equation}
where $(\x_n)_{n=1, \ldots, N}$ is generated under DGP 1, $\I_M$ is the M-dimensional identity matrix and $\A$ is a $M \times M$ random matrix with i.i.d. $\mathcal{N}(0,1)$ components. Table \ref{table:power-dgp4} shows the measured power of the various test statistics for this scenario. Again, 
$\xi_{N,1}$ and $\xi_{N,2}$ show the highest probability of detecting the alternative hypothesis. 

\begin{table}[h]
\centering
\setlength{\tabcolsep}{4pt}
\resizebox{\textwidth}{!}{
\begin{tabular}{lrrrrrrrrrr}
\hline
\toprule
$\sigma$ & \multicolumn{5}{r}{0.1} & \multicolumn{5}{r}{1.0} \\
$\xi$ & $\xi_1$ & $\xi_2 (\Ncal)$ & $\xi_2 (\chi^2)$  & $\xi_3$ & $\xi_{pgy}$ & $\xi_1$ & $\xi_2 (\Ncal)$ & $\xi_2 (\chi^2)$  & $\xi_3$ & $\xi_{pgy}$ \\
$N$ &  &  &  & & &  &  &  & & \\
\midrule
1000 & 0.182 & 0.148 & 0.189 & 0.136 & 0.111 & 1.000 & 1.000 & 1.000 & 1.000 & 0.172 \\
2000 & 0.479 & 0.220 & 0.274 & 0.182 & 0.110 & 1.000 & 1.000 & 1.000 & 1.000 & 0.183 \\
3000 & 0.744 & 0.340 & 0.414 & 0.249 & 0.110 & 1.000 & 1.000 & 1.000 & 1.000 & 0.182 \\
4000 & 0.913 & 0.508 & 0.587 & 0.341 & 0.114 & 1.000 & 1.000 & 1.000 & 1.000 & 0.187 \\
5000 & 0.980 & 0.686 & 0.751 & 0.423 & 0.113 & 1.000 & 1.000 & 1.000 & 1.000 & 0.184 \\
6000 & 0.996 & 0.823 & 0.872 & 0.520 & 0.113 & 1.000 & 1.000 & 1.000 & 1.000 & 0.184 \\
7000 & 1.000 & 0.921 & 0.947 & 0.618 & 0.109 & 1.000 & 1.000 & 1.000 & 1.000 & 0.184 \\
8000 & 1.000 & 0.967 & 0.979 & 0.698 & 0.113 & 1.000 & 1.000 & 1.000 & 1.000 & 0.189 \\
\bottomrule
\hline
\end{tabular}
}

\caption{Measured power of $\zeta_{i,N}$ for DGP 3. $\sigma=0.1$ (left) and $\sigma=1$ (right) . $\phi_m=0.1$ and $\psi_m=0.5$ for all $m$. $\alpha=2/3$, $c=1/2$, $L=\lfloor N^{\frac{1}{4}} \rfloor$, $f: x \mapsto (x-1)^2$. $10^4$ repeats.}
\label{table:power-dgp4}
\end{table}

\subsubsection{DGP 4: Factor Model}
Consider the following factor model:
\begin{equation}
y_{m,n} = \boldsymbol{\lambda_m}^{T}\mathbf{f}_n + \epsilon_{m,n}, \quad n = 1, 2, \ldots, N,
\label{eq:dgp}
\end{equation}
where $(\mathbf{f}_{n})_{n=1,\ldots,N}$ is a $r$-dimensional process generated according to DGP 1, $(\boldsymbol{\lambda_m})_{m=1,\ldots,M}$ are $r \times 1$ deterministic vector of factor loadings, and \\ $(\varepsilon_{m,n})_{m=1,\ldots,M, n=1\ldots,N}$ are iid $\mathcal{N}_c(0, 1)$. Table \ref{table:power-dgp5-r-1} presents the measured power for two Signal to Noise Ratio (SNR) defined as:
$$
SNR = \frac{\sum_{m=1}^M  \|\boldsymbol{\lambda}_m\|^2 \Ebb|\f_{1}|^2 }{\sum_{m=1}^M  \Ebb|\epsilon_{m,1}|^2} 
$$
measured in dB (décibels, i.e. $10 \log_{10} SNR$). Note that this kind of alternative is typically one that is difficult for our test to detect. In \cite{rosuel-vallet-loubaton-mestre-ieeesp-2021}, it was shown that in this so-called spiked model, the eigenvalue distribution of $\hat{\C}_N(\nu)$ converges towards the Marcenko-Pastur distribution as under $\Hcal_0$, but at most $r$ eigenvalues may escape from its support $[\lambda_{-}, \lambda_{+}]$. Since our test statistics is based on a Linear Spectral Statistic of all the eigenvalues, a deviation of only few ones of them will be hard to detect. Table \ref{table:power-dgp5-r-1} confirms this claim
when the Signal to Noise ratio is equal to $-13 dB$, but the performance becomes satisfying when SNR = -7 dB. 
 
\begin{table}[h]

\centering
\setlength{\tabcolsep}{4pt}
\resizebox{\textwidth}{!}{
\begin{tabular}{lrrrrrrrrrr}
\hline
\toprule
SNR(dB) & \multicolumn{5}{r}{-13} & \multicolumn{5}{r}{-7} \\
$\xi$ & $\xi_1$ & $\xi_2 (\Ncal)$ & $\xi_2 (\chi^2)$ & $\xi_3$ & $\xi_{pgy}$ & $\xi_1$ & $\xi_2 (\Ncal)$ & $\xi_2 (\chi^2)$ & $\xi_3$ & $\xi_{pgy}$ \\
$N$ &  &  &  &  &  &  &  &  & & \\
\midrule
1000 & 0.101 & 0.065 & 0.079 & 0.060 & 0.095 & 0.999 & 0.999 & 1.000 & 0.998 & 0.169 \\
2000 & 0.116 & 0.085 & 0.100 & 0.070 & 0.102 & 1.000 & 1.000 & 1.000 & 1.000 & 0.192 \\
3000 & 0.105 & 0.083 & 0.096 & 0.076 & 0.104 & 1.000 & 1.000 & 1.000 & 1.000 & 0.185 \\
4000 & 0.133 & 0.091 & 0.106 & 0.083 & 0.101 & 1.000 & 1.000 & 1.000 & 1.000 & 0.192 \\
5000 & 0.155 & 0.094 & 0.114 & 0.089 & 0.101 & 1.000 & 1.000 & 1.000 & 1.000 & 0.198 \\
6000 & 0.181 & 0.109 & 0.133 & 0.099 & 0.106 & 1.000 & 1.000 & 1.000 & 1.000 & 0.195 \\
7000 & 0.223 & 0.124 & 0.156 & 0.110 & 0.102 & 1.000 & 1.000 & 1.000 & 1.000 & 0.202 \\
8000 & 0.284 & 0.144 & 0.185 & 0.130 & 0.100 & 1.000 & 1.000 & 1.000 & 1.000 & 0.191 \\
\bottomrule
\hline
\end{tabular}
}
\caption{Measured power of $\xi_{i,N}$ for DGP 4 and two values of SNR measured in dB. $\phi_m=0.1$ and $\psi_m=0.5$ for all $m$, $\alpha=2/3$, $c=1/2$, $L=\lfloor N^{\frac{1}{4}} \rfloor$, $f: x \mapsto (x-1)^2$, $10^4$ repeats.}
\label{table:power-dgp5-r-1}
\end{table}

\subsubsection{DGP 3 with non-Gaussian innovations}
In this section, we investigate the robustness of our proposed statistics when applied on time series with non-Gaussian innovations. To assess this, we have specifically considered innovations following a complex Student's distribution with 3 degrees of freedom, known for its heavier tails compared to the Gaussian distribution. The results are compiled in Table \ref{table:power-dgp3-non-gaussian}. A key observation from these results is the stability in the measured power of our proposed statistics, while $\xi_{pgy}$ shows more pronounced differences in view of the results obtained under $\Hcal_0$. 

\begin{table}[h]

\centering
\setlength{\tabcolsep}{4pt}
\resizebox{\textwidth}{!}{
\begin{tabular}{lrrrrrrrrrr}
\hline
\toprule
Innovations & \multicolumn{5}{r}{Gaussian} & \multicolumn{5}{r}{Student (k=3)} \\
$\xi$ & $\xi_1$ & $\xi_2 (\Ncal)$ & $\xi_2 (\chi^2)$ & $\xi_3$ & $\xi_{pgy}$ & $\xi_1$ & $\xi_2 (\Ncal)$ & $\xi_2 (\chi^2)$ & $\xi_3$ & $\xi_{pgy}$ \\
$N$ &  &  &  &  &  &  &  &  & & \\
\midrule
1000 & 0.180 & 0.142 & 0.182 & 0.133 & 0.112 & 0.233 & 0.168 & 0.216 & 0.152 & 0.643 \\
2000 & 0.476 & 0.216 & 0.272 & 0.173 & 0.104 & 0.485 & 0.242 & 0.302 & 0.194 & 0.729 \\
3000 & 0.753 & 0.333 & 0.412 & 0.258 & 0.106 & 0.736 & 0.371 & 0.439 & 0.264 & 0.767 \\
4000 & 0.914 & 0.512 & 0.589 & 0.326 & 0.107 & 0.897 & 0.528 & 0.598 & 0.349 & 0.783 \\
5000 & 0.978 & 0.676 & 0.748 & 0.428 & 0.110 & 0.966 & 0.682 & 0.748 & 0.433 & 0.792 \\
6000 & 0.996 & 0.821 & 0.867 & 0.517 & 0.108 & 0.993 & 0.816 & 0.863 & 0.537 & 0.804 \\
7000 & 1.000 & 0.920 & 0.947 & 0.618 & 0.111 & 0.999 & 0.907 & 0.933 & 0.630 & 0.817 \\
8000 & 1.000 & 0.968 & 0.979 & 0.702 & 0.111 & 1.000 & 0.960 & 0.975 & 0.710 & 0.827 \\
\bottomrule
\hline
\end{tabular}
}
\caption{Measured power of $\xi_{i,N}$ for DGP 3 under Gaussian or Student (df=3) innovations. $\phi_m=0.1$ and $\psi_m=0.5$ for all $m$, $\alpha=2/3$, $c=1/2$, $L=\lfloor N^{\frac{1}{4}} \rfloor$, $f: x \mapsto (x-1)^2$, $10^4$ repeats.}
\label{table:power-dgp3-non-gaussian}
\end{table}

\subsection{Conclusion on the numerical results.}
In this paragraph, we highlight some conclusions from the simulation results.
First, the behaviors of  $\xi_{N,1}(f)$ and $\xi_{N,2}(f)$  when the value of $\alpha$ is close to 7/9 deviates from the CLTs established in this article due to the $\Ocal\left(\frac{B^{5}}{N^{4}}\right)$ bias term which, in principle, tends towards 0, but which in practice can prove to be non-negligible for finite values of $M,B,N$. Second, 
the various numerical results we have presented tend to indicate that it 
is hard to recommend to use $\xi_{N,1}(f)$ rather than  $\xi_{N,2}(f)$
because their first and second type empirical errors appear rather similar. 
We can however clearly conclude that our proposed statistics are more powerful than the statistics presented in PGY.

\section{Final discussion}
We conclude this article with a discussion highlighting future work that we consider relevant. A first direction of research is to study the power analysis of our tests under a large class of alternative. A reasonable model for the alternative $\Hcal_1$ is $\y_n = \sum_{k=0}^{+\infty} \A_k \epsilonbs_{n-l}$, where $(\epsilonbs_n)_{n \in \mathbb{Z}}$ represents a sequence of i.i.d. $\Ncal_c(0, \I)$ distributed random $M$--dimensional vectors. A possible approach would be to establish, under  $\Hcal_1$, a CLT for the LSS of $\hat{\C}(\nu)$ for each $\nu$ , and to study the equivalent of statistics 
$\zeta_{N,1}(f)$ and $\zeta_{N,2}(f)$. \cite{deitmar-2024} established that under $\Hcal_1$, the empirical eigenvalue distribution of 
$\hat{\C}(\nu)$ converges towards a generalized Marcenko-Pastur distribution. However, establishing a CLT requires a much more accurate analysis, and appears to be a challenge. If this line of research is successful, it should be possible, using similar tools, to address problems such as testing the independence between two high-dimensional time series or testing the equality of the spectral densities of two high-dimensional time series (see e.g. \cite{yao-bai-book-2015} and \cite{bodnar-dette-parolya-2019} that studied the case $(\y_n)_{n \in \mathbb{Z}}$ i.i.d.). A second direction is to develop the lag domain approach of \cite{loubaton-mestre-rmta-2022} which studied the behavior of the linear spectral statistics of a normalized version of the sample covariance matrix of vectors $(\y_n^{L})_{n=1, \ldots, N}$, where $\y_n^{L} = (\y_n^{T}, \y_{n+1}^{T}, \ldots, \y_{n+L-1}^{T})^{T}$ and $L = L(N)$ converges toward $+\infty$ in such a way that $\frac{ML}{N} \rightarrow \tilde{c}$, where $0 < \tilde{c} < 1$. In particular, establishing a CLT under $\Hcal_0$ and $\Hcal_1$ on these linear spectral statistics would allow to propose test statistics taking into account implicitly all the frequencies. This would avoid the technical difficulties posed by the frequency domain approach that have been resolved in the present paper thanks to Bartlett's factorization. Finally, the study of the non Gaussian case, discussed in Section \ref{subsub:discussion-assumptions}, is also a quite relevant topic.

\appendix

\section{Proof of Lemma \ref{le:conditional-concentration}}
\label{sec:proof-lemma-concentration}
The proof is a reformulation of various elements presented in \cite{louart-couillet-2023-hanson-wright}. We denote by $\Fcal$ the $\sigma$-algebra of the probability space 
$(\Omega, \Fcal, P)$ on which all the random vectors $(\X_N)_{N \geq 1}$ are defined, and consider the $\sigma$-algebra $\Fcal_{A_N} = \Fcal \cap A_N$ and 
$P_{A_N}$ the probability measure defined on $\Fcal_{A_N}$ by 
$P_{A_N}(B) = P(B)/P(A_N)$. As in  \cite{louart-couillet-2023-hanson-wright} (see the footnote p. 6), we denote by $\X_N/A_N$ the random 
vector defined as the mapping defined on the probability space 
$(A_N, \Fcal_{A_N}, P_{A_N})$ by $(\X_N/A_N)(\omega) = \X_N(\omega)$ for each 
$\omega \in A_N$. We express $P\left( |h_N(\X_N) - \mathbb{E}(h_N(\X_N))| > t \right)$ as
\begin{align*}
  &P\left( |h_N(\X_N) - \mathbb{E}(h_N(\X_N))| > t \right) = 
  \\
  &\qquad\qquad
  P(A_N) P\left( |h_N(\X_N) -  \mathbb{E}(h_N(\X_N)) | > t \, | \, A_N \right) 
  \\
  &\qquad\qquad+ P\left( |h_N(\X_N) -  \mathbb{E}(h_N(\X_N)) | > t, A_N^{c} \right) ,
\end{align*}
and deduce from this that 
\begin{align}
    &P\left( |h_N(\X_N) - \mathbb{E}(h_N(\X_N))| > t \right) \leq 
    \notag\\
    &\qquad P\left( |h_N(\X_N) - \mathbb{E}(h_N(\X_N))| > t \, | \, A_N \right) + e^{-N^{\gamma}}.
    \label{eq:intermediaire-concentration-conditionnee}
\end{align}
It is clear that 
\begin{align*}
&P\left( |h_N(\X_N) - \mathbb{E}(h_N(\X_N))| > t \, | \, A_N \right) = 
\\
&\qquad P_{A_N}\left( | h_N(\X_N|A_N) - \mathbb{E}(h_N(\X_N))| > t \right).
\end{align*}
As $h_N$ is $\sigma_N$-Lipschitz on $\X_N(A_N)= (\X_N|A_N)(A_N)$, Remark 1.5 and Lemma 1.4 in  \cite{louart-couillet-2023-hanson-wright} imply that 
\begin{equation}
    \label{eq:concentration-f-X-conditionne}
    P_{A_N}\left( |h_N(\X_N|A_N) - \mathbb{E}(h_N(\X_N|A_N)| \geq t \right) 
\leq C_1 \exp-C_2\left(\frac{t}{\sigma_N}\right)^{2},
\end{equation}
where it should be understood that $\mathbb{E}(h_N(\X_N|A_N)$ represents the mathematical expectation defined on the probability space $(A_N, \Fcal_{A_N}, P_{A_N})$. In order to conclude, we have to evaluate $\kappa_N = \mathbb{E}(h_N(\X_N)) - \mathbb{E}(h_N(\X_N|A_N))$. For this, we again write that 
\begin{align*}
\mathbb{E}(h_N(\X_N)) & = \mathbb{E}(h_N(\X_N) \mathds{1}_{A_N}) + \mathbb{E}(h_N(\X_N)  \mathds{1}_{A_N^{c}}) \\ 
 & = P(A_N) \mathbb{E}(h_N(\X_N|A_N)) + \mathbb{E}(h_N(\X_N)  \mathds{1}_{A_N^{c}}).
 \end{align*}
 The Schwartz inequality and (\ref{eq:second-moments-f-X}) implies that 
 $$
 |\mathbb{E}(h_N(\X_N)  \mathds{1}_{A_N^{c}})| \leq C e^{-N^{\gamma}/2}.
 $$
As $1 - e^{-N^{\gamma}} \leq P(A_N) \leq 1$, we obtain that 
$$
|\kappa_N | = \left| \mathbb{E}(h_N(\X_N)) -  \mathbb{E}(h_N(\X_N|A_N)) \right| \leq 
C e^{-N^{\gamma}/2}.
$$
We finally remark that 
\begin{align*}
  P_{A_N}\left( | h_N(\X_N|A_N) - \mathbb{E}(h_N(\X_N))| > t \right)  & \\ \leq 
P_{A_N}\left( | h_N(\X_N|A_N) - \mathbb{E}(h_N(\X_N|A_N))| > t - |\kappa_N| \right) 
& \\
\leq \min\left(1, C_1 \exp-\left[C_2( (t - |\kappa_N|)/\sigma_N)^{2}\right]\right),
 \end{align*}
 where the last inequality follows from (\ref{eq:concentration-f-X-conditionne}). Reasoning as in Lemma A.15 in  \cite{these-louart-2023}, we verify that
\begin{align*}
    &\min\left(1, C_1 \exp-\left[C_2( (t - |\kappa_N|)/\sigma_N)^{2}\right]\right) 
    \leq 
    \\
    &\qquad\max\left((C_1, \exp \left[ C_2 (\kappa_N/\sigma_N)^{2} \right] \right) \, \exp-C_2\left(\frac{t}{2\sigma_N}\right)^{2}.
\end{align*}
The hypothesis that $\sigma_N  > \frac{C}{N^{a}}$ and the above evaluation of $\kappa_N$ 
imply that $\exp \left[ C_2 (\kappa_N/\sigma_N)^{2} \right] \leq C$ for some constant $C$
for each $N$. Therefore, we obtain that 
$$
P_{A_N}\left( | h_N(\X_N|A_N) - \mathbb{E}(h_N(\X_N))| > t \right) \leq C_1' 
\exp-C_2'\left(\frac{t}{\sigma_N}\right)^{2} ,
$$
for some constants $C_1'$ and $C_2'$. This and (\ref{eq:intermediaire-concentration-conditionnee}) in turn imply (\ref{eq:conditional-gaussian-concentration}).

\section{Proof of (\ref{eq:moyenne-sup-hanson-wright})}
In order to simplify the notations, we denote $\x_{m,N}$, $\A_{m,N}$ and $\kappa_N$ by $\x_m$, $\A_m$ and $\kappa$, and put 
$\delta_m = |\x_m^* \A_m \x_m - \Tr \A_m|$ as well as $\delta = \sup_{m=1, \ldots, M} \delta_m$. 
In order to take benefit of (\ref{eq:hanson-wright-2}), we express $\mathbb{E}(\delta^{k})$ as 
$$
\mathbb{E}(\delta^{k}) = k \int_0^{+\infty} t^{k-1} P\left( \delta > t\right) \, dt.
$$
For any $\epsilon > 0$, we split the above integral as follows
\begin{equation}
\label{eq:moment-k-sup-hanson-wright}
\mathbb{E}(\delta^{k}) = k \int_0^{N^{\epsilon/k} \kappa}  t^{k-1} P( \delta > t) \, dt + 
 k \int_{N^{\epsilon/k} \kappa}^{+\infty}  t^{k-1} P\left( \delta > t\right) \, dt .
\end{equation}
As $ \mathbb{P}\left( \delta > t\right) \leq 1$, the first term of the right-hand side of (\ref{eq:moment-k-sup-hanson-wright}) verifies 
$$
k \int_0^{N^{\epsilon/k} \kappa}  t^{k-1} \mathbb{P}( \delta > t) \, dt \leq N^{\epsilon}  \kappa^{k}.
$$
In order to evaluate the second term of the right-hand side of (\ref{eq:moment-k-sup-hanson-wright}), we remark 
that if $t \geq N^{\epsilon/k} \kappa$,  (\ref{eq:hanson-wright-2}) implies that 
$$
\mathbb{P}( \delta_m > t ) \leq 2 \exp -C \frac{t}{\kappa},
$$
because $\| \A_m \|_F \leq \kappa$. Therefore, the union bound leads to 
$$
\mathbb{P}(\delta > t) \leq 2 M \exp -C \frac{t}{\kappa},
$$
and to 
\begin{align*}
 k \int_{N^{\epsilon/k} \kappa}^{+\infty}  t^{k-1} \mathbb{P}( \delta > t) \, dt  &  \leq 
2 M k \int_{N^{\epsilon/k} \kappa}^{+\infty}  t^{k-1} \exp -C \frac{t}{\kappa} dt \\
& = 2 M k \kappa^{k} \int_{N^{\epsilon/k}}^{+\infty} t^{k-1} \exp -t \, dt .
\end{align*}
A simple calculation leads to $ \int_{N^{\epsilon/k}}^{+\infty} t^{k-1} \exp -t \, dt = \Ocal\left( N^{\frac{(k-1)\epsilon}{k}} \exp-N^{\epsilon/k} \right)$. Therefore, we also have
$$
k \int_{N^{\epsilon/k} \kappa}^{+\infty}  t^{k-1} \mathbb{P}( \delta > t) \, dt  = o\left(N^{\epsilon} \kappa^{k}\right) .
$$
This completes the proof of (\ref{eq:moyenne-sup-hanson-wright}).

\section{Justification of (\ref{eq:Etrace-tildeQ-Q-precise-1-improved}) and  (\ref{eq:Etrace-tildeQ-Q-precise-2-improved})}
\label{sec:improvement-bias}
(\ref{eq:Etrace-tildeQ-Q-precise-1}) and (\ref{eq:Etrace-tildeQ-Q-precise-2}) are established 
in \cite{loubaton-rosuel-ejs-2021} by evaluating $\frac{1}{M}\mathbb{E}\left( \tilde{\Q}_N(z) - \Q_N(z) \right)$
using the integration by parts formula. The corresponding calculation is long and tedious, so that 
\cite{these-alexis}, Chap. 2 developed a more efficient approach based on the observation that 
matrix $\D_N^{-1/2} \Sigmabs_N$ can be interpreted as a $M \times (B+1)$ large random matrix with mutually independent 
rows having covariance matrices $(\Thetabs_m)_{m=1, \ldots, M}$ given by 
$\Thetabs_m = \I + \Phibs_m$ for each $m$. Therefore, it is possible to use existing large random matrix methods
(see e.g. \cite{kammoun-alouini-ieeeit-2016}) to evaluate the behaviour of the LSS of the eigenvalues of $\tilde{\C}_N = \D_N^{-1/2} \Sigmabs_N \Sigmabs_N^{*}  \D_N^{-1/2}$ and of the expectation $\mathbb{E}(\tilde{\Q}_N(z))$ of the resolvent of $\tilde{\C}_N$. In particular, 
the elements of $\mathbb{E}(\tilde{\Q}_N(z))$ have the same behaviour that the elements of a matrix-valued 
Stieltjes transform $\tilde{\T}_N(z)$. Similarly, if $\bar{\Q}_N(z)$ represents the resolvent of matrix $\Sigmabs_N^{*} \D_N^{-1} \Sigmabs_N$, the elements of the expectation 
$\mathbb{E}(\bar{\Q}_N(z))$ behave as the elements a another matrix-valued 
Stieltjes transform $\bar{\T}_N(z)$. Moreover, $\tilde{\T}_N(z)$ and $\bar{\T}_N(z)$ are defined as the solutions a coupled system of equations. If the covariance matrices $(\Theta_m)_{m=1, \ldots, M}$ were reduced to $\I$, 
$\tilde{\T}_N(z)$ and $\bar{\T}_N(z)$ would be equal to $t_N(z) \I$ and $\tilde{t}_N(z) \I$ respectively. As 
$\| \Thetabs_m - \I \| = \| \Phibs_m \|$ converges towards $0$, it is reasonable to expect that 
$\tilde{\epsilon}_N(z)= \frac{1}{M} \Tr \tilde{\T}_N(z) - t_N(z)$ and 
$\bar{\epsilon}_N(z) =  \frac{1}{M} \Tr \bar{\T}_N(z) - \tilde{t}_N(z)$ converge towards $0$. It is proved 
in (\cite{these-alexis}, Chap. 2) that $\tilde{\epsilon}_N(z)$ is given by 
\begin{align}
    \label{eq:expre-tilde-epsilon-1}
\tilde{\epsilon}_N(z) = & 
 p_N(z)  \left( \frac{1}{B+1} \Tr \left( \frac{1}{M} \sum_{m=1}^{M}  \Phibs_m \right)^{2} \right) - \tilde{p}_N(z) \left( \frac{1}{M} \sum_{m=1}^{M} \frac{1}{B+1} \Tr \Phibs_m \right) + \\ 
 \notag 
 & 
  \mathcal{O}_z\left( \left( \frac{B}{N}\right)^{3} + \frac{1}{N} \right)   \\
   \label{eq:expre-tilde-epsilon-2}
   = & p_N(z) r_N(\nu)  v_N - \tilde{p}_N(z) \tilde{r}_N(\nu) v_N  + 
  \mathcal{O}_z\left( \left( \frac{B}{N}\right)^{3} +  \frac{1}{N} \right) ,
\end{align}
However, a closer look to the proof of Proposition 2-2 in \cite{these-alexis} shows that 
(\ref{eq:expre-tilde-epsilon-1}) can be replaced by 
\begin{align*}
\tilde{\epsilon}_N(z) = & 
 p_N(z)  \left( \frac{1}{B+1} \Tr \left( \frac{1}{M} \sum_{m=1}^{M}  \Phibs_m \right)^{2} \right) - \tilde{p}_N(z) \left( \frac{1}{M} \sum_{m=1}^{M} \frac{1}{B+1} \Tr \Phibs_m \right) + \\ 
 & 
  \mathcal{O}_z\left( \left( \frac{B}{N}\right)^{4} + \frac{1}{N} \right).
\end{align*}
Moreover, it is easily seen that 
\begin{align}
\label{eq:expansion-trace-average-Phi-carre}
\frac{1}{B+1} \Tr \left( \frac{1}{M} \sum_{m=1}^{M}  \Phibs_m \right)^{2} = r_N \, v_N + \mathcal{O}_z\left( \left( \frac{B}{N}\right)^{4} + \frac{1}{N} \right), \\
\label{eq:expansion-trace-average-Phi}
 \frac{1}{M} \sum_{m=1}^{M} \frac{1}{B+1} \Tr \Phibs_m  = \tilde{r}_N \, v_N + \mathcal{O}_z\left( \left( \frac{B}{N}\right)^{4} + \frac{1}{N} \right),
\end{align}
and therefore that 
\begin{equation}
\label{eq:final-tight-evaluation-tilde-epsilon}
\tilde{\epsilon}_N(z) = p_N(z) r_N(\nu)  v_N - \tilde{p}_N(z) \tilde{r}_N(\nu) v_N  + 
  \mathcal{O}_z\left( \left( \frac{B}{N}\right)^{4} +  \frac{1}{N} \right) .
\end{equation}
 (\ref{eq:Etrace-tildeQ-Q-precise-1-improved}) and  (\ref{eq:Etrace-tildeQ-Q-precise-2-improved}) then follow from 
$$
\frac{1}{M} \Tr  \mathbb{E}(\tilde{\Q}_N(z)) - \frac{1}{M} \Tr  \tilde{\T}_N(z)) = \Ocal_z(B^{-2}),
$$
and
$$
\frac{1}{M} \Tr  \mathbb{E}(\Q_N(z)) - t_N(z) = \Ocal_z(B^{-2}),
$$
as well as from $B^{-2} = o(N^{-1})$ because $\alpha > \frac{1}{2}$.

\section{Proof of Proposition \ref{prop:properties-Psi1r-Psi2r}}
\label{sec:proof-properties-Psi1r-Psi2r}
In order to evaluate the properties of $\Psibs_{m,r}^{1}$ and 
$\Psibs_{m,r}^{2}$, we study the behaviour of the entries of the covariance matrix  
 $\Omegabs_{m,r} = \mathbb{E}(\omegabs_{m,r}^{*} \omegabs_{m,r}) $ of $\omegabs_{m,r}$, and take benefit of the identity 
\begin{equation}
\label{eq:decomposition-Omegamr}
\frac{\Omegabs_{m,r}}{s_m}  = (\Psibs_{m,r}^{1})^{*} \Psibs_{m,r}^{1} +  (\Psibs_{m,r}^{2})^{*} \Psibs_{m,r}^{2},
\end{equation}
to deduce (\ref{eq:norme-Psi2r}). We first prove the following Lemma. 
\begin{lemma}
\label{le:entries-Omegamr}
For each $b_1,b_2 \in \{ -B/2, \ldots, B/2 \}$, $b_1 \neq b_2$ , we have
\begin{equation}
\label{eq:entries-Omegamr}
\left(\Omegabs_{m,r}\right)_{b_1,b_2} = \mathbb{E}\left(r_{m,\mathfrak{b}}(\nu+b_2/N) r_{m,\mathfrak{b}}(\nu+b_1/N)^{*}\right) = \mathcal{O}(\frac{1}{N}),
\end{equation}
where the $\mathcal{O}(\frac{1}{N})$ term is uniform w.r.t. $b_1,b_2,\nu$ and $m$.
\end{lemma}
\begin{proof} 
We recall that if
    $$
    h_m(\nu) = \sum_{u=0}^{+\infty} a_{m,u} e^{-2 \pi\nu u}
    $$
    represents the Fourier expansion of $h_m$, then, (\ref{eq:reste-serie-am}) implies that 
    \begin{equation}
      \label{eq:majoration-reste}
      \sum_{u=N}^{+\infty} |a_{m,u}| \leq \frac{C}{N^{\gamma}},
        \end{equation}
        where $\gamma \geq 1$. 
    We follow \cite{walker-1965}, and express $r_{m,\mathfrak{b}}(\nu)$ as 
    $$
    r_{m,\mathfrak{b}}(\nu) = \sum_{u=0}^{+\infty} a_{m,u} e^{-2 i \pi u \nu} z_{m,u}(\nu),
    $$
    where $z_{m,u}(\nu)$ is given by $z_{m,0}(\nu) = 0$ for $u=0$,
    \begin{align}
      \label{eq:ugeqN}
      z_{m,u}(\nu) = \frac{1}{\sqrt{N}} \sum_{v=1-u}^{N-u} \epsilon_{m,v} e^{-2 i \pi (v-1) \nu} - \frac{1}{\sqrt{N}} \sum_{v=1}^{N} \epsilon_{m,v} e^{-2 i \pi (v-1) \nu}, \\
      \end{align}
      for $u \geq N$, and
      \begin{align}
      \label{eq:ulessN}
      z_{m,u}(\nu) = \frac{1}{\sqrt{N}} \sum_{v=1-u}^{0} \epsilon_{m,v} e^{-2 i \pi (v-1) \nu} - \frac{1}{\sqrt{N}} \sum_{v=N-u+1}^{N} \epsilon_{m,v} e^{-2 i \pi (v-1) \nu},
    \end{align}
    for $1 \leq u \leq N-1$. 
    We consider $\nu_1 = \nu + b_1/N$ and $\nu_2 = \nu + b_2/N$, and evaluate $(\Omega_{m,r})_{b_1,b_2} = \mathbb{E}(r_{m,\mathfrak{b}}(\nu_2) r_{m,\mathfrak{b}}(\nu_1)^{*})$. 
    $(\Omega_{m,r})_{b_1,b_2}$ is given by
    $$
    (\Omega_{m,r})_{b_1,b_2} = \sum_{u,u'=0}^{+\infty} a_{m,u} a_{m,u'}^{*} e^{-2i\pi u \nu_2}  e^{2i\pi u' \nu_1} \mathbb{E}\left(z_{m,u}(\nu_2) z_{m,u'}(\nu_1)^{*} \right).
    $$
    If $u \geq N$ and $u^{'} \geq N$, it is easy to check that $\left| \mathbb{E}\left(z_{m,u}(\nu_2) z_{m,u'}(\nu_1)^{*} \right) \right| \leq 1 + \delta_{\nu_1 = \nu_2}$. Therefore, 
    \begin{align*}
    \left| \sum_{u,u'=N}^{+\infty} a_{m,u} a_{m,u'}^{*} e^{-2i\pi u \nu_2}  e^{2i\pi u' \nu_1} \mathbb{E}\left(z_{m,u}(\nu_2) z_{m,u'}(\nu_1)^{*} \right) \right| & \leq 2 \left( \sum_{u=N}^{+\infty} |a_{m,u}| \right)^{2} \\ 
    & 
    = \mathcal{O}\left( \frac{1}{N^{2\gamma}} \right),
    \end{align*}
    If $1 \leq u \leq N-1$ and $1 \leq u' \leq N-1$, the two terms at the right-hand side of  (\ref{eq:ulessN}) are uncorrelated, and we obtain that $\left| \mathbb{E}\left(z_{m,u}(\nu_2) z_{m,u'}(\nu_1)^{*} \right) \right| \leq 2 \frac{\min(u,u')}{N} \leq 2 \frac{u^{1/2} u'^{1/2}}{N}$. This implies that  
    \begin{align*}
    &\left| \sum_{u,u'=1}^{N-1} a_{m,u} a_{m,u'}^{*} e^{-2i\pi u \nu_2}  e^{2i\pi u' \nu_1} \mathbb{E}\left(z_{m,u}(\nu_2) z_{m,u'}(\nu_1)^{*} \right) \right| 
    \\
    & \qquad\leq
    \frac{2}{N} \left(\sum_{u=1}^{+\infty} u^{1/2} |a_{m,u}| \right)^{2} 
    \\ 
    & \qquad = \mathcal{O}\left( \frac{1}{N} \right).
    \end{align*}
    If $1 \leq u \leq N-1$ and $u' \geq N$, we have  $ \mathbb{E}\left(z_{m,u}(\nu_2) z_{m,u'}(\nu_1)^{*} \right) = T_1 + T_2$ with
    $$
    T_1 = \mathbb{E} \left( \frac{1}{\sqrt{N}} \sum_{v=1-u}^{0} \epsilon_{m,v} e^{-2 i \pi (v-1) \nu_2}  \frac{1}{\sqrt{N}} \sum_{v'=1-u'}^{N-u'} \epsilon_{m,v}^{*} e^{2 i \pi (v-1) \nu_1}
    \right),
    $$
    and
    $$
    T_2 = \frac{1}{N} \sum_{v=N-u+1}^{N} e^{-2 i \pi (v-1)(\nu_2 - \nu_1)},
    $$
    It is clear that $|T_2| \leq \frac{u}{N}$. Moreover, if $N-u' < 1 - u$, i.e. if $u'-u > N-1$,  then $T_1 = 0$. If $N-u' \geq 1 - u$, $T_1$ is equal yo $T_1 = \frac{1}{N} \sum_{v=1-u}^{N-u'} e^{-2 i \pi (v-1)(\nu_2 - \nu_1)}$, and $T_1 \leq 1 - \frac{u'-u}{N} \mathds{1}_{u'-u \leq N-1} \leq 1$. We deduce from this that 
    \begin{align*}
    &\left| \sum_{u=1}^{N-1}  \sum_{u'=N}^{+\infty} a_{m,u} a_{m,u'}^{*} e^{-2i\pi u \nu_2}  e^{2i\pi u' \nu_1} \mathbb{E}\left(z_{m,u}(\nu_2) z_{m,u'}(\nu_1)^{*} \right) \right| 
    \\
    &\leq
    \left(\frac{1}{N} \sum_{u=0}^{N-1} u |a_{m,u}|\right) \left(\sum_{u=N}^{+\infty} |a_{m,u}|\right) +  \left( \sum_{u=0}^{N-1}  |a_{m,u}|\right) \left(\sum_{u=N}^{+\infty}  |a_{m,u}| \right) 
    \\
    &= \mathcal{O}\left( \frac{1}{N^{\gamma}} \right).
    \end{align*}
    As $\gamma \geq 1$, this completes the proof of the Lemma. 
    \end{proof}
 As a consequence of Lemma \ref{le:entries-Omegamr}, we obtain immediately that  
\begin{equation}
\label{eq:proprietes-Omegamr}
\| \Omegabs_{m,r} \| = \mathcal{O}(\frac{B}{N}), \; \frac{1}{B+1} \mathrm{Tr} \,  \Omegabs_{m,r} = \mathcal{O}(\frac{1}{N}).
\end{equation}
We therefore deduce from (\ref{eq:decomposition-Omegamr}) that 
\begin{align}
\label{eq:prop-TracePsimr1}
\frac{1}{B+1} \mathrm{Tr} (\Psibs_{m,r}^{1})^{*} \Psibs_{m,r}^{1} =  \mathcal{O}(\frac{1}{N}), \\
\label{eq:prop-TracePsimr2}
\frac{1}{B+1} \mathrm{Tr} (\Psibs_{m,r}^{2})^{*} \Psibs_{m,r}^{2} =  \mathcal{O}(\frac{1}{N}),
\end{align}
as well as $\|  \Psibs_{m,r}^{i} \| = \mathcal{O}\left((\frac{B}{N})^{1/2}\right)$, $i=1,2$. 
We have thus in particular proved that (\ref{eq:norme-Psi2r}) holds.
(\ref{eq:norme-Psi1mr}) is eventually a direct consequence of the following result.
\begin{lemma}
\label{le:entries-Psi1mr}
The entries of $\Psibs_{m,r}^{1}$ verify 
\begin{equation}
\label{eq:entries-Psimr}
 \max_{b_1,b_2} \left|(\Psibs_{m,r}^{1})_{b_1,b_2}\right| \leq \frac{C}{N},
\end{equation}
 for some nice constant $C$
\end{lemma}
\begin{proof}
$\Psibs_{m,r}^{1}$ is equal to  $ \Psibs_{m,r}^{1} = \mathbb{E} \left(\x_{m,\mathfrak{b}}^{*}  \, \frac{\omega_{m,r}}{\sqrt{s_m}}  \right)$. Reasoning as in the proof of Lemma  \ref{le:entries-Omegamr},
we easily obtain (\ref{eq:entries-Psimr}).
\end{proof}

\section{Details of the proof of Theorem \ref{th:simplification-theta1b}}
\label{sec:details-proof-theorem2}
We first provide in Subsection \ref{subsec:needed-results} an overview of the various results that are needed to achieve the 3 steps of the proof, and then present in Subsections 
\ref{subsec:proof-step1}, \ref{subsec:proof-step2} and \ref{subsec:proof-step3} all the detailed proofs. 
\subsection{Overview of the proofs of Steps 1 to 3}
\label{subsec:needed-results}
{\bf Step 1}. We first establish that 
\begin{eqnarray}
\label{eq:domination-norme-Upsilon1}
\| \Upsilonbs_1 \| & = & \Ocal_{\prec}(B^{-3/2}), \\
\label{eq:moments-norme-Upsilon1}
\mathbb{E}( \| \Upsilonbs_1 \|^{k}) & = & \Ocal(B^{-3k/2+\epsilon}),
\end{eqnarray} 
for each $\epsilon > 0$. As $B^{-3/2} = o\left( \sqrt{BN}^{-1}\right)$ because $\alpha > \frac{1}{2}$, this immediately implies that the contribution of $\Upsilonbs_1$ to 
$\mathbb{E} \left(  \frac{1}{M} \Tr \left( \hat{\Q}_{N,\mathfrak{b}}(z) - \tilde{\Q}_{N,\mathfrak{b}}(z)\right) \right)$ and to $\frac{1}{M} \Tr \left( \hat{\Q}_{N,\mathfrak{b}}(z) - \tilde{\Q}_{N,\mathfrak{b}}(z)\right)$, and therefore to 
$ \left(\frac{1}{M} \Tr \left( \hat{\Q}_{N,\mathfrak{b}}(z) - \tilde{\Q}_{N,\mathfrak{b}}(z)\right)\right)^{\circ}$ 
are  terms $\Ocal_z(B^{-3k/2+\epsilon})$ and $\Ocal_{\prec,z}(B^{-3/2})$
respectively. As the moments of $\| \tilde{\C}_{\mathfrak{b}} \|$ are easily seen to be $\Ocal(1)$ terms, 
we also deduce from (\ref{eq:E-norm-hatDb-D}) and (\ref{eq:moments-norme-Upsilon1}) that
\begin{equation}
\label{eq:E-hatCb-tildeCb}
\mathbb{E}\left( \|\hat{\C}_{\mathfrak{b}} - \tilde{\C}_{\mathfrak{b}}\|^{k}\right) = \Ocal\left( B^{-k/2+\epsilon} \right),
\end{equation} 
for each integer $k$ and each $\epsilon > 0$, a result that will be used extensively in the following. 
\\

{\bf Step 2}. We then study the contribution of the 3 quadratic terms in (\ref{eq:decomposition-hatCb-tildeCb}) (i.e. the third, fourth, and fifth term
at the right-hand side of (\ref{eq:decomposition-hatCb-tildeCb})) to \\ $- \frac{1}{M} \Tr \left( \hat{\Q}_{\mathfrak{b}} (\hat{\C}_{\mathfrak{b}} - \tilde{\C}_{\mathfrak{b}}) \tilde{\Q}_{\mathfrak{b}}\right)$. We first justify that replacing $ \hat{\Q}_{\mathfrak{b}}$ and $ \tilde{\Q}_{\mathfrak{b}}$ by $\Q_{\mathfrak{b}}$, and 
$\tilde{\C}_{\mathfrak{b}}$ by $\frac{\X_{\mathfrak{b}} \X_{\mathfrak{b}}^{*}}{B+1}$,  in the terms  
\begin{equation}
  \label{eq:contrib-terme-3-4}
\delta_{34,\mathfrak{b}} =  - \frac{3}{8} \frac{1}{M} \Tr \left( \hat{\Q}_{\mathfrak{b}} \left( \left((\hat{\D}_{\mathfrak{b}} - \D ) \D^{-1}\right)^{2} \tilde{\C}_{\mathfrak{b}} +   \tilde{\C}_{\mathfrak{b}} \left((\hat{\D}_{\mathfrak{b}} - \D ) \D^{-1}\right)^{2} \right) \tilde{\Q}_{\mathfrak{b}}\right),
\end{equation}
and
\begin{equation}
  \label{eq:contrib-terme-5}
\delta_{5,\mathfrak{b}} =  - \frac{1}{4} \frac{1}{M} \Tr \left( \hat{\Q}_{\mathfrak{b}} \left(  (\hat{\D}_{\mathfrak{b}} - \D) \D^{-1} \tilde{\C}_{\mathfrak{b}}   (\hat{\D}_{\mathfrak{b}} - \D) \D^{-1} \right)  \tilde{\Q}_{\mathfrak{b}}\right),
\end{equation}
produces error terms that are negligible w.r.t. $ \frac{1}{\sqrt{BN}}$. More precisely, we have 
\begin{align}
  \label{eq:equivalent1-delta34}
  \delta_{34,\mathfrak{b}}  = & -\frac{3}{4} \frac{1}{M} \Tr (\Q_{\mathfrak{b}} + z \Q_{\mathfrak{b}}^{2}) \left((\hat{\D}_{\mathfrak{b}} - \D ) \D^{-1}\right)^{2} + \Ocal_{\prec,z} \left(\frac{1}{N} + \frac{1}{B^{3/2}}\right), \\
  \label{eq:equivalent1-delta5}
  \delta_{5,\mathfrak{b}}  = &  -\frac{1}{4} \frac{1}{M} \Tr \Q_{\mathfrak{b}}^{2} (\hat{\D}_{\mathfrak{b}} - \D ) \D^{-1} \frac{\X_{\mathfrak{b}} \X_{\mathfrak{b}}^{*}}{B+1}  (\hat{\D}_{\mathfrak{b}} - \D ) \D^{-1}  \\ 
  \notag & 
  + \Ocal_{\prec,z} \left(\frac{1}{N} + \frac{1}{B^{3/2}}\right),
\end{align}
as well as
\begin{align}
  \label{eq:E-equivalent1-delta34}
  \mathbb{E}(\delta_{34,\mathfrak{b}})  = & -\frac{3}{4} \mathbb{E} \left(  \frac{1}{M} \Tr (\Q_{\mathfrak{b}} + z \Q_{\mathfrak{b}}^{2}) \left((\hat{\D}_{\mathfrak{b}} - \D ) \D^{-1}\right)^{2} \right)+ \Ocal_{z} \left(\frac{B^{\epsilon}}{N} + \frac{B^{\epsilon}}{B^{3/2}}\right), \\
  \label{eq:E-equivalent1-delta5}
 \mathbb{E}(\delta_{5,\mathfrak{b}}) = &  -\frac{1}{4} \mathbb{E} \left(\frac{1}{M} \Tr \Q_{\mathfrak{b}}^{2} (\hat{\D}_{\mathfrak{b}} - \D ) \D^{-1} \frac{\X_{\mathfrak{b}} \X_{\mathfrak{b}}^{*}}{B+1}  (\hat{\D}_{\mathfrak{b}} - \D ) \D^{-1} \right) \\ \notag 
 & + \Ocal_{z} \left(\frac{B^{\epsilon}}{N} + \frac{B^{\epsilon}}{B^{3/2}}\right) ,
\end{align}
for each $\epsilon > 0$. In (\ref{eq:equivalent1-delta34}) and (\ref{eq:E-equivalent1-delta34}), we have used the resolvent identity (\ref{eq:resolvent-identity}) $\frac{\X_{\mathfrak{b}} \X_{\mathfrak{b}}^{*}}{B+1} \Q_{\mathfrak{b}}= \I + z \Q_{\mathfrak{b}}$. Using the decomposition (\ref{eq:decomposition-Dx-I}) and the properties of $\D_{2,\mathfrak{b}}$ and $\D_{3,\mathfrak{b}}$, we establish that
\begin{align}
&\frac{1}{\pi} \mathrm{Re} \int_{\Dcal} \bar{\partial}  \Phi_k(f)(z) \,   \left(  \delta_{34,\mathfrak{b}}(z) + 
\frac{3}{4} ( z t(z))^{'} \frac{1}{M} \Tr \left(\D_{\x_{\mathfrak{b}}} - \I\right)^{2} \right)^{\circ}  \diff x \diff y 
    \notag\\
&\qquad =  \Ocal_{\prec}\left(\frac{1}{N} + \frac{1}{B^{3/2}} + \frac{B^{5/2}}{N^{3}} \right) ,
\label{eq:equivalent2-delta34}
\end{align}
and
\begin{align}
\label{eq:equivalent2-E-delta34} 
 &\mathbb{E} (\delta_{34,\mathfrak{b}})  =  -\frac{3}{4}  ( z t(z))^{'} \mathbb{E} \left(  \frac{1}{M} \Tr \left(\D_{\x_{\mathfrak{b}}} - \I\right)^{2} \right)  
 \notag\\
 &\qquad\qquad+  \Ocal_{z} \left(\frac{B^{\epsilon}}{N} + \frac{B^{\epsilon}}{B^{3/2}} + \frac{B^{\epsilon} B^{5/2}}{N^{3}} + \left(\frac{B}{N}\right)^{4} \right),
\end{align}
as well as 
\begin{align}
& \frac{1}{\pi} \mathrm{Re} \int_{\Dcal} \bar{\partial}  \Phi_k(f)(z) \,   \left(  \delta_{5,\mathfrak{b}}(z) +
\frac{1}{4} \frac{1}{M} \Tr \Q_{\mathfrak{b}}^{2} (\D_{\x_{\mathfrak{b}}} - \I)  \frac{\X_{\mathfrak{b}} \X_{\mathfrak{b}}^{*}}{B+1}   (\D_{\x_{\mathfrak{b}}} - \I) \right)^{\circ}  \diff x \diff y  \notag\\
&\qquad = \Ocal_{\prec}\left(\frac{1}{N} + \frac{1}{B^{3/2}} + \frac{B^{5/2}}{N^{3}} \right),
\label{eq:equivalent2-delta5}
\end{align}
and
\begin{align}
 \mathbb{E} (\delta_{5,\mathfrak{b}})  = &  -\frac{1}{4} \mathbb{E} \left( \frac{1}{M} \Tr \Q_{\mathfrak{b}}^{2} (\D_{\x_{\mathfrak{b}}} - \I)  \frac{\X_{\mathfrak{b}} \X_{\mathfrak{b}}^{*}}{B+1}   (\D_{\x_{\mathfrak{b}}} - \I) \right)  \notag \\
& \qquad+  \Ocal_{z} \left(\frac{B^{\epsilon}}{N}+ \frac{B^{\epsilon}}{B^{3/2}} + \frac{B^{\epsilon} B^{5/2}}{N^{3}} + \left(\frac{B}{N}\right)^{4} \right).
\label{eq:equivalent2-E-delta5}  
\end{align}
We eventually prove that 
\begin{align}
\label{eq:stochastic-domination-tr-D-I-square}
& \frac{1}{M} \Tr \left( \left(\D_{\x_{\mathfrak{b}}} - \I\right)^{2} \right)^{\circ} = \Ocal_{\prec}(B^{-3/2}) \\
\label{eq:stochastic-domination-tr-Q2-D-I-XX-D-I}
& \frac{1}{\pi} \mathrm{Re} \int_{\Dcal} \bar{\partial}  \Phi_k(f)(z) \, \left( \frac{1}{M} \Tr \Q_{\mathfrak{b}}^{2} (\D_{\x_{\mathfrak{b}}} - \I)  \frac{\X_{\mathfrak{b}} \X_{\mathfrak{b}}^{*}}{B+1}   (\D_{\x_{\mathfrak{b}}} - \I) \right)^{\circ}  \diff x \diff y  \\ \notag  & =  \Ocal_{\prec}(B^{-3/2}),
\end{align}
and that 
\begin{eqnarray}
\label{eq:E-tr-D-I-square}
& \mathbb{E}\left(  \frac{1}{M} \Tr \left(\D_{\x_{\mathfrak{b}}} - \I \right)^{2} \right)  = \frac{1}{B+1}, \\
\label{eq:E-tr-tr-Q2-D-I-XX-D-I}
& \mathbb{E}\left(  \frac{1}{M} \Tr \Q_{\mathfrak{b}}^{2} (\D_{\x_{\mathfrak{b}}} - \I)  \frac{\X_{\mathfrak{b}} \X_{\mathfrak{b}}^{*}}{B+1}   (\D_{\x_{\mathfrak{b}}} - \I) \right)  = \frac{t'(z)}{B+1} +  \Ocal_{z} \left(\frac{1}{B^{3/2}}\right) . 
\end{eqnarray}
As $\frac{B^{\epsilon}}{N} + \frac{B^{\epsilon}}{B^{3/2}} + \frac{B^{\epsilon} B^{5/2}}{N^{3}} = o\left( \frac{1}{\sqrt{BN}}\right)$ for $\epsilon$ small enough if $\alpha < \frac{4}{5}$, this, in turn, leads to 
\begin{equation}
  \label{eq:equivalent-delta34-plus-delta5}
\frac{1}{\pi} \mathrm{Re} \int_{\Dcal} \bar{\partial}  \Phi_k(f)(z) \,   \left(\delta_{34,\mathfrak{b}}(z) + \delta_{5,\mathfrak{b}}(z) \right)^{\circ} 
\diff x \diff y  =  o_{\prec} \left(\frac{1}{\sqrt{NB}}\right) ,
\end{equation}
and to
\begin{equation}
  \label{eq:equivalent-E-delta34-plus-delta5}
  \mathbb{E}(\delta_{34,\mathfrak{b}} + \delta_{5,\mathfrak{b}}) = -\frac{1}{4(B+1)} \, \left( 3 (zt(z))' + (t(z))' \right) +
  \Ocal_{z}\left( \left( \frac{B}{N}\right)^{4} \right) + o_{z} \left(\frac{1}{\sqrt{NB}}\right) .
\end{equation}
{\bf Step 3} We next address the behaviour of the contributions $\delta_{1,\mathfrak{b}}$ and $\delta_{2,\mathfrak{b}}$ of the two linear terms of 
the right-hand side of (\ref{eq:decomposition-hatCb-tildeCb}) to 
$- \frac{1}{M} \Tr \left( \hat{\Q}_{\mathfrak{b}} (\hat{\C}_{\mathfrak{b}} - \tilde{\C}_{\mathfrak{b}}) \tilde{\Q}_{\mathfrak{b}}\right)$. As these two terms are very similar, we provide more details on the evaluation of $\delta_{1,\mathfrak{b}}$ defined 
by 
\begin{equation}
\label{eq:def-delta}
\delta_{1,\mathfrak{b}} = \frac{1}{2} \frac{1}{M} \Tr \left( \hat{\Q}_{\mathfrak{b}}  (\hat{\D}_{\mathfrak{b}} - \D) \D^{-1} \tilde{\C}_{\mathfrak{b}}    \tilde{\Q}_{\mathfrak{b}}\right).
\end{equation}
We express $\hat{\Q}_{\mathfrak{b}}$ as 
$$
\hat{\Q}_{\mathfrak{b}} = \tilde{\Q}_{\mathfrak{b}} - \hat{\Q}_{\mathfrak{b}} (\hat{\C}_{\mathfrak{b}} - \tilde{\C}_{\mathfrak{b}}) \tilde{\Q}_{\mathfrak{b}},
$$
and plug this expression in (\ref{eq:def-delta}) to obtain that 
$$
\delta_{1,\mathfrak{b}} = \delta_{11,\mathfrak{b}} + \delta_{12,\mathfrak{b}},
$$
where $\delta_{11,\mathfrak{b}} = \frac{1}{2} \frac{1}{M} \Tr \left( \tilde{\Q}_{\mathfrak{b}}  (\hat{\D}_{\mathfrak{b}} - \D) \D^{-1} \tilde{\C}_{\mathfrak{b}}    \tilde{\Q}_{\mathfrak{b}}\right)$
and where $\delta_{12,\mathfrak{b}}$ is given by
$$
\delta_{12,\mathfrak{b}} = - \frac{1}{2} \frac{1}{M} \Tr \left(  \hat{\Q}_{\mathfrak{b}} (\hat{\C}_{\mathfrak{b}} - \tilde{\C}_{\mathfrak{b}})  \tilde{\Q}_{\mathfrak{b}}  (\hat{\D}_{\mathfrak{b}} - \D) \D^{-1} \tilde{\C}_{\mathfrak{b}}   \tilde{\Q}_{\mathfrak{b}}\right).
$$
For $i=1,2$, we evaluate $\mathbb{E}(\delta_{1i,\mathfrak{b}})$ and 
$\frac{1}{\pi} \mathrm{Re} \int_{\Dcal} \bar{\partial}  \Phi_k(f)(z) \, \left(\delta_{1i,\mathfrak{b}} (z)\right)^{\circ} \, \diff x \diff y $. \\

We begin by the study of $\delta_{12,\mathfrak{b}}$, 
and use the expansion (\ref{eq:decomposition-hatCb-tildeCb}) of $\hat{\C}_{\mathfrak{b}} - \tilde{\C}_{\mathfrak{b}}$, 
We remark that the contributions of the non-linear terms w.r.t. $(\hat{\D}_{\mathfrak{b}} - \D) \D^{-1}$ to $\delta_{12,\mathfrak{b}}$ are $\Ocal_{\prec,z}(B^{-3/2}) = 
o_{\prec,z} \left(\frac{1}{\sqrt{NB}}\right)$ terms. Therefore, $\delta_{12,\mathfrak{b}}$ can be written as
\begin{align}
\delta_{12,\mathfrak{b}} =  &  \frac{1}{4} \; \frac{1}{M} \mathrm{Tr} \Biggl[\hat{\Q}_{\mathfrak{b}} \left( (\hat{\D}_{\mathfrak{b}} - \D) \D^{-1} \tilde{\C}_{\mathfrak{b}}  +  \tilde{\C}_{\mathfrak{b}} (\hat{\D}_{\mathfrak{b}} - \D) \D^{-1} \right) 
\notag\\
  &\qquad\tilde{\Q}_{\mathfrak{b}}  (\hat{\D}_{\mathfrak{b}} - \D) \D^{-1} \tilde{\C}_{\mathfrak{b}} \tilde{\Q}_{\mathfrak{b}} \Biggr]
  + \Ocal_{\prec,z}(B^{-3/2}),
  \label{eq:expre-1-delta12b}
\end{align}
Moreover, for each $\epsilon > 0$, we also have
\begin{align}
\label{eq:expre-1-E-delta12b}
\mathbb{E}\left(\delta_{12,\mathfrak{b}}\right) &=  
\frac{1}{4} \mathbb{E}
\Bigl(
    \frac{1}{M} \mathrm{Tr} \Bigl[\hat{\Q}_{\mathfrak{b}} \left( (\hat{\D}_{\mathfrak{b}} - \D) \D^{-1} \tilde{\C}_{\mathfrak{b}}  +  \tilde{\C}_{\mathfrak{b}} (\hat{\D}_{\mathfrak{b}} - \D) \D^{-1} \right) 
    \notag\\
    &\qquad\qquad\qquad\tilde{\Q}_{\mathfrak{b}}  (\hat{\D}_{\mathfrak{b}} - \D) \D^{-1} \tilde{\C}_{\mathfrak{b}} \tilde{\Q}_{\mathfrak{b}} \Bigr]
\Bigr)
  + \Ocal_{z}(B^{-3/2+\epsilon}).
\end{align}

Up to $\Ocal_{\prec,z}\left( N^{-1} + B^{-3/2} \right) = o_{\prec,z} \left(\frac{1}{\sqrt{NB}}\right)$ error terms, it is possible to replace 
$\hat{\Q}_{\mathfrak{b}}$, $\tilde{\Q}_{\mathfrak{b}}$ by $\Q_{\mathfrak{b}}$
and $\tilde{\C}_{\mathfrak{b}}$ by $\frac{\X_{\mathfrak{b}}\X_{\mathfrak{b}}^*}{B+1}$. In other words, 
$\delta_{12,\mathfrak{b}}$ is given by 
\begin{equation}
\label{eq:domination-stochastique-delta12b}
\delta_{12,\mathfrak{b}} = \tilde{\delta}_{12,\mathfrak{b}} + \Ocal_{\prec,z}\left( N^{-1} + B^{-3/2} \right),
\end{equation}
where 
\begin{align}
\label{eq:def-tildedelta12-delta11}
\tilde{\delta}_{12,\mathfrak{b}} &=  
\frac{1}{4} \frac{1}{M} \mathrm{Tr} \Bigl[\Q_{\mathfrak{b}} \left(  (\hat{\D}_{\mathfrak{b}} - \D) \D^{-1}  \frac{\X_{\mathfrak{b}}\X_{\mathfrak{b}}^*}{B+1} + 
\frac{\X_{\mathfrak{b}}\X_{\mathfrak{b}}^*}{B+1} (\hat{\D}_{\mathfrak{b}} - \D) \D^{-1} \right) 
\notag\\
&\qquad\qquad\qquad
\Q_{\mathfrak{b}}   (\hat{\D}_{\mathfrak{b}} - \D) \D^{-1}  \frac{\X_{\mathfrak{b}}\X_{\mathfrak{b}}^*}{B+1} \Q_{\mathfrak{b}} \Bigr] .
\end{align}
We also have
\begin{equation}
\label{eq:E-delta12b-tilde-delta12b}
\mathbb{E}(\delta_{12,\mathfrak{b}}) = \mathbb{E}(\tilde{\delta}_{12,\mathfrak{b}}) + \Ocal_{z}\left( B^{\epsilon} N^{-1} + B^{-3/2+\epsilon} \right).
\end{equation}
As in the context of the evaluation of $\delta_{5,\mathfrak{b}}$ in the course of Step 2, using the properties of $\D_{2,\mathfrak{b}}$
and  $\D_{3,\mathfrak{b}}$, we obtain that 
\begin{equation}
\label{eq:simplification-tildedelta12-rond}
\frac{1}{\pi} \mathrm{Re} \int_{\Dcal} \bar{\partial}  \Phi_k(f)(z) \,   \left(  (\tilde{\delta}_{12,\mathfrak{b}}(z) -  T_{12,\mathfrak{b}}(z))\right)^{\circ}  \diff x \diff y =  \Ocal_{\prec}\left(\frac{1}{N} + \frac{1}{B^{3/2}} + \frac{B^{5/2}}{N^{3}} \right),
\end{equation}
as well as 
\begin{equation}
\label{eq:simplification-E-tildedelta12-rond}
\mathbb{E} \left( \tilde{\delta}_{12,\mathfrak{b}}(z) -  T_{12,\mathfrak{b}}(z))\right) =  \Ocal_{z} \left(\frac{B^{\epsilon}}{N} + \frac{B^{\epsilon}}{B^{3/2}} + \frac{B^{\epsilon} B^{5/2}}{N^{3}} + \left(\frac{B}{N}\right)^{4} \right),
\end{equation}
where $T_{12,\mathfrak{b}}$ is defined by 
\begin{align}
\label{eq:def-T12-delta11}
T_{12,\mathfrak{b}} &=  \frac{1}{4} \frac{1}{M} \mathrm{Tr} \Biggl[\Q_{\mathfrak{b}} \left((\D_{\x_{\mathfrak{b}}} - \I) \frac{\X_{\mathfrak{b}}\X_{\mathfrak{b}}^*}{B+1} + 
\frac{\X_{\mathfrak{b}}\X_{\mathfrak{b}}^*}{B+1} (\D_{\x_{\mathfrak{b}}} - \I) \right) 
\notag\\
&\qquad\qquad
\Q_{\mathfrak{b}} (\D_{\x_{\mathfrak{b}}} - I) \frac{\X_{\mathfrak{b}}\X_{\mathfrak{b}}^*}{B+1} \Q_{\mathfrak{b}} \Biggr] .
\end{align}
We prove that 
\begin{equation}
 \label{eq:behaviour-T12-circ}
\frac{1}{\pi} \mathrm{Re} \int_{\Dcal} \bar{\partial}  \Phi_k(f)(z) \, \left(T_{12,\mathfrak{b}} (z)\right)^{\circ} \, \diff x \diff y   
=  \Ocal_{\prec}\left( B^{-3/2} \right)  ,
\end{equation}
and therefore that 
\begin{equation}
 \label{eq:behaviour-delta12-circ}
\frac{1}{\pi} \mathrm{Re} \int_{\Dcal} \bar{\partial}  \Phi_k(f)(z) \, \left(\delta_{12,\mathfrak{b}}(z)\right)^{\circ} \, \diff x \diff y   
= \Ocal_{\prec}\left( N^{-1} + B^{-3/2} + \frac{B^{5/2}}{N^{3}} \right)  .
\end{equation}
Moreover, using the integration by parts formula, we obtain that 
\begin{equation}
\label{eq:expre-E-T12b}
\mathbb{E}(T_{12,\mathfrak{b}}(z)) = \frac{1}{4(B+1)} \left( 2 t(z) + 3 z t^{2}(z) + 2 z^{2} t(z) t'(z) + z t'(z) \right) + \Ocal_{z}\left( \frac{1}{B^{3/2}}\right) ,
\end{equation}
which itself implies that 
\begin{align}
\label{eq:comportement-final-E-delta12}
\mathbb{E}(\delta_{12,\mathfrak{b}}(z)) = & \frac{1}{4(B+1)} \left( 2 t(z) + 3 z t^{2}(z) + 2 z^{2} t(z) t'(z) + z t'(z) \right) \\ \notag & 
+  \Ocal_{z} \left(\frac{B^{\epsilon}}{N} + \frac{B^{\epsilon}}{B^{3/2}} + \frac{B^{\epsilon} B^{5/2}}{N^{3}}
+ \left(\frac{B}{N}\right)^{4} \right).
\end{align}
In order to evaluate $\delta_{11,\mathfrak{b}}$, we remark that 
$\tilde{\C}_{\mathfrak{b}}   \tilde{\Q}_{\mathfrak{b}} = \I + z  \tilde{\Q}_{\mathfrak{b}}$, 
and express $\delta_{11,\mathfrak{b}}$ as 
$$
\delta_{11,\mathfrak{b}} =  \frac{1}{2} \frac{1}{M} \Tr \left( (\tilde{\Q}_{\mathfrak{b}} + z \tilde{\Q}_{\mathfrak{b}}^{2})   (\hat{\D}_{\mathfrak{b}} - \D) \D^{-1} \right).
$$
We denote by $ T_{11,\mathfrak{b}}$ the term defined by 
\begin{equation}
    \label{eq:def-T11}
 T_{11,\mathfrak{b}} =  \frac{1}{2} \frac{1}{M} \Tr \left( (\Q_{\mathfrak{b}} + z \Q_{\mathfrak{b}}^{2})   (\hat{\D}_{\mathfrak{b}} - \D) \D^{-1} \right)   ,
\end{equation}
and define $\eta(z)$ by 
$\eta = \delta_{11,\mathfrak{b}} - T_{11,\mathfrak{b}} $. We first show that
\begin{eqnarray}
  \label{eq:evaluation-tilde-delta11}
\frac{1}{\pi} \mathrm{Re} \int_{\Dcal} \bar{\partial}  \Phi_k(f)(z) \,(\eta(z))^{\circ}  \, \diff x \diff y  & = \Ocal_{\prec} \left(\frac{1}{N} + \frac{B^{5/2}}{N^{3}} \right), \\
 \label{eq:evaluation-E-tilde-delta11}
\mathbb{E}(\eta(z)) & = 
 \Ocal_{z}\left( N^{-1} + \frac{B^{5/2}}{N^{3}} +  \left(\frac{B}{N}\right)^{4} \right).
\end{eqnarray}
We are thus back to the evaluation of $T_{11,\mathfrak{b}}$. For this, we use the expression 
(\ref{eq:decomposition-Dx-I}) of $(\hat{\D}_{\mathfrak{b}} - \D) \D^{-1}$, 
and obtain that 
$$
T_{11,\mathfrak{b}} = \sum_{i=1}^{3} T_{11,\mathfrak{b}}^{i},
$$
where $T_{11,\mathfrak{b}}^{i}$ represents the contribution of the $i$-th term of 
(\ref{eq:decomposition-Dx-I}) to $T_{11,\mathfrak{b}}$. We notice in particular that 
$T^{1}_{11,\mathfrak{b}} =  \frac{1}{2}  \frac{1}{M} \Tr \left( (\Q_{\mathfrak{b}} + z \Q_{\mathfrak{b}}^{2})   (\D_{\x_{\mathfrak{b}}} - \I) \right)$. 
We prove that  
\begin{eqnarray}
\label{eq:E-T11-1}
& \mathbb{E}(T_{11,\mathfrak{b}}^{1})  = -\frac{1}{2(B+1)} \, \left(t(z) + z t^{2}(z) + z (t(z) + z t^{2}(z))'\right) + \Ocal_{z} \left(\frac{1}{B^{3/2}}\right), \\
  \label{eq:E-T11-2}
&  \mathbb{E}(T_{11,\mathfrak{b}}^{2})  =  \Ocal_{z} \left(\frac{1}{N}\right), \\
   \label{eq:E-T11-3}
&  \mathbb{E}(T_{11,\mathfrak{b}}^{3})  = \frac{1}{2} \tilde{p}_N(z) \tilde{r}_N(\nu) v_N +  \Ocal_{z} \left(\frac{B^{4}}{N^{4}} + \frac{B^{5/2}}{N^{3}}\right) ,
\end{eqnarray}
as well as 
\begin{equation}
 \label{eq:T11-rond}
 \frac{1}{\pi} \mathrm{Re} \int_{\Dcal} \bar{\partial}  \Phi_k(f)(z) \, \left(T_{11,\mathfrak{b}} (z) - T_{11,\mathfrak{b}}^{1}(z)\right)^{\circ} \, \diff x \diff y  
=  \Ocal_{\prec} \left(\frac{1}{N}\right).
\end{equation}
$\delta_{1,\mathfrak{b}}$ thus verifies 
\begin{align}
    \label{eq:detla1b-rond}
& \frac{1}{\pi} \mathrm{Re} \int_{\Dcal} \bar{\partial}  \Phi_k(f)(z) \, \left(\delta_{1,\mathfrak{b}} (z) -  \frac{1}{2} \frac{1}{M} \Tr \left( (\Q_{\mathfrak{b}} + z \Q_{\mathfrak{b}}^{2})   (\D_{\x_{\mathfrak{b}}} - \I) \right) \right)^{\circ} \, \diff x \diff y  
=  \\
\notag
&  \Ocal_{\prec,z}\left( N^{-1} + B^{-3/2} + \frac{B^{5/2}}{N^{3}} \right)  = o_{\prec}\left( \frac{1}{\sqrt{BN}} \right).
\end{align}
Moreover, the evaluation of $\mathbb{E}(T_{12,\mathfrak{b}})$ leads to 
\begin{align}
    \label{eq:E-delta1b}
    \mathbb{E}(\delta_{1,\mathfrak{b}}) = & -\frac{1}{4(B+1)} \left( z t^{2}(z) + z t'(z) + 2 z^{2} t(z) t'(z)  \right) + \frac{1}{2} \tilde{p}_N(z) \tilde{r}_N(\nu) v_N   + \\
    \notag
    &  \Ocal_{z} \left(\frac{B^{\epsilon}}{N} + \frac{B^{\epsilon}}{B^{3/2}} + \frac{B^{\epsilon} B^{5/2}}{N^{3}} + \left(\frac{B}{N}\right)^{4} \right),
\end{align}
and therefore to 
\begin{align}
    \label{eq:E-delta1b-final}
    \mathbb{E}(\delta_{1,\mathfrak{b}}) =  & -\frac{1}{4(B+1)} \left( z t^{2}(z) + z t'(z) + 2 z^{2} t(z) t'(z)  \right) + \frac{1}{2} \tilde{p}_N(z) \tilde{r}_N(\nu) v_N \\ \notag & 
    + \Ocal_z \left(\frac{B}{N}\right)^{4} + o_{z}\left( \frac{1}{\sqrt{BN}} \right).
\end{align}
The behavior of $\delta_{2,\mathfrak{b}}$ is studied using the same arguments. In particular, 
$\delta_{2,\mathfrak{b}}$ can still be written as $\delta_{2,b} = \delta_{21,\mathfrak{b}} + \delta_{22,\mathfrak{b}}$ 
where $(\delta_{2i,\mathfrak{b}})_{i=1,2}$ are defined in the same way than the  $(\delta_{1i,\mathfrak{b}})_{i=1,2}$. As a result, we obtain that 
\begin{equation}
  \label{eq:behaviour-delta2b-circ}
\frac{1}{\pi} \mathrm{Re} \int_{\Dcal} \bar{\partial}  \Phi_k(f)(z) \,  \left(  \delta_{2,\mathfrak{b}}(z)  - \frac{1}{2} \,  \frac{1}{M} \Tr (\Q_{\mathfrak{b}} + z \Q_{\mathfrak{b}}^{2})(\D_{\x_{\mathfrak{b}}} - \I) \right)^{\circ} = o_{\prec} \left(\frac{1}{\sqrt{NB}}\right),
\end{equation}
while 
\begin{align}
    \label{eq:expre-E-delta2b}
    \mathbb{E}(\delta_{2,\mathfrak{b}}) = & -\frac{1}{4(B+1)} \left( 3 z t^{2}(z) + t(z) - t'(z) +  2 z^{2} t(z) t'(z)  \right) + \frac{1}{2} \tilde{p}_N(z) \tilde{r}_N(\nu) v_N \\ \notag & 
    +  \Ocal_z \left(\frac{B}{N}\right)^{4}  +  o_{z} \left(\frac{1}{\sqrt{NB}}\right).
\end{align}
Gathering Steps 1 to 3, we eventually obtain (\ref{eq:approximation-tr-hatQ-tildeQ}). To complete
the proof of 
(\ref{eq:approximation-E-tr-hatQ-tildeQ}), we obtain after some algebra that
\begin{align}
    \label{eq:conclusion-biais}
\mathbb{E}(\delta_{1,b} + \delta_{2,b} + \delta_{34,\mathfrak{b}} + \delta_{5,\mathfrak{b}}) & = - \frac{1}{B+1} \left( t(z) + z t^{2}(z) + z t'(z) + z^{2} t(z) t'(z) \right)  +  \\
\notag
& \tilde{p}_N(z) \tilde{r}_N(\nu) v_N + \Ocal_{z} \left( \frac{B^{4}}{N^{4}} \right) + o_{z}\left( \frac{1}{\sqrt{BN}}\right) \\
 \notag
 & =  \frac{ z t(z) \tilde{t}(z) (zt(z))'}{B+1}   + \tilde{p}_N(z) \tilde{r}_N(\nu) v_N  + \\ \notag 
 & \Ocal_{z} \left( \frac{B^{4}}{N^{4}} \right) + o_{z}\left( \frac{1}{\sqrt{BN}}\right) \\
 \notag
 & = \frac{1}{B+1} \frac{(z t(z) \tilde{t}(z))^{3}}{1 - c (z t(z) \tilde{t}(z))^{2}}   + \tilde{p}_N(z) \tilde{r}_N(\nu) v_N  + \\ \notag 
 & \Ocal_{z} \left( \frac{B^{4}}{N^{4}} \right)  + o_{z}\left( \frac{1}{\sqrt{BN}}\right) \\
 \notag
 & = -\frac{1}{c_N(B+1)} \, p_N(z)  + \tilde{p}_N(z) \tilde{r}_N(\nu) v_N  + \\ \notag 
 & \Ocal_{z} \left( \frac{B^{4}}{N^{4}} \right)  + o_{z}\left( \frac{1}{\sqrt{BN}}\right),
\end{align}
because the equalities 
$$
1 + z t(z) = - z t(z) \tilde{t}(z),
$$
(see Eq. (\ref{eq:equation-MP-tilde-tN},\ref{eq:equation-MP-tN})) and 
$$
 (zt(z))' = \frac{(z t(z) \tilde{t}(z))^{2}}{1 - c (z t(z) \tilde{t}(z))^{2}},
$$
hold. This proves (\ref{eq:approximation-E-tr-hatQ-tildeQ}), and completes the proof of Theorem \ref{th:simplification-theta1b}.

\subsection{Proof of Step 1}
\label{subsec:proof-step1}
(\ref{eq:domination-hatDb-D-alpha-inferieur-4-5}) and (\ref{eq:E-norm-hatDb-D}) immediately imply that in order to prove (\ref{eq:domination-norme-Upsilon1}) and (\ref{eq:moments-norme-Upsilon1}),
it is sufficient to verify that 
\begin{eqnarray}
\label{eq:domination-norme-hatFb}
\| \hat{\F}_{\mathfrak{b}} \| & = & \Ocal_{\prec}(B^{-3/2}), \\
\label{eq:moments-norme-hatFb}
\mathbb{E}( \|  \hat{\F}_{\mathfrak{b}} \|^{k}) & = & \Ocal(B^{-3k/2+\epsilon}),
\end{eqnarray} 
for each $\epsilon > 0$.  
Using $\hat{s}_{m,\mathfrak{b}} - s_m = \Ocal_{\prec}(B^{-1/2})$ and $\hat{\theta}_m \prec 1$ (because $\hat{\theta}_m$
 lies between $\hat{s}_{m,\mathfrak{b}}$ and $s_m$), we obtain immediately (\ref{eq:domination-norme-hatFb}). 
 In order to prove (\ref{eq:moments-norme-Upsilon1}), 
we first remark that 
$$
\| \hat{\F}_{\mathfrak{b}} \| \leq C \, \| \hat{\D}_{\mathfrak{b}} - \D \|^{3} \| \hat{\Thetabs} \|^{-7/2},
$$
where $ \hat{\Thetabs}$ represents the diagonal matrix $\dg\left(\hat{\theta}_m, m=1, \ldots, M\right)$.
The Schwartz inequality and  (\ref{eq:E-norm-hatDb-D}) imply that (\ref{eq:moments-norme-Upsilon1}) holds provided
we verify that for each integer $k$, $\mathbb{E}(\| \hat{\Thetabs} \|^{-k}) = \Ocal(1)$. For this, we notice that
 $$
 \frac{1}{\hat{\theta}_m} \leq  \frac{1}{s_m}  + \frac{1}{\hat{s}_{m,\mathfrak{b}}}   \leq C + \frac{1}{\hat{s}_{m,\mathfrak{b}}},
 $$
for some nice constant $C$. (\ref{eq:expre-hatsb}) and 
$\| \Phibs_{m,\mathfrak{b}} \| = \Ocal\left( \frac{B}{N} \right)$ imply that it exists a nice constant $C$ such that 
$$
 \hat{s}_{m,\mathfrak{b}} \geq C \, \frac{\| \x_{m,\mathfrak{b}}\|^{2}}{B+1},
$$ 
for each $N$ large enough. Therefore, $ \frac{1}{\hat{\theta}_m}$ verifies 
$$
 \frac{1}{\hat{\theta}_m} \leq C \left( 1 + \left( \frac{\| \x_{m,\mathfrak{b}}\|^{2}}{B+1} \right)^{-1} \right),
$$
 and $\mathbb{E}(\| \hat{\Thetabs} \|^{-k}) = \Ocal(1)$
 is a consequence of the following Lemma proved in the Appendix \ref{proof-lemma-moments-Dxb-inverse}. 
\begin{lemma}
\label{le:moments-Dxb-1}
For each integer $k$, we have 
\begin{equation}
\label{eq:moments-Dxb-1}
\mathbb{E}\left( \| \D_{\x_{\mathfrak{b}}} \|^{-k} \right) \leq C_k,
\end{equation}
for some constant $C_k$ depending on $k$.
\end{lemma}
We have thus justified (\ref{eq:moments-norme-Upsilon1}). In order to establish (\ref{eq:E-hatCb-tildeCb}), we first remark that 
the moments of $\| \frac{\X_{\mathfrak{b}}}{\sqrt{B+1}} \|$ are $\Ocal(1)$ terms. Therefore,
(\ref{eq:E-norm-Gammab-Gamma1r}) leads to the conclusion that 
$\mathbb{E}(\| \tilde{\C}_{\mathfrak{b}} \|^{k}) = \Ocal(1)$.
(\ref{eq:E-hatCb-tildeCb}) thus follows directly from the expansion (\ref{eq:decomposition-hatCb-tildeCb}), (\ref{eq:E-norm-hatDb-D}), (\ref{eq:moments-norme-Upsilon1}), and a relevant use of the Schwartz inequality.

\subsection{Proof of Step 2}
\label{subsec:proof-step2}
We first prove (\ref{eq:equivalent1-delta34}), (\ref{eq:equivalent1-delta5}),
and  (\ref{eq:E-equivalent1-delta34}), (\ref{eq:E-equivalent1-delta5}). Using Eq. (\ref{eq:expre-hatQ-tildeQ}), we 
remark that $\| \hat{\Q}_{\mathfrak{b}} - \tilde{\Q}_{\mathfrak{b}} \| \leq 
C(z) \| \hat{\C}_{\mathfrak{b}} - \tilde{\C}_{\mathfrak{b}} \|$. Using (\ref{eq:domination-hatCb-tildeCb})
and the condition $\alpha < \frac{4}{5}$, we obtain that  
\begin{equation}
\label{eq:error-hatQb-tildeQb}
\| \hat{\Q}_{\mathfrak{b}} - \tilde{\Q}_{\mathfrak{b}} \| = \Ocal_{\prec,z}\left( \frac{1}{\sqrt{B}} \right).
\end{equation}
Similarly, the property $\| \tilde{\C}_{\mathfrak{b}} - \frac{\X_{\mathfrak{b}} \X_{\mathfrak{b}}^{*}}{B+1} \| =
\| \tilde{\Deltabs} \| 
\prec \frac{B}{N}$ implies that 
\begin{equation}
\label{eq:error-tildeQb-Qb}
\| \tilde{\Q}_{\mathfrak{b}} - \Q_{\mathfrak{b}} \| = \Ocal_{\prec,z}\left( \frac{B}{N} \right).
\end{equation}
Using that $\| \hat{\D}_{\mathfrak{b}} - \D\| \prec \frac{1}{\sqrt{B}}$, we obtain easily that 
\begin{eqnarray*}
\delta_{34,\mathfrak{b}} & = & -\frac{3}{4} \frac{1}{M} \Tr \Q_{\mathfrak{b}} \left(  (\hat{\D}_{\mathfrak{b}} - \D)\D^{-1} \right)^{2} \frac{\X_{\mathfrak{b}} \X_{\mathfrak{b}}^{*}}{B+1}  \Q_{\mathfrak{b}} + \Ocal_{\prec,z}\left( \frac{1}{B^{3/2}} + \frac{1}{N}\right) .
\end{eqnarray*}
We now justify  (\ref{eq:E-equivalent1-delta34}). For this, we remark that 
(\ref{eq:E-hatCb-tildeCb}) implies that 
$\mathbb{E}(\| \hat{\Q}_{\mathfrak{b}} - \tilde{\Q}_{\mathfrak{b}} \|^{k}) = \Ocal_z(B^{-k/2 + \epsilon})$ for each $\epsilon > 0$. Moreover, (\ref{eq:E-norm-Gammab-Gamma1r}) leads to the evaluation 
$$
\mathbb{E}\left( \left\| \tilde{\C}_{\mathfrak{b}} - \frac{\X_{\mathfrak{b}} \X_{\mathfrak{b}}^{*}}{B+1} \right \|^{k} \right) = \Ocal\left( \left(\frac{B}{N}\right)^{k} \right).
$$
Therefore, we obtain that  
$$
\mathbb{E}(\| \tilde{\Q}_{\mathfrak{b}} - \Q_{\mathfrak{b}} \|^{k})  =  \Ocal_z\left(    \left(\frac{B}{N}\right)^{k} \right),
$$
as well as
$$
\mathbb{E}(\| \hat{\Q}_{\mathfrak{b}} - \Q_{\mathfrak{b}} \|^{k})  =  \Ocal_z\left( \frac{B^{\epsilon}}{B^{k/2}} +  \left(\frac{B}{N}\right)^{k} \right).
$$
The Schwartz inequality thus implies that for each $\epsilon > 0$ small enough, we have 
$$
\mathbb{E}\left( \frac{1}{M} \Tr (\hat{\Q}_{\mathfrak{b}} - \Q_{\mathfrak{b}})\left((\hat{\D}_{\mathfrak{b}} - \D)\D^{-1}\right)^{2} \tilde{\C}_{\mathfrak{b}} \tilde{\Q}_{\mathfrak{b}} \right) = \Ocal_z\left( 
\frac{B^{\epsilon}}{B^{3/2}} + \frac{B^{\epsilon}}{N} \right) .
$$
Using the identity 
$$
 \tilde{\C}_{\mathfrak{b}} \tilde{\Q}_{\mathfrak{b}} - 
 \frac{\X_{\mathfrak{b}} \X_{\mathfrak{b}}^{*}}{B+1}  \Q_{\mathfrak{b}} =  
 \left( \tilde{\C}_{\mathfrak{b}} -  \frac{\X_{\mathfrak{b}} \X_{\mathfrak{b}}^{*}}{B+1}\right) \tilde{\Q}_{\mathfrak{b}} + \frac{\X_{\mathfrak{b}} \X_{\mathfrak{b}}^{*}}{B+1} \left( \tilde{\Q}_{\mathfrak{b}} - \Q_{\mathfrak{b}} \right),
$$
we obtain similarly that 
$$
\mathbb{E}\left( \frac{1}{M} \Tr \left( \Q_{\mathfrak{b}} \left((\hat{\D}_{\mathfrak{b}} - \D)\D^{-1}\right)^{2} \left( \tilde{\C}_{\mathfrak{b}} \tilde{\Q}_{\mathfrak{b}} - 
 \frac{\X_{\mathfrak{b}} \X_{\mathfrak{b}}^{*}}{B+1} \Q_{\mathfrak{b}} \right) \right)\right) = \Ocal_z\left( 
 \frac{B^{\epsilon}}{N} \right) .
$$
Therefore,  (\ref{eq:E-equivalent1-delta34}) holds. We omit the proofs of 
(\ref{eq:equivalent1-delta5}) and (\ref{eq:E-equivalent1-delta5}) which are very similar. \\

We now establish (\ref{eq:equivalent2-delta34}) and (\ref{eq:equivalent2-E-delta34}). For this, we remark that 
$\mathbb{E}( \Q_{\mathfrak{b}}(z) + z  \Q_{\mathfrak{b}}^{2}(z)) = (\beta(z) + z \beta^{'}(z)) \I = (z \beta(z))^{'} \I = ((z t(z))^{'} + \Ocal_z(B^{-2})) \I $. Therefore, 
\begin{align*}
&\frac{1}{M} \Tr \left( ( \Q_{\mathfrak{b}} + z  \Q_{\mathfrak{b}}^{2})  \left((\hat{\D}_{\mathfrak{b}} - \D)\D^{-1}\right)^{2} \right)  = 
\notag\\
&\qquad\qquad \frac{1}{M} \Tr \left( ( \Q_{\mathfrak{b}} + z  \Q_{\mathfrak{b}}^{2})^{\circ}  \left((\hat{\D}_{\mathfrak{b}} - \D)\D^{-1}\right)^{2} \right) 
\notag\\
&\qquad\qquad\qquad+
  (z t(z))^{'}  \frac{1}{M} \Tr  \left((\hat{\D}_{\mathfrak{b}} - \D)\D^{-1}\right)^{2} +  \Ocal_{\prec,z}(B^{-3}).
\end{align*}
We denote by $\eta(z)$ the term $\eta(z) = \frac{1}{M} \Tr \left( ( \Q_{\mathfrak{b}} + z  \Q_{\mathfrak{b}}^{2})^{\circ}  \left((\hat{\D}_{\mathfrak{b}} - \D)\D^{-1}\right)^{2} \right)$
and express $\eta(z)$ as 
$$
\eta(z) = \frac{1}{M} \sum_{m=1}^{M} (\Q_{\mathfrak{b}} + z  \Q_{\mathfrak{b}}^{2})^{\circ}_{m,m} 
\left(\frac{\hat{s}_{m,{\mathfrak{b}}} - s_m}{s_m}\right)^{2}.
$$
Therefore, the term $\gamma =  \frac{1}{\pi} \mathrm{Re} \int_{\Dcal} \bar{\partial}  \Phi_k(f)(z) \, \eta(z) \diff x \diff y$ can be written as 
$$
\gamma = 
\frac{1}{M} \sum_{m} \left( \frac{1}{\pi} \mathrm{Re} \int_{\Dcal} \bar{\partial}  \Phi_k(f)(z) \, \left( \left(\Q_{\mathfrak{b}}^{\circ} + z (\Q_{\mathfrak{b}}^{2})^{\circ}\right)_{m,m} \right) \diff x \diff y \right) \, \left(\frac{\hat{s}_{m,{\mathfrak{b}}} - s_m}{s_m}\right)^{2}.
$$
We claim that the family $(\omega_m)_{m=1, \ldots, M}$ defined by 
$$
\omega_m = \left( \frac{1}{\pi} \mathrm{Re} \int_{\Dcal} \bar{\partial}  \Phi_k(f)(z) \, \left( \left(\Q_{\mathfrak{b}}^{\circ} + z (\Q_{\mathfrak{b}}^{2})^{\circ}\right)_{m,m} \right) \diff x \diff y \right)_{m=1, \ldots, M},
$$
verifies 
\begin{equation}
    \label{eq:proof-behaviour-omegam}
    \omega_m = \Ocal_{\prec}(B^{-1/2}).
\end{equation}
To check this, 
we apply Lemma \ref{le:concentration-integrale-helffer-sjostrand} in the same way than in the proof of (\ref{eq:behaviour-theta2b-f-recentered}). However, it is not necessary to introduce events $(A_N)$ because it holds 
$\| \nabla (\Q_{\mathfrak{b}} + z \Q_{\mathfrak{b}}^{2})_{m,m} \|^{2} \leq \frac{C(z)}{B}$ where $C(z)= P_1(|z|) P_2\left( \frac{1}{\mathrm{Im}z}\right)$ for some nice polynomials $P_1$ and $P_2$. 
Using Property \ref{pr:properties-stochastic-domination} item (i), we obtain that $\omega_m \left(\frac{\hat{s}_{m,{\mathfrak{b}}} - s_m}{s_m}\right)^{2} =  \Ocal_{\prec}(B^{-3/2})$. Lemma \ref{le:domination-moyenne} eventually implies that 
\begin{align*}
 &\frac{1}{\pi} \mathrm{Re} \int_{\Dcal} \bar{\partial}  \Phi_k(f)(z) \,  \frac{1}{M} \Tr \left( ( \Q_{\mathfrak{b}} + z  \Q_{\mathfrak{b}}^{2})^{\circ}  \left((\hat{\D}_{\mathfrak{b}} - \D)\D^{-1}\right)^{2} \right)  \diff x \diff y 
 \\
 &\qquad= \Ocal_{\prec}(B^{-3/2}).
\end{align*}
As $\mathcal{D}$ is compact, $\mathbb{E}(\omega_m^{2})$ verifies 
$$
\mathbb{E}(\omega_m^{2}) \leq C \, \int_{\Dcal} |\bar{\partial}   \Phi_k(f)(z)|^{2} \, \mathbb{E} \left| \left(\Q_{\mathfrak{b}}^{\circ} + z (\Q_{\mathfrak{b}}^{2})^{\circ} \right)_{m,m} \right|^{2} \diff x \diff y .
$$
It is easy to check that the Nash-Poincaré inequality implies that 
$$
\mathbb{E} \left| \left(\Q_{\mathfrak{b}}^{\circ} + z (\Q_{\mathfrak{b}}^{2} \right)_{m,m}^{\circ} \right|^{2} = 
\Ocal_{z}\left( \frac{1}{B} \right),
$$
so that $\mathbb{E}(\omega_m^{2}) = \Ocal(B^{-1})$. The Schwartz inequality and 
$\mathbb{E}(\hat{s}_{m,{\mathfrak{b}}} - s_m)^{4}=  \Ocal(B^{-2})$ in turn leads to 
$\mathbb{E}(\eta) =  \Ocal(B^{-3/2})$. 

In order to complete the proof of (\ref{eq:equivalent2-delta34}) and (\ref{eq:equivalent2-E-delta34}), it is thus sufficient to
establish that 
\begin{equation}
\label{eq:equivalent-trace-Dhat-D-2}
\left( \frac{1}{M} \Tr  \left((\hat{\D}_{\mathfrak{b}} - \D)\D^{-1}\right)^{2} \right)^{\circ} =
\left(  \frac{1}{M} \Tr (\D_{\x_{\mathfrak{b}}} - \I)^{2} \right)^{\circ}  +  
\Ocal_{\prec}\left( \frac{1}{N} + \frac{B^{5/2}}{N^{3}} \right),
\end{equation}
and 
\begin{equation}
\label{eq:equivalent-trace-E-Dhat-D-2}
\mathbb{E} \left( \frac{1}{M} \Tr  \left((\hat{\D}_{\mathfrak{b}} - \D)\D^{-1}\right)^{2} \right)  = 
\mathbb{E} \left( \frac{1}{M} \Tr (\D_{\x_{\mathfrak{b}}} - \I)^{2}  \right)+ \Ocal\left( \frac{1}{N}  + \frac{B^{4}}{N^{4}} \right).
\end{equation}
For this, we use the decomposition (\ref{eq:decomposition-Dx-I}) of 
$(\hat{\D}_{\mathfrak{b}} - \D)\D^{-1}$, and express  \\ $ \frac{1}{M} \Tr  \left((\hat{\D}_{\mathfrak{b}} - \D)\D^{-1}\right)^{2}$ as 
\begin{align*}
& \frac{1}{M} \Tr (\D_{\x_{\mathfrak{b}}} - \I)^{2} + 2 \frac{1}{M} \Tr (\D_{\x_{\mathfrak{b}}} - \I) \D_{2,\mathfrak{b}} + 2 \frac{1}{M} \Tr (\D_{\x_{\mathfrak{b}}} - \I) \D_{3,\mathfrak{b}} + \\ 
& 2  \frac{1}{M} \Tr \D_{2,\mathfrak{b}} \D_{3,\mathfrak{b}} + \frac{1}{M} \Tr \D_{2,\mathfrak{b}}^{2} + \frac{1}{M} \Tr \D_{3,\mathfrak{b}}^{2}.
\end{align*}
As $\D_{\x_{\mathfrak{b}}} - \I$ and $\D_{2,\mathfrak{b}}$ are zero mean, we obtain that 
\begin{align*}
&\mathbb{E} \left(  \frac{1}{M} \Tr  \left((\hat{\D}_{\mathfrak{b}} - \D)\D^{-1}\right)^{2} \right) - \mathbb{E} \left( \frac{1}{M} \Tr (\D_{\x_{\mathfrak{b}}} - \I)^{2}  \right)
=
\\
&2 \mathbb{E} \left(  \frac{1}{M} \Tr (\D_{\x_{\mathfrak{b}}} - \I) \D_{2,\mathfrak{b}} \right) + 
\mathbb{E} \left(  \frac{1}{M} \Tr \D_{2,\mathfrak{b}}^{2} \right) +  \frac{1}{M} \Tr \D_{3,\mathfrak{b}}^{2}.
\end{align*}
The Schwartz inequality and (\ref{eq:trace-Dx-I-carre},\ref{eq:trace-D2-carre}) imply that $\mathbb{E}\left( \frac{1}{M} \Tr (\D_{\x_{\mathfrak{b}}} - \I) \D_{2,\mathfrak{b}} \right) = \Ocal_{\prec}(N^{-1})$ while $\mathbb{E} \left(  \frac{1}{M} \Tr \D_{2,\mathfrak{b}}^{2} \right) = 
\Ocal\left(\frac{B}{N^{2}}\right) = o\left( \frac{1}{N}\right)$ (see \ref{eq:trace-D2-carre}) and $\frac{1}{M} \Tr \D_{3,\mathfrak{b}}^{2} = 
\Ocal \left(  \frac{B^{4}}{N^{4}} \right)$. Therefore, (\ref{eq:equivalent-trace-E-Dhat-D-2}) is valid. We also have 
\begin{align*}
& \left( \frac{1}{M} \Tr  \left((\hat{\D}_{\mathfrak{b}} - \D)\D^{-1}\right)^{2}  -  \frac{1}{M} \Tr (\D_{\x_{\mathfrak{b}}} - \I)^{2}  \right)^{\circ} = 
\\
&\qquad 2 \left(\frac{1}{M} \Tr (\D_{\x_{\mathfrak{b}}} - \I) \D_{2,\mathfrak{b}}\right)^{\circ}  
\\
&\qquad\qquad + 2 \frac{1}{M} \Tr (\D_{\x_{\mathfrak{b}}} - \I) \D_{3,\mathfrak{b}} 
+ 
 2  \frac{1}{M} \Tr \D_{2,\mathfrak{b}} \D_{3,\mathfrak{b}} + \left( \frac{1}{M} \Tr \D_{2,\mathfrak{b}}^{2} \right)^{\circ}.
\end{align*}
(\ref{eq:domination-D2b}) implies that 
$$
\frac{1}{M} \Tr \D_{2,\mathfrak{b}}^{2} = \Ocal_{\prec}\left( \frac{B}{N^{2}} \right) = o_{\prec}\left( \frac{1}{N} \right),
$$
while  $\mathbb{E}\left( \frac{1}{M} \Tr \D_{2,\mathfrak{b}}^{2} \right)=\Ocal\left( \frac{B}{N^{2}}\right)$. Therefore, we have 
$$
\left( \frac{1}{M} \Tr \D_{2,\mathfrak{b}}^{2} \right)^{\circ}= \Ocal_{\prec}\left( \frac{B}{N^{2}} \right) = o_{\prec}\left( \frac{1}{N} \right).
$$
Moreover, it holds that
$$
 \frac{1}{M} \Tr \D_{2,\mathfrak{b}} \D_{3,\mathfrak{b}} = \Ocal_{\prec}\left( \frac{B^{5/2}}{N^{3}} \right) .
$$
Using Property \ref{pr:properties-stochastic-domination} item (i) and Lemma \ref{le:domination-moyenne}, 
we obtain immediately that $$\frac{1}{M} \Tr (\D_{\x_{\mathfrak{b}}} - \I) \D_{2,\mathfrak{b}} = 
\Ocal_{\prec}(N^{-1}).$$ As $\mathbb{E}\left( \frac{1}{M} \Tr (\D_{\x_{\mathfrak{b}}} - \I) \D_{2,\mathfrak{b}} \right) = \Ocal(N^{-1})$, we obtain that 
$$\left( \frac{1}{M} \Tr (\D_{\x_{\mathfrak{b}}} - \I) \D_{2,\mathfrak{b}} \right)^{\circ} =  \Ocal_{\prec}(N^{-1}).$$
Finally, the Hanson-Wright inequality leads to 
$$
\frac{1}{M} \Tr (\D_{\x_{\mathfrak{b}}} - \I) \D_{3,\mathfrak{b}} = \Ocal_{\prec}\left( \frac{B}{N^{2}} \right) .
$$ 
Therefore, we obtain that
$$
\left( \frac{1}{M} \Tr  \left((\hat{\D}_{\mathfrak{b}} - \D)\D^{-1}\right)^{2} - \frac{1}{M} \Tr (\D_{\x_{\mathfrak{b}}} - \I)^{2} \right)^{\circ} = \Ocal_{\prec}\left( \frac{1}{N} +  \frac{B^{5/2}}{N^3} \right) .
$$
as expected. 
 
We now establish (\ref{eq:equivalent2-delta5}) and (\ref{eq:equivalent2-E-delta5}). For this, we prove that 
\begin{align}
\label{eq:inter-domination-stochastique-delta5b-1}
 & \frac{1}{\pi} \mathrm{Re} \int_{\Dcal} \bar{\partial}  \Phi_k(f)(z) \, \left( \frac{1}{M} \Tr \Q_{\mathfrak{b}}^{2} (\hat{\D} - \D)\D^{-1}  \frac{\X_{\mathfrak{b}} \X_{\mathfrak{b}}^{*}}{B+1}  (\hat{\D} - \D)\D^{-1} \right)^{\circ} \diff x \diff y = 
 \notag\\ 
 & \qquad\frac{1}{\pi} \mathrm{Re} \int_{\Dcal} \bar{\partial}  \Phi_k(f)(z) \, \left( \frac{1}{M} \Tr \Q_{\mathfrak{b}}^{2} (\D_{\x_{\mathfrak{b}}} - \I)  \frac{\X_{\mathfrak{b}} \X_{\mathfrak{b}}^{*}}{B+1}   (\D_{\x_{\mathfrak{b}}} - \I) \right)^{\circ} \diff x \diff y  
 \notag\\
 &\qquad\qquad+  \Ocal_{\prec}\left( \frac{B^{5/2}}{N^{3}} + \frac{1}{N} \right),
 \end{align}
 and 
 \begin{align}
\label{eq:inter-E-delta5b-1}
&\mathbb{E} \left(  \frac{1}{M} \Tr \Q_{\mathfrak{b}}^{2} (\hat{\D} - \D)\D^{-1}  \frac{\X_{\mathfrak{b}} \X_{\mathfrak{b}}^{*}}{B+1}  (\hat{\D} - \D)\D^{-1} \right) =  \\ \notag
 & \mathbb{E} \left( \frac{1}{M} \Tr \Q_{\mathfrak{b}}^{2} (\D_{\x_{\mathfrak{b}}} - \I)  \frac{\X_{\mathfrak{b}} \X_{\mathfrak{b}}^{*}}{B+1}   (\D_{\x_{\mathfrak{b}}} - \I) \right) + \Ocal_{z}\left( \frac{B^{4}}{N^{4}} + \frac{B^{\epsilon}B^{5/2}}{N^{3}} + \frac{B^{\epsilon}}{N} \right),
\end{align}
for each $\epsilon > 0$. We expand $ (\hat{\D}_{\mathfrak{b}} - \D)\D^{-1} $ using (\ref{eq:decomposition-Dx-I}), study  
$$
 \frac{1}{M} \Tr \Q_{\mathfrak{b}}^{2} (\hat{\D}_{\mathfrak{b}} - \D)\D^{-1}  \frac{\X_{\mathfrak{b}} \X_{\mathfrak{b}}^{*}}{B+1} \D_{3,b},
$$
and evaluate its contribution to the left-hand side of (\ref{eq:inter-domination-stochastique-delta5b-1}) and of (\ref{eq:inter-E-delta5b-1}). We write that 
\begin{align*}
 &\frac{1}{M} \Tr \Q_{\mathfrak{b}}^{2} (\hat{\D} - \D)\D^{-1}  \frac{\X_{\mathfrak{b}} \X_{\mathfrak{b}}^{*}}{B+1} \D_{3,b}  = 
  \notag\\
  &\qquad
 \frac{1}{M} \Tr \Q_{\mathfrak{b}}^{2} (\D_{\x_{\mathfrak{b}}} - \I)  \frac{\X_{\mathfrak{b}} \X_{\mathfrak{b}}^{*}}{B+1} \D_{3,b} +
      \frac{1}{M} \Tr \Q_{\mathfrak{b}}^{2} \D_{2,\mathfrak{b}} \frac{\X_{\mathfrak{b}} \X_{\mathfrak{b}}^{*}}{B+1} \D_{3,b} 
      \\
      &\qquad\qquad+ \frac{1}{M} \Tr \Q_{\mathfrak{b}}^{2} \D_{3,\mathfrak{b}} \frac{\X_{\mathfrak{b}} \X_{\mathfrak{b}}^{*}}{B+1} \D_{3,b}.
\end{align*}
We first prove that 
\begin{align}
\label{eq:stochastic-domination-Qb2-Dx-I-XX-D3b}
\frac{1}{\pi} \mathrm{Re} \int_{\Dcal} \bar{\partial}  \Phi_k(f)(z) \, \left( \frac{1}{M} \Tr \Q_{\mathfrak{b}}^{2} (\D_{\x_{\mathfrak{b}}} - \I)  \frac{\X_{\mathfrak{b}} \X_{\mathfrak{b}}^{*}}{B+1}   \D_{3,\mathfrak{b}} \right)^{\circ} \diff x \diff y  &= \Ocal_{\prec}\left(\frac{B}{N^{2}}\right) 
\notag\\
&= o_{\prec}\left(  \frac{1}{N} \right) .
\end{align}
We again apply Lemma \ref{le:concentration-integrale-helffer-sjostrand} in the same way than in the proof of (\ref{eq:behaviour-theta2b-f-recentered}), and consider the family of events $A_N(\nu)$ defined by 
$A_N(\nu) = \{ \left\| \frac{\X_N(\nu)}{\sqrt{B+1}} \right\| \leq 3 \}$. It is clear that $\X_N(\nu)(A_N(\nu))$ is a convex subset and that $\sup_{\nu \in [0,1]} P\left(A_N(\nu)^{c}\right) \leq e^{-N^{\gamma}}$ for some constant $\gamma > 0$. A simple calculation leads to 
$$
\left\| \nabla \frac{1}{M} \Tr \Q_{\mathfrak{b}}^{2} (\D_{\x_{\mathfrak{b}}} - \I)  \frac{\X_{\mathfrak{b}} \X_{\mathfrak{b}}^{*}}{B+1} \D_{3,b} \right\|^{2} \leq C(z) \frac{B^{2}}{N^{4}},
$$
on the event $A_N$. Lemma \ref{le:concentration-integrale-helffer-sjostrand}
thus implies (\ref{eq:stochastic-domination-Qb2-Dx-I-XX-D3b}). In order to evaluate 
\\ $\left(  \frac{1}{M} \Tr \Q_{\mathfrak{b}}^{2} \D_{2,\mathfrak{b}} \frac{\X_{\mathfrak{b}} \X_{\mathfrak{b}}^{*}}{B+1} \D_{3,b} \right)^{\circ}$, we remark that 
$$
 \frac{1}{M} \Tr \Q_{\mathfrak{b}}^{2} \D_{2,\mathfrak{b}} \frac{\X_{\mathfrak{b}} \X_{\mathfrak{b}}^{*}}{B+1} \D_{3,b} = \Ocal_{\prec,z}\left( \frac{B^{5/2}}{N^{3}} \right),
$$
because $\| \D_{2,\mathfrak{b}} \| \prec \frac{\sqrt{B}}{N}$ and $\| \D_{3,\mathfrak{b}} \| = \Ocal\left( \frac{B^{2}}{N^{2}}\right)$. Moreover, using 
(\ref{eq:E-norm-D2b}) and (\ref{eq:norm-D3b}), we also obtain that  
\begin{equation}
\label{eq:E-Q2-D2b-XX-D3b}
\mathbb{E} \left( \frac{1}{M} \Tr \Q_{\mathfrak{b}}^{2} \D_{2,\mathfrak{b}} \frac{\X_{\mathfrak{b}} \X_{\mathfrak{b}}^{*}}{B+1} \D_{3,b} \right) =  \Ocal_{\prec,z}\left( \frac{B^{\epsilon}B^{5/2}}{N^{3}} \right),
\end{equation}
for each $\epsilon > 0$. Therefore, $\left(  \frac{1}{M} \Tr \Q_{\mathfrak{b}}^{2} \D_{2,\mathfrak{b}} \frac{\X_{\mathfrak{b}} \X_{\mathfrak{b}}^{*}}{B+1} \D_{3,b} \right)^{\circ} = \Ocal_{\prec,z}\left( \frac{B^{\epsilon} B^{5/2}}{N^{3}} \right)$ for each $\epsilon > 0$, or equivalently, 
$$
\left(  \frac{1}{M} \Tr \Q_{\mathfrak{b}}^{2} \D_{2,\mathfrak{b}} \frac{\X_{\mathfrak{b}} \X_{\mathfrak{b}}^{*}}{B+1} \D_{3,b} \right)^{\circ} = \Ocal_{\prec,z}\left( \frac{B^{5/2}}{N^{3}} \right).
$$
This implies that 
\begin{align*}
\frac{1}{\pi} \mathrm{Re} \int_{\Dcal} \bar{\partial}  \Phi_k(f)(z) \, \left( \frac{1}{M} \Tr \Q_{\mathfrak{b}}^{2}  \D_{2,\mathfrak{b}} \frac{\X_{\mathfrak{b}} \X_{\mathfrak{b}}^{*}}{B+1}   \D_{3,\mathfrak{b}} \right)^{\circ} \diff x \diff y  = \Ocal_{\prec}\left(\frac{B^{5/2}}{N^{3}} \right).
\end{align*}
The use of Lemma \ref{le:concentration-integrale-helffer-sjostrand} with the family of events $A_N(\nu)$ defined by 
$A_N(\nu) = \{ \left\| \frac{\X_N(\nu)}{\sqrt{B+1}} \right\| \leq 3 \}$ leads immediately to 
\begin{align*}
\frac{1}{\pi} \mathrm{Re} \int_{\Dcal} \bar{\partial}  \Phi_k(f)(z) \, \left( \frac{1}{M} \Tr \Q_{\mathfrak{b}}^{2}  \D_{3,\mathfrak{b}} \frac{\X_{\mathfrak{b}} \X_{\mathfrak{b}}^{*}}{B+1}   \D_{3,\mathfrak{b}} \right)^{\circ} \diff x \diff y  &= \Ocal_{\prec}\left(\frac{B^{5/2}}{N^{4}} \right) 
\\
&=  o_{\prec}\left( \frac{1}{N}\right).
\end{align*}
We have thus proved that 
\begin{align*}
&\frac{1}{\pi} \mathrm{Re} \int_{\Dcal} \bar{\partial}  \Phi_k(f)(z) \, \left( \frac{1}{M} \Tr \Q_{\mathfrak{b}}^{2}  (\hat{\D}_{\mathfrak{b}} - \D) \D^{-1} \frac{\X_{\mathfrak{b}} \X_{\mathfrak{b}}^{*}}{B+1}   \D_{3,\mathfrak{b}} \right)^{\circ} \diff x \diff y  
\\
&\qquad\qquad=  \Ocal_{\prec}\left(  \frac{1}{N} + \frac{B^{5/2}}{N^{3}} \right).
\end{align*}
We now evaluate $\mathbb{E} \left( \frac{1}{M} \Tr \Q_{\mathfrak{b}}^{2} (\hat{\D} - \D)\D^{-1}  \frac{\X_{\mathfrak{b}} \X_{\mathfrak{b}}^{*}}{B+1} \D_{3,b}  \right)$
and first prove that 
\begin{equation}
    \label{eq:E-Q2-Dx-I-XX*-D3b}
 \mathbb{E} \left( \frac{1}{M} \Tr \Q_{\mathfrak{b}}^{2} (\D_{\x_{\mathfrak{b}}} - \I)  \frac{\X_{\mathfrak{b}} \X_{\mathfrak{b}}^{*}}{B+1}   \D_{3,\mathfrak{b}}  \right) = \Ocal_z\left( \frac{B}{N^{2}} \right) = o_{z}\left(  \frac{1}{N} \right) .
\end{equation}
We notice that it is straightforward that the left-hand side of (\ref{eq:E-Q2-Dx-I-XX*-D3b}) is of the order
$\Ocal_z\left( \frac{B^{3/2+\epsilon}}{N^{2}} \right)$ term for each $\epsilon > 0$. However, $\frac{B^{3/2}}{N^{2}} = o\left(  \frac{1}{\sqrt{NB}} \right)$ only holds if $\alpha < \frac{3}{4}$, a condition which is not supposed to be verified in the context of the present work. It is thus necessary to improve this rough evaluation of the left-hand side of (\ref{eq:E-Q2-Dx-I-XX*-D3b}). In order to establish (\ref{eq:E-Q2-Dx-I-XX*-D3b}), we remark that 
$\Q_{\mathfrak{b}}^{2}(z) = \Q_{\mathfrak{b}}^{'}(z)$ (where $'$ stands for the differential w.r.t. $z$),  evaluate the order of magnitude of 
$\mathbb{E} \left( \frac{1}{M} \Tr \Q_{\mathfrak{b}} (\D_{\x_{\mathfrak{b}}} - \I)  \frac{\X_{\mathfrak{b}} \X_{\mathfrak{b}}^{*}}{B+1}   \D_{3,\mathfrak{b}}  \right)$, and briefly verify that differentiating w.r.t. $z$ keeps unchanged its order of magnitude. We first use the integration by parts formula to compute 
$\eta_m(z) = \mathbb{E}\left( \left( \Q_{\mathfrak{b}} (\D_{\x_{\mathfrak{b}}} - \I)  \frac{\X_{\mathfrak{b}} \X_{\mathfrak{b}}^{*}}{B+1} \right)_{m,m} \right)$. $\eta_m$ is given by 
$$
\eta_m = \sum_{m',n} \mathbb{E} \left( (\Q_{\mathfrak{b}})_{m,m'} \left( \frac{\|\x_{m',\mathfrak{b}}\|^{2}}{B+1} - 1 \right) \frac{(\X_{\mathfrak{b}})_{m',n} (\bar{\X}_{\mathfrak{b}})_{m,n}}{B+1} \right).
$$
The integration by parts formula leads to 
\begin{align*}
& \mathbb{E}\left( (\Q_{\mathfrak{b}})_{m,m'} \left( \frac{\|\x_{m',\mathfrak{b}}\|^{2}}{B+1} - 1 \right) \frac{(\X_{\mathfrak{b}})_{m',n} (\bar{\X}_{\mathfrak{b}})_{m,n}}{B+1} \right) = \\ 
& \frac{1}{B+1} \mathbb{E} \left( \frac{\partial}{\partial (\bar{\X}_{\mathfrak{b}})_{m',n}} 
\left( (\Q_{\mathfrak{b}})_{m,m'} \left( \frac{\|\x_{m',\mathfrak{b}}\|^{2}}{B+1} - 1 \right)  (\bar{\X}_{\mathfrak{b}})_{m,n} \right) \right).
\end{align*}
After some easy calculations, we obtain that 
\begin{align*}
\eta_m = & \mathbb{E}\left( (\Q_{\mathfrak{b}}(\D_{\x_{\mathfrak{b}}} - \I))_{m,m} \right) - c \, \mathbb{E}\left( \left( \Q_{\mathfrak{b}} \frac{\X_{\mathfrak{b}} \X_{\mathfrak{b}}^{*}}{B+1} \right)_{m,m}  \frac{1}{M} \Tr \Q_{\mathfrak{b}} (\D_{\x_{\mathfrak{b}}} - \I) \right) \\ 
& + \frac{1}{B+1} \mathbb{E}\left( \left( \Q_{\mathfrak{b}} \frac{\X_{\mathfrak{b}} \X_{\mathfrak{b}}^{*}}{B+1} \right)_{m,m} \right) \\
= & \mathbb{E}\left( (\Q_{\mathfrak{b}}(\D_{\x_{\mathfrak{b}}} - \I))_{m,m} \right) - c \, \mathbb{E}\left( \left(1 + z (\Q_{\mathfrak{b}})_{m,m}\right)  \frac{1}{M} \Tr \Q_{\mathfrak{b}} (\D_{\x_{\mathfrak{b}}} - \I) \right) \\ 
& + \frac{1}{B+1} \mathbb{E}\left( 1 + z  (\Q_{\mathfrak{b}})_{m,m} \right) .
\end{align*}
We recall that $\mathbb{E}\left((\Q_{\mathfrak{b}})_{m,m}\right)) = \beta(z)$, and 
remark that $\mathbb{E}\left( (\Q_{\mathfrak{b}}(\D_{\x_{\mathfrak{b}}} - \I))_{m,m} \right)$ does not depend on $m$ because the probability distribution of 
$\X_{\mathfrak{b}}$ is invariant by permutation of its rows. Consequently, we have
$$
\mathbb{E}\left( (\Q_{\mathfrak{b}}(\D_{\x_{\mathfrak{b}}} - \I))_{m,m} \right) = 
\mathbb{E} \left( \frac{1}{M} \mathrm{Tr}  (\Q_{\mathfrak{b}}(\D_{\x_{\mathfrak{b}}} - \I))\right).
$$
Therefore, writing that 
\begin{align*}
 & \mathbb{E}\left( \left(1 + z (\Q_{\mathfrak{b}})_{m,m}\right)  \frac{1}{M} \Tr \Q_{\mathfrak{b}} (\D_{\x_{\mathfrak{b}}} - \I) \right) =  \\ 
 & \mathbb{E}\left(1 + z (\Q_{\mathfrak{b}})_{m,m}\right) \mathbb{E}\left( \frac{1}{M} \Tr \Q_{\mathfrak{b}} (\D_{\x_{\mathfrak{b}}} - \I) \right) + \delta_m(z),
\end{align*}
where $\delta_m(z)$ is given by 
$$
\delta_m(z) =  z \mathbb{E}\left(  (\Q^{\circ}_{\mathfrak{b}})_{m,m})  \left(\frac{1}{M} \Tr \Q_{\mathfrak{b}} (\D_{\x_{\mathfrak{b}}} - \I)\right)^{\circ} \right) ,
$$
we get that 
$$
\eta_m = \mathrm{E}\left( \frac{1}{M}  \Tr  \Q_{\mathfrak{b}} (\D_{\x_{\mathfrak{b}}} - \I) \right) \, 
( 1 - c -c \, z \beta(z)) + \frac{1}{B+1} (1 + z \beta(z)) - c \delta_m(z)  .
$$
Using the integration by parts formula, it is easy to check that 
\begin{equation}
    \label{eq:expre-E-Q-Dx-I}
 \mathbb{E}\left( \frac{1}{M} \Tr \Q_{\mathfrak{b}} (\D_{\x_{\mathfrak{b}}} - \I) \right) = - \frac{\beta(z)(1 + z \beta(z))}{B+1} + \tilde{\delta}(z),
\end{equation}
where $\tilde{\delta}(z) = -\frac{1}{B+1} z \, \frac{1}{M} \sum_{m=1}^{M} \mathbb{E}\left(  (\Q^{\circ}_{\mathfrak{b}})_{m,m})^{2} \right) = \Ocal_z(B^{-2})$.
Therefore, we obtain that $\eta_m$ can be written as 
\begin{align*}
&\eta_m(z) = 
\\
&\qquad
\left( -\frac{1}{B+1} (\beta(z)(1 + z \beta(z))) + \tilde{\delta}(z) \right) ( 1 - c -c \, z \beta(z)) 
\\
&\qquad\qquad
+ \frac{1}{B+1} (1 + z \beta(z)) - c \delta_m(z) ,
\end{align*}
and that 
\begin{align}
\label{eq:before-differentiation}
& \mathbb{E} \left( \frac{1}{M} \Tr \Q_{\mathfrak{b}}  (\D_{\x_{\mathfrak{b}}} - \I)  \frac{\X_{\mathfrak{b}} \X_{\mathfrak{b}}^{*}}{B+1}   \D_{3,\mathfrak{b}}  \right) = 
\notag\\ 
& \Biggl(\Bigl( -\frac{1}{B+1} (\beta(z)(1 + z \beta(z))) + \tilde{\delta}(z) \Bigr) ( 1 - c -c \, z \beta(z))
\notag\\
&\qquad  + \frac{1}{B+1} (1 + z \beta(z)) \Biggr) \frac{1}{M} \Tr  \D_{3,\mathfrak{b}} 
\notag\\
&\qquad\qquad - c z 
\mathbb{E} \left( \frac{1}{M} \Tr \Q_{\mathfrak{b}}^{\circ} \D_{3,\mathfrak{b}}  \frac{1}{M} \Tr \left( \Q_{\mathfrak{b}}(\D_{\x_{\mathfrak{b}}} - \I)\right)^{\circ} \right).
\end{align}
The Nash-Poincaré inequality implies that 
$$
\mathbb{E} \left| \frac{1}{M} \Tr \Q_{\mathfrak{b}}^{\circ} \D_{3,\mathfrak{b}} \right|^{2}= \Ocal_z\left( \frac{B^{2}}{N^{4}}\right) ,
$$
and 
$$
\mathbb{E} \left|  \frac{1}{M} \Tr \left( \Q_{\mathfrak{b}}(\D_{\x_{\mathfrak{b}}} - \I)\right)^{\circ} \right|^{2} = \Ocal_z(B^{-2}),
$$
which, using the Schwartz inequality as well as $\frac{1}{M} \Tr \D_{3,\mathfrak{b}} = \Ocal\left( \frac{B^{2}}{N^{2}}\right)$, leads to the evaluation 
\begin{equation}
    \label{eq:E-Q-Dx-I-XX*-D3b}
 \mathbb{E} \left( \frac{1}{M} \Tr \Q_{\mathfrak{b}} (\D_{\x_{\mathfrak{b}}} - \I)  \frac{\X_{\mathfrak{b}} \X_{\mathfrak{b}}^{*}}{B+1}   \D_{3,\mathfrak{b}}  \right) = \Ocal_z\left( \frac{B}{N^{2}} \right) = o\left( \frac{1}{N}\right).
\end{equation}
Differentiating (\ref{eq:before-differentiation}) w.r.t. $z$, using again the Nash-Poincaré inequality, eventually allows to justify after some extra calculations that (\ref{eq:E-Q2-Dx-I-XX*-D3b}) holds. As we already mentioned 
that (\ref{eq:E-Q2-D2b-XX-D3b}) holds, it remains to evaluate the quantity $ \mathbb{E} \left( \frac{1}{M} \Tr \Q_{\mathfrak{b}}   \D_{3,\mathfrak{b}}  \frac{\X_{\mathfrak{b}} \X_{\mathfrak{b}}^{*}}{B+1}   \D_{3,\mathfrak{b}}  \right)$. Using that $\| \D_{3,\mathfrak{b}} \| = \Ocal\left( \frac{B^{2}}{N^{2}}\right)$, 
we obtain immediately that 
$$
 \mathbb{E} \left( \frac{1}{M} \Tr \Q_{\mathfrak{b}}   \D_{3,\mathfrak{b}}  \frac{\X_{\mathfrak{b}} \X_{\mathfrak{b}}^{*}}{B+1}   \D_{3,\mathfrak{b}}  \right) = \Ocal_{z}\left( \frac{B^{4}}{N^{4}}\right) .
$$
We have thus verified that 
$$
\mathbb{E} \left( \frac{1}{M} \Tr \Q_{\mathfrak{b}}^{2} (\hat{\D} - \D)\D^{-1}  \frac{\X_{\mathfrak{b}} \X_{\mathfrak{b}}^{*}}{B+1} \D_{3,\mathfrak{b}}\right) =  \Ocal_{z}\left( \frac{B^{4}}{N^{4}} + \frac{B^{\epsilon}B^{5/2}}{N^{3}} + \frac{1}{N} \right).
$$

Using the evaluations (\ref{eq:domination-hatDb-D-alpha-inferieur-4-5}), 
(\ref{eq:E-norm-hatDb-D}), (\ref{eq:domination-D2b}), (\ref{eq:E-norm-D2b}),  
it is easy to check that 
\begin{align*}
& \frac{1}{\pi} \mathrm{Re} \int_{\Dcal} \bar{\partial}  \Phi_k(f)(z) \, \left( \frac{1}{M} \Tr \Q_{\mathfrak{b}}^{2} (\hat{\D} - \D)\D^{-1}  \frac{\X_{\mathfrak{b}} \X_{\mathfrak{b}}^{*}}{B+1} \D_{2,\mathfrak{b}} \right)^{\circ} \diff x \diff y  = \Ocal_{\prec}\left(   \frac{1}{N} \right), \\
 & \mathbb{E}\left(  \frac{1}{M} \Tr \Q_{\mathfrak{b}}^{2} (\hat{\D} - \D)\D^{-1}  \frac{\X_{\mathfrak{b}} \X_{\mathfrak{b}}^{*}}{B+1} \D_{2,\mathfrak{b}} \right) =  \Ocal_{z}\left(  \frac{B^{\epsilon}}{N} \right).
\end{align*}
Moreover, using similar arguments, we obtain that the following evaluations hold: 
\begin{align}
 & \frac{1}{\pi} \mathrm{Re} \int_{\Dcal} \bar{\partial}  \Phi_k(f)(z) \, \left( \frac{1}{M} \Tr \Q_{\mathfrak{b}}^{2} (\hat{\D} - \D)\D^{-1}  \frac{\X_{\mathfrak{b}} \X_{\mathfrak{b}}^{*}}{B+1}  (\D_{\x_{\mathfrak{b}}} - \I) \right)^{\circ} \diff x \diff y =
 \notag \\
 & \frac{1}{\pi} \mathrm{Re} \int_{\Dcal} \bar{\partial}  \Phi_k(f)(z) \, \left( \frac{1}{M} \Tr \Q_{\mathfrak{b}}^{2} (\D_{\x_{\mathfrak{b}}} - \I)  \frac{\X_{\mathfrak{b}} \X_{\mathfrak{b}}^{*}}{B+1}   (\D_{\x_{\mathfrak{b}}} - \I) \right)^{\circ} \diff x \diff y  
 \notag\\
 &+ \Ocal_{\prec}\left( \frac{B^{5/2}}{N^{3}} + \frac{1}{N} \right),
 \label{eq:inter-domination-stochastique-delta5b-2}
 \end{align}
 and 
 \begin{align}
\label{eq:inter-E-delta5b-2}
&\mathbb{E} \left(  \frac{1}{M} \Tr \Q_{\mathfrak{b}}^{2} (\hat{\D} - \D)\D^{-1}  \frac{\X_{\mathfrak{b}} \X_{\mathfrak{b}}^{*}}{B+1}  (\D_{\x_{\mathfrak{b}}} - \I) \right) =  
\notag\\
& \mathbb{E} \left( \frac{1}{M} \Tr \Q_{\mathfrak{b}}^{2} (\D_{\x_{\mathfrak{b}}} - \I)  \frac{\X_{\mathfrak{b}} \X_{\mathfrak{b}}^{*}}{B+1}   (\D_{\x_{\mathfrak{b}}} - \I) \right)+ 
\Ocal_{z}\left( \frac{B^{4}}{N^{4}} +
   \frac{B^{\epsilon}B^{5/2}}{N^{3}} + \frac{B^{\epsilon}}{N} \right).
\end{align}
Therefore, the proof of (\ref{eq:equivalent2-delta5}) and (\ref{eq:equivalent2-E-delta5}) is complete. \\

We establish (\ref{eq:stochastic-domination-tr-D-I-square}). 
For this, we use a trick introduced in the proof of
Lemma 7 in \cite{loubaton-rosuel-ejs-2021}. For $\epsilon > 0$, we remark that the set $\tilde{A}_{N,\epsilon}(\nu)$ defined by 
$$
\tilde{A}_{N,\epsilon}(\nu) = \bigcap_{m=1}^{M} \left\{ \frac{\|\x_{m,\mathfrak{b}}(\nu)\|^{2}}{B+1} \in [1 - \frac{B^{\epsilon}}{\sqrt{B}}, 1 + \frac{B^{\epsilon}}{\sqrt{B}}] \right \},
$$
verifies $\sup_{\nu} P(\tilde{A}_{N,\epsilon}(\nu)^{c}) \leq e^{-N^{\gamma}}$ where $\gamma$ is a constant depending only on $\epsilon$. Moreover, considered as a function of the entries 
of $\X_{N,\mathfrak{b}}(\nu)$, the function 
$\frac{1}{M} \Tr\left( \D_{\x_{\mathfrak{b}}} - \I \right)^{2}$ 
is Lipschitz on $\X_{N,\mathfrak{b}}(\nu)(\tilde{A}_{N,\epsilon}(\nu))$
with Lipschitz constant of the order $\Ocal\left(\frac{B^{\epsilon}}{B^{3/2}}\right)$. 
Unfortunately, $\X_{N,\mathfrak{b}}(\nu)(\tilde{A}_{N,\epsilon}(\nu))$ is not a convex set. Therefore,  Lemma 
\ref{le:conditional-concentration} does not imply that 
(\ref{eq:stochastic-domination-tr-D-I-square}) holds. However, it is possible to replace for each $m$ $\frac{\|\x_{m,\mathfrak{b}}\|^{2}}{B+1} - 1$ by a function of $\x_{m,\mathfrak{b}}$, Lipschitz on $\mathbb{C}^{N}$ with constant 
$\Ocal\left(\frac{B^{\epsilon}}{\sqrt{B}}\right)$. This function is given by  $g_{B,\epsilon}\left(\frac{\|\x_{m,\mathfrak{b}}\|^{2}}{B+1}\right)$
where $g_{B,\epsilon}(t)$ is a smooth function verifying 
\begin{align*}
    g_{B,\epsilon}(t) & =  t -1 \; \mathrm{if} \; t \in  [1 - \frac{B^{\epsilon}}{\sqrt{B}}, 1 + \frac{B^{\epsilon}}{\sqrt{B}}] \\
      & = 0 \; \mathrm{if} \; t \notin  [1 - 2 \frac{B^{\epsilon}}{\sqrt{B}}, 1 + 2 \frac{B^{\epsilon}}{\sqrt{B}}],
\end{align*}
and
$$
\sup_{t} |g_{B,\epsilon}(t)| \leq C \frac{B^{\epsilon}}{\sqrt{B}}, \; \sup_{t} |g'_{B,\epsilon}(t) ,\leq C 
$$
where $C$ is a nice constant. We refer to \cite{loubaton-rosuel-ejs-2021} for more details concerning the existence of such a function. We remark that if $\D_{\epsilon,g}$ represents 
the diagonal matrix with diagonal entries
\begin{equation}
\label{eq:def-Depsilong}
\D_{\epsilon,g} = \left(g_{B,\epsilon}\left(\frac{\|\x_{m,\mathfrak{b}}\|^{2}}{B+1}\right)\right)_{m=1, \ldots, M} ,
\end{equation}
then, 
$\D_{\x_{\mathfrak{b}}} - \I = \D_{\epsilon,g}$ on $\tilde{A}_{N,\epsilon}$. Therefore, 
adapting the arguments in \cite{loubaton-rosuel-ejs-2021}, it is easy to check that we have the following implication:
\begin{equation}
    \label{eq:equivalence-Dx-Dg}
\left( \frac{1}{M} \Tr \left( \D_{\epsilon,g} \right)^{2}  \right)^{\circ} \prec \frac{B^{\epsilon}}{B^{3/2}}  \Longrightarrow  \left( \frac{1}{M} \Tr \left( \D_{\x_{\mathfrak{b}}} - \I \right)^{2}  \right)^{\circ} \prec \frac{B^{\epsilon}}{B^{3/2}} .
\end{equation}
The standard Gaussian concentration inequality (\ref{eq:standard-gaussian-concentration}) implies immediately that the left-hand side of 
(\ref{eq:equivalence-Dx-Dg}) holds for each $\epsilon > 0$, which, in turn, leads to (\ref{eq:stochastic-domination-tr-D-I-square})
(we recall Property \ref{pr:properties-stochastic-domination}, item (ii)). \\

We now establish (\ref{eq:stochastic-domination-tr-Q2-D-I-XX-D-I}). For this, we combine the above trick and Lemma \ref{le:concentration-integrale-helffer-sjostrand}. We denote by 
$\eta_N(z)$ and $\gamma_N(f)$ the terms
\begin{align*}
\eta_N(z) = & \frac{1}{M} \Tr \Q_{\mathfrak{b}}^{2} (\D_{\x_{\mathfrak{b}}} - \I)  \frac{\X_{\mathfrak{b}} \X_{\mathfrak{b}}^{*}}{B+1}   (\D_{\x_{\mathfrak{b}}} - \I),  \\
\gamma_N(f) = & \frac{1}{\pi} \mathrm{Re} \int_{\Dcal} \bar{\partial}  \Phi_k(f)(z) \, 
\eta_N(z) \diff x \diff y,
\end{align*}
and by $\eta_{N,g,\epsilon}(z)$ and $\gamma_{N,g,\epsilon}(f)$ their Lipschitz approximations 
\begin{align*}
\eta_{N,g,\epsilon}(z) = & \frac{1}{M} \Tr \Q_{\mathfrak{b}}^{2} \D_{\epsilon,g}   \frac{\X_{\mathfrak{b}} \X_{\mathfrak{b}}^{*}}{B+1}  \D_{\epsilon,g},   \\
\gamma_{N,g,\epsilon}(f) = & \frac{1}{\pi} \mathrm{Re} \int_{\Dcal} \bar{\partial}  \Phi_k(f)(z) \, 
\eta_{N,g,\epsilon}(z) \diff x \diff y.
\end{align*}
We have still the implication 
$$
\gamma_{N,g,\epsilon}(f) = \Ocal_{\prec}\left( \frac{B^{\epsilon}}{B^{3/2}} \right) \Longrightarrow \gamma_{N}(f) = \Ocal_{\prec}\left( \frac{B^{\epsilon}}{B^{3/2}} \right).
$$
We are therefore back to prove that $\gamma_{N,g,\epsilon}(f) = \Ocal_{\prec}\left( \frac{B^{\epsilon}}{B^{3/2}} \right)$. For this, it is sufficient to apply Lemma 
\ref{le:concentration-integrale-helffer-sjostrand} with  
$A_{N}(\nu) = \left \{ \left \| \frac{\X_{N,\mathfrak{b}}(\nu)}{\sqrt{B+1}} \right\| \right \}\leq 3$. The details are left to the reader. \\

We finally briefly consider the proof (\ref{eq:E-tr-tr-Q2-D-I-XX-D-I}) because 
(\ref{eq:E-tr-D-I-square}) is an obvious property that was already mentioned (see Eq. (\ref{eq:trace-Dx-I-carre})). To establish (\ref{eq:E-tr-tr-Q2-D-I-XX-D-I}), 
it is possible to remark that $\Q_{\mathfrak{b}}'(z) = \Q^{2}_{\mathfrak{b}}(z)$ and to use the integration by parts formula already used to establish (\ref{eq:E-Q2-Dx-I-XX*-D3b}). While the calculations are of course more tedious, they are rather straightforward, and are therefore omitted. We however briefly motivate (\ref{eq:E-tr-tr-Q2-D-I-XX-D-I}). For this, 
we remark that
\begin{align}
\label{eq:expre-inter-1-delta5}
& \mathbb{E} \left(\frac{1}{M} \Tr \Q_{\mathfrak{b}}^{2} (\D_{\x_{\mathfrak{b}}} - \I) \frac{\X_{\mathfrak{b}} \X_{\mathfrak{b}}^{*}}{B+1}  (\D_{\x_{\mathfrak{b}}} - \I) \right)= \\
\notag
& (\beta(z))'  \mathbb{E} \left(\frac{1}{M}\Tr(\D_{\x_{\mathfrak{b}}} - \I)  \frac{\X_{\mathfrak{b}} \X_{\mathfrak{b}}^{*}}{B+1} (\D_{\x_{\mathfrak{b}}} - \I) \right) \\ \notag &
+ \mathbb{E} \left(\frac{1}{M} \Tr \left(\Q_{\mathfrak{b}}^{2}\right)^{\circ} \D_{\x_{\mathfrak{b}}} - \I) \frac{\X_{\mathfrak{b}} \X_{\mathfrak{b}}^{*}}{B+1}  (\D_{\x_{\mathfrak{b}}} - \I)\right) .
\end{align}
$\frac{1}{M}\Tr(\D_{\x_{\mathfrak{b}}} - \I)  \frac{\X_{\mathfrak{b}} \X_{\mathfrak{b}}^{*}}{B+1} (\D_{\x_{\mathfrak{b}}} - \I)$ 
coincides with $\frac{1}{M}\Tr(\D_{\x_{\mathfrak{b}}} - \I) \D_{\x_{\mathfrak{b}}} (\D_{\x_{\mathfrak{b}}} - \I) = \frac{1}{M}\Tr(\D_{\x_{\mathfrak{b}}} - \I)^{2} +
 \frac{1}{M}\Tr(\D_{\x_{\mathfrak{b}}} - \I)^{3}$. Therefore, 
 \begin{align*}
 \mathbb{E}\left( \frac{1}{M}\Tr(\D_{\x_{\mathfrak{b}}} - \I)  \frac{\X_{\mathfrak{b}} \X_{\mathfrak{b}}^{*}}{B+1} (\D_{\x_{\mathfrak{b}}} - \I) \right) & =  \mathbb{E}\left( \frac{1}{M}\Tr(\D_{\x_{\mathfrak{b}}} - \I)^{2} \right) + \Ocal\left( \frac{B^{\epsilon}}{B^{3/2}}\right) \\ 
 &= \frac{1}{B+1} + \Ocal\left( \frac{B^{\epsilon}}{B^{3/2}}\right),
 \end{align*}
 for each $\epsilon > 0$. To complete the proof of (\ref{eq:E-tr-tr-Q2-D-I-XX-D-I}) using this approach, it would be necessary to prove that the second term of the right-hand side of (\ref{eq:expre-inter-1-delta5}) is a term $\Ocal_{z}(B^{-3/2})$. However, this property is not obvious. Therefore, it seems difficult to verify (\ref{eq:E-tr-tr-Q2-D-I-XX-D-I}) using the simple approach mentioned above, which is why we evaluated the left-hand side of (\ref{eq:expre-inter-1-delta5}) using the integration by parts formula. \\
 
 (\ref{eq:equivalent-delta34-plus-delta5}) and (\ref{eq:equivalent-E-delta34-plus-delta5}) in turn follow 
 directly from Eq. (\ref{eq:equivalent2-delta34}) to (\ref{eq:E-tr-tr-Q2-D-I-XX-D-I}).
 \subsection{Proof of Step 3}
 \label{subsec:proof-step3}
We omit the proof of (\ref{eq:domination-stochastique-delta12b}), (\ref{eq:E-delta12b-tilde-delta12b}), 
(\ref{eq:simplification-tildedelta12-rond}), and (\ref{eq:simplification-E-tildedelta12-rond}). In order to prove (\ref{eq:behaviour-T12-circ}), 
we again first replace $\D_{\x_{\mathfrak{b}}} - \I$ by the diagonal matrix $\D_{\epsilon,g}$ 
defined by (\ref{eq:def-Depsilong}), and establish that it is sufficient to verify that 
$$
\frac{1}{\pi} \mathrm{Re} \int_{\Dcal} \bar{\partial}  \Phi_k(f)(z) \, 
\left(T_{12,\epsilon,g}(z)\right)^{\circ} \diff x \diff y = \Ocal_{\prec}\left( \frac{B^{\epsilon}}{B^{3/2}} \right),
$$
where $T_{12,\epsilon,g}$ represents the term obtained by replacing $\D_{\x_{\mathfrak{b}}} - \I$ by  $\D_{\epsilon,g}$ in the expression of $T_{12,\mathfrak{b}}$. For this, we
apply Lemma 
\ref{le:concentration-integrale-helffer-sjostrand} with  
$A_{N}(\nu) = \left \{ \left \| \frac{\X_{N,\mathfrak{b}}(\nu)}{\sqrt{B+1}} \right\| \right \}\leq 3$. (\ref{eq:behaviour-delta12-circ}) then follows from (\ref{eq:domination-stochastique-delta12b}), (\ref{eq:simplification-tildedelta12-rond}), and 
(\ref{eq:behaviour-T12-circ}). (\ref{eq:expre-E-T12b}) is still established using the integration by parts formula after tedious, but rather straightforward calculations that are omitted. Finally, (\ref{eq:comportement-final-E-delta12}) follows directly from (\ref{eq:E-delta12b-tilde-delta12b}) and (\ref{eq:expre-E-T12b})  \\

We now prove (\ref{eq:evaluation-tilde-delta11}) and
(\ref{eq:evaluation-E-tilde-delta11}). For this, we express $\eta(z) = \delta_{11,\mathfrak{b}}(z) - T_{11,\mathfrak{b}}(z)$ as $\eta(z) = \eta_1(z) + \eta_2(z)$ where 
\begin{eqnarray}
    \label{eq:expre-delta11-T11-1}
\eta_1(z) & = & \frac{1}{M} \Tr \left(\left(\tilde{\Q}_{\mathfrak{b}}- \Q_{\mathfrak{b}}\right)^{\circ} + z  \left(\tilde{\Q}_{\mathfrak{b}}^{2} - \Q_{\mathfrak{b}}^{2}\right)^{\circ} \right) (\hat{\D}_{\mathfrak{b}} - \D) \D^{-1}, \\
\label{eq:expre-delta11-T11-2}
\eta_2(z) & = & \frac{1}{M} \Tr \mathbb{E}\left( \tilde{\Q}_{\mathfrak{b}}- \Q_{\mathfrak{b}}  + z  \tilde{\Q}_{\mathfrak{b}}^{2} - z \Q_{\mathfrak{b}}^{2}\right) (\hat{\D}_{\mathfrak{b}} - \D) \D^{-1}, 
\end{eqnarray}
and define for $i=1,2$ $\gamma_i(f)$ as $\gamma_i(f) =  \frac{1}{\pi} \mathrm{Re} \int_{\Dcal} \bar{\partial}  \Phi_k(f)(z) \, \eta_{i}(z) \diff x \diff y$. Adapting the proof of (\ref{eq:behaviour-theta2b-f-recentered}), it can easily be shown that the family 
\begin{align*}
&\frac{1}{\pi} \mathrm{Re} \int_{\Dcal} \bar{\partial}  \Phi_k(f)(z) \, \Bigl( \left(\tilde{\Q}_{\mathfrak{b}}(z,\nu)- \Q_{\mathfrak{b}}(z,\nu)\right)_{m,m}^{\circ} 
\\
&\qquad+ z  \left(\tilde{\Q}_{\mathfrak{b}}^{2}(z,\nu) - \Q_{\mathfrak{b}}^{2}(z,\nu)\right)_{m,m}^{\circ} \Bigr) \diff x  \diff y,  
\end{align*}
$m=1, \ldots, M, \nu \in [0,1]$ is a $\Ocal_{\prec}\left( \frac{\sqrt{B}}{N}\right)$ term.  As we have
$\left((\hat{\D}_{\mathfrak{b}} - \D) \D^{-1}\right)_{m,m} = \Ocal_{\prec}\left( \frac{1}{\sqrt{B}}\right)$, Lemma \ref{le:domination-moyenne} implies that $\gamma_1$ is a $ \Ocal_{\prec,z}\left( \frac{1}{N}\right)$ term. Moreover, $\D_{3,\mathfrak{b}}$ deterministic
  implies that $\mathbb{E}(\eta_1)$ is reduced to
  $$
  \mathbb{E}(\eta_1) = \mathbb{E} \left(\frac{1}{M} \Tr \left(\left(\tilde{\Q}_{\mathfrak{b}}- \Q_{\mathfrak{b}}\right)^{\circ} + z  \left(\tilde{\Q}_{\mathfrak{b}}^{2} - \Q_{\mathfrak{b}}^{2}\right)^{\circ} \right) (\D_{\x_{\mathfrak{b}}} - \I + \D_{2,\mathfrak{b}}) \right).
  $$
  The Nash-Poincaré inequality leads to 
  $$
  \mathbb{E}\left| \left(\tilde{\Q}_{\mathfrak{b}}(z,\nu)- \Q_{\mathfrak{b}}(z,\nu)\right)_{m,m}^{\circ} + z  \left(\tilde{\Q}_{\mathfrak{b}}^{2}(z,\nu) - \Q_{\mathfrak{b}}^{2}(z,\nu)\right)_{m,m}^{\circ}  \right|^{2} = \Ocal_z\left( \frac{B}{N^{2}}\right),
  $$
  from which we deduce immediately $\mathbb{E}(\eta_1) = \Ocal_{z}\left( \frac{1}{N} \right)$.
  Therefore, we also have $\mathbb{E}(\gamma_1) =  \Ocal\left( \frac{1}{N} \right)$
and $\gamma_1^{\circ} = \Ocal_{\prec}\left( \frac{1}{N} \right)$.

In order to evaluate $\gamma_2^{\circ}$ and $\mathbb{E}(\eta_2)$, we use the 
decomposition (\ref{eq:decomposition-Dx-I}) of $(\hat{\D}_{\mathfrak{b}} - \D) \D^{-1}$, and of course have
$$
\eta_2^{\circ} = \frac{1}{M} \Tr \mathbb{E}\left( \tilde{\Q}_{\mathfrak{b}}- \Q_{\mathfrak{b}}  + z  \tilde{\Q}_{\mathfrak{b}}^{2} - z \Q_{\mathfrak{b}}^{2}\right) \left( \D_{\x_{\mathfrak{b}}} - \I + \D_{2,\mathfrak{b}} \right),
$$
and 
$$\mathbb{E}(\eta_2) = \frac{1}{M} \Tr \mathbb{E}\left( \tilde{\Q}_{\mathfrak{b}}- \Q_{\mathfrak{b}}  + z  \tilde{\Q}_{\mathfrak{b}}^{2} - z \Q_{\mathfrak{b}}^{2}\right) \D_{3,\mathfrak{b}}.
$$
We first mention that (\ref{eq:E-norm-Gammab-Gamma1r}) implies 
that $\mathbb{E}\left(\tilde{\Q}_{\mathfrak{b}}- \Q_{\mathfrak{b}} \right)_{m,m} = \Ocal_{z}\left(\frac{B}{N}\right)$, and similarly
that $\mathbb{E}\left(\tilde{\Q}^{'}_{\mathfrak{b}}- \Q^{'}_{\mathfrak{b}} \right)_{m,m} =  \mathbb{E}\left(\tilde{\Q}^{2}_{\mathfrak{b}}- \Q^{2}_{\mathfrak{b}} \right)_{m,m}$ is a $\Ocal_{z}\left(\frac{B}{N}\right)$ term. The Hanson-Wright inequality thus leads 
to 
\begin{align*}
& \frac{1}{M} \Tr \mathbb{E}\left( \tilde{\Q}_{\mathfrak{b}}- \Q_{\mathfrak{b}}  + z  \tilde{\Q}_{\mathfrak{b}}^{2} - z \Q_{\mathfrak{b}}^{2}\right) (\D_{\x_{\mathfrak{b}}} - \I)  = \Ocal_{\prec,z}\left(\frac{1}{N}\right), \\
& \frac{1}{M} \Tr \mathbb{E}\left( \tilde{\Q}_{\mathfrak{b}}- \Q_{\mathfrak{b}}  + z  \tilde{\Q}_{\mathfrak{b}}^{2} - z \Q_{\mathfrak{b}}^{2}\right) \D_{2,\mathfrak{b}}  = \Ocal_{\prec,z}\left(\frac{B}{N^{2}}\right)= o_{\prec,z}\left(\frac{1}{N}\right).
\end{align*}
Therefore, $\eta_2^{\circ} =  \Ocal_{\prec,z}\left(\frac{1}{N}\right)$ and
$\gamma_2^{\circ} =  \Ocal_{\prec}\left(\frac{1}{N}\right)$. 
It remains to study $\mathbb{E}(\eta_2)$. For this, we prove in Appendix \ref{sec:proof-inter-moyenne-delta13} the following Lemma.
\begin{lemma}
\label{le:terme-moyenne-delta13}
The equality 
\begin{equation}
\label{eq:expre-inter-moyenne-delta13}
 \frac{1}{M} \Tr \mathbb{E}(\tilde{\Q}_{\mathfrak{b}} + z  \tilde{\Q}_{\mathfrak{b}}^{2}) \, \D_{3,\mathfrak{b}}  = 
 \frac{1}{M} \Tr \mathbb{E}(\Q_{\mathfrak{b}} + z  \Q_{\mathfrak{b}}^{2}) \, \D_{3,\mathfrak{b}} + 
\Ocal_{z} \left( \frac{B^{4}}{N^{4}} + \frac{B^{5/2}}{N^{3}} \right) ,
\end{equation}
holds each $\alpha < 1$
\end{lemma}
Lemma \ref{le:terme-moyenne-delta13} completes the proof of (\ref{eq:evaluation-tilde-delta11}) and (\ref{eq:evaluation-E-tilde-delta11}).  \\

(\ref{eq:E-T11-1}) follows from (\ref{eq:expre-E-Q-Dx-I}) and from the observation that $z \Q_{\mathfrak{b}}^{2} = z \Q_{\mathfrak{b}}^{'}$.  
(\ref{eq:E-T11-2}) is proved using that 
$$
T_{11,\mathfrak{b}}^{2} = (\beta(z) + z \beta'(z)) \frac{1}{M} \Tr \D_{2,\mathfrak{b}} + 
\frac{1}{M} \Tr (\Q_{\mathfrak{b}} + z  \Q_{\mathfrak{b}}^{2})^{\circ} \D_{2,\mathfrak{b}},
$$
and by remarking that $\mathbb{E} \left( \frac{1}{M} \Tr (\Q_{\mathfrak{b}} + z  \Q_{\mathfrak{b}}^{2})^{\circ} \D_{2,\mathfrak{b}} \right)= \Ocal_z\left(\frac{1}{N}\right)$. (\ref{eq:behaviour-moyenne-trace-Phimb}), $\tilde{p}(z) = (z t(z))^{'}$ and $\mathbb{E}(\Q_{\mathfrak{b}} + z  \Q_{\mathfrak{b}}^{2}) = (z \beta(z))^{'} \I = \left((z t(z))^{'} + \Ocal_z(B^{-2}) \right) \I$ imply immediately (\ref{eq:E-T11-3}). 
We omit the proof of (\ref{eq:T11-rond}) and of (\ref{eq:detla1b-rond}). (\ref{eq:E-delta1b}) follows directly from 
(\ref{eq:expre-E-T12b}), (\ref{eq:E-T11-1}) and (\ref{eq:E-T11-3}). \\

We omit the proofs of (\ref{eq:behaviour-delta2b-circ}) and (\ref{eq:expre-E-delta2b}) which are very similar. 

\section{Proof of Lemma \ref{le:moments-Dxb-1}.}
\label{proof-lemma-moments-Dxb-inverse}
The Hanson-Wright inequality implies that for $\epsilon < 1$, the set $A_{\epsilon}$ defined by 
\begin{equation}
\label{eq:def-Aepsilon-proof-le-moments-Dxb-1}
A_{\epsilon} = \{ \|\x_{m,\mathfrak{b}}\|^{2}/(B+1) \in [1 - \epsilon, 1 + \epsilon], m=1, \ldots, M \},
\end{equation}
holds with exponentially high probability. We put $\delta_m = \left(  \|\x_{m,\mathfrak{b}}\|^{2}/(B+1) \right)^{-1}$
and $\delta = \sup_{m=1, \ldots, M} \delta_m = \| \D_{\x_{\mathfrak{b}}}\|^{-1}$. $\mathbb{E}(\delta^{k})$ can be written as 
$$
\mathbb{E}(\delta^{k}) = \mathbb{E}(\delta^{k} \, \mathds{1}_{A_{\epsilon}}) + \mathbb{E}(\delta^{k} \, \mathds{1}_{A_{\epsilon}^{c}}).
$$
$\mathbb{E}(\delta^{k} \, \mathds{1}_{A_{\epsilon}})$ is of course a $\Ocal(1)$ term. The Schwartz inequality 
leads to 
$$
|\mathbb{E}(\delta^{k} \, \mathds{1}_{A_{\epsilon}^{c}})| \leq \left(\mathbb{E}(\delta^{2k})\right)^{1/2} 
\left(\mathbb{P}(A_{\epsilon}^{c})\right)^{1/2}.
$$
As $\mathbb{P}(A_{\epsilon}^{c})$ converges towards $0$ exponentially, it is sufficient to verify that 
$\mathbb{E}(\delta^{2k})$ can be upper-bounded by a term that converges possibly towards $+\infty$ 
at a polynomial rate. For this, we use the explicit expression of the probability density $p(u)$ of the $\chi^{2}$ random variables $(\|\x_{m,\mathfrak{b}}\|^{2})_{m=1, \ldots, M}$, i.e. 
$$
p(u) = \frac{u^{B}}{B!} e^{-u} \mathds{1}_{\mathbb{R}^{+}}(t),
$$
in order to evaluate the probability density $q(t)$ of $\frac{\delta}{B+1} = \sup_{m} \frac{1}{\|\x_{m,\mathfrak{b}}\|^{2}}$. It is easily checked that 
$$
q(t) = \frac{M}{B!} \frac{1}{t^{B+2}} e^{-\frac{1}{t}} \left(\mathbb{P}\left(\|\x_{m,\mathfrak{b}}\|^{2} > \frac{1}{t}\right)\right)^{M-1} \leq \frac{M}{B!} \frac{1}{t^{B+2}} e^{-\frac{1}{t}}.
$$
Therefore, $\mathbb{E}(\delta^{k})$ verifies 
$$
\mathbb{E}(\delta^{k}) \leq \frac{M(B+1)^{k}}{B!} \int_0^{+\infty} \frac{1}{t^{B+2}}  e^{-\frac{1}{t}} dt .
$$
It is easily checked that 
$$
 \int_0^{+\infty} \frac{1}{t^{B+2}}  e^{-\frac{1}{t}} dt  = B!,
$$
so that $\mathbb{E}(\delta^{k}) \leq M(B+1)^{k}$. This establishes that $ \left(\mathbb{E}(\delta^{2k})\right)^{1/2} 
\left(\mathbb{P}(A_{\epsilon}^{c})\right)^{1/2} \rightarrow 0$ as expected.

\section{Sketch of proof of (\ref{eq:expre-inter-moyenne-delta13})}
\label{sec:proof-inter-moyenne-delta13}
We just briefly justify that 
\begin{equation}
\label{eq:expre-appendix-inter-moyenne-delta13}
 \frac{1}{M} \Tr \mathbb{E}(\tilde{\Q}_{\mathfrak{b}}) \, \D_{3,\mathfrak{b}}  =  
 \frac{1}{M} \Tr \mathbb{E}(\Q_{\mathfrak{b}}) \, \D_{3,\mathfrak{b}} + 
\Ocal_{z} \left( \frac{B^{4}}{N^{4}} + \frac{B^{5/2}}{N^{3}} \right) ,
\end{equation}
because it can be shown that (\ref{eq:expre-inter-moyenne-delta13}) can be obtained from (\ref{eq:expre-appendix-inter-moyenne-delta13}) by differentiating w.r.t. $z$. We express $\tilde{\Q}_{\mathfrak{b}} - \Q_{\mathfrak{b}}$
as
$$
\tilde{\Q}_{\mathfrak{b}} - \Q_{\mathfrak{b}} = - \Q_{\mathfrak{b}} \tilde{\Deltabs}_{\mathfrak{b}} \Q_{\mathfrak{b}} + \tilde{\Q}_{\mathfrak{b}} \tilde{\Deltabs}_{\mathfrak{b}} \Q_{\mathfrak{b}} \tilde{\Deltabs}_{\mathfrak{b}} \Q_{\mathfrak{b}}.
$$
As 
$$
\mathbb{E}\left \|  \tilde{\Deltabs}_{\mathfrak{b}}\right \|^{2} = \Ocal \left( \left(\frac{B}{N}\right)^{2} \right),
$$
we obtain that 
$$
 \frac{1}{M} \Tr \mathbb{E}(\tilde{\Q}_{\mathfrak{b}}) \, \D_{3,\mathfrak{b}}  =  
 \frac{1}{M} \Tr \mathbb{E}(\Q_{\mathfrak{b}}) \, \D_{3,\mathfrak{b}} - \frac{1}{M} \Tr \mathbb{E}( \Q_{\mathfrak{b}} \tilde{\Deltabs}_{\mathfrak{b}} \Q_{\mathfrak{b}}) \, \D_{3,\mathfrak{b}} + \Ocal_z \left( \frac{B^{4}}{N^{4}} \right).
$$
As $\mathbb{E} \left\| \frac{\Gammabs_{\mathfrak{b}}\Gammabs_{\mathfrak{b}}^{*}}{B+1} \right\| = \Ocal\left( B^{2}/N^{2}\right)$ (see (\ref{eq:E-norm-Gammab-Gamma1r})), we also have 
\begin{align*}
&\frac{1}{M} \Tr \mathbb{E}(\tilde{\Q}_{\mathfrak{b}}) \, \D_{3,\mathfrak{b}}  =  
 \\
 &\qquad
 \frac{1}{M} \Tr \mathbb{E}(\Q_{\mathfrak{b}}) \, \D_{3,\mathfrak{b}}  - \frac{1}{M} \Tr \mathbb{E}\left( \Q_{\mathfrak{b}} \left(  \frac{\X_{\mathfrak{b}} \Gammabs_{\mathfrak{b}}^{*}}{B+1}  +  \frac{\Gammabs_{\mathfrak{b}} \X_{\mathfrak{b}}^{*}}{B+1} \right)  \Q_{\mathfrak{b}}\right) \, \D_{3,\mathfrak{b}} 
 \\ 
 & \qquad\qquad+ \Ocal_z \left( \frac{B^{4}}{N^{4}} \right).
\end{align*}
We just indicate how to obtain the order of magnitude of the term $\eta$ defined by 
$$
\eta =  \frac{1}{M} \Tr \mathbb{E}\left( \Q_{\mathfrak{b}} \frac{\Gammabs_{\mathfrak{b}} \X_{\mathfrak{b}}^{*}}{B+1}  \Q_{\mathfrak{b}}\right) \, \D_{3,\mathfrak{b}}.
$$
In order to simplify the notations, we define $\W$ as the matrix $\W = \frac{\X_{\mathfrak{b}}}{\sqrt{B+1}}$. Then, $\eta$ can be written as 
\begin{eqnarray*}
\eta & = & \frac{1}{M} \sum_{r=1}^{M} \mathbb{E} \left(  \Q_{\mathfrak{b}} \frac{\Gammabs_{\mathfrak{b}} \X_{\mathfrak{b}}^{*}}{B+1}  \Q_{\mathfrak{b}}\right)_{r,r} \, \frac{1}{B+1} \Tr \Phibs_{r,\mathfrak{b}} \\
  & = &  \frac{1}{M} \sum_{r=1}^{M} \sum_{m=1}^{M} \mathbb{E}\left(  \Q_{\mathfrak{b},r,m} \w_{m} \Psibs_{m,\mathfrak{b}} \W^{*}  \Q_{\mathfrak{b}} \e_r \right) \, \frac{1}{B+1} \Tr \Phibs_{r,\mathfrak{b}} .
\end{eqnarray*}
It is thus necessary to evaluate $\eta_{r,m} = \mathbb{E}\left(  \Q_{\mathfrak{b},r,m} \w_{m} \Psibs_{m,\mathfrak{b}} \W^{*}  \Q_{\mathfrak{b}} \e_r \right)$ for each $r,m$. Using the integration by parts formula, we obtain easily that 
\begin{align}
\label{eq:ipp-1-etarm}
\eta_{r,m} &= \mathbb{E}(  \Q_{\mathfrak{b},r,m}  \Q_{\mathfrak{b},m,r}) \, \frac{1}{B+1} \Tr \Psibs_{m,\mathfrak{b}} 
\notag\\
&\qquad - \frac{1}{B+1} \mathbb{E} \left(\Q_{\mathfrak{b},m,m} (\Q_{\mathfrak{b}} \W \Psibs_{m,\mathfrak{b}} \W^{*} 
\Q_{\mathfrak{b}})_{r,r} \right) 
\notag\\
&\qquad\qquad - \mathbb{E} \left(  \Q_{\mathfrak{b},r,m}  \Q_{\mathfrak{b},m,r} 
 \frac{1}{B+1} \Tr \Q_{\mathfrak{b}} \W  \Psibs_{m,\mathfrak{b}} \W^{*} \right),
\end{align}
or equivalently
\begin{align}
\label{eq:ipp-2-etarm}
\eta_{r,m} &=  \mathbb{E}(  \Q_{\mathfrak{b},r,m}  \Q_{\mathfrak{b},m,r}) \, \frac{1}{B+1} \Tr \Psibs_{m,\mathfrak{b}} 
\notag\\
&\qquad- \frac{1}{B+1} \mathbb{E}(\Q_{\mathfrak{b},m,m}) \mathbb{E} \left((\Q_{\mathfrak{b}} \W \Psibs_{m,\mathfrak{b}} \W^{*} 
\Q_{\mathfrak{b}})_{r,r} \right) 
\notag\\
&\qquad\qquad - \mathbb{E} (\Q_{\mathfrak{b},r,m}  \Q_{\mathfrak{b},m,r})  
\mathbb{E} \left( \frac{1}{B+1} \Tr \Q_{\mathfrak{b}} \W  \Psibs_{m,\mathfrak{b}} \W^{*} \right) + \epsilon_{m,r},
\end{align}
where $\epsilon_{m,r}$ is given by 
\begin{align}
\label{eq:expre-epsilon-mr}
\epsilon_{m,r} = & -\frac{1}{B+1}  \mathbb{E} \left(\Q_{\mathfrak{b},m,m}^{\circ} (\Q_{\mathfrak{b}} \W \Psibs_{m,\mathfrak{b}} \W^{*} \Q_{\mathfrak{b}})_{r,r}^{\circ} \right) \\ \notag & -  \mathbb{E} \left(  (\Q_{\mathfrak{b},r,m}  \Q_{\mathfrak{b},m,r})^{\circ} 
 \left(\frac{1}{B+1} \Tr \Q_{\mathfrak{b}} \W  \Psibs_{m,\mathfrak{b}} \W^{*}\right)^{\circ} \right).
\end{align}
Using the Nash-Poincaré inequality and the Schwartz inequality, we obtain easily that 
$\epsilon_{m,r} = \Ocal_{z} \left( \frac{\|\Psibs_{m,\mathfrak{b}}\|}{B^{2}} + \frac{\|\Psibs_{m,\mathfrak{b}}\|}{B^{3/2}} \right)
 =  \Ocal_{z} \left( \frac{1}{\sqrt{B} N} \right)$. Therefore, we have 
\begin{equation}
\label{eq:behaviour-sum-m-epsilon-mr}
\sum_{m=1}^{M} \epsilon_{m,r} =  \Ocal_{z} \left( \frac{\sqrt{B}}{N} \right),
\end{equation}
and conclude that the contribution of the error terms $(\epsilon_{m,r})_{m,r=1, \ldots,M}$ to 
$\eta$ is a $\Ocal_{z} \left( \frac{\sqrt{B}}{N} \frac{B^{2}}{N^{2}}\right) = \Ocal_{z} \left(  \frac{B^{5/2}}{N^{3}}\right)$ term (we recall that $\frac{1}{B+1} \Tr \Phibs_{r,\mathfrak{b}} = 
\Ocal_{z}\left( \frac{B^{2}}{N^{2}}\right)$). We now evaluate the contribution to $\eta$ of the second term of the right-hand side of (\ref{eq:ipp-2-etarm}). For this, we first recall that  $\mathbb{E}(\Q_{\mathfrak{b},m,m}) = \beta(z)$
does not depend on $m$. We have thus to evaluate the order of magnitude of 
$ \mathbb{E} \left(\Q_{\mathfrak{b}} \W \Psibs_{m,\mathfrak{b}} \W^{*} 
\Q_{\mathfrak{b}})_{r,r} \right)$. For this, we consider any $(B+1) \times (B+1)$ matrix 
$\A$, and compute $\zeta(\A)$ defined by 
\begin{equation}
\label{eq:def-zeta-A}
\zeta(\A) = \mathbb{E} \left((\Q_{\mathfrak{b}} \W \A \W^{*} 
\Q_{\mathfrak{b}})_{r,r} \right),
\end{equation}
using the integration by parts formula. After some algebra, we obtain that 
\begin{align}
\label{eq:zeta-A-ipp-1}
\zeta(\A) = & \mathbb{E}(\Q^{2}_{\mathfrak{b},r,r}) \, \frac{1}{B+1} \Tr \A - c \, \mathbb{E} \left( 
\frac{1}{M} \Tr \Q_{\mathfrak{b}} \, \left(\Q_{\mathfrak{b}} \W \A \W^{*} 
\Q_{\mathfrak{b}})_{r,r} \right)\right) \\ \notag 
& - c \mathbb{E}\left( (\Q^{2}_{\mathfrak{b},r,r}) \frac{1}{B+1} \Tr\Q_{\mathfrak{b}} \W \A \W^{*} \right),
\end{align}
and therefore that 
\begin{equation}
\label{eq:zeta-A-ipp-2}
\zeta(\A) = (\beta(z))' \frac{1}{B+1} \Tr \A - c \, \beta(z) \, \zeta(\A)  - c (\beta(z))' \mathbb{E}\left( \frac{1}{B+1} \Tr\Q_{\mathfrak{b}} \W \A \W^{*} \right) + \omega_r,
\end{equation}
where the error term $\omega_r$ is defined by 
\begin{align*}
\omega_r &= - c \, \mathbb{E} \left( 
\frac{1}{M} \Tr \Q^{\circ}_{\mathfrak{b}} \, \left((\Q_{\mathfrak{b}} \W \A \W^{*} 
\Q_{\mathfrak{b}})_{r,r} \right)^{\circ} \right)  
\\
&\qquad- c \mathbb{E}\left( (\Q^{2}_{\mathfrak{b},r,r})^{\circ}  \left( \frac{1}{B+1} \Tr\Q_{\mathfrak{b}} \W \A \W^{*} \right)^{\circ} \right),
\end{align*}
and verifies $\omega_r =  \Ocal_{z} \left( \frac{\|\A\|}{B^{3/2}} \right)$. Solving (\ref{eq:zeta-A-ipp-2})
w.r.t. $\zeta(\A)$, we obtain that 
$$
\zeta(\A) = \frac{1}{1 + \beta \, c} \left(  \frac{1}{B+1} \Tr \A - \, c \,  (\beta(z))' \, \mathbb{E}\left( \frac{1}{B+1} \Tr\Q_{\mathfrak{b}} \W \A \W^{*} \right) + \omega_r \right).
$$
The term $\mathbb{E}\left( \frac{1}{B+1} \Tr\Q_{\mathfrak{b}} \W \A \W^{*} \right)$ is easily evaluated using the calculations in \cite{loubaton-rosuel-ejs-2021}, Appendix F, from which it can be deduced that 
$$
\mathbb{E}\left( \frac{1}{B+1} \Tr\Q_{\mathfrak{b}} \W \A \W^{*} \right) = \Ocal_z\left( 
\frac{1}{B+1} \Tr \A + \frac{\| \A \|}{B^{2}} \right).
$$
Therefore, we obtain that 
$$
\zeta(\A) = \Ocal_{z} \left( \frac{1}{B+1} \Tr \A + \frac{\|\A\|}{B^{3/2}} \right).
$$
Using this result for $\A = \Psibs_{m,\mathfrak{b}}$, we deduce that the contribution of 
the second term of the right-hand side of (\ref{eq:ipp-2-etarm}) to $\eta$ is a 
$ \Ocal_{z} \left( \frac{B^{4}}{N^{4}} + \frac{B^{3/2}}{N^{3}} \right)$ term (we recall
that $ \frac{1}{B+1} \Tr \Psibs_{m,\mathfrak{b}} = \Ocal\left( B^{2}/N^{2}\right)$). 
The contribution to $\eta$ of the first and third terms of the right-hand side of 
(\ref{eq:ipp-2-etarm}) can be written as 
$$
\frac{1}{M} \sum_{r=1}^{M} \mathbb{E} \left( (\Q_{\mathfrak{b}} \A \Q_{\mathfrak{b}})_{r,r} \right) 
\,  \frac{1}{B+1} \Tr \Phibs_{r,\mathfrak{b}} ,
$$
where $\A$ represents either the diagonal matrix $\dg \left( \frac{1}{B+1} \Tr \Psibs_{m,\mathfrak{b}}, m=1, \ldots, M \right)$ (contribution of the first term), either the diagonal matrix \\
$\dg \left( \mathbb{E} \left( \frac{1}{B+1} \Tr \Q_{\mathfrak{b}} \W  \Psibs_{m,\mathfrak{b}} \W^{*} \right), 
 m=1, \ldots, M \right)$. It thus appears necessary to evaluate $\mathbb{E} \left( (\Q_{\mathfrak{b}} \A \Q_{\mathfrak{b}})_{r,r} \right)$. For this, it is still possible to use the integration by parts formula. As the calculations are easy, but tedious, we omit to establish that 
$$
\mathbb{E} \left( (\Q_{\mathfrak{b}} \A \Q_{\mathfrak{b}})_{r,r} \right) = 
\Ocal_{z} \left( \A_{r,r} + \frac{1}{B+1} \Tr \A + \frac{\| \A \|}{B^{3/2}} \right).
$$
Using this result in the context of the two above mentioned matrices $\A$, we obtain that the contributions to 
$\eta$ of the first and third terms of the right-hand side of (\ref{eq:ipp-2-etarm})
both coincide with a $\Ocal_{z} \left( \frac{B^{4}}{N^{4}} + \frac{B^{3/2}}{N^{3}} \right)$ term. Putting all the pieces together, we obtain  (\ref{eq:expre-inter-moyenne-delta13}). 

 \section{\texorpdfstring{Outline of the computation of the $\Ocal\left( \frac{B}{N}\right)^{4}$ of $\mathbb{E}(\theta_N(f,\nu))$.}{Outline of computation}}
\label{sec:guideline-calculation-high-order-term}
We mentioned above that the computation of the $\Ocal\left( \frac{B}{N}\right)^{4}$ term of $\mathbb{E}(\theta_N(f,\nu))$ might be useful to extend the results of this paper for larger values of $\alpha$. In effect, substracting a consistent estimate of this term from $\theta_N(f,\nu)$ should allow to 
derive CLT on statistics based on the corresponding recentered versions 
of $\theta_N(f,\nu)$. Assuming that the next significant term of the 
mean of the recentered version of $\theta_N(f,\nu)$ is $\Ocal\left(\left( \frac{B}{N}\right)^{6}\right)$, the CLT might be extended to values of $\alpha$ for which 
$\left(\frac{B}{N}\right)^{6} = o\left( \frac{1}{\sqrt{NB}}\right)$, i.e. for $\alpha < \frac{11}{13}$. For those values of 
$\alpha$, we mentioned that the closed form evaluation of the $\Ocal\left(\left( \frac{B}{N}\right)^{4}\right)$ is quite complicated. In this section, we justify this claim by identifying more precisely the various steps that have to be 
achieved for this purpose. For this, we have to consider the decomposition 
(\ref{eq:decomposition-gammab}), and identify the $\Ocal_z\left(\left( \frac{B}{N}\right)^{4}\right)$ terms of $\mathbb{E}( \gamma_{i,N,\mathfrak{b}}(z))$ that should be evaluated in closed form 
for $i=1,2$. \\

\paragraph{Evaluation of the $\Ocal_z\left(\left( \frac{B}{N}\right)^{4}\right)$ term of $\mathbb{E}( \gamma_{1,N,\mathfrak{b}}(z))$.}
We recall that $\gamma_{1,N,\mathfrak{b}}(z))$ is evaluated in Step 1, Step 2, and Step 3 (see Appendix \ref{subsec:needed-results}). We have therefore to reconsider these various steps, and 
to identify the various deterministic $\Ocal_z\left( \frac{B}{N}\right)^{4}$ terms.  
We first mention that as $\alpha$ is potentially larger than 0.8, 
it is no longer true that (\ref{eq:domination-hatDb-D-alpha-inferieur-4-5}) 
and (\ref{eq:E-norm-hatDb-D}) hold (see Remark \ref{re:biais-variance-hats}). Therefore, some simplications that were used in Steps 1, 2, 3 are no longer valid. 
\begin{itemize}
    \item In Step 1. It is easy to check that, while (\ref{eq:moments-norme-Upsilon1}) is no more valid, the moments of $\| \Upsilonbs_1 \|$ 
    are always negligible w.r.t. $\frac{1}{\sqrt{NB}}$ for $\alpha < \frac{11}{13}$
    \item In Step 2, (\ref{eq:E-equivalent1-delta34}) and 
    (\ref{eq:E-equivalent1-delta5}) are still valid, but with different error terms that are negligible w.r.t. $\frac{1}{\sqrt{NB}}$. However, 
    the next calculations presented in the overview of the proof 
    Step 2 (see Appendix \ref{subsec:needed-results}) 
    have to be reconsidered to identify the $\Ocal_z\left(\left( \frac{B}{N}\right)^{4}\right)$ terms of $\delta_{34,\mathfrak{b}}$ and 
    $\delta_{5,\mathfrak{b}}$. For this, one has to use the 
    decomposition (\ref{eq:decomposition-Dx-I}) of $\left(\hat{\D}_{\mathfrak{b}}(\nu) - \D(\nu)\right) \D^{-1}(\nu)$, and to 
    recognize that the $\Ocal_z\left( \frac{B}{N}\right)^{4}$ terms of $\delta_{34,\mathfrak{b}}$ and 
    $\delta_{5,\mathfrak{b}}$ are due to the deterministic term $\D_{3,\mathfrak{b}}$. It is thus necessary to evaluate 
    $$
    \mathbb{E} \left(  \frac{1}{M} \Tr (\Q_{\mathfrak{b}} + z \Q_{\mathfrak{b}}^{2}) \left( \D_{3,\mathfrak{b}} \right)^{2}  \right),
    $$
    and 
    $$
    \mathbb{E} \left(\frac{1}{M} \Tr \Q_{\mathfrak{b}}^{2} \D_{3,\mathfrak{b}} \frac{\X_{\mathfrak{b}} \X_{\mathfrak{b}}^{*}}{B+1}  \D_{3,\mathfrak{b}} \right).
    $$
    This can be done by using the integration by parts formula, and leads to the closed form expression of the $\Ocal_z\left(\left( \frac{B}{N}\right)^{4}\right)$ 
    at the right hand side of (\ref{eq:equivalent-E-delta34-plus-delta5}).     
    \item In Step 3, (\ref{eq:E-delta12b-tilde-delta12b}) is still valid (with a different error term still negligible 
    w.r.t. $\frac{1}{\sqrt{NB}}$), and we now have to 
    evaluate the  $\Ocal_z\left(\left( \frac{B}{N}\right)^{4}\right)$ 
    term at the right hand side of (\ref{eq:simplification-E-tildedelta12-rond}), i.e. the term
    $$
\mathbb{E} \left(\frac{1}{M} \mathrm{Tr} \Bigl[\Q_{\mathfrak{b}} 
\left(\D_{3,\mathfrak{b}} \frac{\X_{\mathfrak{b}}\X_{\mathfrak{b}}^*}{B+1} + 
\frac{\X_{\mathfrak{b}}\X_{\mathfrak{b}}^*}{B+1} \D_{3,\mathfrak{b}} \right)
\Q_{\mathfrak{b}}  \D_{3,\mathfrak{b}} \frac{\X_{\mathfrak{b}}\X_{\mathfrak{b}}^*}{B+1} \Q_{\mathfrak{b}} \Bigr] \right).
$$
This can also be done using the integration by parts formula
after very tedious calculations. A similar job has to be done to 
evaluate the $\Ocal_z\left(\left( \frac{B}{N}\right)^{4}\right)$ at the right hand side of (\ref{eq:evaluation-E-tilde-delta11}). It is finally necessary to compute the $\Ocal_z\left(\left( \frac{B}{N}\right)^{4}\right)$ term at the right hand side of (\ref{eq:E-T11-3}), or equivalently the $\Ocal_z\left(\left( \frac{B}{N}\right)^{4}\right)$ term at the right hand side of (\ref{eq:expre-inter-moyenne-delta13}). For this, it is necessary 
follow the proof of Lemma \ref{le:terme-moyenne-delta13}
in Appendix \ref{sec:proof-inter-moyenne-delta13} and to
identify the $\Ocal_z\left(\left( \frac{B}{N}\right)^{4}\right)$ term 
of $\eta$ defined by 
$$
\eta =  \frac{1}{M} \Tr \mathbb{E}\left( \Q_{\mathfrak{b}} \frac{\Gammabs_{\mathfrak{b}} \X_{\mathfrak{b}}^{*}}{B+1}  \Q_{\mathfrak{b}}\right) \, \D_{3,\mathfrak{b}},
$$
and of
$$
\frac{1}{M} \Tr \mathbb{E}( \Q_{\mathfrak{b}} \tilde{\Deltabs}_{\mathfrak{b}} \Q_{\mathfrak{b}} \tilde{\Deltabs}_{\mathfrak{b}} \Q_{\mathfrak{b}}) \, \D_{3,\mathfrak{b}}.
$$
While the former calculation is rather easy, the later one, that should again be based on the integration by parts formula, 
is extremely tedious.      
    \end{itemize}
    
\paragraph{Evaluation of the $\Ocal_z\left(\left( \frac{B}{N}\right)^{4}\right)$ term of $\mathbb{E}( \gamma_{2,N,\mathfrak{b}}(z))$.}
We refer the reader to Appendix \ref{sec:improvement-bias}
devoted to a sketch of proof of (\ref{eq:Etrace-tildeQ-Q-precise-1}) and (\ref{eq:Etrace-tildeQ-Q-precise-2}). 
While it is easy to evaluate the $\Ocal\left(\left( \frac{B}{N}\right)^{4}\right)$ terms at the right hand side of (\ref{eq:expansion-trace-average-Phi-carre}) 
and (\ref{eq:expansion-trace-average-Phi}), the identification 
of the same term at the right hand side of (\ref{eq:final-tight-evaluation-tilde-epsilon}) appears tremendous. It is hard to be more specific without spending a number of pages to explain intricate calculations. Therefore, the reader 
may convince himself of this by following carefully the various steps of the calculations in \cite{these-alexis}, Chap. 2 in which the $\Ocal_z\left(\left( \frac{B}{N}\right)^{4}\right)$ terms have to evaluated in closed form.

\section{Proof of (\ref{eq:WN-finite-moment})}
\label{sec:proof-eq:WN-finite-moment}
As we will use (\ref{eq:WN-finite-moment}) for finite values of $k$, we do not mention that the constant at the right-hand side of (\ref{eq:WN-finite-moment}) depends on $k$. The Hölder inequality implies that it is sufficient to establish (\ref{eq:WN-finite-moment}) when $k$ is even. We thus prove that 
\begin{equation}
\label{eq:WN-finite-even-moment}
\mathbb{E}|W_N|^{2k} \leq C,
\end{equation}
for some constant $C$. To verify (\ref{eq:WN-finite-even-moment}), we again use that the compactness of $\Dcal$ implies that 
$$
\mathbb{E}(W_N^{2k}) \leq C \, \int_{\Dcal} |h(z)|^{2k} \, \mathbb{E} \left|  \left( M \bar{\gamma}_N \right)^{\circ} \right|^{2k} \, \diff x \diff y,   
$$
where $\bar{\gamma}_N$ is defined by 
$$
\bar{\gamma}_N(z, \nu) = \frac{1}{M} \Tr \left(\Q_{N,\mathfrak{b}}(z,\nu) \D_{\x_{\mathfrak{b}}(\nu)}  + z \Q^{2}_{N,\mathfrak{b}}(z,\nu)(\D_{\x_{\mathfrak{b}}(\nu)} - \I\right).
$$
It is thus sufficient to check that 
\begin{equation}
    \label{eq:Mgammabarrond-moments}
     \mathbb{E} \left|  \left( M \bar{\gamma}_N \right)^{\circ} \right|^{2k} = \Ocal_z(1),
\end{equation}
for each integer $k$. We prove this property by induction on $k$. We consider $k=1$. It is easy to check that 
\begin{equation}
\label{eq:norme-gradient-M-gamma-bar}
\| \nabla \left( M \bar{\gamma}_N \right)^{\circ} \|^{2} = \| \nabla M \bar{\gamma}_N \|^{2} \leq C(z) \left( 1 + \frac{1}{M} \Tr \D_{\x_{\mathfrak{b}}}^{2} \right).   
\end{equation}
The Nash-Poincaré inequality thus implies that  (\ref{eq:Mgammabarrond-moments}) holds for $k=1$. 
We now assume that (\ref{eq:Mgammabarrond-moments}) holds for each $p \leq k$, and establish it for $k+1$. For this, we express $\mathbb{E} \left|  \left( M \bar{\gamma}_{N,\mathfrak{b}} \right)^{\circ} \right|^{2k+2}$ as 
$$
\mathbb{E} \left|  \left( M \bar{\gamma}_{N,\mathfrak{b}} \right)^{\circ} \right|^{2k+2} = \mathrm{Var}\left( \left( M \bar{\gamma}_{N,\mathfrak{b}} \right)^{\circ}\right)^{k+1} + 
\left| \mathbb{E} \left( \left( M \bar{\gamma}_{N,\mathfrak{b}} \right)^{\circ} \right)^{k+1} \right|^{2}.
$$
If $k+1$ is even, the second term of the right-hand side of the above equation is a $\Ocal_z(1)$ term because $k+1 \leq 2k$. 
If $k+1$ is odd, $k+2$ is even, and the Hölder inequality implies that 
$$
\mathbb{E}  \left| \left( M \bar{\gamma}_{N,\mathfrak{b}} \right)^{\circ} \right|^{k+1} \leq \left( \mathbb{E}  \left| \left( M \bar{\gamma}_{N,\mathfrak{b}} \right)^{\circ} \right|^{k+2} \right)^{\frac{k+1}{k+2}}.
$$
As $k+2 \leq 2k$, $ \mathbb{E}  \left| \left( M \bar{\gamma}_{N,\mathfrak{b}} \right)^{\circ} \right|^{k+2}$ is supposed to be a $\Ocal_z(1)$ term. 
Moreover, the inequality 
$\left(\Ocal_z(1)\right)^{(k+1)/(k+2)} \leq 1 + \Ocal_z(1) = \Ocal_z(1)$, and the Hölder inequality leads to the conclusion 
that $\left| \mathbb{E}  \left( \left( M \bar{\gamma}_{N,\mathfrak{b}} \right)^{\circ} \right)^{k+1} \right|^{2} = \Ocal_z(1)$. We finally evaluate 
$ \mathrm{Var}\left( \left( M \bar{\gamma}_{N,\mathfrak{b}} \right)^{\circ}\right)^{k+1}$. For this, we remark that 
$$
\nabla \left( \left( M \bar{\gamma}_{N,\mathfrak{b}} \right)^{\circ}\right)^{k+1} = (k+1) \, \left( \left( M \bar{\gamma}_{N,\mathfrak{b}} \right)^{\circ}\right)^{k} \nabla \left( M \bar{\gamma}_{N,\mathfrak{b}} \right)^{\circ}.
$$
Therefore, using (\ref{eq:norme-gradient-M-gamma-bar}), we obtain that
$$
\left\| \nabla \left( \left( M \bar{\gamma}_{N,\mathfrak{b}} \right)^{\circ}\right)^{k+1} \right \|^{2} \leq 
C(z)  \left( 1 + \frac{1}{M} \Tr \D_{\x_{\mathfrak{b}}}^{2} \right) \, 
\left|\left( M \bar{\gamma}_{N,\mathfrak{b}} \right)^{\circ} \right|^{2k}.
$$
The Nash-Poincaré inequality implies that 
$$
\mathrm{Var}\left( \left( M \bar{\gamma}_{N,\mathfrak{b}} \right)^{\circ}\right)^{k+1} \leq C(z) \mathbb{E} \left[  (1 + \frac{1}{M} \Tr \D_{\x_{\mathfrak{b}}}^{2}) \left|\left( M \bar{\gamma}_{N,\mathfrak{b}} \right)^{\circ} \right|^{2k} \right].
$$
Using the Hölder inequality, we thus have 
\begin{align*}
&\mathrm{Var}\left( \left( M \bar{\gamma}_{N,\mathfrak{b}} \right)^{\circ}\right)^{k+1} 
\leq 
\\
&\qquad C(z) \left( \mathbb{E}  \left|\left( M \bar{\gamma}_{N,\mathfrak{b}} \right)^{\circ} \right|^{2k+2} \right)^{\frac{2k}{2k+2}} \, \left( \mathbb{E}\left( 1 + \frac{1}{M} \Tr \D_{\x_{\mathfrak{b}}}^{2} \right)^{k+1} \right)^{\frac{1}{k+1}}.
\end{align*}
Adding $\left| \mathbb{E} \left( \left( M \bar{\gamma}_{N,\mathfrak{b}} \right)^{\circ} \right)^{k+1} \right|^{2}$ to both sides of this inequality, and using that the later term is $\Ocal_z(1)$, we obtain 
that 
$$
\mathbb{E}  \left|\left( M \bar{\gamma}_{N,\mathfrak{b}} \right)^{\circ} \right|^{2k+2} \leq C(z) \left( 1 + 
\left(\mathbb{E}  \left|\left( M \bar{\gamma}_{N,\mathfrak{b}} \right)^{\circ} \right|^{2k+2}\right)^{\frac{k}{k+1}} \right).
$$
For $x \geq 0$ and $0 < \alpha < 1$, it is clear that $\frac{x}{1+x^{\alpha}} > \frac{x}{2} \mathds{1}_{x \leq 1} + \frac{x^{1-\alpha}}{2} \mathds{1}_{x > 1}$. Therefore, if $x$ verifies $\frac{x}{1+x^{\alpha}} <  C$
for some $C > 0$, $x$ must verify $\frac{x}{2} \mathds{1}_{x \leq 1} + \frac{x^{1-\alpha}}{2} \mathds{1}_{x > 1} < C$. If $x \leq 1$, then $x$ is smaller than $2 C$, while if $x > 1$, $x$ satisfies $x \leq (2 C)^{\frac{1}{1-\alpha}}$. Hence, the inequality $\frac{x}{1+x^{\alpha}} <  C$ implies 
$x \leq 2 C + (2C)^{\frac{1}{1-\alpha}} $. Using this property for $x = \mathbb{E}  \left|\left( M \bar{\gamma}_{N,\mathfrak{b}} \right)^{\circ} \right|^{2k+2}$ and $\alpha = \frac{k}{k+1}$, we obtain that 
$$
\mathbb{E}  \left|\left( M \bar{\gamma}_{N,\mathfrak{b}} \right)^{\circ} \right|^{2k+2} \leq  \left( 2 C(z) + 
(2C(z))^{k} \right),
$$
and that $\mathbb{E}  \left|\left( M \bar{\gamma}_{N,\mathfrak{b}} \right)^{\circ} \right|^{2k+2} = \Ocal_z(1)$
as expected.

\section{Proof of Lemma \ref{le:positivite-sigmaN}}
\label{sec:proof-le-positivite-sigmaN}
In order to establish Lemma \ref{le:positivite-sigmaN}, we evaluate $\sigma^{2}(f)$ defined as the right-hand side of (\ref{eq:sigma2}) when $c_N, t_N(z), \tilde{t}_N(z)$ are replaced by $c, t(z), \tilde{t}(z)$. As $| c (z t(z) \tilde{t}(z))^{2}| < 1$ for all $z \in \Cbb^+$ (see Subsection \ref{subsec:properties-wishart}), we have the series expansion
\begin{equation}
\label{eq:serie-entiere}
\frac{1}{(1 - c z_1 t(z_1)\tilde{t}(z_1)z_2 t(z_2)\tilde{t}(z_2))^{2}} - 1 = \sum_{l=1}^{+\infty} (l+1) \left(c z_1 t(z_1)\tilde{t}(z_1)z_2 t(z_2)\tilde{t}(z_2)\right)^{l},
\end{equation}
for all $z_1,z_2 \in \Dcal = [a_1, a_2] \times [0,1]$ where $[a_1, a_2]$ contains the support 
$[\lambda_{-}, \lambda_{+}]$ of the Marcenko-Pastur distribution with parameter $c$. 
We now justify that the above series and  the integral in (\ref{eq:sigma2}) can be exchanged. 
From the definition of $h$ and the properties of the complex extension $\Phi_k(f)$, it is clear that there exists a nice constant $C$ such that,
\begin{align*}
    \sup_{z \in \Dcal} \left|h(z)\right| \leq C (\Im z)^k ,
\end{align*}
where we recall that $k$ is the order of the complex extension of $\Phi_k(f)$. Moreover, 
inequality (\ref{eq:upper-bound-1-czttilde-carre}) implies that there exists a nice constant 
$C_1$ such that $| c (z t(z) \tilde{t}(z))^{2}| < 1 - C_1 (\Im z)^{4} $ for each $z \in \Dcal$. 
Moreover, it is possible to choose $C_1$ small enough so that 
$$
\left( 1 - C_1 (\Im z)^{4} \right)^{1/2} < 1 - \frac{C_1 (\Im z)^{4}}{4},
$$
for $z \in \Dcal$, i.e. for $\Im z \in [0,1]$. We also have 
$|s(z)| \leq \frac{C}{(\Im z)^{4}}$ for each $z \in \Dcal$. Therefore, we have 
\begin{align*}
&|h(z_1)| |h(z_2)| |s(z_1)||s(z_2)| \left|\sqrt{c} z_1 t(z_1)\tilde{t}(z_1)\right|^{l}  \left|\sqrt{c} z_2 t(z_2)\tilde{t}(z_2)\right|^{l} 
\\
&\qquad\qquad \leq  C \frac{(\Im z_1)^{k} (\Im z_2)^{k}}{(\Im z_1)^{4} (\Im z_2)^{4}} (1 - C_1 (\Im z_1)^{4})^{l/2} (1 - C_1 (\Im z_2)^{4})^{l/2} 
\\
& \qquad\qquad
\leq C \frac{(\Im z_1)^{k} (\Im z_2)^{k}}{(\Im z_1)^{4} (\Im z_2)^{4}} (1 - C_2 (\Im z_1)^{4})^{l} (1 - C_2 (\Im z_2)^{4})^{l} ,
\end{align*}
and consequently
\begin{align*}
    &\sum_{l=1}^{+\infty} (l+1) |h(z_1)| |h(z_2)| |s(z_1)||s(z_2)| \left|\sqrt{c} z_1 t(z_1)\tilde{t}(z_1)\right|^{l}  \left|\sqrt{c} z_2 t(z_2)\tilde{t}(z_2)\right|^{l} 
    \\
    &\qquad\qquad
    \leq  \frac{C (\Im z_1)^{k-4} (\Im z_2)^{k-4}}{\left(1 - (1-C_2(\Im z_1)^{4})(1-C_2 (\Im z_2)^{4})\right)^{2}},
\end{align*}
where $C_2 = \frac{C_1}{4}$. 
It is easy to check that if $x$ and $y$ belong to $[0,1]$, then, 
$x + y - xy \geq \frac{1}{2}(x+y)$. Using this inequality for 
$x = C_2(\Im z_1)^{4}$ and $y = C_2(\Im z_2)^{4}$, we obtain that 
\begin{align*}
    & \sup_{z_1,z_2 \in \Dcal} \sum_{l=1}^{+\infty} (l+1) |h(z_1)| |h(z_2)| |f(z_1)||f(z_2)| \left|\sqrt{c} z_1 t(z_1)\tilde{t}(z_1)\right|^{l}  \left|\sqrt{c} z_2 t(z_2)\tilde{t}(z_2)\right|^{l} 
    \\
    &\qquad\qquad
    \leq  \frac{C (\Im z_1)^{k-4} (\Im z_2)^{k-4}}{\left((\Im z_1)^{4} + (\Im z_2)^{4}\right)^{2}} 
    \\
    &\qquad\qquad
    < C (\Im z_1)^{k-8} (\Im z_2)^{k-8},
\end{align*}
a function that is integrable on $\Dcal$ as soon as $k \geq 8$. This justifies that 
the integral in \eqref{eq:sigma2} can be evaluated by exchanging the above series 
and the integral. Therefore, $\sigma^{2}(f)$  is given by 
\begin{align*}
    \sigma^{2} = \frac{1}{4 \pi^{2} c^{2}} \sum_{l=1}^{\infty} (u_l^{2} + u_l^{*2} + 2 |u_l|^{2}),
\end{align*}
with
\begin{align*}
    u_l = 
    (l+1) \int_{\Dcal} h(z) s(z) \left(\sqrt{c} z t(z) \tilde{t}(z)\right)^{l} \, \drm x \drm y.
\end{align*}
Since $u_l^{2} + u_l^{*2} + 2 |u_l|^{2} \geq 0$ with equality iff $u_l^2 \in \Rbb^-$, we have $\sigma^2 \geq 0$ with equality iff $\Re(u_l) = 0$  for all $l \geq 1$. Next, we notice that the function
\begin{align*}
    z \mapsto s(z) \left(\sqrt{c} z t(z) \tilde{t}(z)\right)^l
\end{align*}
is the Stieltjes transform of a distribution $D_l$ (see Lemma 9.2 in \cite{loubaton-jotp-2016}) carried by the interval $[(1-\sqrt{c})^2,(1+\sqrt{c})^2]$. Therefore, the Helffer-Sjöstrand formula leads to 
\begin{align*}
    \mathrm{Re}(u_l) = \pi (l+1) <D_l, f>.
\end{align*}
Therefore, $\sigma^2 > 0$ holds if there exists $l \geq 1$ such that $\Re(u_l) \neq 0$, a condition equivalent to $<D_l, f> \neq 0$. 

\section{Proof of Proposition \ref{prop:master_eq}}
\label{sec:proof-prop:master_eq}

We only provide the main steps of the proof of Proposition \ref{prop:master_eq} as the computations, which are mostly based on the repeated use of Proposition \ref{prop:nash-poincare-ipp-iid}, are standard (see e.g. \cite{mestre-vallet-ieeeit-2017}).
For the remainder, we use the generic notation $\epsilon(z_1,z_2)$  for any continuous function (depending on $N$) defined on $\Cbb\backslash\Rbb \times \Cbb\backslash \Rbb$, and such that
\begin{align}
  \left|\int_{\Cbb^+} \int_{\Cbb^+} h(z_1) h(z_2) \epsilon(z_1,z_2) \drm z_1 \drm z_2 \right| 
  \leq \Ebb\Biggl[Z \Bigl(|\phi(W)| + |\phi'(W)|\Bigr)\Biggr],
  \label{eq:def_epsilon}
\end{align}
where $Z$ is a positive random variable sharing the same properties as $Y$ in the statement of Proposition \ref{prop:master_eq}. We also use the notation $\epsilon(z)$ if the function only depends on one variable. Note also that the precise value of the function $\epsilon$ is irrelevant, and that it may take different values from one line to another.

In the following, we also make use of the following result compiling various classical and useful approximations, which we provide without proof (see e.g. again \cite{mestre-vallet-ieeeit-2017} for similar results in a different model): if $\beta(z) = \Ebb\left[\frac{1}{B+1} \tr \Q_{\bfrak}(z)\right]$, then we have the following lemma.
\begin{lemma}
    \label{lemma:formula_beta}
    The following holds
    \begin{align*}
        \beta(z) &= t(z) + \frac{\epsilon(z)}{B},
    \\
    \frac{1 + \beta(z)}{1-z(1+\beta(z))} &= t(z) + \frac{\epsilon(z)}{B},
    \\
    \frac{1}{1- z(1+\beta(z))  - \frac{\frac{M}{B+1}}{1-z(1+\beta(z))}}
    &= t(z) + z t'(z) + \frac{\epsilon(z)}{B},
     \\
     \frac{1}{1-z(1+\beta(z))} &= \frac{t(z)}{1+ct(z)} + \frac{\epsilon(z)}{B}.
    \end{align*}
\end{lemma}

\subsection{\texorpdfstring{Expansion of $\Ebb\left[\int_{\Cbb^+} h(z_1) \tr \left(\Q_{\bfrak}(z_1)\D_{\x}\right)^\circ \drm z_1 \phi(W) \right]$}{Some expansion}}

\paragraph{2nd order expansion.} After a first series of computations using Proposition \ref{prop:nash-poincare-ipp-iid} eq. \eqref{eq:ipp-iid}, we have
\begin{align}
 &\Ebb\left[\tr \Q_{\bfrak}(z_1) \D_{\x}\phi(W) \right]
   =
 \notag \\
 & 
  - \Ebb\left[\frac{1}{B+1} \tr \Q_{\bfrak}(z_1) \odot\left(\Q_{\bfrak}(z_1) \frac{\X\X^*}{B+1}\right) \phi(W)\right]
  \notag\\
  & + \Ebb\left[\tr \Q_{\bfrak}(z_1) \phi(W)\right] 
\notag\\
  & + 
  \frac{1}{2\sigma} \int_{\Cbb^+}
  \Bigl(h(z_2) \vartheta(z_1,z_2) +  \overline{h(z_2)} \vartheta(z_1,\bar{z_2})\Bigr) \drm z_2 ,
    \label{eq:trQD}
\end{align}
with
\begin{align*}
  \vartheta(z_1,z_2) =
  \frac{1}{B+1} \sum_{m,j}
  \Ebb
  \Bigl[
  &[\Q(z_1)]_{m,m} \overline{X_{m,j}} \ 
  \overline{\partial}_{m,j} \Bigl\{\tr \left(\Q_{\bfrak}(z_2)\D_{\x}\right)^\circ 
  \\
  &+ z_1 \ \tr \left(\Q_{\bfrak}(z_2)^2(\D_{\x}-\I)\right)^\circ\Bigr\} \phi'(W)\Bigr],
\end{align*}
where $\overline{\partial}_{m,j}$ denotes the operator $\frac{\partial}{\partial \overline{X_{m,j}}}$.
Using again Proposition \ref{prop:nash-poincare-ipp-iid} eq. \eqref{eq:ipp-iid}, we also have the equality
\begin{align*}
  \Ebb\left[\tr \Q_{\bfrak}(z) \D_{\x}\right] = \Ebb\left[\tr \Q_{\bfrak}(z)\right] - \Ebb\left[\frac{1}{B+1} \tr \Q_{\bfrak}(z) \odot \left(\Q_{\bfrak}(z) \frac{\X\X^*}{B+1}\right)\right],
\end{align*}
which further provides
\begin{align}
 &\Ebb\left[\tr \left(\Q_{\bfrak}(z_1) \D_{\x}\right)^\circ \phi(W) \right]
   =
   \notag\\
  &\Ebb\left[\left(\tr \Q_{\bfrak}(z_1)\right)^\circ \phi(W)\right] 
  +\frac{1}{2\sigma} \int_{\Cbb^+}
  \Bigl(h(z_2) \vartheta(z_1,z_2) +  \overline{h(z_2)} \vartheta(z_1,\bar{z_2})\Bigr) \drm z_2 
    +  \Omega_1(z_1),
    \label{eq:trQDrond}
\end{align}
where
\begin{align*}
    \Omega_1(z_1) = 
    \Ebb\left[\left(\frac{1}{B+1} \tr \Q_{\bfrak}(z_1) \odot\left(\Q_{\bfrak}(z_1) \frac{\X\X^*}{B+1}\right)\right)^\circ \phi(W)\right].
\end{align*}
Using Proposition \ref{prop:nash-poincare-ipp-iid} eq. \eqref{eq:nash-poincare-iid}, we can show that
\begin{align*}
    \Vbb\left[\frac{1}{B+1} \tr \Q_{\bfrak}(z_1) \odot\left(\Q_{\bfrak}(z_1) \frac{\X\X^*}{B+1}\right)\right] \leq \frac{1}{B^2}  P_1\left(\frac{1}{|\Im(z_1)|}\right) P_2\left(|z_1|\right) ,
\end{align*}
with $P_1,P_2$ two polynomials with positive coefficients independent of $N$, so that
\begin{align*}
    \Omega_1(z_2) = \epsilon(z_1),
\end{align*}
where we recall that $\epsilon(z_1)$ is a generic notation defined in \eqref{eq:def_epsilon}.

Expanding in the same way $\Ebb\left[\tr \Q_{\bfrak}(z_1) \phi(W)\right]$, we get
\begin{align}
 & \Ebb\left[\tr \Q_{\bfrak}(z_1)  \phi(W)\right] =
\notag\\
  &  \frac{M (1+\beta(z_1))}{1- z_1(1+\beta(z_1))}  \Ebb[\phi(W)] 
  \notag\\
  & + \frac{1}{2 \sigma}  \frac{1}{1- z_1(1+\beta(z_1))} 
  \int_{\Cbb^+} \Bigl(h(z_2) \tilde{\vartheta}(z_1,z_2) +  h(\bar{z_2}) \tilde{\vartheta}(z_1,\bar{z_2})\Bigr) \drm z_2 
    \notag\\
    & +  \frac{1}{1 - z_1(1+ \beta(z_1))} \Ebb\left[\frac{1}{B+1} \tr \Q_{\bfrak}(z_1)^\circ \, \Tr\left(\Q_{\bfrak}(z_1) \frac{\X_{\bfrak}\X_{\bfrak}^*}{B+1}\right) \phi(W)\right] ,
  \notag\\
    \label{eq:trQ}
\end{align}
with
\begin{align*}
  \tilde{\vartheta}(z_1,z_2) =
  \frac{1}{B+1} \sum_{m,n,k}
  \Ebb
  &\Bigl[
  [(z_1)]_{m,n} \overline{X_{m,k}} \ 
  \overline{\partial}_{n,k}\Bigl\{\tr \left(\Q_{\bfrak}(z_2)\D_{\x}\right)^\circ 
  \\
  &+ z_1 \ \tr \left(\Q_{\bfrak}(z_2)^2(\D_{\x}-\I)\right)^\circ\Bigr\} \phi'(W)\Bigr].
\end{align*}
Moreover,
\begin{align*}
   &\Ebb\left[ \frac{1}{B+1} \tr \Q_{\bfrak}(z_1)^\circ \, \Tr\left(\Q_{\bfrak}(z_1) \frac{\X_{\bfrak}\X_{\bfrak}^*}{B+1}\right) \phi(W)\right]
  =
  \notag\\
  &\qquad\qquad \frac{M}{B+1}\frac{1}{1-z_1(1+\beta(z_1))} \Ebb[\Tr \Q_{\bfrak}(z_1)^\circ \phi(W)] +  \Omega_2(z_1) + \Omega_3(z_1),
\end{align*}
with
\begin{align*}
    &\Omega_2(z_1) =
    \frac{1}{2\sigma} \frac{\beta(z_1)}{1+\beta(z_1)}
    \int_{\Cbb^+} \left(h(z_2) \omega_1(z_1,z_2) + \overline{h(z_2)} \omega_1(z_1,\overline{z_2})\right) \drm z_2,
    \notag
\end{align*}
where
\begin{align*}
  &\omega_1(z_1,z_2) =
    \\
  &\Ebb\left[\frac{1}{B+1} \tr \Q_{\bfrak}(z_1) \Q_{\bfrak}(z_2) \D_{\x} \Q_{\bfrak}(z_2) \frac{\X_{\bfrak}\X_{\bfrak}^*}{B+1} \frac{1}{B+1} \Tr \Q_{\bfrak}(z_1)^\circ \phi'(W)\right]
  \notag\\
   & - \Ebb\left[\frac{1}{B+1} \tr \left(\Q_{\bfrak}(z_1)\frac{\X_{\bfrak}\X_{\bfrak}^*}{B+1}\right) \odot \Q(z_2) \frac{1}{B+1} \Tr \Q_{\bfrak}(z_1)^\circ \phi'(W)\right]
  \\
  & - z_2 \Ebb\left[\frac{1}{B+1} \tr \left(\Q_{\bfrak}(z_1) \frac{\X_{\bfrak}\X_{\bfrak}^*}{B+1}\right) \odot \Q_{\bfrak}(z_2)^2 \frac{1}{B+1} \Tr \Q_{\bfrak}(z_1)^\circ \phi'(W)\right]
  \\
  & +z_2 \Ebb\left[\frac{1}{B+1} \tr \Q_{\bfrak}(z_1)\Q_{\bfrak}(z_2)^2 (\hat{\D} - \I) \Q_{\bfrak}(z_2) \frac{\X_{\bfrak}\X_{\bfrak}^*}{B+1} \frac{1}{B+1} \Tr \Q_{\bfrak}(z_1)^\circ \phi'(W)\right]
  \\
  &+ z_2 \Ebb\left[\frac{1}{B+1} \tr \Q_{\bfrak}(z_1)\Q_{\bfrak}(z_2) (\hat{\D} - \I) \Q_{\bfrak}(z_2)^2 \frac{\X_{\bfrak}\X_{\bfrak}^*}{B+1} \frac{1}{B+1} \Tr \Q_{\bfrak}(z_1)^\circ \phi'(W)\right],
\end{align*}
and
\begin{align*}
    &\Omega_3(z_1) =
    \notag\\
    &\frac{1}{B+1} \left(\frac{1}{1-z_1(1 + \beta(z_1))} - \frac{1}{1+\beta(z_1)}\right) \Biggl(\Ebb\left[\frac{1}{B+1} \tr \Q_{\bfrak}(z_1)^3 \frac{\X_{\bfrak} \X_{\bfrak}^*}{B+1} \phi(W)\right]
     \notag\\
     &+\Ebb\left[\frac{1}{B+1} \tr \Q_{\bfrak}(z_1) \frac{\X_{\bfrak} \X_{\bfrak}^*}{B+1} \left(\tr \Q_{\bfrak}(z_1)^\circ\right)^2 \phi(W)\right]\Biggr).
    \notag
\end{align*}
Using again Proposition \ref{prop:nash-poincare-ipp-iid} eq. \eqref{eq:nash-poincare-iid}, we can show that
\begin{align*}
    \Omega_2(z_1)+\Omega_3(z_1) = \epsilon(z_1),
\end{align*}
Using Lemma \ref{lemma:formula_beta}, we have
\begin{align*}
     \Ebb\left[\Tr \Q_{\bfrak}(z_1)\right] - M \frac{1+\beta(z)}{1-z_1\left(1+\beta(z_1)\right)}
     = \epsilon(z),
\end{align*}
so that going back to \eqref{eq:trQ}, we get
\begin{align*}
 & \Ebb\left[\Tr \Q_{\bfrak}(z_1)^\circ \drm z_1 \phi(W)\right]
   =
   \\
  &\frac{1}{2\sigma} \frac{1}{1- z_1(1+\beta(z_1))} \int_{\Cbb^+} \Bigl(h(z_2) \tilde{\vartheta}(z_1,z_2,u) 
  +  \overline{h(z_2)} \tilde{\vartheta}(z_1,\bar{z_2},u)\Bigr) \drm z_2
  \notag\\
  &+ \frac{M}{B+1} \left(\frac{1}{1-z_1(1+\beta(z_1))}\right)^2 \Ebb[\Tr \Q_{\bfrak}(z_1)^\circ \phi(W)]
  \notag\\
    &+ \epsilon(z_1).
\end{align*}
Factorizing again, we get
\begin{align*}
 & \Ebb\left[\Tr \Q_{\bfrak}(z_1)^\circ \drm z_1 \phi(W)\right]
   =
   \\
  &\frac{1}{2\sigma} \frac{\int_{\Cbb^+} \Bigl(h(z_2) \tilde{\vartheta}(z_1,z_2) +  \overline{h(z_2)} \tilde{\vartheta}(z_1,\bar{z_2})\Bigr) \drm z_2}{1- z_1(1+\beta(z_1))  - \frac{\frac{M}{B+1}}{1-z_1(1+\beta(z_1))}} 
 + \epsilon(z_1).
\end{align*}
Going back now to \eqref{eq:trQDrond}, we finally obtain
\begin{align}
 &\Ebb\left[\tr \left(\Q_{\bfrak}(z_1)\D_{\x}\right)^\circ  \phi(W)\right]
   =
 \\
  & 
  \frac{1}{2 \sigma} \int_{\Cbb^+}
  \Bigl(h(z_2) \vartheta(z_1,z_2) +  \overline{h(z_2)} \vartheta(z_1,\bar{z_2})\Bigr) \drm z_2
   \\
  & + \frac{1}{2 \sigma} 
  \frac{\int_{\Cbb^+} \Bigl(h(z_2) \tilde{\vartheta}(z_1,z_2) +  \overline{h(z_2)} \tilde{\vartheta}(z_1,\bar{z_2})\Bigr) \drm z_2}{1- z_1(1+\beta(z_1))  - \frac{\frac{M}{B+1}}{1-z_1(1+\beta(z_1))}}
    \\
  &    + \epsilon(z_1).
  \label{eq:trQrond2}
\end{align}

\paragraph{Computation of $\vartheta(z_1,z_2)$ and $\tilde{\vartheta}(z_1,z_2)$.} 

A direct computation of $\vartheta(z_1,z_2)$ provides 
\begin{align*}
  &\vartheta(z_1,z_2) =
    \\
  &  -\Ebb\left[\frac{1}{B+1} \Tr \Q_{\bfrak}(z_1) \odot \left(\Q_{\bfrak}(z_2)\D_{\x}\Q_{\bfrak}(z_2) \frac{\X_{\bfrak}\X_{\bfrak}^*}{B+1}\right) \phi'(W)\right]
                   \\
    &  + \Ebb\left[\frac{1}{B+1} \Tr \Q_{\bfrak}(z_1) \odot \left(\Q_{\bfrak}(z_2)\D_{\x}\right) \phi'(W)\right]
  \\
  & + z_2 \Ebb\left[\frac{1}{B+1} \Tr \Q_{\bfrak}(z_1) \odot \left(\Q_{\bfrak}(z_2)^2\D_{\x}\right) \phi'(W)\right]
  \\
  & - z_2 \Ebb\left[\frac{1}{B+1} \Tr \Q_{\bfrak}(z_1) \odot \left(\Q_{\bfrak}(z_2)^2(\D_{\x}-\I)\Q_{\bfrak}(z_2) \frac{\X_{\bfrak}\X_{\bfrak}^*}{B+1}\right) \phi'(W)\right]
  \\
  & - z_2 \Ebb\left[\frac{1}{B+1} \Tr \Q_{\bfrak}(z_1) \odot \left(\Q_{\bfrak}(z_2)(\D_{\x}-\I)\Q_{\bfrak}(z_2)^2 \frac{\X_{\bfrak}\X_{\bfrak}^*}{B+1}\right) \phi'(W)\right].
\end{align*}
Using again Proposition \ref{prop:nash-poincare-ipp-iid}, one can show that $\vartheta(z_1,z_2)$ vanishes in the sense that
\begin{align*}
  &\vartheta(z_1,z_2) = \epsilon(z_1,z_2).
\end{align*}
Regarding $\tilde{\vartheta}(z_1,z_2)$, we have
\begin{align*}
  &\tilde{\vartheta}(z_1,z_2)
    =
    \\
  &\Ebb\left[\frac{1}{B+1} \tr \Q_{\bfrak}(z_1) \Q_{\bfrak}(z_2) \D_{\x} \Q_{\bfrak}(z_2) \frac{\X_{\bfrak}\X_{\bfrak}^*}{B+1} \phi'(W)\right]
   \\
    & - \Ebb\left[\frac{1}{B+1} \tr \left(\Q_{\bfrak}(z_1)\frac{\X_{\bfrak}\X_{\bfrak}^*}{B+1}\right) \odot \Q_{\bfrak}(z_2) \phi'(W)\right]
  \\
  &- z_2 \Ebb\left[\frac{1}{B+1} \tr \left(\Q_{\bfrak}(z_1) \frac{\X_{\bfrak}\X_{\bfrak}^*}{B+1}\right) \odot \Q_{\bfrak}(z_2)^2 \phi'(W)\right]
  \\
  & + z_2 \Ebb\left[\frac{1}{B+1} \tr \Q_{\bfrak}(z_1)\Q_{\bfrak}(z_2)^2 (\D_{\x} - \I) \Q_{\bfrak}(z_2) \frac{\X_{\bfrak}\X_{\bfrak}^*}{B+1} \phi'(W)\right]
  \\
  & + z_2 \Ebb\left[\frac{1}{B+1} \tr \Q_{\bfrak}(z_1)\Q_{\bfrak}(z_2) (\D_{\x} - \I) \Q_{\bfrak}(z_2)^2 \frac{\X_{\bfrak}\X_{\bfrak}^*}{B+1} \phi'(W)\right].
\end{align*}
A first approximation provides 
\begin{align*}
  &\tilde{\vartheta}(z_1,z_2)
    =
    \\
  &\Ebb\left[\frac{1}{B+1} \tr \Q_{\bfrak}(z_1) \Q_{\bfrak}(z_2)^2 \frac{\X_{\bfrak}\X_{\bfrak}^*}{B+1}\right] \Ebb\left[\phi'(W)\right]
    \\
  & - \frac{ct(z_1)}{1+ct(z_1)} \left(t(z_2) + z_2 t'(z_2)\right) \Ebb\left[\phi'(W)\right]
  \\
  & + \epsilon(z_1,z_2).
\end{align*}
Moreover, standard computations based on Proposition \ref{prop:nash-poincare-ipp-iid} show that
\begin{align*}
  \Ebb\left[\frac{1}{B+1} \tr \Q_{\bfrak}(z_1) \Q_{\bfrak}(z_2)^2 \frac{\X_{\bfrak}\X_{\bfrak}^*}{B+1}\right]
  = 
  \frac{c t(z_1)\left(t(z_2) + z_2 t'(z_2)\right)}{(1+ct(z_1)) \Gamma\left(z_1,z_2\right)^2}
  + \epsilon(z_1,z_2),
\end{align*}
where
\begin{align*}
  \Gamma(z_1,z_2) = 1 - \frac{ct(z_1)t(z_2)}{(1+ct(z_1))(1+ct(z_2))}.
\end{align*}
Thus,
\begin{align*}
    \tilde{\vartheta}(z_1,z_2) = 
    \frac{c t(z_1)\left(t(z_2) + z_2 t'(z_2)\right)}{(1+ct(z_1))}
    \left(\frac{1}{\Gamma(z_1,z_2)^2} - 1\right)
    +
    \epsilon(z_1,z_2).
\end{align*}

\paragraph{Final form.} Going back to \eqref{eq:trQrond2} and given the fact that
\begin{align*}
    &\left|\frac{1}{1- z_1(1+\alpha(z_1))  - \frac{\frac{M}{B+1}}{1-z_1(1+\alpha(z_1))}}
    -
     \frac{1}{1-z_1(1 + c t(z_1))  - \frac{c}{1-z_1(1 + c t(z_1))}} \right|
     \\
     &\qquad\qquad
     \leq \frac{1}{B^2} \Prm_1(|z_1|) \Prm_2\left(\frac{1}{\Im(z_2)}\right) ,
\end{align*}
for some universal polynomials $\Prm_1,\Prm_2$, as well as the equalities
\begin{align*}
    \frac{\frac{t(z)}{1+c t(z)}}{1-z(1 + c t(z))  - \frac{c}{1-z(1 + c t(z))}}
    =
    \frac{t'(z)}{\left(1+ c t(z)\right)^2}
    =
    t(z) + z t'(z),
\end{align*}
we finally obtain
\begin{align*}
 &\Ebb\left[\tr \left(\Q_{\bfrak}(z_1)\D_{\x}\right)^\circ \phi(W)\right]
  =
 \\
  & 
    \frac{1}{2\sigma} \int_{\Cbb^+} h(z_2) \frac{c t'(z_1) t'(z_2)}{(1+ct(z_1))^2 (1+ct(z_2))^2}
    \left(\frac{1}{\Gamma(z_1,z_2)^2} - 1\right)\drm z_2 \Ebb\left[\phi'(W)\right]
  \\
  &+ \frac{1}{2 \sigma} \int_{\Cbb^+}  \overline{h(z_2)}
  \frac{c t'(z_1) t'(\bar{z_2})}{(1+c t(z_1))^2 (1+c t(\bar{z_2}))^2}
   \left(\frac{1}{\Gamma(z_1,z_2)^2} - 1\right)\drm z_2 \Ebb\left[\phi'(W)\right]
  \\
  &    + \epsilon(z_1).
\end{align*}
Consequently,
\begin{align}
\notag
    &\Ebb\left[\int_{\Cbb^+} h(z_1) \tr \left(\Q_{\bfrak}(z_1)\D_{\x}\right)^\circ \drm z_1 \phi(W) \right]
    =
    \\ \notag
    &\frac{1}{2\sigma} \int_{\Cbb^+}\int_{\Cbb^+}
    \frac{h(z_1) h(z_2) c t'(z_1) t'(z_2)}{(1+ct(z_1))^2 (1+ct(z_2))^2}
    \left(\frac{1}{\Gamma(z_1,z_2)^2} - 1\right)\drm z_2 \drm z_1  \Ebb\left[\phi'(W)\right]
  \\ \notag
  &+ \frac{1}{2 \sigma} \int_{\Cbb^+} \int_{\Cbb^+}  
 \frac{h(z_1)  \overline{h(z_2)} c t'(z_1) t'(\bar{z_2})}{(1+c t(z_1))^2 (1+c t(\bar{z_2}))^2}
   \left(\frac{1}{\Gamma(z_1,z_2)^2} - 1\right)\drm z_2 \drm z_1  \Ebb\left[\phi'(W)\right]
  \\
  &    + \Delta_1,
  \label{eq:master1}
\end{align}
where $\Delta_1$ shares the same properties as $\Delta$ in the same statement of Proposition \ref{prop:master_eq}.

    \subsection{\texorpdfstring{Expansion of $\Ebb\left[\int_{\Cbb^+} h(z_1) \tr\left(\Q_{\bfrak}(z_1)^2\left(\D_{\x} - \I\right)\right)^\circ \drm z_1 \phi(W) \right] $}{Some expansion}}
    
Using computations similar to the previous section (details are omitted), we find that
\begin{align*}
  &\Ebb\left[z_1 \left(\tr\Q_{\bfrak}(z_1)^2(\D_{\x} - \I)\right)^\circ \phi(W)\right]
  =
  \\
&\qquad\frac{1}{2 \sigma} \int_{\Cbb^+} \left(h(z_2) \vartheta(z_1,z_2) + \overline{h(z_2)} \vartheta(z_1,\overline{z_2}) \right) \drm z_2
  +
  \epsilon(z_1),
\end{align*}
where this time
\begin{align*}
  &\vartheta(z_1,z_2) =
\\
  &\qquad  - z_1\Ebb\left[\frac{1}{B+1} \tr \Q_{\bfrak}(z_1)^2 \odot \left(\Q_{\bfrak}(z_2)\D_{\x}\Q_{\bfrak}(z_2) \frac{\X\X^*}{B+1}\right) \phi'(W)\right]
\\    
  &\qquad  + z_1\Ebb\left[\frac{1}{B+1} \tr \left(\Q_{\bfrak}(z_1)^2\D_{\x}\right) \odot \Q_{\bfrak}(z_2) \phi'(W) \right]
  \\
  &\qquad+ z_1 z_2 \Ebb\left[\frac{1}{B+1} \tr \left(\Q_{\bfrak}(z_1)^2\D_{\x}\right) \odot \Q_{\bfrak}(z_2)^2 \phi'(W) \right]
  \\
  &\qquad- z_1 z_2 \Ebb\left[\frac{1}{B+1} \tr \Q_{\bfrak}(z_1)^2 \odot \left(\Q_{\bfrak}(z_2)^2(\D_{\x}-\I)\Q_{\bfrak}(z_2) \frac{\X\X^*}{B+1}\right) \phi'(W)\right]
  \\
  &\qquad- z_1 z_2 \Ebb\left[\frac{1}{B+1} \tr \Q_{\bfrak}(z_1)^2 \odot \left(\Q_{\bfrak}(z_2)(\D_{\x}-\I)\Q_{\bfrak}(z_2)^2 \frac{\X\X^*}{B+1}\right) \phi'(W)\right].
\end{align*}
We can show that
\begin{align*}
    \vartheta(z_1,z_2) = \epsilon(z_1,z_2),
\end{align*}
so that
\begin{align}
    \Ebb\left[\int_{\Cbb^+} h(z_1) \tr\left(\Q_{\bfrak}(z_1)^2\left(\D_{\x} - \I\right)\right)^\circ \drm z_1 \phi(W) \right]
    = \Delta_2,
    \label{eq:master2}
\end{align}
where $\Delta_2$ shares the same properties as $\Delta$ in Proposition \ref{prop:master_eq}.

    \subsection{Final equation}

Gathering \eqref{eq:master1} and \eqref{eq:master2}, we finally obtain
\begin{align}
    &\Ebb
    \left[
    \int_{\Cbb^+} h(z_1) 
    \left(
        \tr \left(\Q_{\bfrak}(z_1)\D_{\x}\right)^\circ 
        +
        z_1 \left(\tr\Q_{\bfrak}(z_1)^2(\D_{\x} - \I)\right)^\circ
    \right)
    \drm z_1 f(W) 
    \right]
    =
    \notag\\
    &\frac{1}{2\sigma} \int_{\Cbb^+}\int_{\Cbb^+}
    \left(h(z_1) h(z_2) \omega(z_1,z_2) + h(z_1)  \overline{h(z_2)} 
 \omega(z_1,\overline{z_2})\right) \drm z_2 \drm z_1  \Ebb\left[\phi'(W)\right]
  + \Delta_3,
  \label{eq:master3}
\end{align}
where $\Delta_3$ shares the same properties as $\Delta$ in Proposition \ref{prop:master_eq}, 
because $\omega(z_1,z_2)$ defined by 
(\ref{eq:def-thetaz1z2}) is easily seen 
to be given by the alternative expression
\begin{align*}
    \omega(z_1,z_2) = 
    \frac{c t'(z_1) t'(z_2)}{(1+ct(z_1))^2 (1+ct(z_2))^2}
    \left(\frac{1}{\Gamma(z_1,z_2)^2} - 1\right).
\end{align*}
This identity follows immediately from 
$$
t'(z) = \frac{t^{2}(z)}{1 - c (z t(z) \tilde{t}(z))^{2}}
$$
and 
$$
\Gamma(z_1,z_2) = 1 - c (z_1 t(z_1) \tilde{t}(z_1))^{2} (z_2 t(z_2) \tilde{t}(z_2))^{2} .
$$
From \eqref{eq:master3} and the definition of $\sigma^2$ in \eqref{eq:sigma2}, we easily deduce Proposition \ref{prop:master_eq}.

\section{Proof of (\ref{eq:domination-stochastique-zeta-tildeQ-tildeQb})}
\label{sec:proof-domination-stochastique-zeta-tildeQ-tildeQb}
We apply Lemma \ref{le:concentration-integrale-helffer-sjostrand} to the case $U^{(N)} = [0,1]$ and 
$\X_N(u) = \tilde{\X}_N(\nu)$, and $q_N(z, \X_N(u), \X_N(u)^{*})=  \frac{1}{M} \Tr \left(\tilde{\Q}_N(z,\nu) - \tilde{\Q}_{N,\mathfrak{b}}(z,\nu) \right) = \eta_N(z, \tilde{\X}(\nu), \tilde{\X}(\nu)^{*})$. For each $\delta > 0$, we consider the event $A_{N,\delta}$ defined by
\begin{align}
    \label{eq:def-event-proof-domination-stochastique-zeta-tildeQ-tildeQb}
A_{N,\delta}(\nu) = & \left\{ \| \frac{\X_{\mathfrak{b}}}{\sqrt{B+1}} \| \leq 3,  \| \frac{\X_{r}}{\sqrt{B+1}} \| \leq 3,  \|\frac{\Gammabs_{\mathfrak{b}}}{\sqrt{B+1}} \| \leq N^{\delta} \frac{B}{N} \right\} \ \cap \\ \notag & \left\{   \| \frac{\Gammabs^{1}_{r}}{\sqrt{B+1}} \| \leq N^{\delta} \frac{B}{N}, \| \frac{\Gammabs^{2}_{r}}{\sqrt{B+1}} \| \leq N^{\delta} \left(\frac{B}{N}\right)^{1/2}\right\}.
\end{align}
Proposition \ref{prop:localisation-eigenvalues-hatC-tildeC}, (\ref{eq:domination-stochastique-Gammab-Gamma1r}) and (\ref{eq:domination-stochastique-Gamma2r}) imply the existence of $\gamma > 0$
for which $\sup_{\nu} P(A_{N,\delta}(\nu)) \leq e^{-N^{\gamma}}$ for each $N$ large enough. 
In order to evaluate the gradient of $\eta$ w.r.t. $\tilde{\X}$, we express $\eta$ as 
$$
\eta = - \frac{1}{M} \Tr \tilde{\Q} \left( \tilde{\C} - \tilde{\C}_{\mathfrak{b}}\right) \tilde{\Q}_{\mathfrak{b}}.
$$
We use the representation (\ref{eq:expre-tildeC-bartlett}) in order to express $ \tilde{\C} - \tilde{\C}_{\mathfrak{b}}$ in terms of 
$\tilde{\X}$, and, after some tedious but straighforward calculations, we obtain that on $\| \frac{\X_{\mathfrak{b}}}{\sqrt{B+1}} \| \leq 3,  \| \frac{\X_{r}}{\sqrt{B+1}} \| \leq 3$, the inequality 
\begin{align}
    \label{eq:majorant-norme-carre-gradient-eta}
& ||\nabla_{\tilde{\X}}  \eta||^{2}  \leq 
\notag\\
&\frac{C(z)}{B^{2}} \left(    \| \frac{\Gammabs^{2}_r \Gammabs_r^{2*}}{B+1} \| +  \| \frac{\Gammabs^{1}_r \Gammabs_r^{1*}}{B+1} \| +    \| \frac{\Gammabs^{2}_r \Gammabs_r^{2*}}{B+1} \|^{2} +  \| \frac{\Gammabs^{1}_r \Gammabs_r^{1*}}{B+1} \|^{2} +   \| \frac{\Gammabs_{\mathfrak{b}} \Gammabs_{\mathfrak{b}}^{*}}{B+1} \|^{2}  \right) 
\notag\\
     & +\frac{C(z)}{NB} \Bigl(1 + \| \frac{\Gammabs_{\mathfrak{b}}}{\sqrt{B+1}} \| + \| \frac{\Gammabs_r^{1}}{\sqrt{B+1}} \| + \| \frac{\Gammabs_r^{2}}{\sqrt{B+1}} \| 
     \notag\\
     &\qquad+   \| \frac{\Gammabs^{2}_r \Gammabs_r^{2*}}{B+1} \| +  \| \frac{\Gammabs^{1}_r \Gammabs_r^{1*}}{B+1} \| +    \| \frac{\Gammabs_{\mathfrak{b}} \Gammabs_{\mathfrak{b}}^{*}}{B+1} \|  \Bigr),
\end{align}
holds. Therefore, on the set $A_{N,\delta}$, we have 
$$
||\nabla_{\tilde{\X}}  \eta||^{2}  \leq C(z) \frac{N^{2\delta}}{NB}.
$$
As $\tilde{\X}_N(A_{N,\delta})$ is convex and  the other conditions mentioned in Lemma \ref{le:concentration-integrale-helffer-sjostrand} are met, we deduce from Lemma \ref{le:concentration-integrale-helffer-sjostrand} that $\omega_{N}(f,\nu) = \Ocal_{\prec}\left( \frac{N^{\delta}}{\sqrt{BN}}\right)$. As this property holds for each $\delta > 0$, (\ref{eq:domination-stochastique-zeta-tildeQ-tildeQb}) is verified.

\section{Proof of Lemma \ref{le:hats-hatsb}}
\label{sec:proof-lemma-hatD-hatDb}
We recall that $\hat{\D} = \dg(\hat{s}_{m}, m=1, \ldots, M)$ and $\hat{\D}_{\mathfrak{b}} = \dg(\hat{s}_{m,\mathfrak{b}}, m=1, \ldots, M)$. We express $\hat{s}_m$ and $\hat{s}_{m,\mathfrak{b}}$ as $\hat{s}_m = \frac{\| \omegabs_m \|^{2}}{B+1}$, $\hat{s}_{m,\mathfrak{b}} = \frac{\| \omegabs_{m,\mathfrak{b}} \|^{2}}{B+1}$. Using the representations (\ref{eq:representation-omegam-bartlett}) and (\ref{eq:representation-omegaw}), we obtain that
 $$
 \hat{s}_{m} - \hat{s}_{m,\mathfrak{b}} = s_m \left( \tilde{\x}_m \frac{\tilde{\Phibs}_m}{B+1} \tilde{\x}_m^{*} \right),
 $$
 where $\tilde{\x}_m$ is the $\mathcal{N}_c(0, \I_{2(B+1)}$ distributed vector $\tilde{\x}_m = (\x_{m,\mathfrak{b}},\x_{m,r})$ and where $\tilde{\Phibs}_m$ is defined as the $2\times2$ block matrix with blocks given by 
\begin{align*}
 [\tilde{\Phibs}_m]_{1,1} &= \Psibs_{m,r}^{1} (I+\Psibs_{m,\mathfrak{b}})^{*} +  (I+\Psibs_{m,\mathfrak{b}})  (\Psibs_{m,r}^{1})^{*} + \Psibs_{m,r}^{1} (\Psibs_{m,r}^{1})^{*}
 \\
 [\tilde{\Phibs}_m]_{1,2} &= (I+\Psibs_{m,\mathfrak{b}} + \Psibs_{m,r}^{1})  (\Psibs_{m,r}^{2})^{*} 
 \\
 [\tilde{\Phibs}_m]_{2,1} &= \Psibs_{m,r}^{2}  (I+\Psibs_{m,\mathfrak{b}} + \Psibs_{m,r}^{1})^{*} 
 \\
 [\tilde{\Phibs}_m]_{2,2} &= \Psibs_{m,r}^{2}  (\Psibs_{m,r}^{2})^{*}.
 \end{align*}
 We first claim that $\mathbb{E}(\hat{s}_m - \hat{s}_{m,\mathfrak{b}}) =  \frac{s_m}{2} \,  \frac{1}{B+1} \mathrm{Tr} \tilde{\Phibs}_m = 
 \mathcal{O}(N^{-1})$, a property which immediately implies (\ref{eq:E-hats-hatsb}) and 
  (\ref{eq:tr-hatD-hatDb}). To verify this, we first use (\ref{eq:prop-TracePsimr1}), (\ref{eq:prop-TracePsimr2})
, and (\ref{eq:norme-Psi1mr}). Moreover, the elements of the diagonal matrix $\Psibs_{m,\mathfrak{b}}$ are $\mathcal{O}\left(\frac{B}{N}\right)$ terms. Therefore, (\ref{eq:entries-Psimr}) 
 leads to $\frac{1}{B+1} \mathrm{Tr} \Psibs^{1}_{m,r} \Psibs_{m,\mathfrak{b}}^{*} = \mathcal{O}\left(\frac{B}{N^{2}}\right) = o(N^{-1})$. In order to evaluate 
 $\hat{s}_m - \hat{s}_{m,\mathfrak{b}}  - \mathbb{E}(\hat{s}_m - \hat{s}_{m,\mathfrak{b}})$, we remark that the Hanson-Wright inequality provides  $ |\hat{s}_m - \hat{s}_{m,\mathfrak{b}} - s_m \frac{1}{B+1} \mathrm{Tr} \tilde{\Phibs}_m| \prec \| \frac{\tilde{\Phibs}_m}{B+1} \|_{F}$. Using the properties of matrices $\Psibs_{m,\mathfrak{b}}, \Psibs^{1}_{m,r}$ and 
 $\Psibs^{2}_{m,r}$, a simple calculation then leads to 
$ \| \frac{\tilde{\Phibs}_m}{B+1} \|_{F} = \mathcal{O}\left( \frac{1}{\sqrt{NB}}\right)$. This completes 
the proof of (\ref{eq:hats-hatsb-rond}), and also implies (\ref{eq:E-hats-hatsb-square}). (\ref{eq:tr-hatD-hatDb-rond}) is a consequence of mutual independence of vectors $(\tilde{\x}_m)_{m=1, \ldots, M}$ and of the Hanson-Wright inequality.    
\section{Proof of Lemma \ref{le:support-mu1}}
\label{sec:proof-lemma-support-mu1}
In order to establish Lemma \ref{le:support-mu1}, we show that for $N$ (and thus $M$ and $B$) fixed, 
$\mu_N^{1}$ coincides with the limit of the sequence of empirical eigenvalue distributions $(\hat{\mu}_P)_{P \geq 1}$ of certain well chosen $P \times P$ matrices. We denote by $(\lambda_m)_{m=1, \ldots, M}$ the eigenvalues of matrix 
$\H_N(\nu)$ (we do not mention these eigenvalues depend on $N$ and $\nu$ because these two parameters are supposed to 
be fixed in the present analysis), and consider $\Lambdabs = \mathrm{diag}\left((\lambda_m)_{m=1, \ldots, M}\right)$. For each 
integer $P$, we express $P$  as $P = kM + p$, $k,p \in \mathbb{N}$, $0 \leq p \leq M-1$, and define $\Lambdabs_P$ as the 
$P \times P$ diagonal matrix given by 
$$
\Lambdabs_P = \left( \begin{array}{ccccc} \I_k \otimes \Lambdabs & 0 & 0 & \ldots & 0 \\
                                    0 & \lambda_1 & 0 & \ldots & 0 \\
                                    0 & 0 & \lambda_2 & \ldots & 0 \\
                                    \vdots & \ddots & \ddots & \ddots & \vdots \\
                                    0 & 0 & 0 & 0 & \lambda_p \end{array} \right),
$$
where $\otimes$ represents the Kronecker product of matrices. It is easy to check that the empirical eigenvalue distribution of $\Lambdabs_P$ converges when $P \rightarrow +\infty$ weakly towards
the empirical eigenvalue distribution $\mu_{\H_N} = \frac{1}{M} \sum_{m=1}^{M} \delta_{\lambda_m}$ of  $\H_N$. For each 
integer $P$, we define $Q = \lfloor \frac{P}{c} \rfloor$, and define a complex Gaussian $P \times Q$ random matrix $\X_P$ with i.i.d. $\Ncal_c(0,1)$ entries. Then, if we denote by $\Y_P$ the random matrix $\Y_P = \Lambdabs_P^{1/2} \frac{\X_P}{\sqrt{Q}}$, 
the empirical eigenvalue $\hat{\mu}_P$ of $\Y_P \Y_P^{*}$ converges weakly $\mu_N^{1}$ (see e.g. \cite{silverstein-jmva-1995}). The eigenvalues $(\lambda_m)_{m=1, \ldots, M}$
of $\H_N$ belong to $[0, \tilde{a}]$. Moreover, for each $P$ large enough, almost surely w.r.t. the probability distribution of matrices $\X_P$, $\left\| \frac{\X_P \X_P^{*}}{Q} \right \| \leq (1 + \sqrt{c} + \epsilon/2)^{2}$
holds for each $P$ large enough. Therefore, for the same values of $P$, we also have $\| \Y_P \Y_P^{*} \| \leq \tilde{a}(1 + \sqrt{c}+ \epsilon/2)^{2}$ . If we consider a continuous function $\phi(\lambda)$ verifying 
\begin{eqnarray*}
    & \phi(\lambda) = 0, \, \lambda \in [0, \tilde{a}(1 + \sqrt{c}+ \epsilon/2)^{2}] \\
    & \phi(\lambda) = 1, \, \lambda \geq \tilde{a}(1 + \sqrt{c}+ \epsilon)^{2},
\end{eqnarray*}
as well as $\phi(\lambda) \geq 0$ for each $\lambda \geq 0$, then, $\int \phi(\lambda) d\hat{\mu}_P(\lambda) = 0$ for each $P$ large enough. As the sequence $(\hat{\mu}_P)_{P \geq 1}$ 
converges almost surely towards $\mu_N^{1}$, we deduce that 
$$
\int \phi(\lambda) d\mu_N^{1}(\lambda) = \lim_{P \rightarrow +\infty} \int \phi(\lambda)  d\hat{\mu}_P(\lambda) = 0.
$$
This, in turn, shows that the support of $\mu_N^{1}$ is included into $[0, \tilde{a}(1 + \sqrt{c}+ \epsilon)^{2}]$.

\section{Proof of Proposition \ref{prop:minorant}}
\label{sec:proof-prop-power-A-random}
We denote by $\eta_N(\nu)$ the term defined by 
$$
\eta_N(\nu) =  \frac{1}{M} \sum_{i \neq j} \frac{|\a_i^{*} \D(\nu) \a_j|^{2}}{ \|\a_i| \|^{2}  \|\a_j \|^{2}} - \mathbb{E} \left(  \frac{1}{M} \sum_{i \neq j} \frac{|\a_i^{*} \D(\nu) \a_j|^{2}}{ (1+\sigma^{2})^{2}} \right) ,
$$
and establish that the family $(\eta_N(\nu))_{\nu \in [0,1]}$ verifies
\begin{equation}
    \label{eq:stochastic-domination-etaN}
    \eta_N(\nu) = \Ocal_{\prec}\left( \frac{1}{\sqrt{M}}\right),
\end{equation}
which, of course, will imply (\ref{eq:behaviour-Tr-H2}). For this, we express 
$\eta_N(\nu)$ as $\eta_N(\nu) = \eta_{1,N}(\nu) + \eta_{2,N}(\nu)$ where 
\begin{align*}
  &  \eta_{1,N}(\nu) =  \frac{1}{M} \sum_{i \neq j} |\a_i^{*} \D(\nu) \a_j|^{2}  \left( \frac{1}{\|\a_i| \|^{2}  \|\a_j \|^{2}} - \frac{1}{(1 + \sigma^{2})^{2}} \right),\\
  & \eta_{2,N}(\nu) = \frac{1}{(1 + \sigma^{2})^{2}} \left( 
 \frac{1}{M} \sum_{i \neq j} \left( |\a_i^{*} \D(\nu) \a_j|^{2} - \mathbb{E} \left(  |\a_i^{*} \D(\nu) \a_j|^{2} \right) \right)\right)    .
\end{align*}
We first evaluate $\eta_{1,N}(\nu)$ and prove that 
\begin{equation}
\label{eq:evaluation-eta-1}
\eta_{1,N}(\nu) = \Ocal_{\prec}\left( \frac{1}{\sqrt{M}}\right).
\end{equation}
We first remark that 
(\ref{eq:evaluation-norme-amN}) implies that the family 
$(\frac{1}{\|\a_i\|^{2} \|\a_j\|^{2}})_{1 \leq i,j \leq M}$
verifies $\frac{1}{\|\a_i\|^{2} \|\a_j\|^{2}} \prec 1$. Therefore, 
(\ref{eq:convergence-norme-colonnes-AN}) leads to 
$$
\frac{1}{\|\a_i| \|^{2}  \|\a_j \|^{2}} - \frac{1}{(1 + \sigma^{2})^{2}} = \Ocal_{\prec}\left(\frac{1}{\sqrt{M}}\right).
$$
In order to complete the proof of (\ref{eq:evaluation-eta-1}), it is sufficient 
to check that the family $\left(|\a_i^{*} \D(\nu) \a_j|^{2}\right)_{i \neq j, \nu \in [0,1]}$ verifies $|\a_i^{*} \D(\nu) \a_j|^{2} \prec \frac{1}{M}$. For this, 
we remark that for $i \neq j$, 
$$
\a_i^{*} \D(\nu) \a_j = \frac{\sigma}{\sqrt{M}} \left( s_i(\nu) \G_{i,j} +
s_j(\nu) \G_{i,j}^{*} + \frac{\sigma}{\sqrt{M}} \g_i^{*} \D(\nu) \g_j \right).
$$
It is clear that $ s_i(\nu) \G_{i,j} +
s_j(\nu) \G_{i,j}^{*} = \Ocal_{\prec}(1)$. Moreover, if $\g$ is the 
$M^{2}$--dimensional vector defined by $\g^{T} = (\g_1^{T}, \ldots, \g_M^{T})$, 
then $\g_i^{*} \D(\nu) \g_j$ can be written as 
$\g_i^{*} \D(\nu) \g_j = \g^{*} \left( \e_i \e_j^{*} \otimes \D(\nu) \right) \g$. The Hanson-Wright inequality thus implies that 
$$
\g_i^{*} \D(\nu) \g_j - \mathbb{E}(\g_i^{*} \D(\nu) \g_j) = \Ocal_{\prec}\left( \| 
 \e_i \e_j^{*} \otimes \D(\nu)\|_{F}\right) = \Ocal_{\prec}(\sqrt{M}).
$$
As $\mathbb{E}\g_i^{*} \D(\nu) \g_j = 0$ for $i \neq j$, we obtain that 
$\a_i^{*} \D(\nu) \a_j = \O_{\prec}\left( \frac{1}{\sqrt{M}}\right)$
and that $|\a_i^{*} \D(\nu) \a_j|^{2} \prec \frac{1}{M}$. 
(\ref{eq:evaluation-eta-1}) is thus a consequence of Lemma \ref{le:domination-moyenne}. 

The evaluation $\eta_{2,N}(\nu) = \Ocal_{\prec}\left( \frac{1}{\sqrt{M}}\right)$
follows from a direct application of Lemma \ref{le:conditional-concentration}
and of Remark \ref{re:evaluation-sigmaN}. This completes the proof 
of (\ref{eq:stochastic-domination-etaN}).

\section{Code and Data Availability}

To facilitate the reproducibility of the numerical experiments and simulation studies presented in this paper, all custom code developed for these analyses is publicly available. The repository contains the necessary scripts (primarily implemented in Python) for running the simulations, reproducing the figures, and executing the computational examples discussed in Section \ref{sec:simulations}. The repository can be accessed at: \url{https://github.com/alexisrosuel/lrv-test}.

\begin{acks}[Acknowledgments]
A. Rosuel thanks the authors of \cite{pan-gao-yang-jasa-2014} for their help related to the implementation of the PGY approach. 
\end{acks}

\bibliographystyle{imsart-number} 
\bibliography{main_ejs}       

\end{document}